\documentclass[english,letterpaper,12pt,leqno]{article}

\usepackage{hyperref}

\usepackage{tikz, tikz-3dplot}
\usetikzlibrary{decorations.markings}
\usepackage{fullpage}

\usepackage[utf8]{inputenc}
\usepackage{babel}

\usepackage{amsxtra}
\usepackage{amsmath}
\usepackage{amssymb}
\usepackage{amsfonts}
\usepackage{amsthm}

\usepackage[all,2cell]{xy}
\UseAllTwocells
\usepackage{mathrsfs}
\usepackage{enumitem}
\usepackage{subfigure}

\setlist[enumerate]{leftmargin=*,labelindent=.5pc}

\newtheorem{thm}[equation]{Theorem}
\newtheorem{cor}[equation]{Corollary}
\newtheorem{lem}[equation]{Lemma}
\newtheorem{prop}[equation]{Proposition}

\newtheoremstyle{example}{\topsep}{\topsep}%
     {}
     {}
     {\bfseries}
     {.}
     {2pt}
     {\thmname{#1}\thmnumber{ #2}\thmnote{ #3}}

   \theoremstyle{example}

   \newtheorem{defi}[equation]{Definition}
   \newtheorem{rem}[equation]{Remark}
   \newtheorem{rems}[equation]{Remarks}
   
   \newtheorem{exas}[equation]{Examples}
   \newtheorem{ex}[equation]{Example}

\newtheorem{princ}[equation]{Universality principle}

\newtheorem{remcom}[equation]{Remarks and Complements}
\newtheorem{remex}[equation]{Remarks and Examples}

\newtheorem{exa}[equation]{Example}

\def\a{{\alpha}}

\def\CC{\mathbb{C}}

\def\FF{\mathbb{F}}
\def\GG{\mathbb{G}}
\def\HH{\mathbb{H}}
\def\KK{\mathbb{K}}
\def\LL{\mathbb{L}}

\def\RR{\mathbb{R}}
\def\ZZ{\mathbb{Z}}
\def\QQ{\mathbb{Q}}
\def\TT{{\mathbb{T}}}

\def\fen{\mathfrak{f}}

\def\hen{\mathfrak{h}}

\def\Hen{\mathfrak{H}}
\def\Nen{\mathfrak{N}}
\def\Sen{\mathfrak{S}}
\def\Den{\mathfrak{D}}

\def\Ac{\mathcal{A}}
\def\Bc{\mathcal{B}}
\def\Kc{\mathcal{K}}
\def\Dc{\mathcal{D}}
\def\Ec{\mathcal{E}}
\def\Fc{\mathcal{F}}
\def\Gc{\mathcal{G}}

\def\Lc{\mathcal{L}}
\def\Mc{\mathcal{M}}

\def\Hc{\mathcal{H}}
\def\Oc{\mathcal{O}}
\def\Pc{\mathcal{P}}
\def\Qc{\mathcal{Q}}
\def\Rc{\mathcal{R}}
\def\Sc{\mathcal{S}}
\def\Tc{\mathcal{T}}
\def\Uc{\mathcal{U}}

\def\Vc{\mathcal{V}}
\def\Wc{\mathcal{W}}

\def\Cen{\mathfrak {C}}
\def\Een{\mathfrak{E}}
\def\Fen{\mathfrak {F}}

\def\Men{\mathfrak{M}}
\def\Sen{\mathfrak{S}}
\def\Wen{\mathfrak {W}}

\def\fb{\mathbf{f}}

\def\Ic{\mathcal{I}}

\def\Aff{\Ac ff}
\def\alg{\underline}
 
 \def\Aut{\operatorname{Aut}\nolimits}
 \def\Bun{{\Bc un}}
 \def\Coh{\mathcal{C}oh}
 \def\Coker{\operatorname{Coker}\nolimits}
 
 \def\dg{\on{dg}}
 \def\End{\operatorname{End}\nolimits}
  \def\Ext{\operatorname{Ext}\nolimits}

\def\Ho{\operatorname{Ho}}
 \def\Hom{\operatorname{Hom}}
 \def\Homc{\mathcal{H}om}
 \def\id{\operatorname{id}}
 
 \def\Id{\operatorname{Id}\nolimits}

 \def\inj{{\on{inj}}}
 \def\Ker{\operatorname{Ker}\nolimits}
 
 \def\Map{\operatorname{Map}}
 
 \def\EHom{{\mathcal Hom}}
 \def\map{\operatorname{map}}
 \def\Mat{\operatorname{Mat}\nolimits}
 \def\Mod{\operatorname{Mod}}
\def\Mor{\operatorname{Mor}\nolimits} 
\def\NC{\operatorname{NC}}
\def\Ob{\operatorname{Ob}\nolimits}
\def\on{\operatorname}
\def\Perf{\operatorname{Perf}\nolimits}
\def\pr{\operatorname{pr}\nolimits}
\def\pt{\operatorname{pt}\nolimits}

\def\rk{\operatorname{rk}\nolimits}

\def\SZ{\overline{\operatorname{Spec}(\ZZ)}}

\def\Vect{\operatorname{Vect}}

\def\3x3{\operatorname{3}\times\operatorname{3}}
\def\1{{\bf 1}}
\def\lra{\longrightarrow}
\def\lla{\longleftarrow}
\def\(({(\hskip -1mm (}
\def\)){)\hskip -1mm )}
\def\((({(\hskip -1mm (}
\def\))){)\hskip -1mm )}
\def\ltb {[\hskip -1mm [}
\def\rtb{]\hskip -1mm ]}
\def\llb{ (\hskip -1mm (}
\def\rlb{ )\hskip -1mm )}
\def\bb{{\bullet\bullet}}
\def\bbs{{\backslash\hskip -1mm \backslash}}
\def\be{\begin{equation}}
\def\ee{\end{equation}}
\def\ed{\end{document}}

\def\ind{\textstyle\varinjlim}
\def\pro{\textstyle\varprojlim}
\def\hoind{\operatorname{ho}\!\ind}
\def\hopro{\operatorname{ho}\!\pro}
\def\hoend{R\int}
\def\twopro{2\varprojlim{}}
 \def\k{\mathbf{k}}

\usepackage{MnSymbol}
\usepackage{euscript}
 
\newcommand{\disc}[1]{ {\prec\!{#1}\!\succ}}
\newcommand{\const}[1]{\underline{#1}}
\newcommand{\itimes}{\boxtimes}

\def\op{{\operatorname{op}}}
\def\DI{\Delta^{\I}}
\def\Dop{\Delta^{\!\op}}
\def\Dp{ {\Delta^{\!+}}}
\def\Dpop{{\Delta^{\!+}}^{\op}}

\def\Fun{\operatorname{Fun}}

\def\h{\operatorname{h}\!}
\def\ho{\operatorname{ho}\!}

\def\hopro{\ho \underleftarrow{\lim}{}}
\def\Hom{\operatorname{Hom}}
\def\hra{\hookrightarrow}

\def\O{{\mathcal O}}

\def\I{{\mathcal I}}
\def\inftyone{(\infty,1)}
\def\inftytwo{(\infty,2)}

\def\J{{\EuScript J}}
\def\Joy{{\mathcal J}}

\def\Kan{\operatorname{Kan}}
\def\h{\operatorname{h}\!}

\def\M{{\mathcal M}}
\def\Map{\operatorname{Map}}

\def\ob{ \operatorname{ob}}

\def\N{\operatorname{N}}
\def\NDop{\N(\Delta)^{\op}}
\def\ND{\N(\Delta)}

\def\PF{{P^\triangleright}}
\def\CF{{C_\triangleright}}
\def\PFp{{P^\triangleright_+}}
\def\PI{{P^\triangleleft}}
\def\CI{{C_\triangleleft}}
\def\PIp{{P^\triangleleft_+}}

\def\RMap{\operatorname{RMap}}

\def\s{{\mathscr S}}
\def\seg{{\mathscr S}}
\def\Seg{{\mathcal Seg}}
\def\Sp{{\mathbb S}}
\def\inftyS{{\mathcal S}}
\def\Cat{{\EuScript Cat}}
\def\pS{{\widetilde{\mathcal S}}}
\def\SX{ H(X)}
\def\HX{ H(X)}
\def\hHX{ \h \HX}
\def\pSX{ \widetilde{H}(X)}

\def\Gr{\Grp}
\def\sS{{\Sp}_{\Delta}}
\def\sC{{\mC}_{\Delta}}

 \def\SW{\mathcal {S}}
 
\def\T{{\mathcal T}}

\def\Top{ {\EuScript Top}}
\def\Set{ {\mathcal Set}}
\def\Set{ {\EuScript Set}}
\def\Hall{ {\mathcal H}}
\def\FC{\operatorname{FC}}
\def\yoneda{\Upsilon}

\def\A{ {\EuScript A}}
\def\P{{\EuScript P}}
\def\B{ {\EuScript B}}
\def\C{{\EuScript C}}
\def\D{{\EuScript D}}
\def\E{{\EuScript E}}
\def\M{{\EuScript M}}

\def\Cc{\C}

\def\J{ {\EuScript J}}

\def\capd{\sqcapdot}

\def\cop{\textstyle\coprod}

\def\Grp{\mathcal{G}r}

 \def\ch{\on{char}}
\def\Rep{{\mathcal Rep}}
\def\Hall{ {\mathcal H}}
\def\Heck{\on{Heck}}

\def\Ben{\mathfrak{B}}
\def\Rep{{\mathcal Rep}}
\def\Fset{{\mathcal {FS}et}}
\def\Act{{\mathcal Act}}
\def\Bld{\on{Bld}}
\def\<<{\langle {}\hskip -.1cm {}\langle}
\def\>>{\rangle \hskip -.1cm \rangle}

\def\one{ {\bf 1}}

\def\FC{\on{FC}}

\setlist[enumerate,1]{label=(\arabic{*})}
\setlist[enumerate,2]{label=(\alph{*})}
\setlist[enumerate,3]{label=(\roman{*})}

\def\s{{\EuScript S}}
\def\Set{{\EuScript Set}}
\def\sSet{{\Set}_{\Delta}}
\def\mSet{{\Set}_{\Delta}^{+}}
\def\Deltac{\Delta^{\!\displaystyle \amalg}}
\def\Mo{{\on{Mo}}}
\def\asd{{\on{asd}}}
\def\Span{\on{Span}^\prime}
\def\Spanl{\on{Span}}
\def\bSpan{\on{BiSpan}}
\def\HSpan{\on{HSpan}}
\def\HSpanp{\on{HSpan}^\prime}
\def\Cat{{\EuScript Cat}}
\def\Past{\wp}
\def\lastedge{\spadesuit}
\def\ND{\N(\Delta)}
\def\mC{ {\mathbf C} }

\newtheorem{warn}[equation]{Warning}

\def\dgalg{\operatorname{dgalg}_{\FF}}
\def\dgcat{\operatorname{dgcat}_{\FF}}
\def\NCI{{\NC^{\infty}}}
\def\mC{{\mathbf C}}
\def\mD{{\mathbf D}}
\def\mult{\mu}
\def\Fenh{\mathfrak{F}^h}
\def\Ndg{\N_{\on{dg}}}
\def\Vectk{\Vect_{\k}}
\def\kk{k}
\def\NCC{N^*}
\def\mM{{\mathbf M}}

\newlist{exaenumerate}{enumerate}{1}
\setlist[exaenumerate]{topsep=0pc,itemindent=2.5pc,leftmargin=0pc,label=(\alph{*}),parsep=0pc,itemsep=0pc}

\numberwithin{paragraph}{section}
\renewcommand\theparagraph{(\arabic{section}.\arabic{paragraph})}
\numberwithin{equation}{section}

\begin{document}

\title{Higher Segal spaces I}
 \author{T. Dyckerhoff\footnote{Department of Mathematics, Yale University,
 10 Hillhouse Avenue, New Haven CT 06520 USA, email:
 {\tt tobias.dyckerhoff@yale.edu, mikhail.kapranov@yale.edu}  }, M. Kapranov\footnotemark[1]
 }

\maketitle

 

\begin{abstract}
	This is the first paper in a series on new higher categorical structures called higher Segal
	spaces. For every $d \ge 1$, we introduce the notion of a $d$-Segal space which is a
	simplicial space satisfying locality conditions related to triangulations of cyclic
	polytopes of dimension $d$. In the case $d=1$, we recover Rezk's theory of Segal spaces. 
	The present paper focuses on $2$-Segal spaces. 
	The starting point of the theory is the observation that Hall algebras, as previously
	studied, are only the shadow of a much richer structure governed by a system of higher
	coherences captured in the datum of a $2$-Segal space. This $2$-Segal space is given by
	Waldhausen's S-construction, a simplicial space familiar in algebraic K-theory.  
	Other examples of $2$-Segal spaces arise naturally in classical topics such as Hecke algebras, 
	cyclic bar constructions, configuration spaces of flags, solutions of the pentagon equation, 
	 and mapping class groups.
\end{abstract}

\tableofcontents

\vfill\eject

\section*{Introduction}
\addcontentsline{toc}{section}{Introduction}

The theory of Segal spaces, as introduced by C. Rezk \cite{rezk}, has its roots in the classical
work of G. Segal \cite{segal.categories} where the notion of a $\Gamma$-space is introduced and used
to exhibit various classifying spaces as infinite loop spaces. Rezk's work analyzes the role of
Segal spaces as a model for the homotopy theory of $\inftyone$-categories. The concept of a Segal
space can be motivated as follows. Given a simplicial set $X$, we have, for each $n \ge 1$, a
natural map
\begin{equation}\label{eq:fn}
	f_n: X_n \lra X_1\times_{X_0} X_1 \times_{X_0}  \cdots \times_{X_0} X_1
\end{equation}
where the right-hand side is an $n$-fold fiber product. The condition that all maps $f_n$ be
bijective is called the  {\em Segal condition} and a simplicial set which satisfies this condition is
called {\em Segal}. The relevance of this condition comes from the fact that it characterizes the
essential image of the fully faithful functor
\[
  	\N : \Cat \to \sSet
\]
which takes a small category to its nerve. Given a Segal simplicial set $X$, we
can recover the corresponding category $\C$: The set of objects is formed by the vertices of
$X$ and morphisms between a pair of objects are given by edges in $X$ between the corresponding 
pair of vertices.
The invertibility of $f_2$ allows us to interpret the diagram
\begin{equation}\label{eq:fund-corr-intro}
\xymatrix{
\mu: \big\{ X_1\times_{X_0} X_1&&\ar[ll]_-{f_2=(\partial_2, \partial_0)} X_2 \ar[r]^{\partial_1}& 
 X_1 \big\}
 }
\end{equation}
as a composition law for $\C$, while the bijectivity of $f_3$ implies the associativity of this law. 
One can view the theory of Segal (simplicial) spaces as a development of this idea in a homotopy theoretic
framework, where simplicial sets are replaced by simplicial spaces, fiber products by their homotopy
analogs, and bijections by weak equivalences. This leads to a weaker notion of coherent
associativity which can be used to describe composition laws in higher categories.
 
The goal of this paper, and the sequels to follow, is to study a ``higher'' extension of Rezk's
theory to what we call $d$-Segal spaces. These are simplicial spaces which are required to satisfy
analogs of the Segal conditions corresponding to triangulations of $d$-dimensional convex polytopes.
We outline the basic idea. Note that the fiber product in \eqref{eq:fn} can be viewed as the set
$\Hom(\J^n, X)$, where we define the simiplicial set
\begin{equation}\label{eq:segment}
	\J^n = \Delta^1 \amalg_{\Delta^0} \Delta^1 \amalg_{\Delta^0} \dots \amalg_{\Delta^0} \Delta^1
\end{equation}
whose geometric realization can be interpreted as an oriented interval, subdivided into $n$ subintervals. 
Further, the Segal map $f_n$ from
\eqref{eq:fn} is obtained by pulling back along the natural inclusion $\J^n \subset \Delta^n$.
This can be generalized as follows. Consider a convex polytope $P \subset \RR^d$ given as the convex
hull of a finite set of points $I \subset \RR^d$.
Choose a numbering of this set $I \cong \{0,1, ..., n\}$. Any triangulation $\T$
of $P$ with vertices in $I$ gives rise to a simplicial subset $\Delta^\T \subset \Delta^n$, and 
we obtain a natural pullback map
\[
	f_\T: X_n \lra X_\T 
\]
where we observe that $X_n \cong \Hom(\Delta^n, X)$ and define $X_{\T} := \Hom(\Delta^\T, X)$. 
For example, the triangulations of the square
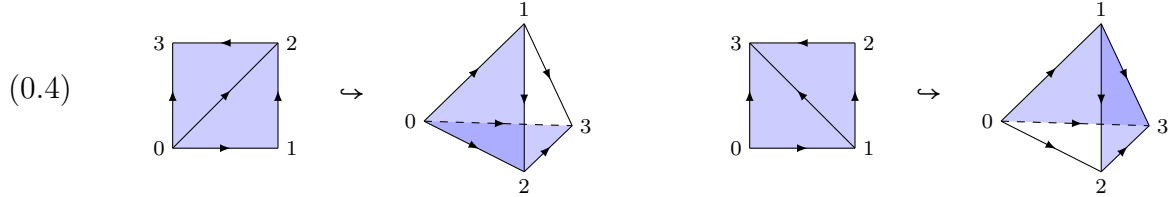
\begin{equation}\label{eq:square}
\tdplotsetmaincoords{100}{170}
\begin{tikzpicture}[>=latex,scale=1.0, baseline=(current  bounding  box.center)]
\begin{scope}[scale=1.4]

\coordinate (A0) at (0,0);
\coordinate (A1) at (0,1);
\coordinate (A2) at (1,1);
\coordinate (A3) at (1,0);

\path[fill opacity=0.4, fill=blue!50] (A0) -- (A3) -- (A2) -- (A1) -- cycle;

\begin{scope}[decoration={
    markings,
    mark=at position 0.55 with {\arrow{>}}}
    ] 
\draw[postaction={decorate}] (A0) -- (A3);
\draw[postaction={decorate}] (A3) -- (A2);
\draw[postaction={decorate}] (A0) -- (A1);
\draw[postaction={decorate}] (A2) -- (A1);
\draw[postaction={decorate}] (A0) -- (A2);
\end{scope}

{\scriptsize
\draw (A0) node[anchor=east] {$0$};
\draw (A1) node[anchor=east] {$3$};
\draw (A2) node[anchor=west] {$2$};
\draw (A3) node[anchor=west] {$1$};
}
\draw (1.7,0.5) node {$\hookrightarrow$};
\end{scope}

\begin{scope}[tdplot_main_coords, >=latex, xshift=4.5cm, yshift=0.5cm]

\coordinate (A0) at (1,-1,0);
\coordinate (A1) at (0,1,1);
\coordinate (A2) at (0,1,-1);
\coordinate (A3) at (-1,-1,0);


\path[fill opacity=0.4, fill=blue!50] (A0) -- (A2) -- (A3) -- cycle;
\path[fill opacity=0.4, fill=blue!50] (A0) -- (A2) -- (A1) -- cycle;

\begin{scope}[decoration={
    markings,
    mark=at position 0.55 with {\arrow{>}}}
    ] 
\draw[postaction={decorate}] (A0) -- (A1);
\draw[postaction={decorate}] (A0) -- (A2);
\draw[postaction={decorate}] (A1) -- (A2);
\draw[postaction={decorate}] (A1) -- (A3);
\draw[postaction={decorate}] (A2) -- (A3);
\draw[postaction={decorate},dashed] (A0) -- (A3);
\end{scope}

{\scriptsize
\draw (A0) node[anchor=east] {$0$};
\draw (A1) node[anchor=south] {$1$};
\draw (A2) node[anchor=north] {$2$};
\draw (A3) node[anchor=west] {$3$};
}
\end{scope}
\end{tikzpicture}
\quad\quad\quad\quad
\begin{tikzpicture}[>=latex,scale=1.0, baseline=(current  bounding  box.center)]
\begin{scope}[scale=1.4]

\coordinate (A0) at (0,0);
\coordinate (A1) at (0,1);
\coordinate (A2) at (1,1);
\coordinate (A3) at (1,0);

\path[fill opacity=0.4, fill=blue!50] (A0) -- (A3) -- (A2) -- (A1) -- cycle;

\begin{scope}[decoration={
    markings,
    mark=at position 0.55 with {\arrow{>}}}
    ] 
\draw[postaction={decorate}] (A0) -- (A3);
\draw[postaction={decorate}] (A3) -- (A2);
\draw[postaction={decorate}] (A0) -- (A1);
\draw[postaction={decorate}] (A2) -- (A1);
\draw[postaction={decorate}] (A3) -- (A1);
\end{scope}

{\scriptsize
\draw (A0) node[anchor=east] {$0$};
\draw (A1) node[anchor=east] {$3$};
\draw (A2) node[anchor=west] {$2$};
\draw (A3) node[anchor=west] {$1$};
}
\draw (1.7,0.5) node {$\hookrightarrow$};
\end{scope}

\begin{scope}[tdplot_main_coords, >=latex, xshift=4.5cm, yshift=0.5cm]

\coordinate (A0) at (1,-1,0);
\coordinate (A1) at (0,1,1);
\coordinate (A2) at (0,1,-1);
\coordinate (A3) at (-1,-1,0);


\path[fill opacity=0.4, fill=blue!50] (A1) -- (A2) -- (A3) -- cycle;
\path[fill opacity=0.4, fill=blue!50] (A0) -- (A1) -- (A3) -- cycle;

\begin{scope}[decoration={
    markings,
    mark=at position 0.55 with {\arrow{>}}}
    ] 
\draw[postaction={decorate}] (A0) -- (A1);
\draw[postaction={decorate}] (A0) -- (A2);
\draw[postaction={decorate}] (A1) -- (A2);
\draw[postaction={decorate}] (A1) -- (A3);
\draw[postaction={decorate}] (A2) -- (A3);
\draw[postaction={decorate},dashed] (A0) -- (A3);
\end{scope}

{\scriptsize
\draw (A0) node[anchor=east] {$0$};
\draw (A1) node[anchor=south] {$1$};
\draw (A2) node[anchor=north] {$2$};
\draw (A3) node[anchor=west] {$3$};
}
\end{scope},
\end{tikzpicture}
\end{equation}
induce two natural maps $X_3 \to X_2 \times_{X_1} X_2$. We call the elements of $X_\T$ {\em
membranes in $X$ of type $\T$}. Similarly, for a simplicial space $X$ we have a natural derived
version of the membrane space, denoted $RX_\T$, which comes equipped with a map $X_n \to RX_\T$.
The Segal map $f_n$ of \eqref{eq:fn} is recovered for $I=\{0,1,\dots,n\} \subset \RR^1$,
in which case  $P= [0,n]$ is an  interval, and  for $\T$ being the triangulation of $[0,n]$ by 
the segments $[i,i+1]$,
in which case $\Delta^\T = \J^n$.  From this perspective,  it is natural to refer to Rezk's notion of 
 Segal spaces,   as {\em $1$-Segal spaces}.

In the present paper, we study the $2$-dimensional theory corresponding to triangulations $\T$ of
convex plane polygons $P_n$, with $I$ being the set of vertices of $P_n$, numbered counterclockwise
(so $P_n$ has $n+1$ vertex). 
  A simplicial space $X$ is called {\em $2$-Segal space} if, for every convex polygon $P_n$ and every
triangulation $\T$, the resulting map $X_n \to RX_\T$ is a weak homotopy equivalence.  Note that, in
contrast to the $1$-dimensional situation, a given convex polygon $P_n$ has many triangulations $\T$,
each involving all the vertices. 
  For a $2$-Segal space $X$, each derived membrane space $RX_\T$ comes
equipped with a weak homotopy equivalence $X_n \to RX_{\T}$. In particular, all derived membrane
spaces corresponding to different triangulations of $P_n$ are weakly equivalent to one another.
Moreover, the $2$-Segal space $X$ exhibits the {\em independence of $RX_\T$ on $\T$ up to a
coherent system of weak equivalences}.

Remarkably, $2$-Segal spaces appear in several areas of current interest:
\begin{itemize}
\item 
	Various associative algebras obtained via correspondences such as Hall algebras, Hecke algebras,
	and various generalizations, appear as shadows of richer structures: $2$-Segal simplicial
	groupoids, stacks, etc. The invariance under change of triangulations of a square (cf.
	\eqref{eq:square}) is the property which is responsible for the associativity of these algebras.
	In this context, the most important example of a $2$-Segal space is given by the {\em Waldhausen
	S-construction} $\SW(\E)$ of an exact category $\E$, and $\infty$-categorical generalizations
	thereof. While the geometric realization of the simplicial space $\SW(\E)$ plays a fundamental
	role in algebraic K-theory, its structural property of being $2$-Segal seems to be a new
	observation and can be viewed as a kind of ``hidden $2$-dimensional symmetry" of classical
	homological algebra. The above mentioned associative algebras are obtained by applying suitable
	theories with transfer to various incarnations of $\SW(\E)$.
 
\item 
	The {\em cyclic nerve} \cite{drinfeld} of any category is a $2$-Segal set. More generally,
	the {\em cyclic bar construction} of an $\infty$-category is a $2$-Segal space. On the one
	hand, this class of examples leads to new associative algebras whose structure constants are
	given by counting certain factorizations. On the other hand, we obtain natural examples of
	$2$-Segal spaces which carry a cyclic structure in the sense of A.  Connes. A detailed study
	of cyclic $2$-Segal spaces will be deferred to the sequel of this work. We provide a more
	detailed outlook at the end of this introduction, explaining relations to mapping class
	groups and potential applications in the context of $2$-periodic derived categories.
 
\item 
	$2$-Segal spaces can be naturally interpreted in the context of model categories. We
	introduce a model category for $2$-Segal spaces which is, in a precise way, compatible with Rezk's
	model structure for $1$-Segal spaces. 
 
\item 
	In analogy to the role of $1$-Segal spaces in higher category theory, we provide several
	higher categorical interpretations of $2$-Segal spaces. A $1$-Segal space encodes a
	coherently associative composition law in which a given pair of composable morphisms admits
	a composition which is unique up to homotopy. More precisely, the space of all possible
	compositions of a fixed pair of composable morphisms is contractible. Informally, a
	$2$-Segal space describes a higher categorical structure in which a composable pair of
	morphisms may admit a multitude of possibly non-equivalent compositions. Nevertheless, the
	$2$-Segal maps provide a coherent notion of associativity among the composition spaces.
	Following a suggestion of J. Lurie, we will make this statement precise by associating to a
	$2$-Segal space an $\inftytwo$-category enriched in $\infty$-categories of presheaves.
	Various alternative structures of higher bicategorical nature can be associated to a
	$2$-Segal space such as a monad in the $\inftytwo$-category of bispans.

\item 
	$2$-Segal simplicial sets provide a combinatorial version of the Clebsch-Gordan formalism for
	semi-simple tensor categories. $2$-Segal simplicial {\em spaces} can be thought of as higher
	categorical generalizations of this formalism. In particular, we expect our theory to be
	relevant in the context of the Reshetikhin-Turaev-Viro tensor category formalism for 
	$3$-dimensional topological quantum field theories (cf. \cite{turaev} and references
	therein).

\item 
	Cluster coordinate systems on various versions of Teichm\"uller spaces, see \cite{fock-goncharov}, can
	be naturally explained in terms of certain $1$- and $2$-Segal spaces. In particular, set-theoretic solutions
	of the pentagon equation \cite{kashaev-sergeev, kashaev-reshitikhin} can be considered as very
	special types of $2$-Segal semi-simplicial sets. 

\end{itemize}
 
The theory of $2$-Segal spaces can be developed in different contexts and at different levels of
generality. In the first part of this paper (Chapters 1-3), we work in the more elementary context
of simplicial {\em topological spaces}, thus reducing to a minimum of background in homotopy theory
required from the reader. This part can be seen as an extended introduction to the rest of the
paper. 
In particular, the motivating example of the Waldhausen's S-construction is studied in Section
\ref{subsec:waldhausen-1}. We generalize Quillen's concept of an exact category to a non-additive
setting and call the resulting class of categories {\em proto-exact}. We show that the definition and
properties of the Waldhausen S-construction extend to this more general framework. The ``belian categories"
of Deitmar \cite{deitmar:belian} and categories of representations of quivers in pointed sets
studied by Szczesny \cite{szczesny:quivers} provide many examples of proto-exact categories.
Another important class of examples is given by various categories of Arakelov vector bundles, see
Example \ref{ex:arakelov}.  The role of the classical Waldhausen S-construction in algebraic K-theory
suggests that our construction should give a natural definition of K-groups in these more
general contexts.

Already the discrete case of $2$-Segal simplicial {\em sets}, requiring no homotopy theoretical
background at all, leads to an interesting theory presented in Chapter \ref{sec:discsegal}.  Such
structures axiomatize the idea of ``associative multivalued compositions". More precisely, for any
simplicial set $X$, we can consider the diagram in \eqref{eq:fund-corr-intro} above as a
correspondence (multivalued map) from $X_1\times_{X_0} X_1$ to $X_1$. The $2$-Segal condition can
then be regarded as the associativity of $\mu$ in the sense of composition of correspondences, the
only sense in which multivalued maps can be meaningfully composed.  Such an ``associative
correspondence" induces an associative multiplication in the usual sense on the linear envelope of
$X_1$, thus giving rise to a linear category $\Hall(X)$ which we call the {\em Hall category} of
$X$, see Section \ref{subsec:hall-algebra-discrete}. In Theorem \ref{thm:segal-2-categories}, we
show how to categorify the Hall category construction one more time so as not to lose any
information and to identify $2$-Segal sets with certain bicategories. As an alternative perspective,
we give an interpretation in terms of operads in Section \ref{subsec:operadic}. We give examples of
discrete $2$-Segal spaces relating to Bruhat-Tits complexes, set-theoretic solutions of the pentagon
equation, and pseudo-holomorphic polygons and conclude with Section \ref{subsec:bir-segal} on
examples of birational Segal schemes.

In the main body of the paper, we work in the general context of combinatorial model categories
which we recall in Chapter \ref{sec:model-cat}.  In particular, we understand spaces
combinatorially, as simplicial sets. In Chapter \ref{sec:2-segal-model}, we construct a model
structure $\seg_2$ on the category $\Sp_\Delta$ of simplicial spaces whose fibrant objects are
exactly the Reedy fibrant $2$-Segal spaces. More precisely, denoting by $\I$ the Reedy model
structure on $\Sp_\Delta$, we construct a chain of left Bousfield localizations
\begin{equation}\label{eq:chain-of-loc-intro}
(\Sp_\Delta, \I) \longrightarrow
(\Sp_\Delta, \seg_2) \longrightarrow
(\Sp_\Delta, \seg_1), 
\end{equation}
where $\seg_1$ is the model structure for $1$-Segal spaces constructed by Rezk \cite{rezk}. The
precise statement will be given in Theorem \ref{thm:S-d-localization} and depends crucially on the
fact that every $1$-Segal space is $2$-Segal. In particular, we obtain a construction of the
``$2$-Segal envelope" of any simplicial space $X$ as the fibrant replacement of $X$ with respect to
$\seg_2$.

In Chapter \ref{sec:pathspace}, we introduce the {\em path space criterion} which characterizes $2$-Segal
spaces in terms of $1$-Segal conditions: A simplicial space $X$ is a $2$-Segal space if and only if its associated simplicial
path spaces $\PI X$ and $\PF X$ are $1$-Segal spaces. 
In the context of $2$-Segal semi-simplicial sets, this criterion provides a natural explanation of
the following remarkable (but originally mysterious) observation of Kashaev and Sergeev
\cite{kashaev-sergeev}: if $C$ is a set and 
\[
s: C\times C \lra C\times C
\]
is a bijection satisfying the pentagon equation \eqref{eq:pentagon-eq-set}, then the first compoment
of $s$, considered as a binary operation $C\times C\to C$, is associative. 
   
In Chapter \ref{sec:2-segal-from-higher}, the path space criterion is essential to verify $2$-Segal
conditions in the context of $\infty$-categories: We use it to show that the Waldhausen
S-construction of an exact $\infty$-category is a $2$-Segal space. Here, we define the new concept
of an exact $\infty$-category as a non-linear higher generalization of Quillen's notion of an exact
category. For example, stable $\infty$-categories are examples of exact $\infty$-categories, hence
our result covers pre-triangulated dg categories and various categories appearing in stable homotopy
theory.  As another application of the path space criterion, we define the cyclic bar construction
of any $\infty$-category and show that it is a $2$-Segal space.

$2$-Segal spaces underlie practically all associative algebras ``formed by correspondences". In
Chapter \ref{sec:hall-alg}, we explain a general procedure of forming such algebras. The input is,
on one hand, a $2$-Segal simplicial object $X$ of a model category $\mC$ and on the other hand, a
{\em theory with transfer} $\hen$ on $\mC$. The latter is a functor compatible with products,
covariant under one class of morphisms and contravariant with respect to another class, satisfying
natural axioms. We then use the diagram $\mu$ above to produce a genuinely associative map 
\begin{equation}\label{eq:multiplication}
      m =  \partial_{1*} (\partial_2,\partial_0)^*:   
      \hen(X_1)\otimes\hen(X_1) \lra \hen(X_1),
\end{equation}
defining an algebra $\Hall(X,\hen)$ which we call the Hall algebra with coefficients in $\hen$.
Taking for $X$ various incarnations of the Waldhausen S-construction, we recover `` classical" Hall
algebras \cite{schiffmann:hall} ($\mC$ is the category of groupoids, $\hen$ is the space of
functions), derived Hall algebras of To\"en \cite{toen-derived} ($\mC$ is the category of spaces,
$\hen$ is the space of locally constant functions), motivic Hall algebras of Joyce \cite{joyce-II}
and Kontsevich-Soibelman \cite{KS-motivic} ($\mC$ is the category of stacks, $\hen$ is given by
motivic functions), etc.  Further, we observe that Hecke algebras arise via a theory with transfer
from a simplicial groupoid which we call the Hecke-Waldhausen space studied in Section
\ref{subsec:hecke-waldhausen}.

Given a $2$-Segal space $X$ and a suitable theory with transfer, the $2$-Segal conditions
corresponding to the triangulations \eqref{eq:square} are responsible for the associativity of the multiplication
\eqref{eq:multiplication}. The relevance of the higher $2$-Segal coherences can be understood in
terms of higher categorical structures. For example, in Chapter \ref{sec:higherhall}, we construct
the Hall monoidal $\infty$-category associated to $X$ which can be interpreted as a categorification
of the ordinary Hall algebra. 
In Chapter \ref{section:higher-bicat} we provide an alternative higher categorical interpretation of
$2$-Segal spaces within a $\inftytwo$-categorical theory of bispans, developed in Chapter
\ref{sec:spans}. In terms of this theory, we can functorially associate to a $2$-Segal space $X$ a
monad $A_X$ in the $\inftytwo$-category of bispans in spaces. If the space $X_0$ is contractible,
then we can reinterpret $A_X$ as an algebra object in the category of spans in spaces, equipped with
the pointwise Cartesian monoidal structure constructed in Chapter \ref{sec:spans}.\\

In a sequel to this work, we provide yet another interpretation of $2$-Segal spaces which is
suitable for a comparison statement between model categories: We can associate to a $2$-Segal space $X$ a
generalized $\infty$-operad $O_X$ in the sense of \cite{lurie.algebra}. On the one hand, the
$\infty$-operad $O_X$ can be easily obtained from the monad $A_X$. On the other hand, we can
construct a Quillen adjunction 
\[
	\sS \longleftrightarrow (\mSet)_{/\ND}
\]
between the category simplicial spaces equipped with the $2$-Segal model structure and the category
of marked simplicial sets over $\ND$, equipped with the model structure for {\em quadratic} operads.
This latter model structure is a localization of the model structure for non-symmetric generalized
$\infty$-operads (constructed using \cite[B.2]{lurie.algebra}). We expect that this Quillen
adjunction is in fact a Quillen equivalence, thus providing a complete description of the homotopy
theory of $2$-Segal spaces in $\infty$-categorical operadic terms. One interesting feature of this
description is the possibility to study algebras for the operad $O_X$.  We expect this notion to
provide a natural higher categorical generalization of the Deligne's theory of determinant functors
\cite{deligne}, and, more generally, the of notion of a charade \cite{kapranov-analogies} due to the
second author.

Let us indicate two further directions which will be taken up in the sequel to this paper. 
The first is the study of {\em cyclic} $2$-Segal spaces such as the cyclic bar
construction. We recall that Connes \cite{connes} has introduced a category $\Lambda$ containing the category
$\Delta$ of simplices, and cyclic objects in a category $\mC$ are contravariant functors
$X:\Lambda\to\mC$. So a cyclic object is a simplicial object with extra structure and we can hence 
speak about $2$-Segal objects in this context.
Above, we observed that, for each $n \ge 2$, the derived membrane space $RX_{\T}$ of a $2$-Segal
space $X$ is weakly independent of the choice of triangulation $\T$ of the convex polygon $P_n$. If
$X$ carries a cyclic structure, then we can ``globalize'' this statement to triangulations $\T$ of a
marked oriented surface $S$. Roughly, this construction goes as follows. The orientation of $S$
equips each of the triangles of $\T$ with a cyclic structure. We can glue these cyclic triangles to
obtain a cyclic set $\Lambda^{\T}$. 
The formalism of homotopy Kan extensions allows us to evaluate the cyclic space $X$ on
$\Lambda^{\T}$ which produces a {\em cyclic derived membrane space}. Again, this
homotopy type can be shown to be weakly independent of $\T$ in a coherent way which, in particular,
implies that it admits an action of the mapping class group of the marked surface $S$.
We expect this result to be particularly interesting in the context of $2$-periodic triangulated
dg categories: heuristic considerations predict the existence of a natural cyclic structure on the
Waldhausen S-construction. This cyclic structure seems to be highly interesting and opens up
potential connections between $2$-periodic triangulated categories (e.g., $2$-periodic orbit
categories, matrix factorization categories) and mapping class groups.

As the title of this paper suggests, we can view $1$- and $2$-Segal spaces as part of a hierarchy
consisting of successively larger classes of $d$-Segal spaces defined for $d \geq 0$, and a chain
of Bousfield localizations extending \eqref{eq:chain-of-loc-intro}. Systematic study of the case
$d\geq 3$ will be done in a sequel to this paper. It is based on R. Street's notion of {\em
orientals} \cite{street}. The main idea behind this notion is to subdivide the boundary of the
$d$-simplex into two combinatorial $(d-1)$-balls 
\[
	\partial\Delta^{d} = \partial_+\Delta^{d} \cup \partial_-\Delta^{d}
\]
with $\partial_+$, resp. $\partial_-$ obtained as the union of the faces $\partial_i$ with even,
resp. odd $i$. So for each simplicial set $X$, the correspondence \eqref{eq:fund-corr-intro} is
included (as a particular case $d=2$) into a hierarchy of correspondences
\[
\mu_d =\bigl\{ \Hom(\partial_-\Delta^{d+1}, X) \lla X_{d+1} = \Hom(\Delta^{d+1}, X) 
\lra \Hom(\partial_+\Delta^{d+1}, X)\bigr\}
\]
each of which can be viewed as a coherence condition for the previous one.  For $d=3$ the
$\partial_\pm\Delta^3$ form the two triangulations of the 4-gon, with $\Delta^3$ itself providing
the flip between them.  The $d$-Segal condition on a simplicial space $X$ is obtained, in the first
approximation, by forming a homotopy analog of $\mu_{d+1}$ and requiring that one or both of its arrows
be weak equivalences. This should be further complemented by ``associativity" conditions involving
various triangulations of {\em cyclic polytope} $C(n, d)\subset\RR^d$ with $n+1$ vertices, see
\cite{kapranov-voevodsky, rambau} which plays the role of a convex $(n+1)$-gon $P_n$.\\

\noindent
{\bf Acknowledgements.} We would like to thank A. Goncharov, P. Lowrey, J. Lurie, I. Moerdijk, P.
Pandit, and B. To\"en for useful discussions which influenced our understanding of the subject.

The first author was a Simons Postdoctoral Fellow while this work was carried out. The research of
the second author was partially supported by an NSF grant, by the Max-Planck-Institut f\"ur
Mathematik in Bonn and by the Universit\'e Paris-13.

\vfill\eject

\numberwithin{equation}{subsection}

\vfill\eject

\section{Preliminaries}\label{sec:preliminaries}

\subsection{Limits and Kan extensions}\label{subsec:kan}

We recall some aspects of the basic categorical concepts of limits and Kan extensions.  For more
background on this classical material see \cite{schubert,maclane,kelly,kashiwara-schapira}. 

Given a small category $A$, an {\em $A$-indexed diagram} (or simply $A$-diagram) in a category $\Cc$
is defined to be a covariant functor $F:A\to\Cc$. It is traditional to denote the value of $F$ on an object $a\in A$ by $F_a$ and to write
the diagram as $(F_a)_{a\in A}$, suppressing the notation for the values of $F$ on morphisms in $A$. 
We denote by
\[
  \Cc^A =\Fun(A,\Cc), \quad \Cc_A = \Fun(A^{\op}, \Cc)
\]
the categories of $A$-indexed (resp. $A^{\op}$-indexed) diagrams where the morphisms are given by
natural transformations.  The projective limit (or simply {\em limit}) and the inductive limit (or
{\em colimit}) of an $A$-indexed diagram $(F_a)_{a\in A}$ will, if they exist, be denoted by
$\varprojlim^\Cc_{a\in A} F_a$ and $\varinjlim_{a\in A}^\Cc F_a$, respectively. 
If $\Cc$ has all inductive and projective limits, we obtain functors
\begin{alignat*}{2}
	 \varinjlim{}^\Cc &: \Cc^A \lra \Cc, &\quad\quad & \varprojlim{}^\Cc : \Cc^A \lra \Cc,
\end{alignat*}
which are, respectively, left and right adjoint to the constant diagram functor
\[
	\kappa: \Cc\lra \Cc^A, \quad X\mapsto (X)_{a\in A}. 
\]
 
More generally, let $\phi: A\to B$ be a functor of small categories, and consider the pullback
functor
\[
	\phi^*: \Cc^B\lra \Cc^A, \quad (\phi^*G)(a) = G(\phi(a)), 
\]
reducing to $\kappa$ for $B=\pt$. The left and right adjoints to $\phi^*$
are, assuming they exist, known as the {\em left}, resp. {\em right Kan extension}
functors along $\phi$, denoted by
\begin{alignat*}{2}
	 \phi_! &: \Cc^A \lra \Cc^B, &\quad\quad & \phi_* : \Cc^A \lra \Cc^B,
\end{alignat*}
If $\Cc$ has all inductive and projective limits, then $\phi_!$ and $\phi_*$ exist and their values
on a functor $F: A\to \Cc$ are given by the formulas \cite[\S X.3, Thm. 1]{maclane}:
\begin{equation}\label{eq:kan-extensions-pointwise-limits}
 \begin{split}
	(\phi_!F)(b) & \cong \ind^\Cc_{\{\phi(a)\to b\}\in \phi/b} F(a),\\
	(\phi_*F)(b) & \cong \pro^\Cc_{\{b\to \phi(a)\}\in b/\phi} F(a). 
\end{split}
\end{equation}
Here the {\em comma category} $\phi/b$ has as objects pairs $(a,\phi)$, consisting of an object $a \in A$ and a
morphism $\phi(a) \to b$ in $B$, and similarly for $b/\phi$.
Further, the values of $\phi_!F$ and $\phi_*F$ on an arrow $b\to b'$ in $B$ can be found from
the pointwise formulas \eqref{eq:kan-extensions-pointwise-limits} by using the functoriality of
the limits.  
      
\vfill\eject

\subsection{Simplicial objects}

Let $\Delta$ be the category of finite nonempty standard ordinals and monotone maps. As usual, we
denote the objects of $\Delta$ by $[n]=\{0,1,...,n\}$, $n \ge 0$. A {\em simplicial object} in a category $\Cc$
is a functor $X:\Dop\to\Cc$. Since any finite nonempty ordinal is canonically isomorphic to a
standard ordinal, we may canonically extend $X$ to {\em all} finite nonempty ordinals; we leave this
extension implicit and use the notation $X_I$ for the value of $X$ on
an ordinal $I$. Further, we write $X_n$ for the object $X_{[n]}$ of $\C$. The objects $\{X_n\}$ are
related by the 
{\em face} and {\em degeneracy morphisms}
\[
  \partial_i: X_n\lra X_{n-1}, \; i=0, ..., n, \quad s_i: X_n\lra X_{n+1}, \; i=0, ..., n,
\]
satisfying the standard simplicial identities, see \cite{gabriel-zisman}.  
To emphasize that $X$ is a simplicial object, we sometimes
write it as $X_\bullet$ or $(X_n)_{n\geq 0}$. Using the notation introduced above, the category of simplicial objects in 
$\Cc$ will be denoted by $\Cc_\Delta$.
 
\begin{ex} In this paper we will be mostly interested in simplicial objects in the three following 
   categories. First, the category $\Cc = \Set$ of sets, so that objects of $\Set_\Delta$ are
   simplicial sets. We denote by $\Sp=\Set_\Delta$ the category of simplicial sets. 
   Second, the category $\Cc = \Top$
   of compactly generated topological spaces, see, e.g., 
   \cite{hovey}. Objects of
   $\Top_\Delta$ will be called {\em simplicial spaces}. Third, the category $\Cc = \Sp$.
   Simplicial objects in $\Sp$ will be called {\em combinatorial simplicial spaces} and can be identified with {\em bisimplicial sets}. 
\end{ex}
  
Let $\Delta_{\inj}\subset\Delta$ denote the subcategory formed by injective morphisms. By a {\em
semi-simplicial object} in a category $\C$ we mean a contravariant functor $\Delta_\inj\to \C$. The
category of such objects will be denoted by $\C_{\Delta_\inj}$. Thus a semi-simplicial object in
$\C$ gives rise to a sequence $\{X_n\}$ of objects in $\C$, related by face maps as above, but without
degeneracy maps. For example, semi-simplicial objects in $\Set$ have been studied in
\cite{rourke-sanderson} under the name of $\Delta$-sets. Any simplicial object can be considered
as a semi-simplicial object by restricting the functor from $\Delta$ to $\Delta_\inj$. Even though we
focus on simplicial objects, much of the theory developed in this work will also be applicable to {\em semi}-simplicial objects. 

For a natural number $n \ge 0$, we introduce
the {\em standard $n$-simplex} $\Delta^n \in \Set_\Delta$, 
which is the representable functor
\[
  \Delta^n: \Dop \lra \Set,\; [m] \mapsto \Hom_{\Delta}([m],[n]).
\]
We have a natural isomorphism $\Hom_{\Sp}(\Delta^n, D) \cong D_n$ for any simplicial set $D$.
Occasionally, it will be convenient to define the {\em $I$-simplex}
$\Delta^I := \Hom_{\Delta}(-, I) \in \Sp$
for any finite ordinal $I$, where as above, we canonically identify $\Delta$ with the
category of {\em all} finite nonempty ordinals.
Any simplicial set $D$ can be realized as an inductive limit 
of a diagram indexed by its {\em category of simplices},
by which we mean the comma category $\Delta/D$
formed by all morphisms $\Delta^n\to D$ in $\Sp$, $n \ge 0$:
\begin{equation}
  \label{eq:sset-lim-simplices}
  D \cong \varinjlim{}_{\{(\Delta^n \to D) \in \Delta/D\}}^{\Sp}\; \Delta^n.
\end{equation}
 In fact, this is a general property of functors 
from any category to $\Set$: any such functor is an inductive limit of representable functors. 

We further denote by
\[
   |\Delta^I| =\bigl\{ p=(p_i)_{i\in I}\in\RR^{I}| p_i\geq 0, \sum p_i=1\bigr\}\in\Top
\]
the {\em geometric $I$-simplex}. Here, the {\em geometric realization} $|D|$
of a simplicial set $D$ is the topological space obtained by replacing $\Delta^n$
with $|\Delta^n|$ in \eqref{eq:sset-lim-simplices}:
\[
    |D| = \varinjlim{}_{\{\Delta^n\to D\}}^{\Top} \; |\Delta^n|.
\]
 
\begin{rem} More generally, one can define the geometric realization
 of any simplicial space $X\in\Top_\Delta$ by gluing the spaces
 $X_n \times |\Delta^n|$ or, more precisely, forming the coend (see \cite{maclane})
 of the bivariant functor 
 $$X_\bullet\times |\Delta^\bullet|: \Dop\times\Delta\lra \Top, \quad ([m], [n])\mapsto X_m\times |\Delta^n|.$$
\end{rem}
  
We introduce some standard examples of simplicial objects.
  
\begin{exas}\label{ex:nerve,fat-simplex}
  \begin{exaenumerate}
    \item For a set $I$ we define the {\em fat $I$-simplex} to be the simplicial set
	 $(\Delta^I)'$ given by
	 \[
	 (\Delta^I)'_J = \Hom_{\Set}(J,I),
	 \]
	 where we consider {\em all} maps between the sets underlying the ordinals $J$ and $I$. 
	 As usual, in the case $I=[n]$, we write $(\Delta^n)'$ for $(\Delta^I)'$.
   \item For a small category $\Cc$ we denote by $\N\Cc$ the {\em nerve} of $\Cc$. This is a simplicial set,
      with $\N_n\Cc$ being the set of functors $[n]\to\Cc$, 
      where the ordinal $[n]$ is considered as a category. Explicitly, we have the formula
      \[
      \N_n\Cc = \coprod_{x_0, ..., x_n\in\Ob(\Cc)} \Hom_\Cc (x_0, x_1)\times \cdots
      \times \Hom_\Cc (x_{n-1}, x_n).
      \]
      For instance, the fat simplex $(\Delta^I)'$ is the nerve of the category with the set of
      objects $I$ and one morphism between any two objects. We write $B\Cc = |\N \Cc|$ for the
      geometric realization of the nerve and call it the {\em classifying space} of $\Cc$. 
      
      More generally, by a {\em semi-category} we mean a structure consisting of objects,
      morphisms and their associative composition (as in the ordinary concept of a category) but
      without requiring the existence of identity morphisms. For instance, a semi-category with one
      object is the same as a semigroup, while a category with one object
      is a monoid (a semigroup with unit). For any semi-category $\Cc$, we can define its nerve
      $\N\Cc$ as a semi-simplicial set.  
	  
   \item Let $\Cc$ be a small {\em topological category}, i.e., a small category enriched
     in $\Top$. Then we can define a topological nerve $\N_{on{top}}\Cc$ which is naturally a simplicial space. 
 
   \item Any (semi-)simplicial set $X$ gives rise to the {\em discrete (semi-)simplicial space} $\disc{X}$
 	so that $\disc{X}_n = X_n$ considered with discrete topology.
       Any topological space $Z\in\Top$ gives rise to a 
       {\em constant simplicial space}, also denoted by $Z$, so that $Z_n=Z$ and
       all face and degeneracy morphisms are identity maps. 
  \end{exaenumerate}
\end{exas}

\subsection{Homotopy limits of diagrams of spaces.}\label{subsec:homotopy-limits-spaces}

Homotopy limits were originally introduced by Bousfield and Kan \cite{bousfield-kan}
using explicit constructions, usually referred to as bar and cobar constructions, 
which we now recall. 

\begin{defi}\label{def:hopro}
Let $Y = (Y_a)_{a\in A}$ be a diagram in $\Top$. 
\begin{enumerate}[topsep=0.5pc,label=(\alph{*})]
	
  \item Assume that each space $Y_a$ is a retract of a CW-complex. Then the 
  {\em homotopy inductive limit} (or {\em homotopy colimit}) of $Y$, denoted by  
	 $\hoind_{a\in A} Y_a$, is the geometric realization of the simplicial space $\underrightarrow Y_\bullet$,
	 defined by
	 \[
	 \underrightarrow Y_n = \coprod_{a_0\to ...\to a_n} Y_{a_0},
	 \]
	 where we take the disjoint union over all chains of composable morphisms in $A$. 

	\item The {\em homotopy projective limit} (or {\em homotopy limit}) of $Y$, denoted by
	 $\hopro_{a\in A} Y_a$, is the topological space 
	 formed by the following data:
	 \begin{enumerate}[label=(\arabic{*})]
	\item For each object $a\in A$, a point $y_a\in Y_a$;

	\item For each morphism $a\buildrel u\over\rightarrow b$
	 in $A$, a path (singular 1-simplex) $y_{a\to b}: [0,1]\to Y_b$ with $y_{a\to b}(0)= u_*(y_a)$ and
	 $y_{a\to b}(1)=y_b$. 

	\item For each composable pair of morphisms $a\buildrel u\over\to b\buildrel v\over\to c$ in $A$,
	a singular triangle $y_{a\to b\to c}: |\Delta^2|\to Y_c$ whose restrictions to the three sides of $\Delta^2$
	are $y_{b\to c}$, $y_{a\to c}$, and $v_*(y_{a\to b})$.

	\item For each composable triple of morphisms $a\buildrel u\over\to b\buildrel v\over\to c\buildrel w\over\to d$
	in $A$, a singular tetrahedron $y_{a\to b\to c\to d}: |\Delta^3|\to Y_d$ whose restriction to the 2-faces
	are $y_{b\to c\to d}$, $y_{a\to c\to d}$, $y_{a\to b\to d}$ and $w_*(y_{a\to b\to c})$.

\quad \vdots

\item[(n)] The analogous data for each composable $n$-chain of morphisms in $A$. 
\end{enumerate}
	The topology on $\hopro_{a\in A} Y_a$ is induced from the compact-open topology on
	mapping spaces.
\end{enumerate}
\end{defi}

There are various frameworks which allow for a more conceptual definition of homotopy limits. For
an approach using model categories, see Chapter \ref{sec:model-cat} and specifically \S \ref
{subsec:holim-model} below. In the later chapters, we will also utilize the $\infty$-categorical theory
of limits. In both contexts, one can show that, for diagrams of spaces, homotopy limits can be
computed using the formula given in Definition \ref{def:hopro}.

For now, it will be sufficient to introduce the notion of a {\em weak equivalence} in $\Top$ which
is a morphism $f: X\to Y$ inducing a bijection on $\pi_0$ and, for every $i \ge 1$, an isomorphism
$\pi_i(X,x)\to\pi_i(Y, f(x))$. Further, a morphism $f: (Y_a)_{a\in A}\to (Y'_a)_{a\in A}$ of
diagrams in $\Top$ will be called a weak equivalence, if each $f_a$ is a weak equivalence.

Note that we have natural maps
\begin{equation}
   \label{eq:lim-holim}
   \varprojlim{}_{a\in A}^{\Top} Y_a \lra\hopro{}_{a\in A}Y_a, \quad \quad
   \hoind_{a\in A} Y_a \lra \ind{}_{A\in A}^{\Top} Y_a. 
\end{equation}
Further, note that, on the level of connected components, homotopy limits are given by 
set-theoretic limits:
\begin{equation}
   \label{eq:pi-0-lim-holim}
   \pi_0 \hoind_{a\in A} Y_a = \ind{}^{\Set}_{a\in A} \pi_0(Y_a), \quad 
   \pi_0 \hopro_{a\in A} Y_a = \pro{}^{\Set}_{a\in A} \pi_0(Y_a). 
\end{equation}
   
\begin{exas}  
  \begin{exaenumerate}
    \item The homotopy limit
	\[ 
      	 X\times^R_Z Y := \hopro \bigl\{ X\buildrel f\over \lra Z\buildrel g\over \lla Y\bigr\}
	\]
	 is known as the {\em homotopy fiber product} of $X$ and $Y$ over $Z$. 
	 Up to weak equivalence, this is the space consisting of
	 triples $(x,y, \gamma)$, where $x\in X$, $y\in Y$ and $\gamma$ is a path in $Z$, joining $f(x)$ and $g(y)$. 

     \item The homotopy limit
	 \[
	 Rf^{-1}(y) =\hopro \bigl\{ X\buildrel f\over\lra Y\lla \{y\} \bigr\}, \quad y\in Y,
	 \]
	 is known as the {\em homotopy fiber} of $f$ over $y$. Up to weak equivalence, the homotopy
	 fiber is given by the space consisting of pairs $(x,\gamma)$, where $x\in X$
	 and $\gamma$ is a path joining $f(x)$ and $y$. 
  \end{exaenumerate}
\end{exas}

The following is a crucial property of homotopy limits.  

\begin{prop}\label{prop:invariance-holim-top}
Let $f: (Y_a\to Y'_a)_{a\in A}$ be a weak equivalence of diagrams in $\Top$.
Then the induced map 
\[
\hopro(f): \hopro_{a\in A} Y_a\to\hopro_{a\in A} Y'_a
\]  
is a weak equivalence. Assume further that all spaces $Y_a$, $Y'_a$ are retracts of CW-complexes.
Then we have a weak equivalence
\[
\hopro(f): \hoind_{a\in A} Y_a\to\hoind_{a\in A} Y'_a.
\]  
\end{prop}

We now recall a concept related to that of the homotopy limit. Denote by $\Cat$ the category of
small categories with morphisms given by functors. By a {\em diagram of categories} we mean a
functor from a small category $A$ to $\Cat$. 

\begin{defi}\label{def:2lim}
Let $(\Cc_a)_{a\in A}$ be a diagram of categories.  
The {\em projective 2-limit} $\twopro_{a\in A} \Cc_a$ is the category whose objects
are data consisting of:

\begin{itemize}
\item[(0)] An object $y_a\in\Cc_a$, given for each $a\in\Ob(A)$. 

\item[(1)] An isomorphism $y_u: u_*(y_a)\to y_b$ in $\Cc_b$, given for each morphism $u: a\to b$ in $A$. 

\item[(2)] The $y_u$ are required to satisfy the compatibility condition: For each
each composable pair of morphisms $a\buildrel u\over\to b\buildrel v\over\to c$ in $A$, we should have
$y_{vu}=y_v\circ v_*(y_u)$. 
\end{itemize}
A morphism in $\twopro_{a\in A} \Cc_a$ from $(y_a, y_u)$ to $(y'_a, y'_u)$
is a system of morphisms $y_a\to y'_a$ in $\Cc_a$ commuting with the $y_u$ and $y'_u$. 
\end{defi}

In particular, we have the {\em $2$-fiber product} of categories
  \[
  \Cc\times_{\Dc}^{(2)}\Ec = \twopro\bigl\{ \Cc \buildrel p\over\lra \Dc
\buildrel q\over\lla \Ec\bigr\}.
\]

\begin{prop}\label{prop:invariance-2lim}
If $(\Cc_a\to \Cc'_a)_{a\in A}$ is a morphism of diagrams in $\Cat$
consisting of equivalences of categories, then the induced morphism 
$\twopro_{a\in A} \Cc_a\to\twopro_{a\in A} \Cc'_a $ is an equivalence 
of categories as well. 
\end{prop}

We recall that a {\em groupoid} is a category with all morphisms invertible. 

\begin{prop}\label{prop:2lim-holim}
(a) For any diagram of categories $(\Cc_a)_{a\in A}$
we have a natural morphism of spaces
\[
f: B\bigl(\twopro{}_{a\in A} \Cc_a\bigr) \lra \hopro{}_{a\in A} B\Cc_a. 
\]

(b) Assume that $(\Cc_a)_{a\in A}$ is a diagram of groupoids. Then $\twopro{}_{a\in A} \Cc_a$ is a groupoid,
and $f$ is a weak equivalence.
\end{prop}

\begin{proof} (a) A vertex of $N\bigl(\twopro{}_{a\in A} \Cc_a\bigr)$,
i.e., an object of $\twopro{}_{a\in A} \Cc_a$, gives a datum as in
Definition \ref{def:hopro}, in fact a datum consisting of a combinatorial
$n$-simplex $\Delta^n\to N\Cc_{a_n}$ for each composable chain
$a_0\to ...\to a_n$ of $n$ morphisms in $A$
(which then gives a singular $n$-simplex in $B\Cc_a$). This datum
gives therefore a point of $\hopro{}_{a\in A} B\Cc_a$. 
  Further, for a 
combinatorial $p$-simplex $\sigma: \Delta^p\to N\bigl(\twopro{}_{a\in A} \Cc_a\bigr)$
we get, in the same way, a  morphism 
of simplicial sets $\Delta^n\times \Delta^p\to N\Cc_{a_n}$, and these
morphisms give a map $f_\sigma: |\Delta^p|\to \hopro{}_{a\in A} B\Cc_a$. 
It is straightforward to see that the $f_\sigma$ assemble into the claimed map $f$. 

(b) The fact that $\twopro{}_{a\in A} \Cc_a$ is a groupoid is obvious from the
definition of its morphisms. We now construct a homotopy inverse for $f$. 
For a space $Y\in\Top$ let $\operatorname{Sing}(Y)$ be its singular
simplicial set, so that the natural map $|\operatorname{Sing}(Y)|\to Y$ is a
homotopy equivalence. Let also $\Pi_1(Y)$ be the fundamental groupoid
of $Y$, so $\Ob(\Pi_1(Y))=Y$ and $\Hom_{\Pi_1(Y)}(x,y)$ is the set of
homotopy classes of paths from $x$ to $y$. We have a natural morphism of
simplicial sets $h_Y: \operatorname{Sing}(Y)\to N\Pi_1(Y)$. If all
the connected components of $Y$ have $\pi_{\geq 2}=0$, then
$|h_Y|$ is a homotopy equivalence. This is true, in particular, if $Y=B\Cc$
where $\Cc$ is a groupoid. In that case we also have that
the natural functor of groupoids $\Cc\to \Pi_1(B\Cc)$ is an
 equivalence. 

Further, if $(Y_a)_{a\in A}$ is any diagram in $\Top$, then we have a morphism of
simplicial sets
$$g: \operatorname{Sing}\bigl( \hopro_{a\in A} Y_a\bigr) \lra
N\bigl( \twopro{}_{a\in A} \Pi_1(Y_a) \bigr). 
$$
We apply this to $Y_a=B\Cc_a$.
Propositions \ref{prop:invariance-holim-top} and \ref{prop:invariance-2lim}
together with the above equivalences
 imply that $g$ is homotopy inverse to $f$. 
\end{proof}

We will also need a slight generalization of homotopy limits:
the homotopy version of the concept of the end of a bifunctor.
Let us present an explicit definition using a kind of cobar construction.
  
Let $A$ be a small category and $Y: A^\op\times A\to\Top$ be a bifunctor. Thus for each morphism
$u: a\to b$ and each object $c$ of $A$ we have the maps
\[
	u_*: Y(c,a) \lra Y(c,b), \quad u^*: Y(b,c)\lra Y(a,c).
\]

\begin{defi}\label{def:derived-end-topological}
The {\em homotopy end} of $Y$, denoted by $\hoend_{a\in A} Y(a,a)$
is the topological space  formed by the following data:
\begin{enumerate}
	\item[(0)] For each object $a\in A$, a point $y_a\in Y(a,a)$.
  
	\item[(1)] For each morphism $a\buildrel u\over\rightarrow b$
	    in $A$, a path (singular 1-simplex) $y_{a\to b}: [0,1]\to Y(a,b)$ with $y_{a\to b}(0)= u_*(y_a)$ and
	    $y_{a\to b}(1)=u^*y_b$. 

	\item[(2)] For each composable pair of morphisms $a\buildrel u\over\to b\buildrel v\over\to c$ in $A$,
	    a singular triangle $y_{a\to b\to c}: |\Delta^2|\to Y(a,c)$ whose restrictions to the three sides of $\Delta^2$
	    are $u^*y_{b\to c}$, $y_{a\to c}$, and $v_*(y_{a\to b})$.
	
	    \quad \vdots

	\item[($n$)] And so on for composable chains of morphisms of any length $n-1\geq 0$.
\end{enumerate}
\end{defi}
 
Thus, if $Y(a,b)=Y_b$ is constant in the first argument, then
\[
	\hoend_{a\in A} Y(a,a) =\hopro_{a\in A} Y_a. 
\]
Similarly to Proposition \ref{prop:invariance-holim-top}, the homotopy end takes weak equivalences
of functors $A^\op\times A\to\Top$ to weak equivalences in $\Top$.
 
\vfill\eject

\section{Topological 1-Segal and 2-Segal spaces}
\label{sec:topsegal}

\subsection{Topological 1-Segal spaces and higher categories}
\label{subsec:1-segal}
  
Informally, a ``higher category" should be given by 
\begin{itemize}
	\item[(0)] a collection of objects,
	\item[(1)] for objects $x,y$ a collection of $1$-morphisms between $x$ and $y$,
	\item[(2)] for objects $x,y$ and $1$-morphisms $f,g$ between $x$ and $y$ a collection of
	  $2$-morphisms between $f$ and $g$,
	\item[\vdots] 
       \item[($n$)] a collection of $n$-morphisms involving analogous data,
	\item[\vdots] 
\end{itemize}
together with composition laws which are weakly associative up to coherent homotopy.
For example, the classical concept of a {\em bicategory} involves data (0), (1) and (2), see
Appendix \ref{app.bicategories}   for details.

The most accessible so far has been a class of higher categories, called {\em $\inftyone$-categories}, 
in which all $k$-morphisms, $k>1$, are invertible. Several different approaches to
$\inftyone$-categories have been shown to be equivalent in \cite{bergner.survey}, not unlike 
\v Cech, Dolbeault and other realizations for ``cohomology". One of these approaches is Rezk's theory
of Segal spaces (cf. \cite{rezk, segal.categories, lurie.tqft}). It is based on the following
observation.
     
\begin{prop}\label{prop:nerve-segal-set}
The functor $\N: \Cat\to\Set_\Delta$, associating to a small category its nerve, is fully faithful.
The essential image of $\N$ consists of those simplicial sets $K$ such that, for each $n \ge 2$, the map 
\[
K_n \lra K_1 \times_{K_0} K_1 \times_{K_0} \dots \times_{K_0} K_1, 
\]
induced by the inclusions $\{i,i+1\} \hookrightarrow [n]$, is a bijection. 
\end{prop}

Let $X$ be a simplicial space. For $n \ge 2$, the inclusions 
$\{i,i+1\} \hookrightarrow [n]$ as above, 
and the canonical map from $\pro$ to $\hopro$, give rise to the diagram of spaces
\[
	X_n \lra X_1 \times_{X_0} X_1 \times_{X_0} \dots \times_{X_0} X_1
	\overset{\eqref{eq:lim-holim}}{\lra} X_1 \times_{X_0}^R X_1 \times^R_{X_0} \dots
	\times_{X_0}^R X_1.
\]
We denote the composite map by $f_n$ and refer to the collection $\{f_n |\; n\ge 2\}$ as {\em $1$-Segal maps}.
  
\begin{defi}\label{defi:1-segal}
	A simplicial space is called {\em $1$-Segal space} if,
	for every $n\ge 2$, the map $f_n$ is a weak equivalence of
	topological spaces.
\end{defi}

Our definition is a topological variant of Rezk's combinatorial notion of a Segal space \cite{rezk},
following \cite[Definition 2.1.15]{lurie.tqft}. 

\begin{prop}\label{prop:1-segal-basic} Let $X$ be a simplicial space. Then the following are equivalent:
	\begin{enumerate}
	 \item $X$ is a $1$-Segal space.
	 \item For every $0\le i_1 < i_2 < \dots < i_k \le n$, the map
	 \[
	 X_n \lra X_{i_1} \times^R_{X_0} X_{i_2-i_1} \times^R_{X_0} \dots
	 \times^R_{X_0} X_{n-i_k}
	 \]
	 induced by the inclusions $\{0,\dots,i_1\}, \{i_1,\dots,i_2\}, \dots,
	 \{i_k,\dots,n\} \hookrightarrow [n]$, is a
	 weak equivalence.
	 \item For every $0 \le i \le n$, the map
	 \[
	 X_n \lra X_i \times^R_{X_0} X_{n-i}
	 \]
	 induced by the inclusions $\{0,\dots,i\}, \{i,\dots,n\} \hookrightarrow [n]$, is a
	 weak equivalence.
	\end{enumerate}
\end{prop}
\begin{proof} This is an immediate consequence of the 2-out-of-3 property of weak equivalences.
\end{proof}

\begin{ex}[(Discrete nerve and categorified nerve)]\label{ex.1segclass} 
	Let $\C$ be a small category. There are two immediate ways to associate to $\C$ a $1$-Segal
	space:
	\begin{enumerate}[label=(\alph{*})] 
	 \item The {\em discrete nerve} $\disc{\N(\C)}$ is, by Proposition
	   \ref{prop:nerve-segal-set}, a $1$-Segal space and every discrete $1$-Segal spaces is
	   isomorphic to the discrete nerve of a small category.
	 	
	 \item The set $\N(\C)_n$ of composable chains of morphisms in $\C$ is in fact the set of
	   objects of the {\em category} $\Fun([n],\C)$. Denote by $\C_n \subset \Fun([n], \C)$ the
	   groupoid of all isomorphisms in $\Fun([n],\C)$. Then the collection $\{\C_n\}$ assembles
	   to a simplicial groupoid $\C_\bullet$ which we call the {\em categorified nerve} of $\C$.
	   Passing to classifying spaces, the simplicial space $X_{\bullet}$ obtained by setting
	   $X_n = B(\C_n)$, $n \ge 0$, is a $1$-Segal space. This follows at once from Proposition
	   \ref{prop:2lim-holim}: the $1$-Segal maps identify the groupoid $\C_n$ with the $2$-fiber
	   product $\C_1 \times^{(2)}_{\C_0} \C_1 \times^{(2)}_{\C_0} \dots \times^{(2)}_{\C_0}
	   \C_1$. Within Rezk's theory, this categorified nerve is the preferred way to model a
	   small category as a $1$-Segal space, since it satisfies a completeness condition which
	   will be explained in more detail in \S \ref{subsec:quasicat-1-segal}.
	\end{enumerate}
\end{ex}

By Example \ref{ex.1segclass}, we can associate a $1$-Segal space to any small category. Vice versa,
given a $1$-Segal space $X$ we can define the
{\em homotopy category of $X$}, denoted $\h X$, as follows. The set of objects $\Ob(\h X)$ is given
by the set underlying the space $X_0$.
For objects $x,y\in X_0$, we define
\[
  \Hom_{\h X}(x,y) = \pi_0 \bigl( \{x \} \times_{X_0}^R X_1 \times_{X_0}^R \{y\} \bigr),
\]
where the homotopy fiber product involves the face maps $\partial_1$ and $\partial_0$. 
To compose morphisms $f: x \to y$ and $g: y \to z$, we consider the span diagram
\begin{equation}\label{eq:1segalspan}
  \xymatrix{ \{x \} \times_{X_{\{0\}}}^R X_2 \times_{X_{\{2\}}}^R \{z\} \ar[d]^p
    \ar[r]^q & \{x\} \times_{X_{\{0\}}}^R X_{\{0,2\}}\times_{X_{\{2\}}}^R \{z\}\\
    \{x \} \times_{X_{\{0\}}}^R X_{\{0,1\}} \times_{X_{\{1\}}}^R
    X_{\{0,1\}} \times_{X_{\{2\}}}^R \{z\} &.
  }
\end{equation}
The pair $(f,g)$ singles out a connected component of the bottom space in \eqref{eq:1segalspan}.
Since the vertical map in \eqref{eq:1segalspan} is a weak equivalence, we obtain a well-defined
connected component $q \circ p^{-1} (f,g)$ of $\{x\} \times_{X_{\{0\}}}^R
X_{\{0,2\}}\times_{X_{\{2\}}}^R \{z\}$ which we define to be the composition of $f$ and $g$. A
similar argumentation, using the fact that the $1$-Segal map
\[
  X_3 \lra X_1 \times_{X_0}^R X_1 \times_{X_0}^R X_1
\]
is a weak equivalence, shows that the above composition law is associative.
The identity morphism of an object $x \in X_0$ is obtained by interpreting the image of $x$ under the
degeneracy map $X_0 \to X_1$ as an element of $\{x\} \times_{X_0}^R
X_{1}\times_{X_{0}}^R \{x\}$.

Note that the definition of $\h X$ only involves the $3$-skeleton of $X$. The additional data
contained in $X$, allows us to define {\em mapping spaces}
\begin{equation}\label{eq:segalmapping}
	\Map_X(x,y) := \{x \} \times_{X_0}^R X_1 \times_{X_0}^R \{y\},
\end{equation}
together with maps
\[
  \Map_X(x_1,x_2) \times \Map_X(x_2,x_3) \times \dots \times \Map_X(x_{n-1},x_n) \lra \Map_X(x_1,x_n),
\] 
which form a coherently associative system of composition laws.

\begin{rem}\label{rem:semicat} Let $\Top_{\Delta_\inj}$ be
	  the category of semi-simplicial spaces.  Note that the $1$-Segal maps $f_n$ in Definition
	  \ref{defi:1-segal}, being defined in terms of the injections $\{i, i+1\}\to [n]$, make sense for any
	  $X\in \Top_{\Delta_\inj}$.  We say that $X$ is $1$-Segal, if they are weak eqiuivalences.  The
	  (discrete) semi-simplicial nerve construction (Example \ref{ex:nerve,fat-simplex}(b)) gives an
	  equivalence of categories
\[
	\bigl\{ \text{Small semi-categories}\bigr\} \lra \bigl\{
	\text{Discrete $1$-Segal semi-simplicial spaces}\bigr\},
	\quad \Cc \mapsto \disc{\N\Cc}. 
\]
	Similarly, the categorified nerve construction from Example \ref{ex.1segclass}(b) applies to
	any small semi-category and associates to it a different $1$-Segal semi-simplicial space.
\end{rem}

Up to a completeness condition which will be recalled in \S \ref{sec:2-segal-from-higher}, we have
the following informal statements.

\begin{princ}\label{Universality-principle} 
\begin{enumerate}[label=(\alph{*}),itemsep=0pt]
	  \item $1$-Segal simplicial spaces model any reasonable concept of $\inftyone$-categories.
	  \item $1$-Segal semi-simplicial spaces model $\inftyone$-analogs of semi-categories.
\end{enumerate}
\end{princ}

For now, we leave the statement at this informal level; the main purpose of formulating the principle 
at this point is to put the theory of $2$-Segal spaces below into a context.
 
\vfill\eject

\subsection{Membrane spaces and generalized Segal maps}
\label{subsec:membranes}

Let $X$ be a simplicial space and $D$ a simplicial set. First forgetting the topology of the spaces
$\{X_n\}$, we consider $X$ as a simplicial set and form the set
\[
	(D,X) :=\Hom_{\Set_\Delta} (D, X) \subset \prod_{n\geq 0} X_n^{D_n}. 
\]
The topology on $\{ X_n\}$ naturally makes $(D,X)$ a topological space which we call the {\em space of
$D$-membranes in $X$}. The general formula \eqref{eq:sset-lim-simplices} implies the identification
\begin{equation}\label{eq:membranes-top}
	(D,X) \cong \pro_{ \{ \Delta^p \to D\}\in\Delta/D}^{\Top} X_p. 
\end{equation}
Further, we define the {\em derived space of $D$-membranes in $X$} by
\begin{equation}\label{eq:derived-membranes-top}
	(D,X)_R =  \hopro_{ \{\Delta^p \to D\}\in \Delta/D} X_p. 
\end{equation}
 
\begin{exa} \label{ex:mem} 
\begin{exaenumerate}
	\item Taking $D=\Delta^n$, we find that the category $\Delta/\Delta^n$ has
	a final object, given by $\id: \Delta^n\to\Delta^n$, and so
	\[
		 ( \Delta^n, X)_R \simeq ( \Delta^n, X) \cong X_n.
	\]

	\item \label{ex:I[n]} For the segment
	\[
		\J^n := \Delta^1 \cop_{\Delta^0} \cdots \cop_{\Delta^0}
		\Delta^1 = \bullet\lra\bullet\lra\cdots\lra\bullet
	\]
	of $n$ edges, we have
	\[
		  ({\J^n}, X)_R\simeq X_1\times_{X_0}^R X_1\times_{X_0}^R \cdots \times_{X_0}^R X_1,
	\]
	and similarly for the ordinary space of membranes, with ordinary fiber products instead of 
	homotopy fiber products. 
 
	\item Let $Z\in\Top$ be a topological space, considered as a constant simplicial space. For
		any simplicial set $D$, the membrane spaces
	\[
		(D, Z) =\Map (\pi_0 |D|, Z) \quad \text{and} \quad (D, Z)_R = \Map(|D|, Z),
	\]
	are given by the space of locally constant maps and the space of all continuous maps from
	$|D|$ to $Z$, respectively.
\end{exaenumerate}
\end{exa}

We will be interested in the behavior of the membrane spaces $(D,X)$ and $(D,X)_R$ with respect
to colimits in the first argument. To this end, we introduce some terminology.

\begin{defi}\label{def:acyclic}
	Let $A, B$ be small categories. A diagram $(D_b)_{b \in B}$ in $\Set_A$ is called 
	{\em acyclic} if, for every $a \in A$, the natural map
	\[
		\hoind_{b\in B} D_b(a) \lra \ind_{b\in B} D_b(a)
	\]
	is a weak homotopy equivalence of spaces. Here, the diagram $(D_b(a))_{b \in B}$,  
	obtained from $(D_b)_{b \in B}$ by evaluating at $a$, is to be interpreted as a diagram of discrete
	topological spaces.
\end{defi} 

In this section, we will mostly apply this concept in the case when $A=\Delta$, so
$(D_b)_{b\in B}$ is a diagram of simplicial sets. Let $(D_b)_{b\in B}$ be such a diagram and denote
its colimit by $D$.
For $n \ge 0$ and a simplex $\sigma \in D_n$, we define a category $B_{\sigma}$ as
follows: 
\begin{itemize} 
	\item The objects of $B_\sigma$ are given by pairs $(b, \tau)$ where $b \in B$ and $\tau \in
		(D_b)_n$ such that $\tau \mapsto \sigma$ under the canonical map $D_b \to D$. 
	\item A morphism $(b,\tau) \to (b', \tau')$ is given by a morphism $b \to b'$ in $B$ such that 
		$\tau \mapsto \tau'$ under the induced map $D_b \to D_{b'}$.
\end{itemize}
With this terminology, the following statement follows immediately from the definition.

\begin{prop}\label{prop:acyclic-crit} The diagram $(D_b)_{b\in B}$ is acyclic if and only if, for every $n \ge 0$ and every
	simplex $\sigma \in D_n$, the classifying space of the category $B_{\sigma}$ is weakly
	contractible.
\end{prop}
 
\begin{prop} \label{prop:acyclic-topological}
Let $(D_b)_{b\in B}$ be a diagram of simplicial sets and $X$ a simplicial space. Then:
\begin{enumerate}[label=(\alph{*})]
	  \item We have a natural homeomorphism
	   \[
		 \bigl( \ind_{b\in B}^{\Set_\Delta} D_b, X\bigr) \cong \pro_{b\in B}^\Top (D_b, X).
	  \]
  
	  \item If $(D_b)_{b\in B}$ is acyclic, then we have a natural weak equivalence
	  \[
		  \bigl( \ind_{b\in B}^{\Set_\Delta} D_b, X\bigr)_R \simeq  
		  \hopro_{b\in B} (D_b, X)_R. 
	  \]
\end{enumerate}
\end{prop}
\begin{proof} Part (a) is obvious. To prove (b), we formulate a more general statement
holding for arbitrary diagrams $(D_b)$ which reduces to (b) when $(D_b)$ is acyclic. 
 
First, for any two simplicial spaces $Y,X\in\Top_\Delta$, we introduce the ordinary and derived mapping
spaces as the ordinary and homotopy ends (see Definition \ref {def:derived-end-topological})
\[
\Map(Y,X) := \int_{[n]\in\Delta} \Map(Y_n, X_n), \quad 
R\Map(Y,X) := \hoend_{[n]\in\Delta} \Map(Y_n, X_n).
\]
Here $\Map(Y_n, X_n)$ is the space of continuous maps with compact--open topology. Then
\begin{equation}\label{eq:membranes-map-Rmap}
(D,X)= \Map(\disc{D}, X), \quad (D,X)_R = R\Map(\disc{D}, X),
\end{equation}
where $\disc{D}$ is the discrete simplicial space corresponding to $D$. 

Let now $(Y_b)_{b\in B}$, be a diagram in $\Top_\Delta$. 
As in any category of diagrams, colimits in the category $\Top_\Delta$
are calculated componentwise, so
\[
\bigl( \ind_{b\in B}^{\Top_\Delta} Y_b\bigr)_n = \ind_{b\in B}^\Top Y_{b, n}. 
\]
Define the {\em homotopy colimit} of $(Y_b)$ to be the simplicial space $\hoind_{b\in B} Y_b$
obtained by applying homotopy colimits componentwise:
\[
\bigl( \hoind_{b\in B} Y_b\bigr)_n : = \hoind_{b\in B} Y_{b, n}. 
\]
Then, straight from the definitions, we obtain a natural homeomorphism
\begin{equation}\label{eq:Rmap-colim-holim}
 R\Map\bigl(\hoind_{b\in B} Y_b, X\bigr) \simeq \hopro_{b\in B} R\Map(Y_b, X)
\end{equation}
for arbitrary simplicial spaces $X$ and $Y$. Now, the condition that $(D_b)$ is an acyclic diagram
of simplicial sets, means that the natural map
\[
\hoind_{b\in B} \disc{D_b} \lra \ind_{b\in B} \disc{ D_b} 
\]
is a weak equivalence of simplicial spaces. 
Combining \eqref{eq:Rmap-colim-holim} with \eqref{eq:membranes-map-Rmap} and with the homotopy
invariance of $R\Map$, we obtain the statement (b). 
\end{proof}

The following statement shows that in formula \eqref{eq:derived-membranes-top} it suffices to consider nondegenerate simplices of the
simplicial set $D$. 

\begin{prop}\label{prop:membranes-semi-simplicial}
Let $D$ be a simplicial set. Then we have a natural homeomorphism, resp. weak equivalence 
 \[
 (D,X) \cong \pro_{ \{ \Delta^p \hookrightarrow D\}\in\Delta_\inj/D}^\Top X_p,
 \quad 
 (D,X)_R \simeq \hopro_{ \{ \Delta^p \hookrightarrow D\}\in\Delta_\inj/D}^\Top X_p. 
 \]
\end{prop} 
\begin{proof} 
From the finality of the embedding $\Delta_{\inj}/D \to \Delta/D$, and
\eqref{eq:sset-lim-simplices}, we deduce 
\begin{equation}\label{eq:membranes-semi-simplicial}
	D = \ind^{\Set_\Delta}_{\{\sigma: \Delta^p \hookrightarrow D\}\in
	\Delta_\inj/D} \Delta^p. 
\end{equation}
The formula involving $(D,X)$ follows from part (a) of Proposition 
\ref{prop:acyclic-topological}. To deduce the formula for $(D,X)_R$
from part (b) of the same proposition, we need to verify that the diagram in \eqref{eq:membranes-semi-simplicial}
is acyclic. For $n \ge 0$ and $\sigma \in D_n$, the category $(\Delta_{\inj}/D)_\sigma$ from
Proposition \ref{prop:acyclic-crit} has an initial object given by the unique nondegenerate simplex
$\Delta^k \hookrightarrow D$ of which $\sigma$ is a degeneration. Therefore, Proposition
\ref{prop:acyclic-crit} implies that the diagram under consideration is acyclic.
\end{proof}
  
The formulas in Proposition \ref{prop:membranes-semi-simplicial} imply that for $D \subset \Delta^I$
the (derived) membrane space $(D,X)$ depends only on the underlying semi-simplicial structure (face maps) of
$X$. We will use these formulas to extend the definition of $(D,X)$ and $(D,X)_R$ to 
semi-simplicial spaces $X$.  
We will be particularly interested in $D$-membranes where $D$
is a subset of a standard simplex $\Delta^I$. 
Let $D, D' \subset \Delta^I$ be simplicial subsets. We define the {\em intersection}
and the {\em union} of $D$ and $D'$ by 
\[
 D \cap D' := D \times_{\Delta^I} D', \quad 
	 D \cup D' := D \cop_{D \cap D'} D' \quad \subset \Delta^I.
\]
The set of $p$-simplices of $D\cap D'$, resp. $D\cup D'$ is the intersection,
resp. the union of the sets $D_p$ and $D'_p$. 
Passing to geometric realizations, we recover the intersection and union of topological
subspaces of $|\Delta^I|$.
  
\begin{prop} \label{prop.decomp}
For simplicial sets $D, D' \subset \Delta^I$, and a (semi-)simplicial space $X$, 
we have 
\[
(D \cup D',X)_R \simeq 
(D, X)_R \times^R_{(D \cap D', X)_R} (D, X)_R.
\]
\end{prop}
\begin{proof}
	By Proposition \ref{prop:acyclic-topological}, it suffices to show that the diagram of
	simplicial sets
	\[
		D \lla D\cap D' \lra D'
	\]
	is acyclic. To show this, we use Proposition \ref{prop:acyclic-crit} with 
	$B = \{\pt\leftarrow \pt \to\pt \}$. For $n \ge 0$ and $\sigma \in D \cup D'$, the category
	$B_\sigma$ is either the trivial category with one object, or the full index category $B$.
	In both cases, the respective classifying space is contractible such that Proposition
	\ref{prop:acyclic-crit} implies the statement.
\end{proof}
 
Combinatorially, a simplicial subset $D \subset \Delta^I$ can be constructed from a collection
of subsets $\I \subset 2^I$: Any subset $J \subset I$ defines a subsimplex $\Delta^J \subset
\Delta^I$ and we define  
\begin{equation}
	\label{eq.DI}
	\DI := \bigcup_{J\in\I} \Delta^J \subset\Delta^I.
\end{equation}

\begin{prop}\label{prop:collection}
	\begin{enumerate} 
	\item Let $\Delta/\I$, resp. $\Delta_\inj /\I$, be the full
	subcategory of the overcategory $\Delta/I$, resp. $\Delta_\inj /I$, spanned by those maps $J \to I$ whose image is contained
	in one of the sets in $\I$. Then we have representations as colimits of acyclic diagrams:
	\[
	\DI \cong \ind^{\Set_\Delta}_{\{J \to I\} \in \Delta/\I}\; \Delta^J
	 \cong \ind^{\Set_\Delta}_{\{J \hookrightarrow I\} \in \Delta_\inj /\I}\; \Delta^J.
	\]

	\item For two collections $\I, \I'\subset 2^I$ we have
	 \[
		\Delta^\I \cup \Delta^{\I'} \cong \Delta^{\I \cup \I'} \quad \text{and}\quad
		\Delta^\I \cap \Delta^{\I'} \cong \Delta^{\I \capd \I'}, 
	\]
	where $\I \cup \I'$ is the union of $\I$ and $\I'$ as subsets of $2^I$, and $\I \capd \I'$ is the set
	formed by all pairwise intersections of elements of $\I$ and $\I'$.
\end{enumerate}
\end{prop}
\begin{proof} The only statement requiring proof is the acyclicity of the diagrams in (1). But 
this follows as in the proof of Proposition \ref{prop:membranes-semi-simplicial}.
\end{proof}

\begin{ex} \label{I.1segal} Consider the collection
	 ${\I_n} = \left\{ \{0,1\}, \{1,2\}, \cdots, \{n-1,
	 n\} \right\}$
	 of subsets of $[n]$ for a fixed $n \ge 2$. Then 
	 \[
\DI =	 \Delta^{\{0,1\}} \cop_{\Delta^{\{1\}}}
	 \Delta^{\{1,2\}} \cop_{\Delta^{\{2\}}} \quad \cdots \quad\cop_{\Delta^{\{n-1\}}} \Delta^{\{n-1,n\}}
	 = \J^n
	 \]
	 is the segment with $n$ edges from Example 
	 \ref{ex:mem}\ref{ex:I[n]}.
	 \end{ex}
		 
\begin{ex}\label{ex:convex-polytope} 
\begin{exaenumerate} \item Let $I \subset \RR^d$ be a finite set of points and let $P$ be the convex hull of $I$,
a convex polytope in $\RR^d$. Suppose that $I$ is given a total order, so that the simplicial
set $\Delta^I$ is defined. Let $\T$ be any triangulation of $P$ into (straight geometric) simplices
with vertices in $I$. Every simplex $\sigma \in \T$ is uniquely determined
by its subset of vertices $\operatorname{Vert}(\sigma)\subset I$, so $\T$ itself can be
viewed as a subset $\T \subset 2^I$. Hence, the triangulation $\T$ defines simplicial subset $\Delta^\T \subset \Delta^I$. 
Its realization is a CW-subcomplex in the geometric simplex $|\Delta^I|$, homeomorphic to $P$. 
By definition of the convex hull, $P$ is the image of the map
\[
p: |\Delta^I| \lra \RR^d, \quad (p_i)_{i\in I} \mapsto \sum_i p_i\cdot i.
\]
This map projects the subcomplex $|\Delta^\T|$ onto $P$ in a homeomorphic way. 

\item More generally, by a {\em polyhedral subdivision} of $P$ with vertices in $I$
we mean a decomposition $\Pc$ of $P$ into a union of convex polytopes, each having
vertices in $I$, so that any two such polytopes intersect in a (possibly empty) common face. Each polytope $Q$ of $\Pc$
is completely determined by its set of vertices $\on{Vert}(Q)\subset I$,
so we can view $\Pc$ as a subset of $2^I$, and, generalizing the convention of (a) we obtain an
inclusion of simplicial sets $\Delta^{\Pc} \subset \Delta^I$.
The set of polyhedral subdivisions of $P$ with vertices in $I$ 
is partially ordered by refinement, its unique minimal element is the subdivision $\{P\}$ consisting of $P$ alone.
In this case $\Delta^{\{P\}}=\Delta^I$. The maximal elements are precisely the triangulations as
defined in (a).
\end{exaenumerate}
\end{ex}

\begin{ex} \label{I.2segal}
The two triangulations of a square in $\RR^2$ are given by collections $\T =
\{\{0,1,2\},\{0,2,3\}\}$ and $\T' = \{\{0,1,3\},\{1,2,3\}\}$ of subsets of $I = [3]$. The
corresponding 
embeddings $\Delta^{\T} \subset \Delta^3$ and $\Delta^{\T'} \subset \Delta^3$ can be depicted
as follows:
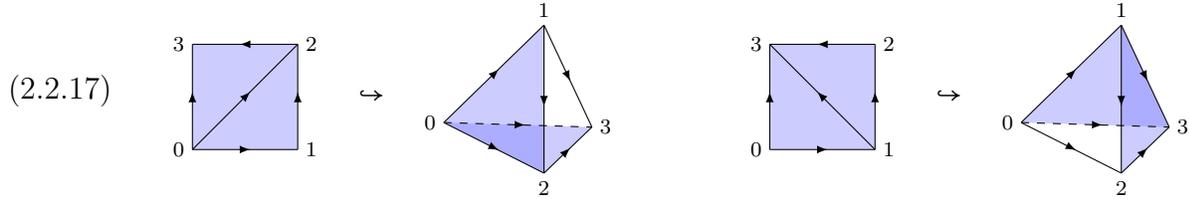
\begin{equation} \label{eq:emb}
\tdplotsetmaincoords{100}{170}
\begin{tikzpicture}[>=latex,scale=1.0, baseline=(current  bounding  box.center)]
\begin{scope}[scale=1.4]

\coordinate (A0) at (0,0);
\coordinate (A1) at (0,1);
\coordinate (A2) at (1,1);
\coordinate (A3) at (1,0);

\path[fill opacity=0.4, fill=blue!50] (A0) -- (A3) -- (A2) -- (A1) -- cycle;

\begin{scope}[decoration={
    markings,
    mark=at position 0.55 with {\arrow{>}}}
    ] 
\draw[postaction={decorate}] (A0) -- (A3);
\draw[postaction={decorate}] (A3) -- (A2);
\draw[postaction={decorate}] (A0) -- (A1);
\draw[postaction={decorate}] (A2) -- (A1);
\draw[postaction={decorate}] (A0) -- (A2);
\end{scope}

{\scriptsize
\draw (A0) node[anchor=east] {$0$};
\draw (A1) node[anchor=east] {$3$};
\draw (A2) node[anchor=west] {$2$};
\draw (A3) node[anchor=west] {$1$};
}
\draw (1.7,0.5) node {$\hookrightarrow$};
\end{scope}

\begin{scope}[tdplot_main_coords, >=latex, xshift=4.5cm, yshift=0.5cm]

\coordinate (A0) at (1,-1,0);
\coordinate (A1) at (0,1,1);
\coordinate (A2) at (0,1,-1);
\coordinate (A3) at (-1,-1,0);


\path[fill opacity=0.4, fill=blue!50] (A0) -- (A2) -- (A3) -- cycle;
\path[fill opacity=0.4, fill=blue!50] (A0) -- (A2) -- (A1) -- cycle;

\begin{scope}[decoration={
    markings,
    mark=at position 0.55 with {\arrow{>}}}
    ] 
\draw[postaction={decorate}] (A0) -- (A1);
\draw[postaction={decorate}] (A0) -- (A2);
\draw[postaction={decorate}] (A1) -- (A2);
\draw[postaction={decorate}] (A1) -- (A3);
\draw[postaction={decorate}] (A2) -- (A3);
\draw[postaction={decorate},dashed] (A0) -- (A3);
\end{scope}

{\scriptsize
\draw (A0) node[anchor=east] {$0$};
\draw (A1) node[anchor=south] {$1$};
\draw (A2) node[anchor=north] {$2$};
\draw (A3) node[anchor=west] {$3$};
}
\end{scope}
\end{tikzpicture}
\quad\quad\quad\quad
\begin{tikzpicture}[>=latex,scale=1.0, baseline=(current  bounding  box.center)]
\begin{scope}[scale=1.4]

\coordinate (A0) at (0,0);
\coordinate (A1) at (0,1);
\coordinate (A2) at (1,1);
\coordinate (A3) at (1,0);

\path[fill opacity=0.4, fill=blue!50] (A0) -- (A3) -- (A2) -- (A1) -- cycle;

\begin{scope}[decoration={
    markings,
    mark=at position 0.55 with {\arrow{>}}}
    ] 
\draw[postaction={decorate}] (A0) -- (A3);
\draw[postaction={decorate}] (A3) -- (A2);
\draw[postaction={decorate}] (A0) -- (A1);
\draw[postaction={decorate}] (A2) -- (A1);
\draw[postaction={decorate}] (A3) -- (A1);
\end{scope}

{\scriptsize
\draw (A0) node[anchor=east] {$0$};
\draw (A1) node[anchor=east] {$3$};
\draw (A2) node[anchor=west] {$2$};
\draw (A3) node[anchor=west] {$1$};
}
\draw (1.7,0.5) node {$\hookrightarrow$};
\end{scope}

\begin{scope}[tdplot_main_coords, >=latex, xshift=4.5cm, yshift=0.5cm]

\coordinate (A0) at (1,-1,0);
\coordinate (A1) at (0,1,1);
\coordinate (A2) at (0,1,-1);
\coordinate (A3) at (-1,-1,0);


\path[fill opacity=0.4, fill=blue!50] (A1) -- (A2) -- (A3) -- cycle;
\path[fill opacity=0.4, fill=blue!50] (A0) -- (A1) -- (A3) -- cycle;

\begin{scope}[decoration={
    markings,
    mark=at position 0.55 with {\arrow{>}}}
    ] 
\draw[postaction={decorate}] (A0) -- (A1);
\draw[postaction={decorate}] (A0) -- (A2);
\draw[postaction={decorate}] (A1) -- (A2);
\draw[postaction={decorate}] (A1) -- (A3);
\draw[postaction={decorate}] (A2) -- (A3);
\draw[postaction={decorate},dashed] (A0) -- (A3);
\end{scope}

{\scriptsize
\draw (A0) node[anchor=east] {$0$};
\draw (A1) node[anchor=south] {$1$};
\draw (A2) node[anchor=north] {$2$};
\draw (A3) node[anchor=west] {$3$};
}
\end{scope}
\end{tikzpicture}
\end{equation}

\end{ex}

For a
collection $\I \subset 2^I$ and a (semi-)simplicial space $X$, we write
\[
	X_{\I} := (\Delta^\I, X), \quad RX_{\I} := (\Delta^\I, X)_R  
\]
for the ordinary and derived spaces of $\Delta^\I$-membranes.
Using this notation, Proposition \ref{prop.decomp} reads:
\begin{prop}\label{prop.colim}
	Let $X$ be a (semi-)simplicial space and $\I', \I''\subset 2^I$. 
	Then there is a natural weak equivalence
	\[
	RX_{\I' \cup \I''} \stackrel{\simeq}{\longrightarrow} RX_{\I'} \times^R_{RX_{\I' \capd
	\I''}} RX_{\I''}.  \qed
	\] 
\end{prop}

The following statement   allows us to simplify, for some triangulations $\T$,
 the indexing diagrams  
for $RX_\T = (\Delta^\T, X)_R$ 
given by  \eqref{eq:derived-membranes-top}.

\begin{prop} Let $\Tc$ be a triangulation of a convex polytope $P$
with vertices in $I$, as defined in Example \ref{ex:convex-polytope}. Suppose that each geometric
simplex of $\Tc$ of dimension $\leq p$ is a face of $P$. Then we have the formula
\[
	RX_\T \simeq \hopro_{\{J\in\T, |J| \ge p+1\} }  X_J,
\]
expressing the derived membrane space as the homotopy limit over geometric simplices of $\T$ of
dimension $p+1$ or higher. 
\end{prop}
\begin{proof} Assume that $P$ is $d$-dimensional. Then 
\[
	\Delta^\T \cong \ind_{\{J\in\T, |J|\ge p+1\} }^{\Set_\Delta} \Delta^J
\]
for any $p \le d-2$. To deduce our statement from Proposition \ref{prop:acyclic-topological}, it
suffices to show that, under our assumptions, the diagram in the above limit is acyclic. 
To show this, we use Proposition \ref{prop:acyclic-crit} where $B$ is the full subcategory of
$\Delta_{\inj}/\Delta^{\T}$ spanned by the nondegenerate simplices of dimension $\ge p+1$. For $n \ge 0$ and $\sigma \in
\Delta^{\T}_n$, let $\overline{\sigma}: \Delta^k \hra \Delta^{\T}$ be the minimal simplex of which
$\sigma$ is a degeneration. If $k \ge p+1$, then $(\overline{\sigma}, \sigma)$ is a final object of
the category $B_\sigma$ which therefore has a contractible nerve. If $k \le p$, then the category
$B_\sigma$ is given by the poset of all geometric simplices of $\T$ that contain
$\overline{\sigma}$. By our assumption, $\overline{\sigma}$ is a $k$-dimensional face of $P$. This
implies that $B_{\sigma}$ can be identified with the poset of positive-dimensional cones of a
subdivision of the normal cone to $\overline{\sigma}$ in $P$ into convex subcones. Therefore, the
classifying space of $B_\sigma$ is contractible. 
\end{proof}
  
\begin{ex}
If each element of $I$ is a vertex of $P$, then we can disregard $0$-dimensional simplices when
computing $RX_\T$. For example, for the triangulation $\T$ of the square from Example \ref{I.2segal}, we get
\[
	RX_\T \simeq X_{\{0,1,2\}}\times^R_{X_{\{0,2\}} } X_{\{0,2,3\}} \cong X_2\times^R_{X_1} X_2,
\]
where the last homotopy fiber product taken with respect to the map $\partial_1: X_2\to X_1$ for the
first factor and $\partial_2: X_2\to X_1$ for the second factor. 
Note that this formula also follows immediately from Proposition \ref{prop.colim}. 
The space $RX_{\T'}$ is given by a similar homotopy fiber product but with respect to different face maps. 
\end{ex}

Let $\I \subset 2^I$ be a collection of subsets. Composing the pullback $X_I \to X_{\I}$ along the inclusion $\DI \subset
\Delta^I$ with the natural map $X_{\I} \to RX_{\I}$, we obtain a map
\begin{equation}\label{def.XImap}
	f_{\I}: X_I \lra RX_{\I}
\end{equation}
which we call the {\em $\I$-Segal map}.

\begin{ex}
\label{ex.1segal} 
	For the collection $\I_n$ of Example \ref{I.1segal}, the map $X_n \to RX_{\I_n}$ 
	reproduces the $n$th $1$-Segal map
	\[
	f_n: X_n \lra X_1 \times_{X_0}^R X_1 \times_{X_0}^R \dots \times_{X_0}^R X_1
	\]
	from Definition \ref{defi:1-segal}.
\end{ex}
 
\vfill\eject

\subsection{2-Segal spaces}
\label{subsec:2-segal-spaces} 

We fix a convex $(n+1)$-gon $P_n$ in $\RR^2$ with a chosen total order on the set of vertices,
compatible with the counterclockwise orientation of $\RR^2$. The chosen order provides a canonical
identification of the set of vertices of $P_n$ with the standard ordinal $[n]$. Any polygonal
subdivision $\mathcal P$ of $P_n$ as in Example \ref{ex:convex-polytope} can be identified with a
collection of subsets of $[n]$. Note that the class of collections thus obtained does not depend on
a specific choice of $P_n$: any two convex $(n+1)$-gons are combinatorially equivalent.  

As explained in \S \ref{subsec:membranes}, the subdivision $\Pc$ gives rise to a simplicial subset $\Delta^\Pc
\subset \Delta^n$ and to the corresponding $\Pc$-Segal map
\[
	f_\Pc: X_n \lra RX_\Pc=(\Delta^\Pc, X)_R 
\]
from \eqref{def.XImap}. The map $f_\Pc$ will be called the {\em $2$-Segal map corresponding to
$\Pc$}. We are now in a position to give the central definition of this work. 

\begin{defi}\label{def:2-segal-top} Let $X$ be a (semi-)simplicial space. 
We call $X$ a {\em $2$-Segal space} if, for every $n\geq 2$ and
every triangulation $\T$ of the polygon $P_n$, the corresponding $2$-Segal map $f_\T$ is a weak equivalence of
topological spaces. 
\end{defi}

The following statement is the analog of Proposition \ref{prop:1-segal-basic} in the context of
$2$-Segal spaces.

\begin{prop}\label{prop:2-segal-basic} 
Let $X$ be a (semi-)simplicial space. Then the following are equivalent:
	\begin{enumerate}
		 \item\label{item:2-segal-basic-1} $X$ is a $2$-Segal space.
		 \item\label{item:2-segal-basic-2} For every polygonal subdivision $\Pc$ of $P_n$ the
		 map $f_\Pc$
		 is a weak equivalence.
		 \item\label{item:2-segal-basic-3} For every $n \ge 3$ and $0 \le i < j \le n$, the map
		 \[
		 X_n \lra X_{\{0,1,\dots,i, j,j+1,\dots, n\}} \times^R_{X_{\{i,j\}}} X_{\{i,i+1,\dots,j\}}
		 \]
		 induced by the inclusions $\{0,1,\dots,i, j,j+1, \dots, n\},\{i,i+1,\dots,j\} \subset [n]$ is a
		 weak equivalence.  
		 \item\label{item:2-segal-basic-4} The same condition as in \ref{item:2-segal-basic-3} but we
		 only allow $i=0$ or $j=n$. 
	 \end{enumerate}
\end{prop}

\begin{proof} 
	Note, first of all, that we have obvious implications
	\ref{item:2-segal-basic-1}$\Leftarrow$\ref{item:2-segal-basic-2}$\Rightarrow$\ref{item:2-segal-basic-3}$\Rightarrow$\ref{item:2-segal-basic-4}. 
	The implication 
	\ref{item:2-segal-basic-1}$\Rightarrow$\ref{item:2-segal-basic-2} 
	follows inductively from the 2-out-of-3 property of weak
	equivalences. The implication 
	\ref{item:2-segal-basic-4}$\Rightarrow$\ref{item:2-segal-basic-1} 
	follows by a similar inductive argument, using the fact that each
	triangulation $\T$ of $P_n$ has a diagonal of the form $\{0,j\}$ or $\{i,n\}$.
\end{proof}

Following our general point of
view on $2$-Segal spaces as generalizations of categories
(see Introduction),
 we start with a basic comparison result between the notions of $1$-Segal and $2$-Segal spaces.

\begin{prop}\label{prop.1segal2segal} 
	Every $1$-Segal (semi-)simplicial space is $2$-Segal.
\end{prop}
\begin{proof}
	Let $X$ be a $1$-Segal space. Consider a triangulation $\T$ of $P_n$ and let 
	\[
		\I_n = \{ \{0,1\},\{1,2\}, \dots, \{n-1,n\} \}
	\]
	denote the collection from Example \ref{I.1segal}.
	The inclusions of simplicial sets $\Delta^{\I_n} \subset
	\Delta^{\T} \subset \Delta^I$ induce a commutative diagram
	\[
	\xymatrix{
	X_I \ar[r]^{f_{\T}} \ar[dr]_h & X_{\T} \ar[d]^g\\
	& X_{\I_n}
	}\text{.}
	\]
	We have to show that the $2$-Segal map $f_{\T}$ is a weak equivalence. Since $X$ is a $1$-Segal
	space, the $1$-Segal map $h$ is a weak equivalence and, by the two-out-of-three property, it suffices to show that $g$ is a
	weak equivalence. To prove this, we argue by induction on $n$.

	There exists a unique $0 < i < n$ such that $\{0,i,n\} \in \T$. We show how to argue for $1
	< i < n-1$, the cases $i \in \{1,n-1\}$ are similar but easier.
	We define the collections $\T_1 = \{ I \in \T |\; I \subset \{0,1,\dots,i\}\}$ and $\T_2 = \{ I \in \T |\; I \subset \{i,i+1,\dots,n\}\}$. Applying Proposition
	\ref{prop.colim} twice, we obtain a weak equivalence
	\[
	X_{\T} \stackrel{\cong}{\lra} RX_{\T_1} \times_{X_{\{0,i\}}}^R X_{\{0,i,n\}}
	\times_{X_{\{i,n\}}}^R RX_{\T_2}
	\text{.}
	\]
	Further, since $X$ is a $1$-Segal space, we obtain a weak equivalence
	\[
	X_{\{0,i,n\}} \stackrel{\simeq}{\lra} X_{\{0,i\}}
	\times_{X_{\{i\}}}^R X_{\{i,n\}} \text{.}
	\]
	Composing these maps, we obtain a weak equivalence
	\[
	g':\; RX_{\T} \stackrel{\simeq}{\lra} RX_{\T_1} \times_{X_{\{i\}}}^R RX_{\T_2}
	\text{.}
	\]
	By induction, we have weak equivalences
	\[
	g_1:\; RX_{\T_1} \stackrel{\simeq}{\lra} X_{\{0,1\}}
\times_{X_{\{1\}}}^R X_{\{1,2\}} \times_{X_{\{2\}}}^R \dots \times_{X_{\{i-1\}}}^R X_{\{i-1,i\}}
	\]
	and
	\[
	g_2:\; RX_{\T_2} \stackrel{\simeq}{\lra} X_{\{i,i+1\}}
\times_{X_{\{i+1\}}}^R X_{\{i+1,i+2\}} \times_{X_{\{i+2\}}}^R \dots \times_{X_{\{n-1\}}}^R
X_{\{n-1,n\}}\text{.}
	\]
	We conclude that the map $g = (g_1,g_2) \circ g'$ is a weak equivalence as well.
\end{proof}

\begin{prop}\label{prop:2-segal-product}
If $X, X'$ are $2$-Segal simplicial spaces, then so is $X\times X'$.
\end{prop}

\begin{proof} Let $\T$ be a triangulation of $P_n$. Then the map
\[
f_{\T, X\times X'}: (X\times X')_n= X_n\times X'_n \lra  R(X\times X')_\T = R X_\T \times
RX'_\T
\]
is the product of the maps $f_{\T, X}$ and $f_{\T, X'}$ and so is a weak equivalence. 
Similarly for the maps $f_{n,i}$. 
\end{proof}

\vfill\eject

\subsection{Proto-exact categories and the Waldhausen S-construction} 
\label{subsec:waldhausen-1}

In this section, we present the example which initiated our study of $2$-Segal spaces: the
Waldhausen S-construction originating in Waldhausen's work \cite{waldhausen} on algebraic
$K$-theory. We will generalize this example to the context of $\infty$-categories in \S \ref{subsec:waldhausen-exact-infty}. 

We start with generalizing Quillen's notion of an exact category to the non-additive case. Let $\Ec$ be a category.
A commutative square
\begin{equation}
	\label{eq:comm-square}
	\xymatrix{A_2 \ar[r]^{i} \ar[d]_{j_2} &A_1\ar[d]^{j_1}\cr
	A'_2 \ar[r]^{i'}&A'_1
	}
\end{equation}
in $\A$ is called {\em biCartesian}, if it is both Cartesian and coCartesian. 

\begin{defi}\label{def:proto-exact}
A {\em proto-exact category} is a category $\Ec$ equipped with
two classes of morphisms $\Men$, $\Een$, whose elements are called
{\em admissible monomorphisms} and {\em admissible epimorphisms}
such that the following conditions are satisfied:
\begin{enumerate}[label=(PE\arabic{*})]
\item $\Ec$ is pointed, i.e., has an object $0$ which is both initial and final.
Any morphism $0\to A$ is in $\Men$, and any morphism $A\to 0$ is in $\Een$. 

\item The classes $\Men, \Een$ are closed
under composition and contain all isomorphisms.  

\item A commutative square \eqref{eq:comm-square} in $\Ec$ with $i,i'$
admissible mono and $j_1, j_2$ admissible epi, is
Cartesian if and only if it is coCartesian.

\item Any diagram in $\Ec$
\[
A_1\buildrel j_1\over\lra A'_1\buildrel i'\over\lla A'_2
\]
with $i'$ admissible mono and $j_1$ admissible epi, can be completed
to a biCartesian square \eqref{eq:comm-square} with $i$ admissible mono
and $j_2$ admissible epi.

\item Any diagram in $\Ec$
\[
A'_2\buildrel j_2\over\lla A_2\buildrel i\over\lra A_1
\]
with $i$ admissible mono and $j_2$ admissible epi, can
be completed to a biCartesian square \eqref{eq:comm-square} with $i'$ admissible mono
and $j_1$ admissible epi.
\end{enumerate}
\end{defi}

\begin{ex} Any exact category in the sense of Quillen is proto-exact, with the
same classes of admissible mono- and epi-morphisms. In particular, any abelian
category $\Ac$ is proto-exact, with $\Men$ consisting of all
categorical monomorphisms, and $\Een$
consisting of all categorical epimorphisms in $\A$.
\end{ex}

\begin{ex}[(Pointed sets)]\label{ex:pointed-sets} 
\begin{exaenumerate} 
	\item Let $\Set_*$ be the category of pointed sets $(S, s_0)$
	and morphisms preserving base points. Let $\Men$ consist of all injections 
	of pointed sets and $\Een$ consist of surjections $p:(S, s_0)\to (T, t_0)$
	such that $|p^{-1}(t)|=1$ for $t\neq t_0$. This makes $\Set_*$ 
	into a proto-exact category. The full subcategory $\Fc\Set_*$ of finite pointed sets
	is also proto-exact.

	\item Let $A$ be a small category and $\Ec$ a proto-exact category. The category
	$\Fun(A, \Ec)$ of $A$-diagrams in $\Ec$ is again proto-exact, with the 
	componentwise definition of the classes $\Men, \Een$.
	In particular, the category of representations of a given quiver (or a monoid)
	in pointed sets is proto-exact. Such categories have been studied
	in \cite{szczesny:quivers, szczesny:semigroups} from the 
	the point of view of Hall algebras.
\end{exaenumerate}
\end{ex}

\begin{rem} The categories from Example
\ref{ex:pointed-sets} belong to the class of {\em belian categories},
a non-additive generalization of the concept of abelian categories
introduced by A. Deitmar \cite{deitmar:belian}. Each belian category $\Bc$
has two natural classes of morphsms: $\Men$, consisting of all
categorical monomorphisms, and $\Een$, consisting of
{\em strong epimorphisms}, i.e., morphisms $f: A\to B$ which
can be included into a biCartesian square
\[
\xymatrix{
K\ar[d] \ar[r] &A\ar[d]^f
\\
0\ar[r]&B
}
\]
see \cite{deitmar:belian}, Def. 1.1.4. 
In the examples we know, these two classes form a proto-exact structure. 
\end{rem}

\begin{ex}[(Quadratic forms)] By a  {\em quadratic space} we mean a pair
$(V,q)$, where $V$ is a finite-dimensional $\RR$-vector space and $q$ is a positive
definite quadratic form on $V$. A {\em morphism of quadratic spaces} $f: (V',q')\to(V,q)$
is an $\RR$-linear operator $f: V'\to V$ such that $q(f(v'))\leq q'(v')$ for each $v'\in V'$. 
We denote by $\Qc\Sc$ the category of quadratic spaces. 

Call an {\em admissible monomorphism} a morphism $i: (V',q')\to(V,q)$ 
in $\Qc\Sc$ such that
$i$ is injective and $q(i(v'))= q'(v')$ for $v'\in V'$, i.e., $q'$ is the pullback of $q$ via $i$.
Call an {\em admissible epimorphism} a morphism $j: (V, q)\to (V'', q'')$ in $\Qc\Sc$
such that $j$ is surjective and
$
q''(v'') =\min_{j(v)=v''} q(v)
$
for each $v''\in V''$, see  \cite{KSV} for more details. This makes $\Qc\Sc$ into a proto-exact category.

One has a similar proto-exact category $\Hc\Sc$
 of {\em Hermitian spaces} formed by finite-dimensional
$\CC$-vector spaces and positive-definite Hermitian forms. 
\end{ex}

\begin{ex}[(Arakelov vector bundles)]
\label{ex:arakelov}
  By an {\em Arakelov vector bundle on $\SZ$}
we mean a triple $E=(L, V, q)$, where $(V,q)$ is a quadratic space and $L\subset V$
be a lattice (discrete free abelian subgroup) of maximal rank.
 A morphism
$E'=(L',V', q')\to E=(L,V,q)$ is a morphism of quadratic spaces $f: (V',q')\to(V,q)$
such that $f(L)\subset L'$. This gives a category $\Bun(\SZ)$. 
The {\em rank} of $E$ is set to be $\rk(E)=\rk_\ZZ(L)=\dim_\RR(V)$.
The set of isomorphism classes of Arakelov bundles of rank $r$ is thus the
classical double coset space of the theory of automorphic forms
\[
\operatorname{Bun}_r(\SZ)= GL_r(\ZZ)\backslash GL_r(\RR)/O_r. 
\]
Call an {\em admissible monomorphism} a morphism $i: (L', V',q')\to(L, V,q)$ 
in $\Bun(\SZ)$ such that $i:(V', q')\to(V,q)$ is an admissible monomorphism in
$\Qc\Sc$ and $i: L'\to L$ is an embedding of a direct summand.
Call an {\em admissible epimorphism} a morphism $j: (L, V, q)\to (L'', V'', q'')$ in 
$\Bun(\SZ)$
such that $j: (V,q)\to(V'', q'')$ is an admissible epimorphism in $\Qc\Sc$
and $j: L\to L''$ is surjective, see \cite{KSV} for more details.
 This makes $\Bun(\SZ)$ into a proto-exact
category.

One similarly defines proto-exact categories consisting of vector bundles
on other arithmetic schemes compactified at the infinity in the 
sense of Arakelov, see \cite{manin, soule} for more background.
\end{ex}
  
We now give a version of the classical construction of Waldhausen \cite{waldhausen, gillet} which
associates to a proto-exact category $\Ec$ a simplicial space.  Let $T_n=\Fun([1], [n])$ be the
poset (also considered as a category) formed by ordered pairs $(0\leq i\leq j\leq n)$, with
$(i,j)\leq (k,l)$ iff $i\leq k$ and $j\leq l$.  A functor $F: T_n\to\Ec$ is therefore a 
commutative diagram 
\[
	 \xymatrix{
		 F(0,0)\ar[r]&F(0,1)\ar[r]\ar[d]&\cdots\ar[r]&F(0, n-1)\ar[r]\ar[d]& F(0,n)\ar[d]
		 \cr
		 &F(1,1)\ar[r]&\cdots\ar[r]&F(1, n-1)\ar[r]\ar[d]&F(1,n)\ar[d]\cr
		 &&\ddots&\vdots\ar[d]&\vdots\ar[d]\cr
		 &&&F(n-1, n-1)\ar[r]&F(n-1, n)\ar[d]\cr
		 &&&&F(n,n)
	 }
\]
formed by objects $F(i,j)\in\Ec$ and morphisms $ F(i, j)\to F(k, l)$ given whenever $i\leq k$ and
$j\leq l$.  Let $\Wc_n(\Ec)$ be the full subcategory in $\Fun(T_n, \Ec)$ formed by diagrams $F$ as
above satisfying the following properties:
\begin{enumerate}[label=(W\arabic{*})]
\item For every $0 \le i \le n$, we have $F(i, i) \simeq 0$. 

\item All horizontal morphisms are in $\Men$, and all vertical
morphisms are in $\Een$. 

\item Each square in the diagram is biCartesian.
\end{enumerate}

Let $\SW_n(\Ec)$ be the subcategory in $\Wc_n( \Ec)$ formed by all objects and their isomorphisms.
One easily verifies that the construction $\SW_n \Ec$ is functorial in $[n]$ and defines a
simplicial category (groupoid) $\SW_{\bullet} \Ec$. We call it the {\em Waldhausen simplicial
groupoid} of $\Ec$. Assume that $\Ec$ is small.  Passing to the classifying spaces, we then obtain a
simplicial space $S_\bullet(\Ec)= B\SW_\bullet(\Ec)$ which we call the {\em Waldhausen space} of
$\Ec$.

\begin{prop}\label{prop:waldhausen-proto-segal}
For any small proto-exact category $\Ec$ the Waldhausen space $S_\bullet(\Ec)$ is $2$-Segal. 
\end{prop}

\begin{proof} Informally, an object of $\Sc_n(\Ec)$ or $\Wc_n(\Ec)$ can be  
seen as an object $F(0, n)$ of $\Ec$
equipped with an ``admissible filtration" of length $n$
together with a specified choice of quotient objects. More precisely,  
Let $\Men_n$, resp. $\Een_n$
be the groupoid formed by chains of $(n-1)$ admissible mono- resp. epi-morphisms
 and by isomorphisms of such chains.

\begin{lem}\label{leq:waldhausen-filtration-abelian}
  (a) The functor $ \mu_n: \Sc_n(\Ec)\to \Men_n$ which associates to $F$ the subdiagram
$$F(0, 1)\lra F(0, 2)\lra \cdots \lra F(0, n),$$
 is an equivalence.  

(b) Similarly, the functor $\epsilon_n: \Sc_n(\Ec)\to \Een_n$ which associates to $F$ the subdiagram
$$F(0,n)\lra F(1,n)\lra \cdots \lra F(n-1, n),$$
 is an equivalence.  
\end{lem}

\begin{proof} (a) Given a sequence of objects $F(0,i)$ and monomorphisms as 
stated, we first put
$F(i,i)=0$ for all $i$, and then define the $F(i,j)$ inductively,
filling the second (from top) row left to right, then the third (from top) row left to right, etc. 
 by successively forming coCartesian squares using (PE4) :
\[
F(1,2) =\varinjlim{}^\Ec \left\{
\begin{matrix}
F(0,1)&\to&F(0,2)\cr
\downarrow&&\cr
F(1,1)&&
\end{matrix}
\right\}, \quad 
F(1,3) =\varinjlim{}^\Ec \left\{
\begin{matrix}
F(0,2)&\to&F(0,3)\cr
\downarrow&&\cr
F(1,2)&&
\end{matrix}
\right\}\quad\text{etc.}
\]
 This gives a functor which is
quasi-inverse to $\mu_n$.

(b) Similar procedure, by successively forming Cartesian squares using (PE3).  
\end{proof}

We now prove the proposition. Write $\Sc_n$ for $\Sc_n(\Ec)$.
To prove that $S(\Ec)$ is $2$-Segal, it suffices to verify the conditions in Part
\ref{item:2-segal-basic-4} of Proposition 
\ref{prop:2-segal-basic}. Using
  Proposition \ref {prop:2lim-holim}, we rewrite these conditions in terms of
  2-fiber products of categories. That is, it is enough to prove that the functors
 \[
  \begin{gathered}
\Phi_j :  \Sc_n \lra \Sc_{\{ 0, 1, ..., j \}}
  \times^{(2)}_{\Sc_{\{0,j\}}} \Sc_{\{j, j+1, ..., n\}}, \quad j=2, ..., n-1,\\
  \Psi_i: \Sc_n \lra \Sc_{\{0,1,..., i,n\}}\times^{(2)}_{\Sc_{\{i,n\}}}\Sc_{\{i, i+1, ..., n\}}, \quad
  i=1, ..., n-2,
  \end{gathered}
  \]
  are equivalences. 
    In order to prove that $\Phi_j$
 is an equivalence, we include it into a commutative diagram
\[
\xymatrix{
 \Sc_n \ar[r]^{\hskip -2cm\Phi_j}
 \ar[d]_{\mu_n} &\Sc_{\{ 0, 1, ..., j \}}
  \times^{(2)}_{\Sc_{\{0,j\}}} \Sc_{\{j, j+1, ..., n\}}\ar[d]^{\mu_j\times\mu_{n-j+1}}
  \\
  \Men_n \ar[r]^{\hskip -1cm\phi_j}& \Men_j\times^{(2)}_{\Sc_{\{0,j\}}} \Men_{n-j+1}
}
\]
with vertical arrows being equivalences by Lemma
\ref{leq:waldhausen-filtration-abelian}. 
  Now, the functor $\phi_n$
is obviously an equivalence: two objects
\[
\{F(0,1) \to ... \to F(0,j)\}\in\Men_j, \quad 
 \{F'(0,j)\to ...\to F'(0,n)\}\in\Men_{n-j+1}
 \] together with an isomorphism $F(0,j)\to F'(0,j)$
combine canonically to give an object of $\Men_n$. Therefore $\Phi_j$
is an equivalence as well.

In order to prove that $\Psi_i$ is an equivalence, we include it into a similar diagram
with bottom row
\[
\Een_n \buildrel \psi_i\over\lra \Een_i\times^{(2)}_{\Sc_{\{i,n\}}} \Een_{n-i+1}
\]
via the equivalences $\epsilon_n$ and $\epsilon_i\times \epsilon_{n-i+1}$.
Again, $\psi_i$ is an equivalence for obvious reasons. \end{proof}

\vfill\eject

\subsection{Unital 2-Segal spaces}

Recall from the Universality Principle \ref{Universality-principle} that in the context of $1$-Segal spaces, 
a semi-simplicial space $X$
corresponds to a nonunital higher category. The existence of a simplicial structure on $X$ implies the existence of units.
For $2$-Segal spaces, the situation is more subtle. The existence of a simplicial
structure is not sufficient to give a reasonable notion of units -- we require an additional
condition which we introduce in this section.

For $n \ge 2$ and $0 \le i \le n-1$, consider the commutative square
\[
	\xymatrix{
	 [n-1] & \ar[l] \{i\}\\
	 [n] \ar[u]^{\sigma_i} & \ar[l] \ar[u] \{i,i+1\}.
	}
\]
in $\Delta$, where $\sigma_i$ denotes the $i$-th degeneracy map, so that $\sigma_i$ is surjective and
$\sigma_i^{-1}(i) = \{i,i+1\}$. Given a simplicial space $X$, we have an induced square
\begin{equation}\label{eq:unitalsquare}
\xymatrix{
X_{n-1} \ar[r] \ar[d] & X_{\{i\}} \ar[d]\\
	X_n \ar[r] & X_{\{i,i+1\}}
}
\end{equation}
of topological spaces.

\begin{defi}\label{def:unital-2-segal-top}
Let $X$ be a $2$-Segal simplicial space. We say $X$ is {\em unital} if, for every $n \ge 2$ and 
$0\leq i < n$, the square \eqref{eq:unitalsquare} is homotopy Cartesian.
\end{defi} 

We have the following strengthening of Proposition \ref{prop.1segal2segal}.

\begin{prop}\label{prop.1segalunital2segal} 
	Every $1$-Segal simplicial space is a unital $2$-Segal simplicial space. 
\end{prop}
\begin{proof} 
	We can refine the square \eqref{eq:unitalsquare} to the diagram
	\[
		\xymatrix{
			X_{n-1} \ar[r]^-{f_{n-1}} \ar[d] & X_{\{0,1\}} \times_{X_{\{1\}}}^R \dots \times_{X_{\{n-2\}}}^R X_{\{n-2,n-1\}} \ar[r] \ar[d] & X_{\{i\}} \ar[d]\\
			X_n \ar[r]^-{f_n} & X_{\{0,1\}} \times_{X_{\{1\}}}^R \dots 
			\times_{X_{\{n-1\}}}^R X_{\{n-1,n\}} \ar[r] & X_{\{i,i+1\}},
		}
	\]
	where $f_{n-1}$ and $f_n$ are $1$-Segal maps and hence by assumption weak equivalences.
	In particular, the lefthand square of \eqref{eq:unitalsquare} is homotopy Cartesian. The
	righthand square of \eqref{eq:unitalsquare} is homotopy Cartesian by inspection, and
	we therefore deduce that \eqref{eq:unitalsquare} is homotopy Cartesian as well.
\end{proof}

\begin{exas} 
\begin{exaenumerate} 
	\item As a special case of Proposition \ref{prop.1segalunital2segal}, we obtain that the nerve
	of a small category is a unital $2$-Segal simplicial set. 

	\item The Waldhausen space of a proto-exact category $\Ec$ is a unital $2$-Segal simplicial space.

	\item If $X, X'$ are unital $2$-Segal simplicial spaces, then so is their product $X\times X'$.
\end{exaenumerate}
\end{exas}
 
\vfill \eject

\subsection{The Hecke-Waldhausen space and relative group cohomology}
\label{subsec:hecke-waldhausen}

For a small groupoid $\Gc$ we denote by $\pi_0(\Gc)$ the set of isomorphism classes of objects of
$\Gc$.
Let $G$ be a group acting on the left on a set $E$.  Then we have the {\em quotient groupoid}
$G\backslash\hskip -1mm \backslash E$. It has $\Ob(G\bbs E)=E$, with $\Hom_{G\bbs E} (x,y)$ being
the set of $g\in G$ such that $gx=y$. Thus the source and target diagram of $G\bbs E$ has the form
\begin{equation}\label{eq:source-target-quotient}
	\xymatrix{
	\Mor(G\bbs E) = G\times E \ar@<.5ex>[r]^{\hskip .5cm s}\ar@<-.5ex>[r]_{\hskip .5cm t}&
	E=\Ob(G\bbs E) 
	}, \quad s(g,x)=x, t(g,x)=gx.
\end{equation}
In particular, $\pi_0(G\bbs E)$ is the orbit space $G\backslash E$.

For any $n\geq 0$ consider $E^{n+1}$ with the diagonal action of $G$ and put $\Sc_n(G,E) = G\bbs E^{n+1}$ 
to be the corresponding quotient groupoid.  The collection of categories
$(\Sc_n(G,E))_{n\geq 0}$ is made into a simplicial category $\Sc_\bullet(G, E)$ in an obvious way:
the simplicial operations (functors) $\partial_i, s_i$ are defined by forgetting or repeating
components of an element from $E^{n+1}$. We define $S_n(G,E)$ to be the classifying space of
$\Sc_n(G,E)$, so $S_\bullet(G,E)$ is a simplicial space which we call the {\em Hecke-Waldhausen
space} associated to $G$ and $E$.

\begin{ex}\label{ex:hecke-waldhausen}
\begin{exaenumerate}
	\item Let $E=G/K$, where $K\subset G$ is a subgroup. Let $e\in G/K$ be the distinguished
		point (corresponding to $K$ itself). Any element $(x_0, ..., x_n)\in E^n$ can be
		brought by an appropriate $g\in G$ to an element $(x'_0, ..., x'_n)$ with $x_0=e$,
		and such a $g$ is defined uniquely up to left multiplication by $K$. This means that
		we have an equivalence of categories
		\[
		\Sc_n(G, E) =  G\bbs (G/K)^{n+1} \simeq K\bbs (G/K)^n. 
		\]
		In particular, 
		\[
		\pi_0(\Sc_0(G,E)) = \pt, \quad \pi_0(\Sc_1(G, E)) = 
		K\backslash G/K.
		\]

		\hskip .5cm  If $K=G$, then $E=G/G=\pt$, and $S_n (G, G/G)=BG$ for each $n$. In other words,
		$S_\bullet(G, G/G)$ is the constant simplicial space corresponding to
		$BG$. 

		\hskip .5cm If $K=\{1\}$, then $E=G$. In this case $G$ acts freely on each $G^{n+1}$, so
		$S_n(G,G)$ is the discrete category corresponding to the set
		$G\backslash G^{n+1}=N_nG$. In other words, $S_\bullet(G,G)=\disc{NG}$
		is the discrete simplicial space corresponding to the simplicial set $NG$. 

	\item Let $\k$ be a field, and $\k\ltb t\rtb $ resp. $\k\llb t\rlb$ be the ring, resp. field of
		formal Taylor, resp. Laurent series with coefficients in $\k$. Fix $r\geq 1$ and let
		$G= GL_r(\k\llb t\rlb )$.  Put $K=GL_r(k\ltb t\rtb )$. Then $E=G/K$ can be identified
		with the set of lattices (free $\k\ltb t\rtb$-submodules of rank $r$) $L\subset \k\llb
		t\rlb ^r$. This set is partially ordered by inclusion, and the action of $G$
		preserves the order.  Put $$E^{n+1}_\leq =\bigl\{ (L_0, ..., L_n)\in E^{n+1} \bigl|
		L_0\subset \cdots \subset L_n\bigr\},$$ and further put $\Sc_n^\leq(G,E) = G\bbs
		E^{n+1}_\leq$. This gives a simplicial subcategory $\Sc_\bullet^\leq (G,E) \subset
		\Sc_\bullet(G,E)$. On the other hand, let $\Ac$ be the abelian category of
		finite-dimensional $\k\ltb t \rtb$-modules. We then have a functor of simplicial categories
		$$\Sc_\bullet^\leq(G,E) \lra \Sc_\bullet(\Ac), \quad (L_0\subset\cdots\subset L_n)
		\longmapsto (L_j/L_i)_{i\leq j}.$$ This makes it natural to think of $S_\bullet(G,E)$ for
		general $G$ and $E$ as a group-theoretic analog of the Waldhausen space. 
\end{exaenumerate}
\end{ex}

\begin{prop}\label{prop:hecke-waldhausen-1-segal}
The simplicial space $S_\bullet(G, E)$ is $1$-Segal.
\end{prop}
\begin{proof}
Let $\Sc_n=\Sc_n(G,E)$. By Proposition \ref{prop:2lim-holim}, it suffices to verify the $1$-Segal
condition at the level of groupoids, i.e., show that the natural functor  
\[
	\phi_n: \Sc_n \lra \Sc_1\times_{\Sc_0}^{(2)}\Sc_1\times_{\Sc_0}^{(2)} \cdots \times_{\Sc_0}^{(2)} \Sc_1
	\quad
	\text{($n-1$ times})
\]
is an equivalence of categories.  
Explicitly, an object of the iterated 2-fiber product on the right is a set of data 
\begin{equation}\label{eq:set-of-data}
	 \bigl( (x_0^{(0)}, x_1^{(0)}), (x^{(1)}_1, x^{(1)}_2), ..., 
	(x^{(n-1)}_{n-1}, x^{(n-1)}_n), g_1, ..., g_{n-1}\bigr), \quad x^{(i)}_\nu\in E, g_i\in G,
	g_i(x^{(i)}_{i+1})=x^{(i+1)}_{i+1}. 
\end{equation}
A morphism from such a set of data to another one, say to
\[
	  \bigl( (y_0^{(0)}, y_1^{(0)}), (y^{(1)}_1, y^{(1)}_2), ..., 
	(y^{(n-1)}_{n-1}, y^{(n-1)}_n), h_1, ..., h_{n-1}\bigr)
\]
is a sequence $(\gamma_1, ..., \gamma_{n-1})$ of elements of $G$ such that 
\begin{equation}\label{eq:gamma-i-condition}
	\begin{gathered}
	\gamma_i(x^{(i)}_\nu) = y^{(i)}_\nu, \quad i=0, ..., n-1, \nu=i, i+1;\\
	\gamma_{i+1}g_i=h_i\gamma_i, \quad i=0, ..., n-2.
	\end{gathered}
\end{equation}
The functor $\phi_n$ takes an object $(x_0, ..., x_n)\in\Sc_n = G\bbs E^{n+1}$ into the system of
data consisting of 
\begin{equation}\label{eq:phi-n-group-explicit}
	\begin{gathered}
	 x^{(0)}_0=x_0, x^{(0)}_1=x^{(1)}_1=x_1, \cdots, x^{(n-2)}_{n-1}=x^{(n-1)}_{n-1}=x_{n-1}, 
	  x^{(n-1)}_n=x_n,\\
	 g_1 = \cdots = g_{n-1} = 1. 
	 \end{gathered}
\end{equation}
A morphism $(x_0, ..., x_n)\to (y_0, ..., y_n)$ in $\Sc_n$ corresponding to $g\in G$ such that
$g(x_i)=y_i$, is sent into the sequence $(\gamma_1, ..., \gamma_{n-1})$ with all $\gamma_i=g$.

We now prove that $\phi_n$ is fully faithful.  Let $(x_0, ..., x_n)$ and $(y_0, ..., y_n)$ be two
objects of $\Sc_n$ and $(\gamma_1, ..., \gamma_{n-1})$ be a morphism between the corresponding
systems \eqref{eq:phi-n-group-explicit}. Then the second condition in \eqref{eq:gamma-i-condition}
gives $\gamma_{i+1}=\gamma_i$ for each $i=0, ..., n-2$, so all $\gamma_i=g$ for some $g\in G$,
whence the statement. 

We next prove that $\phi_n$ is essentially surjective. Indeed, for any object \eqref{eq:set-of-data}
of the iterated 2-fiber product as above we have an isomorphism
\[
	\begin{gathered}
	\phi_n\bigl( x^{(0)}_0, g_1^{-1}(x^{(1)}_1), 
	g_1^{-1} g_2^{-1}(x^{(2)}_2), ..., g_1^{-1} ...g_{n-1}^{-1}(x^{(n-1)}_{n-1}), 
	g_1^{-1} ...g_{n-1}^{-1}(x^{(n-1)}_{n})\bigr)\lra\\
	\bigl( (x_0^{(0)}, x_1^{(0)}), (x^{(1)}_1, x^{(1)}_2), ..., 
	(x^{(n-1)}_{n-1}, x^{(n-1)}_n), g_1, ..., g_{n-1}\bigr)
	\end{gathered}
\]
given by $\gamma_i= g_ig_{i-1} ... g_1$. This finishes the proof of the proposition.
\end{proof}

We now consider $|S_\bullet(G,E)|$, the realization of the simplicial space $S_\bullet(G,E)$. As
each space $S_n(G,E)$ is, in its turn, the realization of the nerve of $\Sc_n(G,E)$, we have a
bisimplicial set $S_{\bullet\bullet}$, with
\[
	S_{nm}=N_m\Sc_n(G,E)
\]
being the set of chains of $m$ composable morphisms in $\Sc_n(G,E)$. Then
$|S_\bullet(G,E)|=\|S_{\bullet\bullet}\|$ is the double realization (or, what is the same, the
realization of the diagonal) of $S_{\bullet\bullet}$.

\begin{prop}\label{prop:waldhausen-BG}
The space $|S_\bullet(G,E)|$ is homotopy equivalent to $BG$. 
\end{prop}
\begin{proof}
Consider the simplicial space $S'_\bullet$ formed by the realizations of the slices of
$S_{\bullet\bullet}$ with respect to the second simplicial direction: $S'_m= |S_{m\bullet}|$. Then
$|S'_\bullet| = \|S_{\bullet\bullet}\| = |S_\bullet(G,E)|$. 
To prove our statement, it suffices to construct, for each $m$, a homotopy equivalence between
$S'_m$ and the set $N_mG=G^m$ (considered as a discrete topological space), in a way compatible with
simplicial operations. To do this, notice that \eqref{eq:source-target-quotient} applied to $E^n$
instead of $E$, implies that $S_{mn}= G^m\times E^n$, and the simplicial operations in the
$n$-direction consist of forgetting or repeating elements of $E$. In other words,
$S_{m\bullet}=G^m\times (\Delta^E)'$, where $(\Delta^E)'$ is the fat simplex (Example
\ref{ex:nerve,fat-simplex}), known to be contractible.  So $S'_m =  G^m\times |(\Delta^E)'|\to G^m$
is a  homotopy equivalence.
\end{proof} 

\begin{rem} In fact, Proposition \ref{prop:waldhausen-BG} can be refined to identify the higher
	category modelled by $X = S_{\bullet}(G,E)$. A straightforward calculation shows that, the
	homotopy category $\h X$ is given by the category with set of objects $E$, and, for every
	pair of elements $e,e'$, the set $\Hom_{\h X}(e,e')$ can be identified with $G$. The
	composition is given by the composition law of the group $G$. Therefore, the category $\h X$
	is equivalent to the groupoid with one object and endomorphism set $G$. Further, the mapping
	spaces \eqref{eq:segalmapping} associated to $X$ are unions of contractible components. These observations imply
	that the higher category modelled by the $1$-Segal space $X$ is in fact weakly equivalent to
	the {\em ordinary} category $\h X$. Therefore, the completion of the $1$-Segal space $X$ is
	simply given by the constant simplicial space $BG$. In particular, the higher category
	associated to $X$ does not depend in any way on the action of $G$ on the set $E$.

	However, since the $1$-Segal space $X$ is {\em not} complete, it captures information
	which is lost after passing to the completion. This information is retained if we interpret $X$
	as a $2$-Segal space. As such, the Hecke-Waldhausen space will reappear in \S
	\ref{subsec:groupoids-classical-hall}, where we explain its relevance for Hecke algebras. 

\end{rem}

\begin{rem} Let $K\subset G$ be a subgroup. The simplicial space $S_\bullet(G, G/K)$ with realization $BG$ is a
group-theoretic analog of the filtered complex used to construct the Hochschild-Serre spectral
sequence (HSSS) for a Lie algebra with respect to a Lie subalgebra \cite{fuks}. More precisely, for
a $G$-module $A$  the {\em relative cohomology groups} $H^n(G, K; A)$ are defined by means of the
cochain complex
\[
	 C^n(G,K; A) =\Map_G ( (G/K)^{n+1}, A),
	 \quad
	 (df)(\bar g_0, ..., \bar g_{n+1}) = \sum_{i=0}^{n+1} (-1)^i f(\bar g_0, ..., \widehat{\bar {g}}_i, ..., 
	 \bar g_{n+1}), 
\]
see \cite{adamson, hochschild}. On the other hand, $A$ defines an obvious functor from
$\Sc_n(G,G/K)$ to abelian groups (each object goes to $A$, each morphism corresponding to $g\in G$ goes to
$g: A\to A$) and so gives a local system $\underline A_n$ on $S_n(G, G/K)$. These local systems are
compatible with the simplicial maps so give a local system $\underline A_\bullet$ on $S_\bullet(G,G/K)$ 
and thus a spectral sequence 
\[
	\begin{gathered}
	E_1^{pq} =   H^q(S_p(G, G/K); \underline A_p) \Rightarrow 
	H^{p+q}(|S_\bullet(G,G/K)|; \underline A_\bullet)=
	H^{p+q}(G;A),\\
	E_2^{p0} = H^p(G,K; A).
	\end{gathered}
\]
This is an analog of the group-theoretic HSSS for the case of a not necessarily normal subgroup.
Cf. \cite{butler-horrocks} where a spectral sequence like this was constructed using ``relative
homological algebra". 
\end{rem}
 
\vfill\eject

\section{Discrete 2-Segal spaces}
\label{sec:discsegal}

 In this chapter we study the $2$-Segal condition in the more immediate,
  non-homotopy setting:
 that of semi-simplicial {\em sets}, not spaces. A semi-simplicial set $Y$ will be
 called $2$-Segal, if $\disc Y$, the discrete semi-simplicial space associated to $Y$, is
 $2$-Segal. This simply means that for any $n\geq 2$ and 
 any triangulation $\T$ of the $(n+1)$-gon $P_n$, the map
 $
 f_\T: Y_n\lra Y_\T
 $
  is a bijection of sets.

  Let $\C$ be any category with finite projective limits. A semi-simplicial
  object $Y = (Y_n)\in \C_{\Delta_\inj}$ will be called $2$-Segal, if, for any object $U\in\C$
  the semi-simplicial set 
  \[
  \Hom_\C(U, Y) =  (\Hom_\C(U, Y_n))_{n\geq 0}
  \]
   is $2$-Segal. 
  Alternatively, for any triangulation $\T$ as above we define an object $Y_\T\in \C$
  as a projective limit in $\C$, and the condition is that each 
    $f_\T$ is an isomorphism in $\C$. 
In particular, we can speak about $2$-Segal schemes, analytic spaces etc. 
 
\subsection{Examples: Graphs, Bruhat-Tits complexes}
\label{subsec:graphs}
 
 We start with some very simple examples and then provide several generalizations.
 
 \begin{exa} (a) We say that a simplicial set $Y$ is {\em 1-skeletal},
 if all simplices of $Y$ of dimension $\geq 2$ are degenerate.  
 We will also refer to 1-skeletal simplicial
 sets   as {\em oriented graphs}.
  If $\T$ is a triangulation of 
  $P_n$ as above, we include $f_\T$ into a commutative diagram
  \[
  \xymatrix{
  Y_n \ar[r]^{f_\T}\ar[dr]& Y_\T\ar[d] \cr
  & Y_{\{0,n\}}=Y_1
  }
  \]
  If $Y$ is 1-skeletal, then the two other arrows in the diagram are bijections,
  which implies that $f_\T$ is a bijection, so 
    $Y$ is $2$-Segal.

    (b) An oriented graph $Y$ is $1$-Segal
  if and only if it has no pair of composable arrows
  (this includes arrows whose source and target coincide, as such an arrow
  is considered composable with
  itself). 
  \end{exa}
  
  Applying Proposition \ref{prop:2-segal-product}, we obtain the following:
  
  \begin{cor}
 \label{cor:product-graphs-2-segal}
  Any finite product of oriented graphs is $2$-Segal. 
  \end{cor}

 \begin{defi}\label{def:Z-order}
  (a) A $\ZZ_+$-order on a set $I$ is a pair $(\leq, F)$, where $\leq$ is a partial
 order on $I$, and $F: I\to I$ is an order-preserving map. A $\ZZ$-order is a
 $\ZZ_+$-order such that $F$ is a bijection.

 (b) Given a $\ZZ_+$-ordered set $I$, its {\em building} $\Bld(I)$ is defined to be the simplicial subset
 in $\N(I, \leq)$, the nerve of $I$, whose $n$-simplices are chains
   \[
 a_0\leq a_1\leq\cdots\leq a_n\leq F(a_0). 
 \]
 \end{defi} 
 
 \begin{prop}
 For any $\ZZ_+$-ordered set $I$, the building $\Bld(I)$ is $2$-Segal.
  \end{prop}
  
  \begin{proof} 
   Let $\T$ be a triangulation of the $(n+1)$-gon $P_n$.
   As $\Bld(I)$ is a simplicial subset in $N(I)$, we have the
   commutative diagram
   \[
   \xymatrix{\Bld_n(I)\ar[r]^{\subset}\ar[d]_{f_{\T, \Bld(I)}}& N_n(I)\ar[d]
   ^{f_{\T, N(I)}}
   \cr
  \Bld_\T(I) \ar[r]^{\subset} &N_\T(I)
    }  
   \]
   As the nerve of any category, $N(I)$ is $1$-Segal and therefore $2$-Segal,
   so $f_{\T, N(I)}$ is a bijection. This implies that $f_{\T, \Bld(I)}$
   is an injection. 
   
   Let us prove surjectivity. Let $\sigma: \Delta^\T\to \Bld(I)$ be a membrane in
   $\Bld(I)$ of type $\T$. As $N(I)$ is $2$-Segal, there is a unique
   $n$-simplex $\Sigma\in N_n(I)$ which maps to $\sigma$ under
   $f_{\T, N(I)}$. This simplex is a chain of elements
   $a_0\leq \cdots \leq a_n$. Let us prove that $\Sigma\in \Bld_n(I)$,
   i.e., that the additional condition $a_n\leq T(a_0)$ is satisfied.  
   For this, look at the unique triangle $\{0,j,n\}$ of $\T$ which contains the
   side $\{0,n\}$ of $P_n$. The image of this triangle under $\sigma$ is a
   2-simplex of $\Bld(I)$, i.e., a triple of elements of $I$ of the form
   \[
   a_0\leq a_j\leq a_n\leq T(a_0),
   \]
   so the additional condition is indeed satisfied. \end{proof}

  \begin{exa}
    Let $\k$ be a field. Denote
  by $\mathbf K=\k\llb t\rlb$ and $\Oc = \k\ltb t \rtb$
  be the field of formal Laurent series and the ring of
  formal Taylor series
  with coefficients in $\k$. Fix a finite-dimensional $\mathbf K$-vector space $V$
  and let $d=\dim (V)$.
  By a {\em lattice} in $V$ we mean a free $\Oc$-submodule $L\subset V$
  of rank $d$. Let  
    $\Gamma = \Gamma(V)$ be the set of all lattices in 
   $V$. The group $GL(V)$ acts transitively on $\Gamma$. 
    For the coordinate vector space $V=\mathbf K^d$ the set $\Gamma$
   is identified with the coset space $GL_d(\Oc)\backslash GL_d(\mathbf K)$. 
   The set $\Gamma$ is partially ordered by inclusion. 
   Define a bijection $F:\Gamma\to\Gamma$ by $F(L)=t^{-1}L$.
   With this data, $\Gamma$ becomes a $\ZZ$-ordered set.
   The building $\Bld(\Gamma)$ is known as the
   {\em Bruhat-Tits building} of $V$ and denoted $\on{BT}(V)$.
   By the above, $\on{BT}(V)$ is $2$-Segal. 
   
    \end{exa}
  
    \begin{exa}\label{ex: apartment} Consider the simplicial subset
    $A\subset \on{BT}(\mathbf K^d)$ whose
   vertices are lattices of the form
   \[
   t^{i_1}\Oc \oplus \cdots \oplus t^{i_d}\Oc, \quad (i_1,..., i_d)\in\ZZ^d
   \]
   and higher-dimensional simplices are all chains of such lattices
   satisfying the condition
   \[
  L_0\subset L_1\subset \cdots \subset L_n\subset t^{-1}L_0.
  \]
 This subset is known as the
   {\em standard apartment} in the building $\on{BT}(\mathbf K^d)$.
    
    Let $I_\ZZ$ be the oriented graph with the set of vertices $\ZZ$ and one oriented
  edge from $i$ to $i+1$ for each $i$, so that $|I_\ZZ|$ is the subdivision of $\RR$ into
  unit intervals:
  \[
  \cdots\lra\bullet\lra\bullet\lra\bullet\lra\cdots
  \]
  Then $A$ is isomorphic to the 
    $d$th Cartesian power $I_\ZZ^d$, which is $2$-Segal by 
    Corollary \ref{cor:product-graphs-2-segal}. 
    
  The 
   building $\on{BT}(K^d)$ 
   is the union of the translations of $A$ under
   the elements of $GL_d(\mathbf K)$.  
  See \cite{brown.buildings, goldman-iwahori}
   for more details. 
   \end{exa}

 \begin{prop}\label{prop:2-segal-quotient}
  Let $Y$ be a simplicial set and $G$ be a group acting on
 $Y$ by automorphisms of simplicial sets. Suppose that the $G$-action
 on each $Y_n$ is free. Then the quotient simplicial set
 \[
 G\backslash Y = (G\backslash Y_n)_{n\geq 0}
 \]
 is $2$-Segal if and only if $Y$ is $2$-Segal. 
 
 \end{prop}

 \begin{proof} Suppose $Y$ is $2$-Segal. To prove that $G\backslash Y$
 is $2$-Segal, means to show that for any $n\geq 2$ and any triangulation $\T$ of $P_n$,
  any morphism of simplicial sets $\sigma: \Delta^\T\to G\backslash Y$  
  can be uniquely extended to a morphism
  $\Sigma: \Delta^n\to G\backslash Y$ (such a morphism in the same as an 
    $n$-simplex). To indeed show this, suppose that $\T,\sigma$ are given. 
Because $G$ acts on each $Y_n$ freely, the canonical projection $\pi: Y\to G\backslash Y$  
induces an unramified covering of geometric realizations. Since $P_n=|\Delta^\T|$
is simply connected, $\sigma$ has a lifting $\widetilde\sigma$, as in the diagram:
\[
\xymatrix{
\Delta^n \ar[r]^{\widetilde\Sigma}& Y\ar[d]^{\pi}
\cr
\Delta^\T\ar[r]^{\sigma}\ar[u]^{\text{incl.}}\ar[ur]^{\widetilde\sigma}& G\backslash Y
}
\]
Since $Y$ is $2$-Segal, $\widetilde\sigma$ can be uniquely extended to a morphism
$\widetilde \Sigma$ as in the diagram. Then $\Sigma = \pi\circ\widetilde\Sigma$ is a
required extension of $\sigma$. This proves that an extension exists. To show uniqueness,
suppose $\Sigma', \Sigma ''$ are two extensions of $\sigma$. Because both $\Delta^\T$
and $\Delta^n$ are simply connected, we can find liftings 
$\widetilde\Sigma',
\widetilde\Sigma'': \Delta^n\to Y$
which restrict to the same lifting $\widetilde\sigma: \Delta^\T\to Y$ of $\sigma$
and must therefore be equal. This equality implies that $\Sigma'=\Sigma''$. 
 
 This proves that $G\backslash Y$ is $2$-Segal, if $Y$ is. The proof in the opposite
 direction is similar and left to the reader. \end{proof}

 We apply this to the action of $G=\ZZ$ on $\on{BT}(V)$ generated by the 
 transformation $F$ which acts on simplices as follows:
 \[
 F(L_0, ..., L_n) = (tL_0, ..., tL_n).
 \]
 Clearly, this action is free. The apartment $I_\ZZ^n$ is preserved under the action, and $F$ acts on
 $\ZZ^d$, the set of its vertices, by adding the vector $(1, ..., 1)$. We denote by
 \[
	 \overline{\on{BT}}(V) =\ZZ\backslash \on{BT}(V), \quad \overline I_\ZZ^d= 
		 \ZZ\backslash I_\ZZ^d
 \]
 the quotient simplicial sets. By Proposition \ref{prop:2-segal-quotient} they are $2$-Segal.

  \vfill\eject

\subsection{The twisted cyclic nerve}
\label{subsec:twisted-cyclic-nerve}

Let $\C$ be a small category and $F: \C\to\C$ be an endofunctor.
The {\em $F$-twisted cyclic nerve} of $\C$ is the simplicial set $N^F\C$
with $N_n^F\C$ being the set of chains of arrows in $\C$ of the form
\begin{equation}\label{eq:chain-sigma}
\Sigma =
 \bigl\{ x_0\buildrel u_{01}\over\lra x_1\buildrel u_{12}\over\lra x_2
 \buildrel u_{23}\over\lra \cdots \buildrel u_{n-1, n}\over\lra x_n
 \buildrel u_{n0}\over\lra F(x_0)\bigr\}.  
\end{equation}
The simplicial structure is defined as follows.
For $\Sigma$ as above and $1\leq i\leq n$ 
the chain $\partial_i(\Sigma)$ is obtained from $\Sigma$
by omitting $x_i$ and composing the two arrows going in and out of it.
For $i=0$ we put
\[
\partial_0(\Sigma) =
 \bigl\{ x_1\buildrel u_{12}\over\lra x_2\buildrel u_{23}\over\lra x_3
 \buildrel u_{34}\over\lra \cdots \buildrel u_{n-1, n}\over\lra x_n
 \buildrel u_{n0}\over\lra F(x_0)\buildrel F(u_{01})\over\lra F(x_1)\bigr\}.
\]
For any $0\leq i\leq n$ 
the chain $s_i(\Sigma)$ is obtained from $\Sigma$ by replacing $x_i$ 
with the fragment
 $x_i\buildrel \Id \over\lra x_i$. One verifies directly that the simplicial identities hold.

 \begin{exas} 
 (a) If $\C=(I,\leq)$ is a poset, then $F$ is a monotone map,
 so $(I,F)$ is a $\ZZ_+$-ordered set and $N^F\C = 
 \Bld(I)$ is the building associated to it 
 (Definition \ref{def:Z-order}).

 (b) The twisted cyclic nerve $N^{\Id}\C$ corresponding to $F=\Id_\C$,
 will be called simply the {\em cyclic nerve} of $\C$ and denoted
 $\NC(\C)$, see \cite{drinfeld}.  In the case when $\C$ has one object
 (i.e., reduces to a monoid),  the cyclic nerve is a particular case
 of the {\em cyclic bar-construction} of Waldhausen
 \cite[\S 2.3]{waldhausen-II}.

(c) Assume that $\C$ is a groupoid.
In this case $\NC(\C)$ is identified with the nerve of 
  the functor category
 \[
L\C=\Fun (\ZZ, \C),
\]
where $\ZZ$ is the additive group of integers considered as a 
category with one object. This category is a groupoid, known as
the {\em inertia groupoid} of $\C$. This observation is essentially due to D. Burghelea
   \cite{burghelea} who treated the case 
    when $\C = G$ is a group considered as a category with one
object. In this case
 \[
 \Ob(L\C) = G, \quad \Hom_{L\C}(g, g') =  \{ u\in G: g'=ugu^{-1}\},
 \]
 so isomorphism classes of objects in $L\C$ are the same as
 conjugacy classes in $G$. 

  \end{exas}

   \begin{thm}\label{thm:cyclic-nerve}
   For any small category $\C$ and any endofunctor $F: \C\to\C$, 
   the simplicial set $N^F\C$ is
   $2$-Segal. 
      \end{thm}
      
      \begin{proof} Denote $X=N^F\C$. Let $\T$ be
      a triangulation of the polygon $P=P_n$ with vertices
      $0,1,...,n$. We need to prove that $f_{\T, X}: X_n\to X_\T$
      is a bijection. By induction in $n$ we can assume that the
      statement is true for any triangulation of any $P_m$ with
      $m<n$. Now, looking at the unique triangle $\{0,i,n\}$
      of $\T$ containing the edge $\{0,n\}$, we see that there
      is $0<i<n$ such at least one of the two pairs
      $\{0,i\}$, $\{i,n\}$ is an edge of $\T$. 
      Assume that the first pair is an edge, the second case is treated similarly. 
      
      \begin{lem} The map
   \[
   g: X_n\lra X_{\{0,1,..., i\}}\times_{X_{\{0,i\}}} X_{\{0,i,i+1,..., n\}}
   \]
   is a bijection. 
   \end{lem}    
   
 The lemma implies bijectivity of $f_{\T, X}$. Indeed, the edge $\{0,i\}$
subdivides $P$ into two subpolygons: $P'$, with
vertices $0,1,..., i$, and $P''$, with vertices $0,i, i+1, ..., n$.
The triangulation $\Tc$ induces then triangulations $\T', \T''$
of $P', P''$, and $X_\T = X_{\T'}\times_{X_{\{0,i\}}} X_{\T''}$. 
 The map $f_{\T, X}$ is therefore the composition of $g$ and 
\[
f_{\T', X}\times f_{\T'', X}: X_{\{0,1,..., i\}}\times_{X_{\{0,i\}}} X_{\{0,i,i+1,..., n\}}\lra 
X_{\T'}\times_{X_{\{0,i\}}} X_{\T''}= X_\T,
\]
which is a bijection by the inductive assumption. 

\noindent {\sl Proof of the lemma:} We first prove that $g$ is injective.
  Given an $n$-simplex of $X$, i.e., a chain $\Sigma$ as in
\eqref{eq:chain-sigma}, the two simplices corresponding to it via $g$, are the
chains
\begin{equation}\label{eq:two-chains}
\begin{gathered}
\Sigma'=\bigl\{x_0\buildrel u_{01} \over\lra x_1\buildrel u_{12}\over\lra \cdots 
\buildrel u_{i,i-1}\over\lra x_i\buildrel u_{i0} \over\lra F(x_0)\bigr\},
\cr
\Sigma'' =\bigl\{x_0\buildrel v_{0i} \over\lra x_i\buildrel v_{i,i+1}\over\lra \cdots 
\buildrel v_{n,n-1}\over\lra x_n\buildrel v_{n0} \over\lra F(x_0)\bigr\},
\end{gathered}
\end{equation}
such that 
  $v_{p,p+1}=u_{p,p+1}$, $i\leq p\leq n-1$ and, in addition,
  \begin{equation}\label{eq:v-0-i-u-0-i}
v_{0i}=u_{i-1,i}\circ u_{i-2,i-1}\circ\cdots\circ u_{01}, \quad u_{i0}=
v_{n0}\circ v_{n-1,n}\circ \cdots \circ v_{i, i+1}. 
\end{equation}
   Among the arrows of these two
chains, we find all the arrows in $\Sigma$, which shows the injectivity of $g$.

We next prove that $g$ is surjective. Suppose we have two chains
$\Sigma'$ and $\Sigma''$ as in \eqref{eq:two-chains}. The fact that
the simplices represented by these chains have a common edge $\{0,i\}$,
means that we have \eqref{eq:v-0-i-u-0-i}
 But this precisely means that putting $u_{p,p+1}=v_{p,p+1}$, $i\leq p\leq n-1$,
we define a chain $\Sigma$ such that $g(\Sigma)=(\Sigma', \Sigma'')$.
This finishes the proof of the lemma and of Theorem \ref{thm:cyclic-nerve}.
\end{proof}

 \vfill\eject
  
\subsection{The multivalued category point of view}\label{subsec:mult-cat}

Let $\Cc$ be a category with fiber products.  
By a {\em span} (or {\em correspondence}) between two objects
$Z, Z'$ of $\Cc$ we will mean a diagram
\begin{equation}\label{eq:span-in-C}
\sigma =\bigl\{ Z\buildrel s\over \lla W\buildrel p \over\lra Z'\bigr\}
\end{equation}
and write 
$\xymatrix{
    \sigma: Z\ar@{~>}[r]&Z'}$. 
    All spans from $Z$ to $Z'$ form a category $\Spanl_\Cc (Z,Z')$, with morphisms being commutative
    diagrams
    \[
    \xymatrix{
	 &\ar[dl]_{s_1} W_1\ar[dd]_f \ar[dr]^{p_1}&\\
	 Z & & Z'\\
	 & \ar[ul]^{s_2}W_2\ar[ur]_{p_2}&
	 } 
    \]
    The {\em composition} of two spans
    \[
 \sigma' =\bigl\{ Z'\buildrel s' \over \lla W' \buildrel p'\over\lra Z''\bigr\}
  \quad \text{ and } \quad
 \sigma =\bigl\{ Z\buildrel s\over \lla W\buildrel p\over \lra Z'\bigr\}
    \]
    is defined by taking the fiber product:
  \[
  \sigma'\circ\sigma =\bigl\{ Z \buildrel s \over\lla W\buildrel \on{pr}_W\over\lla
  W\times_{Z'}W'\buildrel \on{\pr}_{W'}\over\lra W'\buildrel p'\over\lra Z\bigr\}.  
    \]
    Composition is associative: for any three spans of the form
    \[
    \xymatrix{
    Z\ar@{~>}[r]^{\sigma}&Z'\ar@{~>}[r]^{\sigma'}&Z''\ar@{~>}[r]^{\sigma''} & Z'''
    }
    \]
   the spans $(\sigma''\circ\sigma')\circ\sigma$ and $\sigma''\circ(\sigma'\circ\sigma)$
   are connected by a natural isomorphism in the category $\Spanl_\Cc (Z, Z''')$.

   One can express these properties more precisely
    by saying that the collection of categories $\Spanl_\Cc (Z, Z')$
   and composition functors connecting them, forms a bicategory $\Spanl_\Cc$ with
   the same objects as $\Cc$. This bicategory was introduced by Benabou
   \cite{benabou}.

   \begin{rem} A span \eqref {eq:span-in-C} in the category of sets can be thought of
   as a ``multivalued map" from $Z$ to $Z'$, associating to $z\in Z$ the set
   $s^{-1}(z)$ (which is mapped into $Z'$ by $p$). As $s^{-1}(z)$ may be empty,
   this understanding of
    ``multivalued" includes ``partially defined". 
    \end{rem}

    \begin{defi}
    A {\em multivalued category} (a $\mu$-{\em category}, for short) is a
    a weak category object in the bicategory $\Spanl_\Set$. 
      \end{defi}

    Explicitly, a $\mu$-category is a
    datum
    $\Cen$ consisting of:
    \begin{enumerate}[label=($\mu$C\arabic{*})]
  \item Sets $\Cen_0, \Cen_1$ (objects and morphisms of $\Cen$) and maps
  $s,t:\Cen_1\to \Cen_0$ (source and target). 
  
  \item A span $\xymatrix{
   \mu: \Cen_1\times_{\Cen_0}\Cen_1
    \ar@{~>}[r]&\Cen_1}$ in $\Set$ (multivalued composition).
    
    \item An isomorphism (associator)
      \[
    \alpha: \mu\circ(\mu\times\Id) \lra \mu\circ(\Id\times\mu)
    \quad\quad \text{in} \quad \Spanl_{\Set}(\Cen_1\times_{\Cen_0}\Cen_1\times_{\Cen_0} \Cen_1, \Cen_1).
    \]

    \item A map $e: \Cen_0\to\Cen_1$ (unit)
    and
     isomorphisms
    \[
    \lambda: \mu\circ(et, \Id) \lra \Id,\quad \rho: \mu\circ(\Id, e s)\lra\Id
    \]
  in $\Spanl_\Set(\Cen_1,\Cen_1)$ (left and right unitality).
   \end{enumerate}
  
\noindent These data are required to satisfy the properties
familiar from the theory of monoidal categories and bicategories:

\begin{enumerate}[resume,label=($\mu$C\arabic{*})]
  \item (Mac Lane pentagon constraint)  
The diagram 
\[
\xymatrix{
& \mu\circ((\mu\circ(\mu\times\Id))\times\Id)
\ar[dr]^{\mu\circ(\alpha\times\Id)}
\ar[dl]_{(\mu\times\Id\times\Id)^*\alpha}
&
\cr
\mu\circ(\mu\times\mu)
\ar[d]_{(\Id\times\Id\times\mu)^*\alpha}
&& \mu\circ((\mu\circ(\Id\times\mu))\times\Id)
\ar[d]^{(\Id\times\mu\times\Id)^*\alpha}
\cr
\mu\circ(\Id\times(\mu\circ(\Id\times\mu)))&&\ar[ll] _{\mu\circ(\Id\times\alpha)}
\mu\circ(\Id\times(\mu\circ(\mu\times\Id)))
}
\]
in the category
$\Spanl_{\Set}(\Cen_1\times_{\Cen_0}\Cen_1\times_{\Cen_0}\Cen_1\times_{\Cen_0}\Cen_1, \Cen_1)$
is commutative.

\item (Unit coherence) The following diagram
 in $\Spanl_\Set(\Cen_1\times_{\Cen_0}\Cen_1, 
\Cen_1)$ is commutative:
\[
\xymatrix{
\mu\circ(\mu\circ(\Id,es)\times\Id)
\ar[d]_{\alpha(\Id\times et\times\Id)}
\ar[drr]^{\mu(\rho\times\Id)}
&&
\\
\mu\circ(\Id\times\mu\circ(et,\Id))\ar[rr]_{\hskip 1.3cm\mu(\Id\times\lambda)}&&\mu
}
\]

\end{enumerate}

\begin{remcom}
\begin{exaenumerate}
	\item By a $\mu$-{\em semicategory} we will mean a ``$\mu$-category but possibly without units",
 i.e., the datum of ($\mu$C1-3) satisfying the condition ($\mu$C5).

 \item As usual, we will use the term $\mu$-{\em monoid}, resp $\mu$-{\em semigroup}
 to signify a $\mu$-category (resp. a $\mu$-semicategory) $\Cen$ with one object,
 i.e., with $\Cen_0=\pt$.

 \item Given any category $\C$ with fiber products, one can speak about
 $\mu$-categories
 in $\C$ by replacing morphisms and spans in $\Set$ by morphisms and spans in $\C$. 
 Similarly for $\mu$-semicategories, $\mu$-monoids, $\mu$-semigroups.
\end{exaenumerate}
 \end{remcom}

  \begin{defi}
   Let $\Cen, \Den$ be two $\mu$-categories with composition spans
  \[
 \xymatrix{
   \mu_\Cen: \Cen_1\times_{\Cen_0}\Cen_1
    \ar@{~>}[r]&\Cen_1,} \quad\quad
    \xymatrix{
   \mu_\Den: \Den_1\times_{\Den_0}\Den_1
    \ar@{~>}[r]&\Den_1.
    } 
     \]
     A (single-valued) functor $F: \Cen\to\Den$ is a datum of maps
     $F_i: \Cen_i\to\Den_i, i=0,1$, commuting with $s,t,e$, and of a morphism of spans
     \[
     \widetilde F_2: F_1\circ\mu_\Cen \lra \mu_\Den\circ (F_1\times_{F_0}F_1), \quad
     \widetilde F_2\in\Spanl_\Set(\Cen_1\times_{\Cen_0}\Cen_1, \Den_1),
     \]
     commuting with $\alpha, \lambda$ and $\rho$. We denote by $\mu\Cat$
     the category formed by $\mu$-categories and their functors.
    \end{defi}
    
    Similarly, a functor $F: \Cen\to\Den$ between two $\mu$-semicategories
    is a datum of $F_0, F_1$ commuting with $s,t$ and of $\widetilde F_2$ commuting with $\alpha$. 
    We denote by $\mu \Sc\Cat$ the resulting category of $\mu$-semicategories.
    
    The following is the main result of this section. 
    
    \begin{thm}\label{thm:2-segal-mu-cat}
    (a) The category of $2$-Segal semi-simplicial sets is equivalent to $\mu\Sc\Cat$.
    
    (b) The category of unital $2$-Segal simplicial sets is equivalent to $\mu\Cat$. 
    
    \end{thm}
    
    \begin{proof} (a) Let $X$ be a $2$-Segal semi-simplicial set. We associate
 to $X$ a $\mu$-semicategory $\Cen=\Cen(X)$ as follows. 
 We put $\Cen_i=X_i$ for $i=0,1$. Further, we define the composition span in $\Cen$ 
 to be the diagram
 \begin{equation}\label{eq:fund-corr}
 \mu =\bigl\{
\xymatrix{
 X_1\times_{X_0} X_1 && \ar[ll]_{\hskip .7cm f_2=(\partial_0, \partial_2)}X_2\ar[r]^{\partial_1} &X_1
}
\bigr\}.
 \end{equation}
  which we call the {\em fundamental correspondence}
 of $X$.
   To construct the associator $\alpha$,  
  let $\nu$ be the span
 \[
 \nu = \bigl\{
 \xymatrix{
 X_1\times_{X_0}X_1\times_{X_0} X_1&&&\ar[lll] _{\hskip 1.3cm
 (\partial_{\{2,3\}}, \partial_{\{1,2\}}, \partial_{\{0,1\}})
 }
 X_3\ar[r]^{\partial_{\{0,3\}}}&X_1
 \bigr\}.
 }
 \]
 Consider the two triangulations of the 4-gon: 
 \[
	 \T'=\bigl\{ \{0,1,3\}, \{1,2,3\}\bigr\} \quad \text{and} \quad 
	 \T''=\bigl\{\{0,1,2\}, \{0,2,3\}\bigr\} 
 \]
  or, pictorially,
   
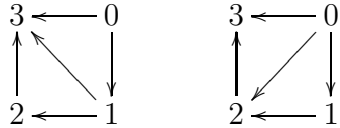
\[
   \xymatrix{
 3&0\ar[l] \ar[d]
 \cr
 2\ar[u]&1\ar[l]\ar[ul]}
 \quad\quad\quad
 \xymatrix{
 3&0\ar[l]\ar[dl] \ar[d]
 \cr
 2\ar[u] &1\ar[l]
 }
 \]
 Since $X$ is $2$-Segal, these triangulations define isomorphisms of spans
 \[
   \mu\circ(\mu\times\Id)
 \buildrel f_{\Tc'}\over
  \lla \nu 
  \buildrel f_{\Tc''}\over
  \lra 
  \mu\circ(\Id\times\mu).    
 \]
We d efine $\alpha$ to be the morphism of spans
 \begin{equation}\label{eq:span-associator}
 \alpha = f_{\Tc''}\circ f_{\Tc'}^{-1}: \mu\circ(\mu\times\Id)
 \buildrel \sim\over \lra
  \mu\circ(\Id\times\mu).
\end{equation}

\begin{prop}\label{prop:associator-pentagon-2-segal-set}
 Let $X$ be any $2$-Segal semi-simplicial set. 
Then $\alpha$ satisfies the pentagon constraint ($\mu$C5),
thus making $\Cen=\Cen(X)$ into a $\mu$-semicategory.
\end{prop}

\begin{proof}
Denote the Mac Lane pentagon in ($\mu$C5) by $\Mc$.
Consider also the pentagon $P_4$ (not to be confused with $\Mc$) and its five
triangulations, which we denote by $\Tc_i, i=0, ..., 4$,
 so that $\T_i$ consists of 3 triangles
 with common vertex $i$. 
 Then the five spans in the vertices of $\Mc$ have the form
 \begin{equation}\label{eq:span-T-i}
 \xymatrix{
X_1\times_{X_0}X_1\times_{X_0}X_1\times_{X_0}X_1 &&&& X_{\Tc_i}
\ar[llll]_{\hskip 2cm
(\partial_{\{0,1\}}, \partial_{\{1,2\}}, \partial_{\{2,3\}}, \partial_{\{3,4\}})
}
 \ar[r]^{\partial_{\{0,4\}}} & X_1
}, \quad i=0, ..., 4. 
 \end{equation}
 For instance, $\mu\circ((\mu\circ(\mu\times\Id))\times\Id)$ corresponds to $\Tc_4$, etc. 
  The morphisms in $\Mc$ corresponds to elementary flips of 
 triangulations which connect $\T_i$ with $\T_{i\pm 2 (\on{mod} 5)}$.

 Let now $\Pi$ be the poset of all polyhedral subdivisions of the pentagon $P_4$, ordered
 by refinement, so that the $\Tc_i$ are the maximal elements. 
 For $j\equiv i\pm 2(\on{mod} 5)$
 we denote by $\Pc_{ij}$ the subdivision consisting of one 4-gon and 1-triangle of which
 both $\T_i$ and $\T_j$ are refinements. These $\T_i$, $\Pc_{ij}$ together with the subdivision 
 consisting of $P_4$ alone, exhaust all elements of $\Pi$, so the nerve of $\Pi$ looks
 like the barycentric subdivision of a pentagon or, more precisely of the Mac Lane pentagon $\Mc$.

 As in Example \ref{ex:convex-polytope} (b), we view any subdivision $\Pc\in\Pi$
 as a subset of $2^{[4]}$ and associate to it the simplicial subset $\Delta^\Pc\subset\Delta^4$
 and the corresponding space (set, in our case) of membranes 
 $X_\Pc\leftarrow X_4$. The correspondence $\Pc\mapsto X_\Pc$ is thus a covariant functor
 from $\Pi$ to $\Set$, so we have a commutative diagram having the shape of the barycentric
 subdivision of $\Mc$:

 \begin{equation}\label{eq:huge-pentagon}
  { 
\begin{tikzpicture}

 \node (origin) at (0,0) (origin) {$X_{\{01234\}}= X_4$}; 
\node (P0) at (0:4cm) { $X_{\T_0} $}; 
\node (P2) at (1*72:4cm) {$X_{\T_2}$};
 \node (P4) at (2*72:4cm) {$X_{\T_4}$}; 
 \node (P1) at (3*72:4cm) {$X_{\T_1}$}; 
  \node (P3) at (4*72:4cm) {$X_{\T_3}$}; 
  \path (P0) -- (P2)
  node [midway] (P02) {$X_{\Pc_{02}}$}; 
  \path (P2) -- (P4) node [midway] (P24) {$X_{\Pc_{24}}$}; 
  \path (P4) -- (P1) node [midway] (P14) {$X_{\Pc_{14}}$ };
  \path (P1) -- (P3) node [midway] (P13) {$X_{\Pc_{13}}$};
  \path (P0) -- (P3) node [midway] (P03) {$X_{\Pc_{03}}$}; 
\draw [-latex, thick] (P02) -- (P0);
\draw [-latex, thick] (P02) -- (P2);
 \draw [-latex, thick] (origin) -- (P0);
  \draw [-latex, thick] (origin) -- (P1);
   \draw [-latex, thick] (origin) -- (P2);
    \draw [-latex, thick] (origin) -- (P3);
     \draw [-latex, thick] (origin) -- (P4);
      \draw [-latex, thick] (origin) -- (P02);
       \draw [-latex, thick] (origin) -- (P24);
        \draw [-latex, thick] (origin) -- (P14);
         \draw [-latex, thick] (origin) -- (P13);
          \draw [-latex, thick] (origin) -- (P03);
 \draw [-latex, thick] (P24) -- (P2); 
  \draw [-latex, thick] (P24) -- (P4); 
   \draw [-latex, thick] (P14) -- (P4); 
    \draw [-latex, thick] (P14) -- (P1); 
     \draw [-latex, thick] (P13) -- (P3); 
      \draw [-latex, thick] (P13) -- (P1); 
       \draw [-latex, thick] (P03) -- (P0); 
        \draw [-latex, thick] (P03) -- (P3); 
         
 \end{tikzpicture}
 }
 \end{equation}

Because $X$ is $2$-Segal, all the maps in this diagram are bijections. 
Extending \eqref{eq:span-T-i}, for each $\Pc\in\Pi$ we define $F(\Pc)$
to be the span in $\Set$ given by 
\[
 \xymatrix{
X_1\times_{X_0}X_1\times_{X_0}X_1\times_{X_0}X_1 &&&& X_{\Pc}
\ar[llll]_{\hskip 2cm
(\partial_{\{0,1\}}, \partial_{\{1,2\}}, \partial_{\{2,3\}}, \partial_{\{3,4\}})
}
 \ar[r]^{\partial_{\{0,4\}}} & X_1
} 
\]
  The $F(\Pc)$ form then a commutative
diagram of isomorphisms in $\Spanl_{\Set}$ of the same shape $\Pi$ as 
\eqref{eq:huge-pentagon}. This diagram contains the Mac Lane pentagon $\Mc$: any arrow
$F(\Tc_i)\to F(\Tc_j)$ in $\Mc$ can be seen as the composite arrow
\[
F(\Tc_i)\lla F(\Pc_{ij})\lra F(\Tc_j)
\]
after reversing the isomorphism on the left. Therefore $\Mc$ is commutative.
\end{proof}

  It is clear that a morphism of $2$-Segal semi-simplicial
sets $X\to Y$ defines a functor $\Cen(X)\to\Cen(Y)$.

We now describe a reverse construction, associating to any $\mu$-semicategory $\Cen$
a semi-simplicial set $\N= \N\Cen$. 
We put $\N_i=\Cen_i$ for $i=0,1$, and 
  define $\N_2$ as the middle term
of the composition span:
\[
\mu_\Cen =\bigl\{ \Cen_1\times_{\Cen_0}\Cen_1
\buildrel p\over \lla \N_2\buildrel q\over \lra \Cen_1 \bigr\}.
\]
Define maps $\partial_i: \N_2 \to\N_1$, $i=0,1,2$, by putting $\partial_1=q$ and
$(\partial_0, \partial_2)=p$. Let also $\partial_1=s, \partial_2=t: \N_1\to \N_0$. 
These data make $(\N_p)_{p\leq 2}$ into a semi-simplicial 2-skeleton, i.e.,
into a functor $\Delta_\inj[0,2]^\op\to\Set$, where $\Delta_\inj[0,2]\subset\Delta_\inj$ is the
full subcategory
on objects isomorphic to $[0], [1], [2]$. Therefore, for any triangulation $\Tc$ of the polygon $P_n$ we can form
the set
\[
\N_\Tc=\pro _{\{\Delta^p\hookrightarrow \Delta^\Tc\}_{p\leq 2}}\N_p.
\]
Recall that triangulations of $P_n$ correspond to bracketed products of $n$ factors.
Note further that $\N_\T$
fits into a span
\[
\underbrace{ \Cen_1\times_{\Cen_0} \cdots \times_{\Cen_0}\Cen_1}_n \lla \N_\T \lra\Cen_1
\]
which is nothing but the bracketed iteration of $\mu$ corresponding to the triangulation $\Tc$. 
So the same argument
as in the Mac Lane coherence theorem shows 
that we have a transitive system of bijections $f_{\T, \T'}: \N_\T\to\N_{\T'}$
 coming from iterated applications of $\alpha$.
 In particular, for any $0\leq i<j<k<l\leq n$, the associator $\alpha$ gives a bijection
 \[
 \alpha_{ijkl}: N_{\{i,j,l\}}\times_{N_{\{i,j\}}} N_{\{j,k,l\}} \lra
 N_{\{i,k,l\}}\times_{N_{\{i,k\}}} N_{\{i,j,k\}}.
 \]
 Consider the limit
 \[
\widetilde \N_n = \pro _{\{\Delta^p\hookrightarrow \Delta^n\}_{p\leq 2}}\N_p.
 \]
 For an element $\bf x$ of $\widetilde \N_n$ 
and $0\leq i<j<k\leq n$ we will denote by $x_{ijk}$ the component of $x$ corresponding to
the embedding $\Delta^2\to\Delta^n$ sending $0\mapsto i$, $1\mapsto j$ and $2\mapsto k$. 
A monotone injection 
 $\phi: [m]\to[n]$  
gives rise to the
 embedding of simplices, also denoted $\phi: \Delta^m\to\Delta^n$. 
 Then composing with $\phi$
 defines a map $\widetilde N_n\to \widetilde\N_m$, so $\widetilde \N$ is a semi-simplicial set.

\begin{prop}\label{prop:nerve-mu-cat}
 For $n\geq 0$ let 
$\N_n\Cen\subset \widetilde \N_n$ consist of $\bf x$ such that
\[
(x_{ikl}, x_{ijk}) =\alpha_{ijkl}(x_{ijl}, x_{jkl}). 
\]
Then $\N\Cen = (\N_n\Cen)_{n\geq 0}$ is a semi-simplicial subset in $\widetilde N$.
It is $2$-Segal.
\end{prop}

\begin{proof} Both statements follow from the 
Mac Lane coherence argument (the transitivity of the bijections $f_{\T,\T'}$ above).
\end{proof}

 We call $\N\Cen$ the {\em nerve} of the $\mu$-semi-category $\Cen$.

 Further, let $F: \Cen\to\Den$ be a functor of $\mu$-semicategories. Note that the datum $\widetilde F_2$
 in $F$ contains the same information as a map $F_2$ making the following diagram
 commutative:
 \[
 \xymatrix{
 \Cen_1\times_{\Cen_0}\Cen_1
 \ar[d]_{F_1\times_{F_0} F_1} &\ar[l] \N_2\Cen\ar[d]^{F_2}\ar[r]&\Cen_1
 \ar[d]^{F_1}\\
  \Den_1\times_{\Den_0}\Den_1&\ar[l] \N_2\Den\ar[r]&\Den_1
 }
 \]
 This implies that $F$ gies rise to a morphism of semi-simplicial sets $\N F: \N\Cen\to\N\Den$. 
 Therefore we have a functor from $\mu\Sc\Cat$ to the category of $2$-Segal semi-simplicial spaces.
 It is now straightforward to verify that the two functors are inverse to each other,
 thus finishing the proof of part (a) of Theorem \ref {thm:2-segal-mu-cat}.

To prove part (b), assume that $X$ is a unital $2$-Segal simplicial set. We then make $\Cen=\Cen(X)$
into a $\mu$-category as follows.
The map $e:\Cen_0\to\Cen_1$ is defined to be the degeneracy map
$s_0: X_0\to X_1$. 
To construct the isomorphism $\lambda$, notice that $\mu\circ(et, \Id)$,
is, by definition, the span $X_1\leftarrow W\to X_1$ at the bottom of the following diagram
obtained by forming a Cartesian square:
\begin{equation}\label{eq:diagram-mu-et-id}
\xymatrix{
X_1\times_{X_0} X_1& \ar[l]_{\hskip 0.7cm (\partial_0, \partial_2)} X_2 \ar[r]^{\partial_1}&X_1
\\
X_1\ar[u]^{(s_0\partial_0, \Id)}
&\ar[l] W\ar[u]^u\ar[ur]&
}
\end{equation}
We claim that $W=X_1$, i.e., that the square
\[
\xymatrix{
X_1\times_{X_0} X_1 & \ar[l]_{\hskip 0.7cm (\partial_0, \partial_2)} X_2 \\
X_1\ar[u]^{(s_0\partial_0, \Id)}& \ar[l]_{\Id}X_1\ar[u]_{s_0}
}
\]
is Cartesian. But this follows at once from the Cartesianity of the square
\[
\xymatrix{
X_1&\ar[l]_{\partial_0}X_2
\\X_0\ar[u]^{s_0}&\ar[l]^{\partial_0} X_1\ar[u]_{s_0}
}
\]
which is the instance $n=2, i=1$ of the square \eqref {eq:unitalsquare}
in the definition of ``unital". 
 This defines $\lambda$. The construction of $\rho$ is similar, by using the instance $n=2, i=0$
of the square \eqref {eq:unitalsquare}. Further, the condition ($\mu$C6) follows by
considering the instance $n=3, i=1$ of the same square.

Conversely, let $\Cen $ be a $\mu$-category. We then make $\N=\N\Cen$
into a simplicial set as follows. We define $s_0: \N_0\to\N_1$
to be $e$ and $s_0, s_1: \N_1\to\N_2$ to be given by 
the inverses of $\lambda$ and $\rho$ respectively.  
More precisely, $\mu\circ(et, \Id)$ is given by the correspondence a the bottom
of the diagram \eqref{eq:diagram-mu-et-id} with $X_i = \N_i$,
so $\lambda$ gives a bijection $\lambda: W\to N_1$, and we 
put $s_0=u\lambda^{-1}$. Similarly for $s_1$ and $\rho$. 
This makes $(\N_p)_{p\leq 2}$ into a functor on the full subcategory $\Delta[0,2]\subset\Delta$
on objects isomorphic to $[0], [1], [2]$. Further, $\widetilde \N_n$ can be identified with the
limit
\[
\pro _{\{\Delta^p\rightarrow \Delta^n\}_{p\leq 2}}\N_p
\]
taken over all, not necessarily injective morphisms, using the functoriality on $\Delta[0,2]$.
This makes $(\widetilde N_n)_{n\geq 0}$ into a simplicial set and $\N=(\N_n)_{n\geq 0}$
is a simplicial subset so it inherits the structure.

We now prove that $2$-Segal simplicial set $\N$ is unital, i.e., the square \eqref {eq:unitalsquare}
for $X=\N$ is Cartesian for any $n\geq 2$ and any $i=0, ..., n-1$. 
For $n=2$ this is true because $\lambda$ and $\rho$ are isomorphisms of spans.
Let $n>2$ and $x\in \N_n$ be such that the 1-face $\partial_{\{i, i+1\}}(x)$
is degenerate. Then every 2-face $\partial_{\{j, i, i+1\}}(x)$, $j<i$ or
$\partial_{\{i, i+1, j\}}$, $j>i+1$, is degenerate by the case $n=2$ above.
Now, by construction, $x$ is determined by the collection of its 2-faces, and
we conclude that $x$ is itself in the image of $s_i: \N_{n-1}\to\N_n$. 

 This concludes the proof of Theorem \ref{thm:2-segal-mu-cat}. 
 \end{proof}

   \vfill\eject
   
\subsection{The Hall algebra of a discrete 2-Segal space}
\label{subsec:hall-algebra-discrete}
   
  The multivalued category $\Cen(X)$ associated to a $2$-Segal
  set $X$, can often be ``linearized" to yield a linear category in the usual sense.

   Let $\k$ be a field.
  For a set $B$ let $\Fc_0(B)$ be the set of all functions $B\to\k$ with finite support.
  For a map of sets $\phi: B\to B'$ we have the pushforward map
  \[
  \phi_*: \Fc_0(B)\lra\Fc_0(B'), \quad (\phi_*f)(b') =\sum_{\phi(b)=b'} f(b).
  \]
  Call $\phi$ {\em proper},
  if it has finite fibers. For a proper $\phi$ 
  we also have the pullback map $\phi^*:\Fc_0(B')\to\Fc_0(B)$.
  Any span $\sigma$ in $\Set$ as in \eqref{eq:span-in-C}   
  with $s$ proper, gives a linear map
  \begin{equation}\label{eq:Sigma}
  \sigma_* = p_* s^*: \Fc_0(Z)\lra\Fc_0(Z').
  \end{equation}
  Spans with the property that $s$ is proper, are closed under composition.
   Moreover, composition of such spans gives rise to the composition of the corresponding linear maps.

  Let $X$ be a semi-simplicial set.  
  For any $a,a'\in X_0$ we put
  \begin{equation}\label{eq:Bxy}
  B_a^{a'}=\bigl\{ b\in X_1: \partial_1(b)=a, \partial_0(b)=a'\bigr\}
  \quad = \quad\bigl\{
  \xymatrix{
  a\ar[r]^b&a'
  }
  \bigr\}
  \end{equation}
  to be the set of 1-simplices going from $a$ to $a'$. 
  For any $b,b',b''\in X_1$ we put
\[
	  C_{bb'}^{b''} =\bigl\{c\in X_2: \partial_0(c)=b,\partial_2(c)=b',
	  \partial_1(c)=b''\bigr\}
\]
  to be the set of triangles in $X$ with edges $b, b', b''$.
  A necessary condition for $C_{bb'}^{b''}\neq\emptyset$ is that 
    $(b, b', b'')$ form a $\partial\Delta^2$-{\em triple},
  i.e., there are $a,a',a''\in X_0$ such that
  $b\in B_{a'}^{a''}, b'\in B_a^{a'}, b''\in B_a^{a''}$:
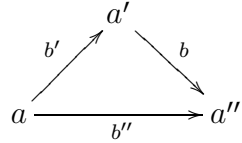
\[
	  \xymatrix{&a'\ar[dr]^b&\\
		  a\ar[ur]^{b'}\ar[rr]_{b''}&&a''
	  }
\]
  In this case $C_{bb'}^{b''}$ is contained in the set 
  \[
	  K_{aa''}^{a'}=\bigl\{ c\in X_2| \partial_{\{0\}}(c)=a, 
	  \partial_{\{1\}}(c) =a',\partial_{\{2\}}(c)=a''
  \bigr\}
  \]
 of 2-simplices of $X$ with vertices $a, a', a''$.
 
 The following is a consequence of the construction of 
 the
  associator map $\alpha$ from \eqref{eq:span-associator}.

  \begin{cor}\label{cor:assoc-set-theor}
  Assume that $X$ is $2$-Segal, and let
  
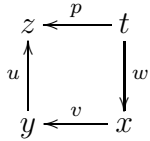
\[
  \xymatrix{
  z&\ar[l]_p t \ar[d]^w
  \\
  y\ar[u]^u&\ar[l]_v x
  }
  \]
  be any system of 0- and 1-simplices of $X$ with endpoints as indicated. Then  
  $\alpha$ defines a bijection of sets
  \[
  \alpha_{uvw}^p: \coprod_{x\buildrel r\over\lra z} C_{uv}^r\times C_{rw}^p \lra 
  \coprod_{t\buildrel s\over\lra y} C_{us}^p\times C_{vw}^s.
  \]
  \end{cor}

  Assume now that $X$ is a unital simplicial $2$-Segal set such that 
     the $1$-Segal map $f_2$ in 
  \eqref{eq:fund-corr}
  is proper. 
  This implies that each
  \[
  c_{bb'}^{b''} := |C_{bb'}^{b''}| \in\ZZ_+
  \]
 is a finite number. Moreover, for each $b,b'$ there are only finitely many $b''$ such that
 $c_{bb'}^{b''}\neq 0$.

 In this situation we can associate to $X$ a $\k$-linear category $\Hall(X)$,
 which we call the {\em Hall category} of $X$. By definition, objects of $\Hall(X)$
 are vertices of $X$, i.e., elements of $X_0$. The abelian group
 $\Hom_{\Hall(X)}(a,a')$ is the $\k$-vector space spanned by edges (1-simplices) $b\in B_a^{a'}$.
     We denote by $\one_b$ the basis vector correspond to the edge $u$. 
  The composition of morphisms is given by ``counting triangles":
   \begin{equation}\label{eq:hall-cat-disc-1}
  \one_b * \one_{b'} =\sum_{b''} c_{bb'}^{b''}\cdot \one_{b''}. \end{equation}
    Alternatively, for $a, a', a''\in X_0$ consider the part of the fundamental correspondence
  \eqref{eq:fund-corr}
  dealing with simplices with vertices among $a, a', a''$:
  \begin{equation}\label{eq:partial-fund-correspondence}
  \mu_{aa''}^{a'} =\bigl\{ B_{a'}^{a''}\times B_a^{a'} \buildrel f_2\over\lla K_{aa''}^{a'}
  \buildrel \partial_1\over\lra B_{a}^{a''}\bigr\}. 
  \end{equation}
  Note that $\Hom_{\Hall(X)}(a,a') = \Fc_0(B_a^{a'})$
  as a vector space. The composition in $\Hall (X)$ can be 
 written as follows:
 \[
	 \Fc_0(B_{a'}^{a''})\otimes\Fc_0(B_a^{a'}) =  \Fc_0(B_{a'}^{a''}\times B_a^{a'}) 
	 \buildrel (\mu_{aa''}^{a'})_*\over\lra \Fc_0(B_a^{a''}),
 \]
 where $(\mu_{aa''}^{a'})_*$ (action of a correspondence on functions)
 is defined by \eqref{eq:Sigma}.

    \begin{prop}\label{prop:hall-category-assoc-unit}
  The composition law \eqref{eq:hall-cat-disc-1} 
  is associative and makes $\Hc(X)$ into a $\k$-linear category, with the unit morphism of
  $a\in X_0=\Ob(\Hc(X))$ given by $\one_{s_0(a)}$,
  where $s_0: X_0\to X_1$ is the degeneration map. \end{prop}

  \begin{proof} 
	  The associativity of composition follows from Corollary
  \ref{cor:assoc-set-theor} . 
  The fact that $\one_{s_0(a)}$ is the unit morphism of $a$
  follows from Theorem \ref{thm:2-segal-mu-cat} (b), since $s_0: X_0\to X_1$ is the unit
  of the $\mu$-category 
  $\Cen(X)$. 
  \end{proof}

 \begin{rems} (a) The particular case when $X$ is $1$-Segal  
 corresponds to the map $f_2$
 being not just proper but a bijection. In this case
 $X$ is 
 the nerve of a category $\C$,  
 and $\Hc(X)$ is the $\k$-linear envelope of $\C$.

(b) If $X$ is a $2$-Segal semi-simplicial set, the above construction defines a $\k$-linear
 semi-category $\Hall(X)$: we still have vector spaces 
 $\Hom_{\Hall(X)}(a,a')$ and associative composition maps among them, but may not have
 identity morphisms. 
 \end{rems}

 \begin{ex}[(The Hall algebra)]
  For any vertex $a\in X_0$ we have therefore the associative
 algebra
 \[
 H(X,a)=\End_{\Hall(X)}(a),
 \]
   which we call the {\em Hall algebra}
 of $a$. 
 In the case when $X_0=\pt$ the category $\Hall(X)$ is reduced to this algebra
 which we then denote $H(X)$ and call the Hall algebra of $X$ itself. 
 
 \end{ex}

 \begin{ex}[(Algebra of factorizations)] Let  
 $M$ be a  monoid (semigroup with unit),
 considered as a category with one object. By Theorem \ref{thm:cyclic-nerve},
 the cyclic nerve $\NC(M)$ is a $2$-Segal simplicial set. 
 Suppose that $|M|<\infty$. Then $\NC(M)$ satisfies the properness
 condition and its Hall category $\Hall(\NC(M))$ is defined.
Objects of this category, i.e., vertices of $\NC(S)$, are elements of $M$. 
 So for each $w\in M$ we have an associative algebra
 \[
 \Phi_w = H(\NC(M), w), 
 \]
 which we call the {\em algebra of factorizations} of $w$. Its $\k$-basis
 is labelled by edges of $\NC(M)$ beginning and ending at $w$, 
 i.e., by pairs $(A, B)\in M^2$ such that $AB=BA=w$
 (``factorizations of $w$").
 We denote by 
 $ \one_{A,B}$ the basis element corresponding to such a pair. 
 Similarly, 2-simplices with all three vertices equal to $w$ correspond to
`` triple factorizations", i.e., triples
 \[
 (\alpha, \beta,\gamma)\in M^3, \quad \alpha\beta\gamma=\beta\gamma\alpha = \gamma\alpha\beta =w, 
   \]
 with the face maps given by 
 \[
   \partial_0(\alpha,\beta,\gamma)= (\gamma\alpha, \beta), \quad \partial_1(\alpha,\beta,\gamma)=
 (\alpha, \beta\gamma), \quad \partial_2(\alpha,\beta,\gamma)=(\alpha\beta, \gamma).
 \]
 Therefore the structure constants
 in the product
 \[
 \one_{A,B} * \one_{C,D} =\sum_{E,F} c_{ABCD}^{EF} \one_{E,F}
 \]
 are easily found to be given by
 \[
 c_{ABCD}^{EF}=\begin{cases}
 1,\quad \text{if} \quad ED=A, BE=C, DB=F;\cr
 0, \quad \text{otherwise}.
 \end{cases}
 \]
 This means that
 \[
  \one_{A,B} * \one_{C,D} =\sum_{E: ED=A, BE=C} \one_{E, DB}.
 \]
 
 \end{ex}

     \vfill\eject

\subsection{The bicategory point of view. } 
\label{subsec:2-segal-monoidal} 

\renewcommand\theparagraph{}

The linearlization of the multivalued category $\Cen(X)$ described in 
\S \ref{subsec:hall-algebra-discrete} involves some loss of information.
Here we describe a related construction which avoids this loss and 
  allows us to ``identify" $2$-Segal sets with some particular
  2-categorical structures in a more traditional sense.

\paragraph{A. Action of correspondences on sheaves}
Let $B$ be a set. By $\Set_B$ we denote the category of sets over $B$.
 Thus, an object of 
$\Set_B$ consists of a set $F$ and a map $p: F\to B$. In particular, any $b\in B$
gives rise to the one-element set $\{b\}\in \Set_B$.

One can view an object of $\Set_B$ 
as a sheaf of sets on $B$ as a discrete topological space. The category $\Set_B$
can therefore serve 
as a categorical analog of the vector space of functions on $B$.

Any map of sets $\phi: B\to B'$ gives rise to the pullback and pushforward functors
\[\begin{gathered}
			\phi^*: \Set_{B'}\lra\Set_B, \quad \phi^*\bigl\{ F'\buildrel p'\over\lra B'\bigr\} =\
			\bigl\{ F\times_{B'} B\buildrel \on{pr}_B\over\lra B\bigr\}, \\
			\phi_*: \Set_B\lra\Set_{B'}, \quad \phi_*\bigr\{ F\buildrel p\over\lra B\bigr\}
			=
			\bigl\{ F\buildrel \phi \circ p\over\lra B'\bigr\}. 
\end{gathered}
\]
 Any span in $\Set$
  \[
  \sigma =\bigl\{ Z\buildrel s\over\lla W\buildrel p\over\lra Z'\bigr\}
  \]
  gives a  functor
  \[
	  \sigma_* = p_* s^*: \Fc_0(Z)\lra\Fc_0(Z').
  \]

   \begin{prop}
   \label{prop:action-spans-sheaves}
   For any two composable spans in $\Set$
    \[
    \xymatrix{
    Z\ar@{~>}[r]^{\sigma}&Z'\ar@{~>}[r]^{\sigma'}&Z''
    }
    \]
    we have a natural isomorphism of functors
    \[
    (\sigma'\circ\sigma)_* \Rightarrow \sigma'_* \circ\sigma_*: \Fc_0(Z)\lra \Fc_0(Z'').
    \]
    More precisely, 
    these isomorphisms make the correspondence $Z\mapsto\Fc_0(Z)$, $\sigma\mapsto\sigma_*$
    into a 2-functor from the bicategory $\Spanl_\Set$ into the bicategory $\Cat$ of
    categories. 
    \end{prop}
    
    \begin{proof} Follows from the base change isomorphism for the pullbacks
    and pushforwards functors corresponding to a Cartesian square of sets. 
    \end{proof}
    
    \paragraph{B. The Hall 2-category} 
As in Example \ref{ex:classical-(2,1)}, by a semi-bicategory we mean a structure
similar to a bicategory but without the requirements of existence of unit 1-morphisms.

Let $X$ be a $2$-Segal semi-simplicial set. We associate to $X$ a semi-bicategory
$\HH=\HH(X)$ (called the {\em Hall 2-category of $X$}) as follows.
 We put $\Ob(\HH)=X_0$. For $a,a'\in X_0$ we define
the category
\[
\Homc_\HH(a,a') =\Set_{B_a^{a'}}.
\]
Here $B_a^{a'}$ is defined by \eqref{eq:Bxy}.
The composition functors $\otimes$
are defined by
\[
\Set_{B_{a'}^{a''}}\times\Set_{B_a^{a'}} \buildrel\times\over\lra
\Set_{B_{a'}^{a''}\times B_{a}^{a'}} \buildrel (\mu_{aa''}^{a'})_*\over\lra\Set_{B_a^{a''}}.
\]
Here the partial fundamental correspondence $\mu_{aa''}^{a'}$ is defined by
\eqref{eq:partial-fund-correspondence} 
For example, on one-element sets the composition has the form
\[
\{b\}\otimes \{b'\} =\coprod_{b''\in B_a^{a''}} C_{bb'}^{b''} \times \{b''\},\quad
b\in B_{a'}^{a''}, b'\in B_a^{a'},
\]
compare with the formula \eqref{eq:hall-cat-disc-1}
for the Hall category. In other words, the sets $C_{bb'}^{b''}$ appear
as Clebsch-Gordan multiplicity sets. 

Further, the 
associator $\alpha$ for the fundamental correspondence
$\mu$, see \eqref{eq:span-associator}, defines associativity isomorphisms
\[
\alpha_{F,G,H}: (F\otimes G)\otimes H \lra F\otimes(G\otimes H). 
\]

\begin{prop}
(a) For any $2$-Segal semi-simplicial set $X$ 
the functors $\otimes$ and the associators $\alpha_{F,G,H}$ make 
$\HH(X)$ into a semi-bicategory.

(b) Let $X$ be a unital $2$-Segal simplicial set. Then the semi-bicategory $\HH(X)$ is
a bicategory, with the unit 1-morphism of any object $a\in X_0$ being
$\{s_0(a)\}\in\Set_{B_a^a}$.\end{prop}

\begin{proof}
	This is a direct consequence of Theorem
 \ref{thm:2-segal-mu-cat} and of Proposition \ref{prop:action-spans-sheaves}.
\end{proof}

\begin{ex} [(Hall monoidal categories)]\label{ex:hall-mon-cat} (a) 
Each $2$-Segal semi-simplicial set $X$ and each vertex $a\in X_0$ gives rise therefore
to a monoidal category
\[
\HH(X,a) =(\Homc_{\HH(X)}(a, a), \otimes),
\]
which has a unit object $s_0(a)$ if $X$ is unital simplicial.

(b) Consider the case when $X_0=\pt$. In this case the semi-bicategory $\HH(X)$
is reduced to the above monoidal category which we still denote $\HH(X)$.
 As a category, $\HH(X)=\Set_B$, where $B=X_1$.
This category has a final object: $B$ itself (with the identity map to $B$). 
Note that we have identifications
\[
X_0=\pt, X_1=B, X_2=B\otimes B, \cdots, X_n=B^{\otimes n}, \cdots
\]
In other words, $B^{\otimes n}\in\Set_B$ is identified with $X_n\buildrel \partial_{\{0,n\}}
\over\lra X_1=B$. More precisely, each tensor power $B^{\otimes n}$ should,
strictly speaking, be
understood with respect to some particular bracketing. Such bracketings correspond
to triangulations $\Tc$ of the $(n+1)$-gon $P_n$. The bracketed tensor product corresponding to
$\Tc$, is precisely $X_\Tc\buildrel \partial_{\{0,n\}}
\over\lra X_1=B$, which is identified with $X_n$ via the $2$-Segal map $f_\Tc$.

 \end{ex}

 \paragraph{C. $\bigsqcup$-semisimple bicategories} 

 We now want to characterize semi-bicategories appearing as $\HH(X)$ for
 $2$-Segal semi-simplicial sets $X$.

A category $\Vc$ equivalent to $\Set_B$ for some $B$, will be called $\sqcup$-{\em semisimple},
and an object of $\Vc$ isomorphic to (the image under such an equivalence of)
an object of the form $\{b\}$, will be
called {\em simple}. We denote by $\|\Vc\|$ the set of isomorphism classes of
simple objects of $\Vc$.

A functor $F: \Vc\to\Wc$ between $\sqcup$-semisimple categories will
be called {\em additive}, if it preserves coproducts. An additive functor is called
{\em simple additive} if, in addition, it takes simple objects to simple objects.
We denote by $\Cat^\sqcup$ the bicategory formed by $\sqcup$-semisimple categories,
their additive functors and their natural transformation. Let also $\Cat^{\sqcup !}$
be the sub-bicategory on the same objects, simple additive functors and their
natural transformations. 
Proposition \ref{prop:action-spans-sheaves} admits the following refinement.

\begin{prop}
\label{prop:semisimple-spans}
(a) For any span of sets $\xymatrix{Z\ar@{~>}[r]^{\sigma}&Z'}$
the functor $\sigma_*: \Set_Z\to\Set_{Z'}$ is additive. The correspondence
$Z\mapsto \Set_Z$, $\sigma\mapsto \sigma_*$ extends to 
a 2-equivalence of bicategories $\Spanl_\Set \to \Cat^\sqcup$,
the inverse 2-equivalence taking $\Vc$ to $\|\Vc\|$. 

(b) Under the equivalence in (a), the category $\Set$ itself becomes 2-equivalent
to the bicategory $\Cat^{\sqcup !}$. 
\end{prop}

\begin{proof} The main point in (a) is that any additive functor
$F: \Set_Z\to\Set_{Z'}$ is isomorphic to a functor of the form $\sigma_*$ for
some span $\sigma = \{Z\buildrel s\over\leftarrow W\buildrel p\over\to Z'\}$.
For this, we note that $s^{-1}(z)$, $z\in Z$, is recovered as $F(\{z\})$.
Part (b) is obvious.
\end{proof}

\begin{defi}  
A semi-bicategory $\Cc$ will be called $\sqcup$-{\em semisimple}, if:
\begin{itemize}
\item[(1)] $\Ob(\Cc)$ is a set.

\item[(2)] Each category $\Homc_\Cc(x,y)$ is $\sqcup$-semisimple.

\item[(3)] The composition functors 
 \[
\otimes: \Homc_\Cc (y,z) \times\Homc_\Cc(x,y) \lra\Homc_\CC (x,z)
 \]
 are additive in each variable. 
\end{itemize}

\end{defi}

\noindent 
Next, we describe what kind of ``morphisms" between $\sqcup$-semisimple bicategories
we want to consider.

Recall, first of all, that a 
   {\em lax 2-functor} $\Phi: \Cc \to \Dc$ between two semi-bicategories
consists of a map $\Phi: \Ob(\Cc)\to\Ob(\Dc)$, a collection of usual functors
\[
\Phi= \Phi_{c,c'}: \Homc_\Cc(c, c')\lra\Homc_\Dc(\Phi(c), \Phi(c'))
\]
and of natural {\em morphisms} (not required to be isomorphisms!)
\[
\Phi^{F, F'}: \Phi_{c, c''}(F\otimes F') \lra \Phi_{c', c''}(F)\otimes\Phi_{c,c'}(F'), \quad F\in\Homc_\Cc(c',c''), 
F'\in\Homc_\Cc(c,c')
\]
which commute with the associativity isomorphisms in $\Cc$ and $\Dc$. 

\begin{defi}\label{def:admissible-lax-functors}
\begin{itemize}
\item[(a)] 
A lax 2-functor $\Phi$ between $\sqcup$-semisimple semi-bicategories is called
{\em admissible}, if each functor $\Phi_{x,y}$ is simple additive. 

\item[(b)] Two admissible 2-functors $\Phi, \Psi: \Cc \to\Dc$ are called {\em equivalent}, if:
\begin{itemize}
\item[(1)] We have $\Phi(c)=\Psi(c)$ for each $c\in\Ob(\Cc)$. 

\item[(2)] There exist equivalences of categories $T_{c, c'}$ and isomorphisms of functors
$U_{c, c'}$, given for all $c, c'\in\Ob(\Cc)$, of the form
\[
\xymatrix{
   &
\Homc_\D(\Phi(c), \Phi(c'))
\ar[dd]^{T_{c,c'}}
\\
\Homc_\Cc(c, c')\ar[ur]^{\Phi_{c,c'}}
\ar[dr]_{\Psi_{c,c'}}
\rtwocell<\omit>{\hskip .5cm U_{c,c'}}
&
\\
&\Homc_\Dc(\Psi(c), \Psi(c'))
}
\]
which commute with the $\Phi^{F,F'}$ and $\Psi^{F,F'}$ as well as with the
associativity isomorphisms.

\end{itemize}
\end{itemize}

\end{defi} 

\begin{rem}
Assume that $\Dc$ is a bicategory, i.e., it has unit 1-morphisms. Then
lax 2-functors from $\Cc$ to $\Dc$ form themselves a bicategory
$\Lc ax(\Cc, \Dc)$, cf. \cite{benabou}. The condition (2) of Definition 
\ref{def:admissible-lax-functors}(b) can be reformulated in this case
by saying that  
  $\Phi$ and $\Psi$ are equivalent
  as objects of the bicategory $\Lc ax(\Cc, \Dc)$. 
\end{rem}
 
 Let $\Cc$ be a $\sqcup$-semisimple bicategory, 
so $\Cc$ has unit objects $\1_a\in\Homc_\Cc(a,a)$
for each $a\in\Ob(\Cc)$. We say that $\Cc$ {\em has simple units}, if each $\1_a$ is a simple
object of $\Homc_\Cc(a,a)$.

\begin{thm}\label{thm:segal-2-categories}
The following categories are equivalent:
\begin{enumerate}
\item[(i)] The category of $2$-Segal semi-simplicial sets (resp. unital $2$-Segal simplicial sets).

\item[(ii)] The category of $\sqcup$-semisimple semi-bicategories
(resp. $\sqcup$-semisimple bicategories with simple units),
with morphisms being equivalence classes of admissible lax 2-functors. 
\end{enumerate}
The equivalence takes a $2$-Segal set $X$ into its Hall 2-category $\HH(X)$.
\end{thm}

\begin{proof}This is a consequence of Proposition 
\ref{prop:semisimple-spans} and Theorem \ref{thm:2-segal-mu-cat}.  
Indeed, a $\sqcup$-semisimple semicategory $\Cc$ gives rise to a $\mu$-semicategory
$\Cen = \Cen(\Cc)$ with 
\[
	\Cen_0 = \Ob(\Cc), \quad \Cen_1 = \coprod_{x,y\in\Ob(\Cc)} \|\Homc_\Cc(x,y)\|, 
\]
and $\mu$ obtained from $\otimes$ by applying Proposition 
\ref{prop:semisimple-spans}(a). 
 We leave further details to the reader. 
 \end{proof}

\begin{ex} [(The Clebsch-Gordan nerve)]\label{ex:CGN}
Let $\Cc$ be a $\sqcup$-semisimple semi-bicategory.
The $2$-Segal semi-simplicial set corresponding to $\Cc$
can be described as the
nerve of the $\mu$-semicategory $\Cen(\Cc)$,
see Proposition \ref{prop:nerve-mu-cat}. In terms of $\Cc$ itself,
this means the following. 

   For any $a, a'\in\Ob(\Cc)$,
choose a set $(E_b)$ of simple generators of the $\sqcup$-semisimple
category $\Homc_\Cc(a, a')$. Here $b$ runs in some index set which we denote 
$B_a^{a'}$.
 By a {\em Clebsch-Gordan triangle} we mean a 2-morphism (triangle) in $\Cc$
 of the form
 \[
 \xymatrix{
 &a'
 \ar[dr]^{E_b}&
 \cr
 a\ar[rr]_{E_{b''}}
 \ar[ur]^{E_{b'}}_{\hskip .4cm u \Uparrow}
 &&a''
 }
 \]
 for some $a,a', a''\in\Ob(\Cc)$ and $b\in B_{a'}^{a''}$, $b'\in B_a^{a'}$,
 $b''\in B_a^{a''}$.

 Let $\N\Cc$ be the semi-simplicial nerve of $\Cc$, so $\N_n\Cc$
  consists of commutative $n$-simplices in $\Cc$
  (Example \ref{ex:classical-(2,1)}). Such a simplex will be called a
  {\em Clebsch-Gordan $n$-simplex}, if all its 2-faces are Clebsch-Gordan triangles. 
   Defining $\on{CGN}_n(\Cc)$
 to be the set of Clebsch-Gordan $n$-simplices, we get a semi-simplicial subset
 $\on{CGN}(\Cc)\subset \N\Cc$ which we call the {\em Clebsch-Gordan nerve} of $\Cc$.
 Then $\N(\Cen(\Cc))=\on{CGN}(\Cc)$.
 In particular, the Clebsch-Gordan nerve is $2$-Segal.  
 Note that the nerve of a bicategory, even of a strict one, is not, in general, $2$-Segal.
It is the requirement that all edges be labelled by simple objects that
ensures the $2$-Segal property.

 \end{ex}

\paragraph{D. Non-simple units} 
 We now discuss how to extend Theorem \ref{thm:segal-2-categories}
 to the case when $\C$ has unit 1-morphisms but they are not simple. 
   Note, first of all, that any
$\sqcup$-semisimple bicategory $\Cc$ gives rise to a  $\sqcup$-semisimple monoidal category  
$\Mat(\Cc)$ with 
\[
\Ob(\Mat(\Cc)) =\prod_{a, a'\in\Ob(\Cc)} \Homc_\Cc(a, a').
\]
Thus an object of $\Mat(\Cc)$ can be seen as a matrix $E=(E_{aa'}: a\to a')$
of 1-morphisms in $\Cc$. 
The monoidal operation $\otimes$ on $\Mat(\Cc)$ is given by mimicking matrix multiplication
\[
(E\otimes F)_{aa''} =\bigsqcup_{a'\in\Ob(\Cc)} E_{aa'}\otimes F_{a'a''}. 
\]
 The object $\1\in \Mat(\Cc)$ with $\1_{aa}=\1_a$ and $\1_{aa'}=\emptyset$
for $a\neq a'$, is a unit object but it is not simple.

\begin{prop}\label{prop:decompositon-of-unit}
 Any $\sqcup$-semisimple monoidal category $(\Ac, \otimes)$ with a unit object $\1$
is equivalent to $\Mat(\Cc)$ where $\Cc$ is a 
 $\sqcup$-semisimple bicategory with simple units. 
\end{prop}

\begin{proof} We assume $\Ac=\Set_B$ as a category.
Let $\1=\bigsqcup_{b\in B} I_b\times \{b\}$ for some
sets $I_b$. Let $A\subset B$ be the set of $b$ such that $I_b\neq\emptyset$.
Let $a\in A$. We claim that $|I_a|=1$. Note that there is $a'\in A$ such that $\{a\}\otimes \{a'\}\neq\emptyset$,
otherwise $\{a\}\otimes \1\simeq \{a\}$ is impossible. But then $\1\otimes\{a'\}$ contains
$I_a\times(\{a\}\otimes \{a'\})$ and cannot be isomorphic to a simple object $\{a'\}$, if $|I_a|>1$.

We have therefore $\1=\bigsqcup_{a\in A} \{a\}$. 
  By writing $\1\otimes\1\simeq \1$, we see that the $\{a\}, a\in A$, are orthogonal idempotents
with respect to $\otimes$:
\[
\{a\}\otimes \{a'\}=\begin{cases}
\emptyset, \quad a\neq a'; \cr
\{a\}, \quad a=a'.
\end{cases}
\]
Therefore, if we put
\[
B_a^{a'}=\bigl\{ b\in B| \{a'\}\otimes \{b\} \simeq \{b\}\otimes \{a\} \simeq \{b\}\bigr\},
\]
we get $B=\bigsqcup_{a, a'\in A} B_a^{a'}$. Further, the monoidal structure $\otimes$ restricted
to the subcategories $\Set_{B_a^{a'}}\subset\Set_B$, gives functors
\[
\otimes: \Set_{B_{a'}^{a''}}\times\Set_{B_a^{a'}}\lra\Set_{B_a^{a''}},
\]
i.e., defines a bicategory $\Cc$ with the set of objects $A$ and $\Homc_\Cc (a, a')=\Set_{B_a^{a'}}$.
The object $\1_a:=\{a\}\in\Set_{B_a^a}$ is then the unit 1-morphism of the object $a$. This
proves the proposition. 
\end{proof}

More generally, let $\Cc$ be a $\sqcup$-semisimple bicategory and $\rho: \Ob(\Cc)\to D$
be a surjection of sets. We then define a bicategory $\Mat_\rho(\Cc)$ with set of
objects $D$ and
\[
\Homc_{\Mat_\rho(\Cc)}(d, d') =\prod_{\rho(a)=d, \rho(a')=d'}
 \Homc_\Cc (a, a').
\]
This is again a $\sqcup$-semisimple bicategory.

\begin{prop} Any $\sqcup$-semisimple bicategory $\Dc$ is equivalent to $\Mat_\rho(\Cc)$
for some $\sqcup$-semisimple bicategory $\Cc$ with simple units.
\end{prop}

\begin{proof} We apply Proposition
\ref{prop:decompositon-of-unit} to each monoidal category $\Homc_\Dc(d,d)$
to get the set $A(d)$. We then split the composition in $\Dc$ to construct a bicategory on the set of objects
$A=\bigsqcup_{d\in\Ob(\Dc)} A(d)$. 
The details are straightforward.
\end{proof}

Summarizing, we can say that unital $2$-Segal simplicial sets are ``the same as" 
 $\sqcup$-semisimple bicategories. 

 \vfill\eject

\subsection{The operadic point of view}
\label{subsec:operadic}
 
In this section we show that 2-Segal simplicial sets $X$ with $X_0=\pt$ can be identified with
certain operads. We recall a version of the concept known variously under the names of colored
operads \cite{moerdijk}, pseudo-tensor categories \cite{beilinson-drinfeld} and multilinear
categories \cite{linton}.  
  
\begin{defi}\label{def:operads-colored}
Let $(\M, \otimes, \1)$ be a symmetric monoidal category. An $\M$-valued {\em (colored) operad} $\O$
consists of the folllowing data:
\begin{enumerate}[label=(OP\arabic{*})]
\item A set $B$, whose elements are called {\em colors}. 

\item Objects $\O(b_1, ..., b_n|b_0)\in\M$ given for all choices of $n\geq 0$ and $b_0, ..., b_n\in B$.

\item The ``composition" morphisms
\[
\begin{gathered}
\O(b_1, ..., b_n|b_0) 
\otimes \O(b^1_1, ..., b^1_{m_1}|b_1) \otimes \cdots\otimes \O(b^n_1, ..., b^n_{m_n}|b_n) \lra \\
\lra \O(b^1_1, ..., b^1_{m_1}, \cdots,  b^n_1, ..., b^n_{m_n}|b_0)
\end{gathered}
\]
given for each $b_0, ..., b_n, b^i_j\in B$ as described.

\item The ``unit" morphisms $ \Id_b: \1\to \O(b|b)$ given for each $b\in B$. 

\end{enumerate}

These data are required to satisfy the standard associativity and unit axioms, cf. \cite[\S 1.2]{moerdijk}.
Dually, an $\M$-valued {\em cooperad} is the same as an operad with values in $\M^\op$.  A cooperad
$\Qc$ has {\em cocomposition  and counit morphisms} going in the directions opposite to those in
(OP3) and (OP4). 
\end{defi}
  
\begin{remex}\label{ex:operads-cooperads}
\begin{exaenumerate}

	\item Note  that  we do not require any data involving
permutations of the arguments, i.e., relating
$\Oc(b_1, ..., b_n|b_0)$ with $\Oc(b_{w(1)}, ..., b_{w(n)}|b_0)$,
$w\in S_n$.
So our concept  can be more precisely called
a  {\em non-symmetric} colored operad. In fact,
for Definition \ref{def:operads-colored} to make sense, it
is enough that 
$(\M, \otimes)$ be a braided, not necessarily a symmetric
monoidal category, but we will not use this generality.

\item Let $(\B,\boxtimes, I)$ be an $\M$-enriched monoidal category (not assumed braided
or symmetric). Then for any subset of objects $B\subset\Ob(\B)$
we have an $\M$-valued operad $\O$ with the set of colors $B$ and
\[
\O(b_1, ..., b_n|b_0) \,\,=\,\, \Hom_\B (b_1\boxtimes \cdots\boxtimes b_n, b_0).
\]
Here, the empty $\boxtimes$-product for $n=0$ is set to be $I$. 
Similarly to (a), notice that to speak of $\M$-enrichment,
it is enough that $\M$ be a braided monoidal category, see
\cite{joyal-street}. 

\item If $B=\{\pt\}$ consists of one element, then the data in $\O$ reduce to the
objects $\O(n) = \O(\pt, ..., \pt)$ ($n$ times) and we get a more familiar
concept of a (non-symmetric) operad. The operadic composition
and unit maps can then
be written as 
\[
\nu_{m_1, ..., m_n}:
\O(n) \otimes  \bigl( \O(m_1) \otimes \cdots \otimes\O(m_n)\bigr) \lra \O(m_1 + ... + m_n),\quad
\Id: \1\to\O(1). 
\]

\item Let $\M=\Set$ with $\otimes$ given by the Cartesian product. For a 
$B$-colored   operad
$\O$  in $(\Set, \times)$ the elements of $\O(b_1, ..., b_n|b_0)$ are called
$n$-ary {\em operations} in $\O$. We put 
\[
\O(n) \,\,\,=\,\,\,\coprod_{b_0, ...,b_n\in B} \O(b_1, ..., b_n|b_0)
\]
to be the set of all possible $n$-ary operations. 
Then the colorings define maps
$
\pi_i: \O(n) \to B, \quad i=0, ..., n. 
$
The operadic composition maps can then be simultaneously written as
\[
\nu_{m_1, ..., m_n}:
\O(n)^{(\pi_1, ..., \pi_n)} \times^{(\pi_0, ..., \pi_0)}_{B^n} \bigl( \O(m_1) \times \cdots \times\O(m_n)\bigr) 
\lra \O(m_1 + ... + m_n).
\]
As the fiber product is a subset in the full product, the $\O(n)$ do not,
in general,  form a 1-colored operad,
unless $|B|=1$. 

\item For a $B$-colored cooperad $\Qc$ in $(\Set, \times)$ we define the sets $\Qc(n)$ and projections
$\pi_i: \Qc(n)\to B$ in the same way as in (d).  Then the cooperadic cocomposition
maps  in $\Qc$ give rise to the maps in the direction opposite from these in (c): 
\be\label{eq:cooperad-cocomposition}
f_{m_1,..., m_n}:
\Qc(m_1 + ... + m_n)\lra   \Qc(n)^{(\pi_1, ..., \pi_n)} \times^{(\pi_0, ..., \pi_0)}_{B^n} \bigl( \Qc(m_1) \times \cdots \times
\Qc(m_n)\bigr).
\ee
Note that  the $f_{m_1, ..., m_n}$ can now  be seen as taking values in the full Cartesian product
and thus the $\Qc(n)$ always form a 1-colored cooperad. The structure of a $B$-colored cooperad in
$(\Set, \times)$ is thus a refinement of a structure of a 1-colored cooperad.

\end{exaenumerate}
\end{remex}

\begin{ex}[(Standard simplices as an operad)]\label{ex:simplices-operad}
Let $\M=\Sp=\Set_\Delta$ be the category of simplicial sets. Equip $\Sp$ with the symmetric
monoidal structure given by $\sqcup$, the disjoint union. 
Remarkably, the collection $(\Delta^n)$ of standard
simplices forms a (1-colored) operad in $(\Sp, \sqcup)$. The operadic composition maps
\[
\nu_{m_1, ..., m_n}:  \Delta^n\sqcup \bigl( \Delta^{m_1} \sqcup  \cdots\sqcup \Delta^{m_n} \bigr) \lra\Delta^{m_1+...+m_n}
\] 
are defined as follows. The $i$th vertex of $\Delta^n$ is mapped into the vertex of $\Delta^{m_1+...+m_n}$
with number $m_1+... + m_i$ (which is set up to be $0$ for $i=0$).
The $j$th vertex of $\Delta^{m_i}$ is mapped into the vertex of $\Delta^{m_1+...+m_n}$
with number $m_1+...+m_{i-1}+j$. The unit maps are the unique embeddings of $\emptyset$
which is the unit object for $\sqcup$. 
The verification of the operad axioms is straightforward. This example
is important for the 2-Segal point of view on simplicial sets.
\end{ex}

\begin{ex}[(Simplicial sets as cooperads)]\label{ex:simplicial-sets-cooperads}
Let $X$ be a simplicial set, so that $X_n=\Hom(\Delta^n, X)$
is the set of $n$-simplices of $X$. Applying the previous example and the adjunction between
$\sqcup$ and $\times$, we conclude that the collection of sets $(X_n)$ forms a 
(1-colored) cooperad
in $(\Set, \times)$. The cocomposition map
\[
f_{m_1, ..., m_n}: X_{m_1+...+m_n} \lra X_n\times (X_{m_1}\times \cdots\times X_{m_n})
\]
is the 2-Segal map corresponding to the polygonal subdivision $\P$ of $P_{m_1+...+m_n}$
consisting of one polygon with vertices $0, m_1, m_1+m_2, ..., m_1+...+m_n)$
and $n$ polygons with vertices $m_{i-1}, m_{i-1}+1, ..., m_{i}$ for $i=1, ..., n$. 

Moreover, for any $b_1, ..., b_n\in X_1$ put 
\[
\Qc_X(b_1, ..., b_n|b_0)\,\,=\,\,\bigl\{ x \in X_n |\,\,   \partial_{\{1,2\}}(x)=b_1, \, 
\partial_{\{2,3\}}(x) = b_2, ..., \partial_{\{n-1, n\}}(x)=b_n, \,\,\partial_{\{0,n\}}(x)=b_0
\bigr\}. 
\] 
Then the $f_{m_1, ..., m_n}$ give rise to the cooperadic cocompositions,
making $\Qc_X$ into a $X_1$-colored cooperad in $(\Set, \times)$. 
\end{ex}

\begin{prop}\label{prop:functor-Q}
Let $B$ be a set. Then the correspondence $X\mapsto\Qc_X$  gives a fully faithful functor $\mathfrak Q$ 
from the category of simplicial sets $X$ with   $X_1=B$ (and  their morphisms identical on
$B$) to the category of $B$-colored cooperads in $(\Set, \times)$.     
\end{prop} 

\begin{proof} Call a morphism $\phi: [n]\to[q]$ in $\Delta$ {\em wide}, if $\phi(0)=0$ and $\phi(n)=q$.
Let $\on{Wid}$ be the class of all wide morphisms. It is closed under composition, contains all
degeneration maps $\sigma^n_i: [n+1]\to[n]$, as well as all the face maps
$\delta_i^n: [n-1]\to [n]$ for $i=1, ..., n-1$. 

Call $\phi$  {\em narrow}, if  $\phi$ identifies $[n]$ with an interval $\{a, a+1, ..., a+n\}\subset [q]$.
Let $\on{Nar}$ be the class of all narrow morphisms. It is closed under composition and contains
the face maps $\delta^n_0, \delta^n_n: [n-1]\to [n]$. 

Since $\on{Wid}$ and $\on{Nar}$ contain all the identity maps, we can consider them
as subcategories in $\Delta$ with the full set of objects. By the above, these categories
together generate all
(morphisms) of $\Delta$. Note also that $\on{Nar}\cap\on{Wid}$ consists only of isomorphisms
in $\Delta$ (i.e., only of identity morphisms among the standard objects $[n]$). 

Look now at the  morphism $\nu_{m_1, ..., m_n}$ from Example 
\ref{ex:simplices-operad}. The first component of $\nu_{m_1, ..., m_n}$
(i.e., its restriction to $\Delta^n$) is a wide morphism, and all wide morphisms
are obtained in this way. 
The other components
(restrictions to the $\Delta^{m_i}$) of $\nu_{m_1, ..., m_n}$
are narrow morphisms, and all narrow morphisms are found in this way.

This implies that the functor $\mathfrak Q$ is fully faithful.
Indeed, for two simplicial sets $X,Y$ with $X_1=Y_1=B$, a morphism of colored cooperads
$u: \Qc_X\to\Qc_Y$ consists  of maps $u_n: X_n\to Y_n$ for all $n$ which, by the above,
commute with the actions of morphisms from $\on{Wid}$ and $\on{Nar}$ and therefore
with all morphisms of $\Delta$, so $u$ is a morphism of simplicial sets. 
\end{proof}

\begin{defi}
Let $B$ be a set. A $B$-colored cooperad $\Qc$
(resp. a $B$-colored operad $\Oc$) in $(\Set, \times)$ is called
{\em invertible}, if:
\begin{enumerate}[label=(\arabic{*})]
\item For each $b\in B$ the counit map 
$\Qc(b|b)\to\pt$  (resp. the unit map $\pt\to\Oc(b|b)$) is bijective.
\item For each $m_1, ..., m_n$ the cocomposition map $f_{m_1, ..., m_n}$
from \eqref {eq:cooperad-cocomposition} (resp. the composition map $\nu_{m_1, ..., m_n}$
from Example \ref {ex:operads-cooperads}(c)) 
is bijective. 
\end{enumerate}
\end{defi}

\begin{thm}\label{thm:2-segal-invertible}
Let $B$ be a set. The following categories are equivalent:
\begin{enumerate}[label=(\roman{*})]
\item 2-Segal simplicial sets with $X_0=\pt$ and $X_1=B$
(with morphisms identical on 1-simplices).

\item Invertible $B$-colored cooperads in $(\Set, \times)$.

\item Invertible $B$-colored operads in $(\Set, \times)$.
\end{enumerate}
\end{thm} 

\begin{proof}
For an invertible $B$-colored  cooperad $\Qc$ we can invert the $f_{m_1, ..., m_n}$, getting an
invertible
$B$-colored  operad  $\Qc^{-1}$  in $(\Set, \times)$
with $\Qc^{-1}(n) = \Qc(n)$ and
\[
\nu_{m_1, ..., m_n} \,\,=\,\, f_{m_1, ..., m_n}^{-1}.
\]
This establishes an equivalence (ii)$\Leftrightarrow$(iii). Further,  
because the functor $\mathfrak Q$ in Proposition
\ref{prop:functor-Q} is fully faithful, the equivalence (i)$\Leftrightarrow$(ii) reduces to the following.

\begin{prop}\label{prop:2-Segal-invertible}
\begin{enumerate}[label=(\alph{*})]
	\item Let $X$ be a simplicial set with $X_0=\pt$. Then $X$ is 2-Segal, if
	and only if the cooperad $\Qc_X$ is invertible.
	\item Let $\Qc$ be an  invertible $B$-colored cooperad.
	Then $\Qc\simeq \Qc_X$ for some simplicial set $X$ with $X_0=\pt, X_1=B$. 
\end{enumerate}
\end{prop}

\noindent {\sl Proof of Proposition \ref{prop:2-Segal-invertible}:} (a) 
Let $m_1, ..., m_n$ be given.
The cocomposition map $f_{m_1, ..., m_n}$
for $\Qc_X$, see   \eqref{eq:cooperad-cocomposition},
is nothing but the 2-Segal map $f_{\Pc_{m_1, ..., m_n}, X}$ for a particular polygonal subdivision
$\Pc_{m_1, ..., m_n}$ of the polygon $P_{m_1+...+m_n}$. Explicitly,
$\Pc_{m_1, ..., m_n}$ consists of the following polygons (indicated by their
sets of vertices):
\[
\begin{gathered}
\{0, m_1, m_1+m_2, m_1+m_2+m_3, ..., m_1+...+ m_n\}, \\
\{ m_1+...+m_i , m_1+...+ m_i+1, \cdots, m_1+... + m_i + m_{i+1}\}, \quad i=0, ..., n-1. 
\end{gathered}
\]
So the 2-Segal  property of $X$ implies that the cooperad $\Qc_X$ is invertible. Conversely, suppose that
$\Qc_X$ is invertible, i.e., that all the 2-Segal maps  $f_{\Pc_{m_1, ..., m_n}, X}$ are invertible.
Then $X$ is 2-Segal in virtue of Proposition \ref{prop:2-segal-basic} \eqref{item:2-segal-basic-4}.

(b) Let $\Qc$ be given. Consider the span of sets
\[
B\times B \buildrel (\pi_1,\pi_2)\over\lla \Qc(2)
\buildrel\pi_0\over\lra B.
\]
Using this span, we define a functor
\[
\otimes \,\,=\,\, \pi_{0*} (\pi_1, \pi_2)^*: \,\,\Set_B\times\Set_B \lra\Set_B. 
\]
We claim that $\otimes$ makes $\Set_B$ into a ($\sqcup$-semisimple) monoidal
category. Indeed, the associativity isomorphism
\[
\alpha: \, \otimes \circ (\otimes\times\Id) \Rightarrow \otimes\circ(\Id\times\otimes)
\]
is obtained from  the bijections
\[
\Qc(2)\times_{B}^{(\pi_0, \pi_1)} \Qc(2)\buildrel f_{2,1}\over \lla \Qc(3)\buildrel f_{1,2}\over \lra  
\Qc(2)\times_{B}^{(\pi_0, \pi_2)} \Qc(2).
\]
We thus define $X$ as the Clebsch-Gordan nerve of $(\Set_B, \otimes)$, and our
statement follows from Theorem \ref{thm:segal-2-categories}. We leave further
details to the reader. 
\end{proof}

Invertible $B$-colored operads   can be seen as providing a set-theoretic analog of the concept of a
quadratic operad in the category $(\Vect_\k, \otimes)$ of vector spaces over a field $\k$, as
introduced in \cite{ginzburg-kapranov}.  Let us explain this point of view in more detail, recalling
analogs of necessary constructions from {\em loc. cit.} 

Let $B$ be a set. We denote by $\on{Bin}^B$ the set of isomorphism classes of plane rooted  binary
trees with all the edges (including the outer edges) labelled (``colored") by elements of $B$. Thus,
a tree $T\in\on{Bin}^B$ has a certain number $n+1\geq 3$ outer edges (called  {\em tails}), of which
one is designated as the ``root" (or {\em output}) tail, and the remaining $n$  tails are totally
ordered by the plane embedding, and are called the {\em inputs} of $T$.  For $b_0, ..., b_n\in B$ we
denote by $\on{Bin}^B(b_1, ..., b_n|b_0)\subset\on{Bin}^B$ the set of  $T$ which have $b_0$ as the
color of the output and $b_1, ..., b_n$ as the colors of the input tails, in the order given by the
plane embedding.     Further, for $T\in\on{Bin}^B$ we denote by $\on{Vert}(T)$ the set of vertices
of $T$. A vertex $v\in\on{Vert}(T)$ has two input edges $\on{in}'(v)$, $\on{in}''(v)$ (order fixed
by the plane embedding) and one output edge $\on{out}(v)$. 

Let  $\Ec = \{\Ec(b_1, b_2|b_0)_{b_0, b_1, b_2\in B}\}$    be a collection of sets labelled by
triples of elements of $B$. We think of elements of $\Ec(b_1, b_2|b_0)$ as formal binary operations
from $b_1\otimes b_2$ to $b_0$.  In this situation, we have  the  {\em free $B$-colored
(non-symmetric) operad} $\Fc_\Ec$  in $(\Set, \times)$ generated by $\Ec$. It consists of sets 
\[
\Fc_\Ec(b_1, ..., b_n|b_0) \,\,=\,\, \coprod_{T\in\on{Bin}^B(b_1, ..., b_n|b_0)}\,\,\,  \prod_{v\in\on{Vert}(T)}
\Ec\bigl(\on{in}'(v), \on{in}''(v)|\on{out}(v)\bigr). 
\]
The composition maps are given by grafting of trees. A $B$-colored operad $\Oc$ in $(\Set, \times)$
is called {\em binary generated} if there exists an $\Ec$ as above and a surjection of operads
$\Fc_\Ec\to\Oc$.    In this case $\Ec$ is recovered as the set of binary operations in $\Oc$, i.e.,
$\Ec(b_1, b_2|b_0) = \Oc(b_1, b_2|b_0)$.

Among binary generated operads $\Oc$  we are interested in those for which all the relations among
generators in $\Ec = \Oc(2)$ follow from those holding already in $\Oc(3)$. More precisely, let
$\on{Bin}=\on{Bin}^{\pt}$ be the set of topological types of binary rooted trees, and let
$\on{Bin}(n)$ be the set of such trees with $n$ inputs. We have then the projection (forgetting the
coloring)
\[
\pi: \on{Bin}^B(b_1, ..., b_n|b_0) \lra \on{Bin}(n).
\]
For $\tau\in\on{Bin}(n)$ we denote $\on{Bin}^{B, \tau}(b_1, ..., b_n|b_0)$ the set of colored trees
of topological type $\tau$, i.e., the preimage $\pi^{-1}(\tau)$.

\begin{defi} A $B$-colored operad $\Oc$ in $(\Set, \times)$ is called
{\em quadratic}, if
for any $\tau\in\on{Bin}(n)$,  the map
\[
\nu_\tau^{(b_1, ..., b_n|b_0)} : 
\coprod_{T\in\on{Bin}^{B,\tau} (b_1, ..., b_n|b_0)}\,\,\,  \prod_{v\in\on{Vert}(T)}
\Oc\bigl(\on{in}'(v), \on{in}''(v)|\on{out}(v)\bigr) \lra \Oc(b_1, ..., b_n|b_0)
\]
induced by the composition in $\Oc$, is a bijection. 
\end{defi}

\begin{prop} A $B$-colored operad $\Oc$ is quadratic, if and only if it is
invertible.
\end{prop}

\begin{proof}
Note that the set $\on{Bin}(n)$ is identified with the set of
triangulations of $P_n$ by associating to a triangulation  its Poincar\'e dual tree:
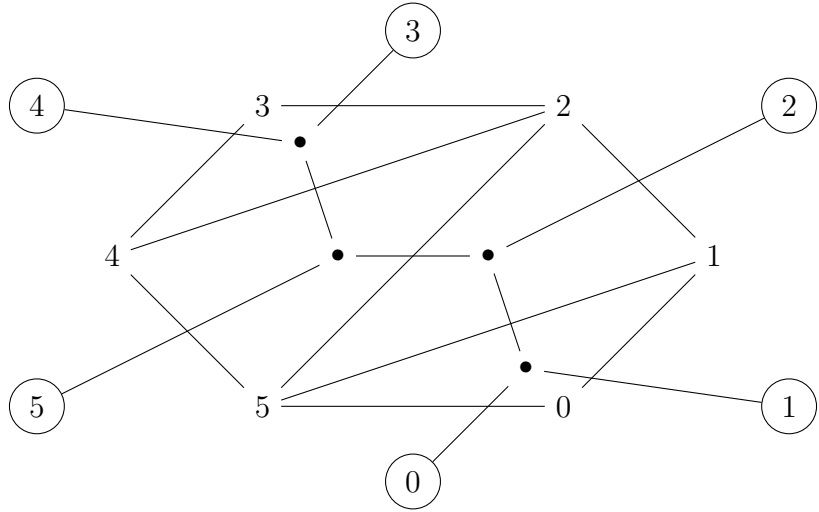
\[
\begin{tikzpicture}
\node (0) at (2, -2) {0};
\node (2) at (2,2){2};
\node(1) at (4,0) {1}; 
\node(3) at (-2,2) {3};
\node (4) at (-4,0) {4}; 
\node (5) at (-2,-2) {5};

\node (b2) at (-1,0) {$\bullet$};
\node (b3) at (1,0) {$\bullet$}; 
\node (b1) at (-1.5, 1.5) {$\bullet$};
\node (b4) at (1.5, -1.5) {$\bullet$}; 

\node[draw, shape=circle] (i2) at (5,2) {2};
\node[draw, shape=circle] (i1) at (5,-2) {1};
\node[draw, shape=circle] (i3) at (0,3) {3};
\node[draw, shape=circle] (i0) at (0,-3) {0};
\node[draw, shape=circle] (i4) at (-5,2) {4};
\node[draw, shape=circle] (i5) at (-5,-2) {5};

\draw (0) -- (1);
\draw (1) -- (2);
\draw (2) -- (3);
\draw (3) -- (4);
\draw (4) -- (5);
\draw (5) -- (0); 

\draw (4) -- (2);
\draw (5) -- (2); 
\draw (5) -- (1); 

\draw (b4)--(i0);
\draw (i1)--(b4);
\draw (i2) -- (b3); 
\draw (b3) -- (b4); 
\draw (b2) -- (b3); 
\draw (i5) -- (b2); 
\draw (b1) -- (b2); 
\draw (i3) -- (b1);
\draw (i4) -- (b1); 

\end{tikzpicture}
\]
So a quadratic operad is directly translated into a 2-Segal simplicial set, and thus the statement
follows from Theorem \ref{thm:2-segal-invertible}. 
\end{proof}

\begin{rem} Any quadratic operad is clearly binary generated. Further, 
note that $\on{Bin}(3)$ consists of two elements $\tau_1$ and $\tau_2$
represented by the following trees:
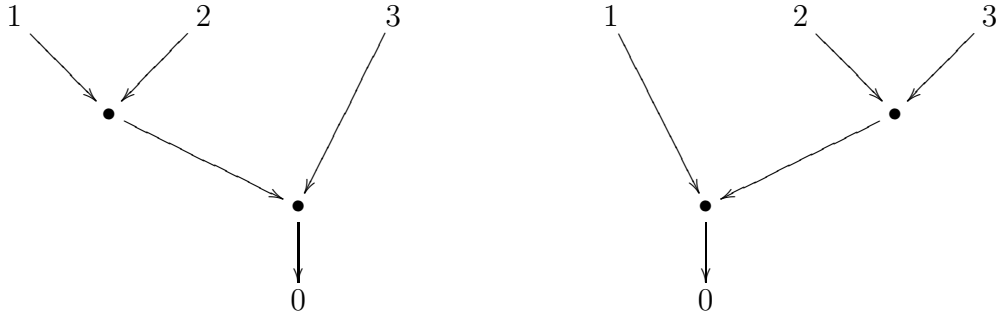
\[
\xymatrix{ 1\ar[dr]&&\ar[dl]2&&3\ar[ddl]\\
&\bullet\ar[drr]&&&\\
&&&\bullet\ar[d]&\\
&&&0&
} \quad\quad \quad\quad\quad
\quad
\xymatrix{
1\ar[ddr]&&2\ar[dr]&&3\ar[dl]\\
&&&\bullet\ar[dll]&\\
&\bullet\ar[d]&&&\\
&0&&&
}
\]
For a quadratic operad $\Oc$ the identifications
\[
\begin{gathered}
\coprod_{T\in\on{Bin}^{B, \tau_1} (b_1,  b_2, b_3|b_0)} \quad  \prod_{v\in\on{Vert}(T)}
\Oc\bigl(\on{in}'(v), \on{in}''(v)|\on{out}(v)\bigr) \lra \Oc(b_1, b_2|b_0) \lla   \\
\lla 
\coprod_{T\in\on{Bin}^{B, \tau_2} (b_1,  b_2, b_3|b_0)} \quad    \prod_{v\in\on{Vert}(T)}
\Oc\bigl(\on{in}'(v), \on{in}''(v)|\on{out}(v)\bigr)
\end{gathered}
\]
can be seen as quadratic relations among the binary generators of $\Oc$.  Furthermore, any two
triangulations of any $P_n$ are obtained from each other by a series of flips on 4-gons.
Equivalently, any two elements of any $\on{Bin}(n)$ can be obtained from each other by applying
local modifications consisting in replacing a subtree of form $\tau_1$ by a subtree of form $\tau_2$
or the other way around.  This means that all identifications (relations) among formal compositions
of binary generators which hold in the  $\Oc(n)$, $n\geq 3$,  follow from those holding already in
$\Oc(3)$. 
\end{rem}

\vfill\eject

\subsection{Set-theoretic solutions of the pentagon equation}
\label{subsec:set-theoretic-pentagon}

Let us illustrate the results of 
 \S \ref{subsec:2-segal-monoidal} on an extreme (but still nontrivial)
 class of $2$-Segal sets. 
 
 Let $X$ be a $2$-Segal semi-simplicial set such that $X_0=X_1=\pt$,
 and let $C=X_2$. 
 Theorem \ref{thm:segal-2-categories} associates to $X$ a distributive monoidal structure $\otimes$
 on the category $\Set$, which is given by
 \begin{equation}\label{eq:otimes-set-theoretic-pentagon}
 F\otimes F' = C\times F\times F'.
 \end{equation}
 Note that $\otimes$ does not have a unit object, unless $|C|=1$. 
 The associativity isomorphisms 
 \[
 \alpha_{F, F', F''}: (F\otimes F')\otimes F''
 \lra F\otimes(F'\otimes F'')
 \]
  for this structure all reduce to the case
 when $F=F'=F''=\pt$, i.e., to one bijection
 \begin{equation}\label{eq:bijection-alpha}
 \alpha: C\times C\lra C\times C. 
 \end{equation}
 The Mac Lane pentagon condition for this $\alpha$ now reads as
  the equality
 \begin{equation}\label{eq:pentagon-eq-set}
 \alpha_{23}\circ \alpha_{13}\circ \alpha_{12}= \alpha_{12}\circ \alpha_{23}
 \end{equation}
 of self-maps of $C^3= C\times C\times C$. Here, for instance, $\alpha_{23}$ means the transformation of
 $C^3$ which acts as $\alpha$ on the 2nd and 3rd coordinates and leaves the first coordinate intact. 
 A datum consisting of a set $C$ and a bijection $\alpha$ as in 
 \eqref{eq:bijection-alpha}, satisfying
   \eqref{eq:pentagon-eq-set}, is known as a
  {\em set-theoretic solution of
 the pentagon equation} \cite{kashaev-sergeev, kashaev-reshitikhin}.
 So Theorem \ref{thm:segal-2-categories} specializes, in our case, to the following:
 
 \begin{cor}\label{cor:2-segal-pentagon-point}
 Let $C$ be a set.
 The following categories are equivalent:
 \begin{itemize}
 
 \item[(i)] The category $2\Seg(\pt, \pt, C)$ formed by
 $2$-Segal semi-simplicial sets $X$ with $X_0=X_1=pt$, $X_2=C$ and their morphisms
 identical on $C$.

 \item[(ii)] The set of set-theoretic solutions $\alpha: C^2\to C^2$ of the pentagon equation.
  
\end{itemize}
 That is, the category $2\Seg(\pt, \pt, C)$ is discrete and isomorphism classes of its
 objects are in bijection with solutions $\alpha$ as in (ii). 
\end{cor}

We will call $2$-Segal set $X$ corresponding to a solution $(C,\alpha)$ the {\em nerve}
of $(C,\alpha)$ and denote by $\Nen(C,\alpha)$. The 
 bar-construction
description of $X$
 in Example \ref{ex:hall-mon-cat}(b) specializes, in our case, to the following. 
 We have $B=X_1=\pt$, the 1-element set, so by the form 
 \eqref{eq:otimes-set-theoretic-pentagon} of the monoidal operation,
 we have
\begin{equation}\label{eq:nerve-of-pentagon}
\Nen_n(C, \alpha) =B^{\otimes n} :=(\cdots(B\otimes B)\otimes \cdots )\otimes B)
 = C^{n-1}, \quad n\geq 1, \quad \Nen_0(C,\alpha)=\pt. 
\end{equation} 
 Alternatively, 
    the Clebsch-Gordan nerve construction (Example \ref{ex:CGN})
  identifies $\Nen_n(C,\alpha)$, $n\geq 2$, with the set of
       systems ${\bf x} = (x_{ijk}\in C)_{0\leq i<j<k\leq n}$,
        satisfying the following 
     ``nonabelian 2-cocycle
     condition"
     for each 4-tuple $0\leq i<j<k<l\leq n$:
\[
     (x_{ikl}, x_{ijk}) = \
     \alpha(x_{ijl}, x_{jkl}). 
\]

 As pointed out in \cite{kashaev-sergeev, kashaev-reshitikhin}, the map $\alpha$
  can be written in terms of two binary operations on $M$:
  \[
	  \alpha(x,y) =(x\bullet y, x*y),
  \]
 or, pictorially: 
   
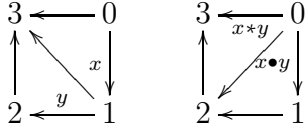
\[
	   \xymatrix{
		   3&0\ar[l] \ar[d]_{ x}
		   \cr
		   2\ar[u] &1\ar[l]_y\ar[ul]
	   }
	   \quad \quad
	   \xymatrix{
		   3&0\ar[l]^{x* y}\ar[dl] \ar[d]_{x\bullet y}
		   \cr
		   2\ar[u] &1\ar[l]
	   } \quad\quad\quad
   \]
 That is, given $x, y\in C=X_2$, we find the unique 3-simplex $d(x,y)$ whose even 2-faces
 are $x$ and $y$, and then $x\bullet y$ and $x*y$ are found as the odd 2-faces of $d(x,y)$,
 as indicated.

 Further, one can rewrite the pentagon equation as three 
 identities satisfied by these operations, of which
 we note the remarkable fact that $\bullet$ is associative:  
  \begin{equation}\label{eq:three-identities-pentagon}
 \begin{gathered}
 (x\bullet y)\bullet z =x\bullet(y\bullet z),\cr
 (x * y) \bullet ((x \bullet y) * z) = x * (y \bullet z), \cr
 (x * y) * ((x\bullet y) * z) = y * z.
 \end{gathered}
 \end{equation}
 Thus $\alpha$ gives rise, in particular, to a semigroup structure on $C$. 
 See Example \ref{ex:pentagon-path-criterion} below for a conceptual explanation
 of this associativity.

  \begin{ex}\label{ex:group-pentagon-solution}
    Let $G$ be a group. Then $\alpha: G^2\to G^2$ given by
 \[
  \alpha (x,y) = (xy, y),\quad x\bullet y= xy, x*y=y,
  \]
  is a solution of the pentagon equation,
  seee \cite{kashaev-double, kashaev-sergeev}.
For example, if $G=\ZZ$, then $\alpha: \ZZ^2\to\ZZ^2$ is the elementary matrix
  \[
  \alpha = e_{12} =\begin{pmatrix} 1&1\cr 0&1\end{pmatrix}\in SL_2(\ZZ),
   \]
   and the pentagon relation incarnates as the {\em Steinberg relation}
   among the elementary matrices:
   \[
   e_{12} e_{23} =e_{23} e_{13} e_{12} \in SL_3(\ZZ). 
   \]
 It was shown in {\em loc. cit.}
  that any solution $(G,\alpha)$ for which $x\bullet y$ makes $G$ into a group,
  is obtained in this way, i.e., has $x*y=y$.

     In particular, any group $G$ gives rise to a $2$-Segal semi-simplicial set
    $\Nen(G) = \Nen(G,\alpha)$ with $\Nen_n(G)=G^{n-1}$ for $n\geq 1$ and
    $\Nen_0(G)=\pt$. Denoting by $[g_1, ..., g_{n-1}]\in\Nen_n(G)$ the element
    corresponding to $(g_1, ..., g_{n-1})$ by 
    \eqref{eq:nerve-of-pentagon}, we find the face operations to be:
    \[
      \partial_i[g_1, ..., g_{n-1}]=
    \begin{cases}
     [g_2, ..., g_{n-1}],\quad i=0; \cr
      [g_2, ..., g_{n-1}],\quad i=1; \cr
 [g_1, ..., g_{i-1} g_{i}, ..., g_{n-1}], \quad i=2, ..., n-1;\cr
 [g_1, ..., g_{n-2}], \quad i=n. 
     \end{cases}
       \]
        Note that $\partial_1, ..., \partial_{n-1}$ are, up to shift,
      given by the same formulas as
        the faces in $N_{n-1}(G)$, the $(n-1)$st level of the usual
        nerve of $G$, while $\partial_0$ repeats $\partial_1$.
        This means that the $2$-Segal semi-simplicial set $\Nen(G)$
        (and thus the above solution of the pentagon
        equation) is obtained from the $1$-Segal simplicial set $N(G)$ by 
        a semi-simplicial version of taking the suspension. 
         See Proposition \ref{prop:suspension-$2$-Segal} below for more details. 
         
     \end{ex}
         
     \begin{ex}\label{ex:conf-plus}
      Let $V$ be a 2-dimensional oriented $\RR$-vector space. Let
     \[
     V^{\oplus (n+1)}_\circlearrowleft=\bigl\{ (v_0, ..., v_n)\in V^{\oplus(n+1)}\bigl|
     v_i\neq 0, \arg(v_0) <\arg(v_1) < \cdots < \arg(v_n) < \arg(v_0)
     \bigr\}
     \]
     be the space of tuples of nonzero vectors whose arguments are in the 
     strict anti-clockwise
order with respect to the chosen orientation. Let
\[
\on{Conf}_n^+ = GL^+(V)\backslash V^{\oplus (n+1)}_\circlearrowleft, \quad GL^+(V) =\{g\in GL(V)|
\det(g)>0\},
\]
be the corresponding configuration space. 
Thus $\on{Conf}_n^+ = \pt$ for $n=0,1$. Further, $\on{Conf}_2^+$ is identified with 
 $\RR_{>0}^2$ with coordinates $\lambda_0, \lambda_2$ by associating to $(v_0, v_1, v_2)$
 the coefficients in the expansion
$v_1 =\lambda_0 v_0 + \lambda_2 v_2$.
For $n\geq 2$ one sees easily that
 $\on{Conf}_n^+$ is a topological space homeomorphic to $\RR^{2n-2}$.

The collection $\on{Conf}^+=(\on{Conf}_n^+)_{n\geq 0}$ forms a semi-simplicial topological
space in an obvious way. Note that repeating a vector would violate the condition
of strict increase of the arguments, so there is no obvious way to make $\on{Conf}^+$
simplicial. 

    \end{ex}

\begin{prop}
$\on{Conf}^+$ is $2$-Segal (as a semi-simplicial set). 
\end{prop}

\begin{proof} We prove by induction that for any $n\geq 2$
and any triangulation $\T$ of $P_n$ the map $f_\T: \on{Conf}^+_n\to\on{Conf}^+_\T$
is a bijection (in fact, a homeomorphism).  
For $n=2$ the statement is obvious. Assume 
       that the statement
    holds for any $n'<n$ and any triangulation $\T'$ of $P_{n'}$.
    Let $\T$ be a triangulation of $P_{n}$ 
    and choose the unique $i$ such that
    $\{0,i,n\}\in\T$. Suppose that $1<i<n-1$, the cases $i\in\{1, n-1\}$ are similar.
    Then
     the edge $\{0,i\}$ dissects the $(n+1)$-gon $P_{n}$
    into two polygons $P^{(1)}$ and $P^{(2)}$ with induced triangulations $\Tc_1$ and $\Tc_2$. 
    We have then the commutative
    diagram
    \[
    \xymatrix{
    \on{Conf}^+_n\ar[r]^{\hskip -2.5cm u} \ar[d]_{f_\T}& \on{Conf}^+_{\{0, ..., i\}} 
    \times_{\on{Conf}^+_{\{0,i\}}}\on{Conf}^+_{\{0,i, i+1, ..., n\}}
    \ar[d]^{f_{\T_1}\times f_{\T_2}}\cr
    \on{Conf}^+_\T\ar[r]^{\hskip -1.5cm \simeq }& \on{Conf}^+_{\T_1}
    \times_{\on{Conf}^+_{\{0,i\}}}\on{Conf}^+_{\T_2}
      }
    \]
    Its lower horizontal arrow is a homeomorphism since $\Tc$ is composed out of $\Tc_1$
    and $\Tc_2$. 
     By induction, 
    $f_{\T_1}$ and $f_{\T_2}$ are homeomorphisms.
        So the same is true for right vertical arrow
  in the diagram. It remains to prove the same property for the arrow $u$.
  Note that $\on{Conf}^+_{0,i}=\pt$, so the fiber products in the diagram are the usual Cartesian products. 
  Given elements, i.e., 
   orbits ${\bf v}= GL^+(V)(v_0, ..., v_i)\in \on{Conf}^+_{\{0, ..., i\}}$ and 
   ${\bf w}= GL^+(V)(w_0, w_i, ..., w_n) 
  \in\on{Conf}^+_{\{0,i,..., n\}}$, 
  there is unique $g\in GL^+(V)$ taking the basis $(v_0, v_i)$ to the basis
  $(w_0, w_i)$. So we can assume that $v_0=w_0$, $v_i=w_i$. Then we see there can be at most
  one element ${\bf x}\in \on{Conf}^+_n$ such that $u({\bf x})=({\bf v}, {\bf w})$.
  This can only be the sequence ${\bf x}=(v_0, ..., v_i, w_{i+1}, ..., w_{n})$.
  On the other hand, our assumptions imply that this $\bf x$ indeed satisfies
  the anti-clockwise argument condition and so indeed lies in $\on{Conf}^+_n$.
    \end{proof}

    The $2$-Segal semi-simplicial set $\on{Conf}^+$ gives rise to a solution
    of the pentagon equation
      \[
     \alpha: \RR_{>0}^2 \times \RR_{>0}^2
     \buildrel\sim\over \lra \RR_{>0}^2\times \RR_{>0}^2, 
         \]
    which is a classical example of a ``cluster transformation", see
    \cite{kashaev-dilog}, Eq. (10). 
    To write $\alpha$ in the explicit form, as a map
     \[
     \alpha: (\lambda, \mu) = \bigl( (\lambda_0, \lambda_2), (\mu_0, \mu_2)\bigr) \longmapsto
     \bigl( (\lambda'_0, \lambda'_2), (\mu'_0, \mu'_2)\bigr) = (\lambda',\mu'),
     \quad \lambda,\mu,\lambda',\mu'\in\RR_{>0}^2,
     \] 
     one has to consider a generic 4-tuple $(v_0, v_1, v_2, v_3)\in\on{Conf}_3^+$
    and to compare two sets of linear relations corresponding to the two triangulations of $P_3$:
     \[
 \xymatrix{
 3&0\ar[l] \ar[d]_{\lambda}
 \cr
 2\ar[u] &1\ar[l]_\mu \ar[ul]
 }  
 \quad 
 \xymatrix{
 3&0\ar[l]^{\lambda'}\ar[dl] \ar[d]_{\mu'}
 \cr
 2\ar[u] &1\ar[l]
} 
 \quad\quad
 \xymatrix{
 v_1=\lambda_0 v_0 + \lambda_2 v_3, \quad v_1=\lambda'_0 v_0 + \lambda'_2 v_2 \cr
 v_2=\mu_0v_1+ \mu_2 v_3 
 \quad v_2=\mu'_0 v_0 + \mu'_2 v_2,
 }
 \]
 cf. \cite{doliwa-sergeev}, Eqs. (4.7-8), whose approach is the closest
 to our $2$-Segal point of view. This gives
 \begin{equation}\label{eq:simplest-cluster}
 \begin{cases}
 \lambda'_0 = (1+\lambda_2\mu_2^{-1}\mu_0)^{-1} \lambda_0\\
 \lambda'_2 = (1+\lambda_2\mu_2^{-1}\mu_0)^{-1}\lambda_2\mu_2^{-1},
 \end{cases}
 \quad\quad
 \begin{cases}
 \mu'_0 = \mu_0\lambda_0\\
 \mu'_2 = \mu_0\lambda_2+ \mu_2.
 \end{cases}
 \end{equation}
 Note that the second column of formulas describes,
 in agreement with 
 \eqref{eq:three-identities-pentagon}, 
 an associative 
binary operation on $\RR_{>0}^2$. This operation is given by multiplying matrices
of the form $\begin{pmatrix} \lambda_0&0\\ \lambda_2&1
\end{pmatrix}$, $\lambda_i>0$, and makes $\RR_{>0}^2$ into a semigroup but not a group. 

\vfill\eject

\subsection{Pseudo-holomorphic polygons as a 2-Segal space}
\label{subsec:pseudo-holomorphic}

In this section we describe a large class of discrete $2$-Segal spaces
associated to almost complex manifolds.

We start with an interpretation of the $2$-Segal space $\on{Conf}^+$
from Example \ref{ex:conf-plus} in terms of
the decorated Teichm\"uller spaces of Penner \cite{penner}. Let $\HH$ be the Lobachevsky plane
 with the group of motions $SL_2(\RR)$.
 We can realize $\HH$ inside the complex plane $\CC$ in one of the two classical ways:
 \begin{enumerate}
 \item[(1)] As the upper half plane $\{\on{Im}(z)>0\}$, with $SL_2(\RR)$ acting by fractional
 linear transformations. It is equipped with the Riemannian metric
 \[
 ds^2 =  \frac{dxdx}{y^2}, \quad z=x+iy 
 \]
 of constant curvature $(-1)$. 
 
 \item[(2)] As the unit disk $\{|z|<1\}$, obtained as the image of the upper half plane
 under the fractional linear transformation $w=(z-i)/(z+i)$. 
 \end{enumerate}
 
The {\em absolute} (or {\em ideal}) {\em boundary} $\partial\HH$ of $\HH$ is identified with $\RR P^1$, the real projective line.
In the realization (1) the identification of $\partial\HH$ with $\RR\cup \{\infty\}=\RR P^1$ is immediate;
in the relalization (2), $\partial \HH$ is identified with the unit circle which is more convenient for
drawing pictures. We equip $\HH$ with the orientation coming from the complex structure. This defines
a canonical orientation of $\partial\HH$ which we will refer to as the {\em counterclockwise orientation}.
This is indeed counterclockwise in the realization (2).

 {\em Geodesics} in $\HH$ are represented, in either realization, by circle arcs (or straight lines) meeting the absolute
 at the right angles. Any two distinct points $x,x'\in\HH$ give rise to a geodesic arc $[x,x']\subset \HH$. 
 Further, any two distinct ponts $b, b'\in\partial \HH$ give rise to an infinite geodesic $(b,b')$ with
 limit positions $b, b'$. We equip it with the orientation going from $b$ to $b'$.
 Note that the intrinsic geometry on $(b, b')$ (coming from the Riemannian metric above)
 is that of a torsor over the
 additive group $\RR$. A choice of a base point $x\in (b,b')$ identifies $(b,b')$ with $\RR$.

 By an {\em ideal polygon} in $\HH$ we mean a geodesic polygon $P$ with vertices
 $b_0, ..., b_n$ lying on the
 absolute and numbered in the counterclockwise order.
Such a  polygon as above is uniquely
 determined by the choice of the $b_i$ and will be denoted by $P(b_0, ..., b_n)$.  
  Here we assume that $n\geq 1$.
 For $n=1$, an ideal 2-gon is understood to be a geodesic $P(b_0, b_1)= (b_0, b_1)$.

 Two ideal polygons $P, P'$ are called {\em congruent}
 if there is a rigid motion taking $P$ to $P'$ and preserving the numeration of vertices. 
 Thus, denoting by $\TT_n$ the set of congruence classes of ideal $(n+1)$-gons, 
 we have
 \[
\TT_n \simeq SL_2(\RR) \backslash (\RR P^1)^n_\circlearrowleft,  
 \]
 where $(\RR P^1)^n_\circlearrowleft\subset (\RR P^1)^n$ is the set of tuples of points $(b_0, ..., b_n)$ going in the counterclockwise
 order (in particular, distinct from each other). Note that $\TT_n$ can be regarded as the simplest Teichm\"uller space:
 the set of hyperbolic structures (identiffications with an ideal polygon)
 on the standard $(n+1)$-gon $P_n$ from \S \ref{subsec:2-segal-spaces}.

 For $x,y\in\HH$ we denote $d(x,y)$ the geodesic distance from $x$ to $y$. For $r\in\RR_+$ we denote
 by $S_r(x)$ the geodesic circle with center $x$ and radius $r$. 
 By a {\em horocycle} in $\HH$ one means an orbit of a subgroup conjugate to 
 \[
 \begin{pmatrix} 1&t\\0&1
 \end{pmatrix}\subset SL_2(\RR). 
 \]
 It is convenient to think of horocycles as ``circles of infinite radius" with center at the boundary,
 i.e., as limit positions of geodesic circles $S_r(x)$ when $x\to b\in\partial \HH$ and $r\to\infty$.
Thus, a horocycle $\xi$ has a {\em center} $b\in\partial\HH$ and consists, informally, of points
$y$ lying at a fixed (but infinite) distance from $b$. One can refine this picture by saying that
a horocycle $\xi$ with center $b$ has a {\em radius} which is not a number but an element of
a certain torsor $\on{Hor}_b$ over the additive group $\RR$, and writing
\[
d(b,y) \in \on{Hor}_b, \quad y\in \xi. 
\]
Thus for $y, y'\in\HH$ the ``infinite distances" $d(b,y), d(b,y')\in\on{Hor}_b$ can be compared, i.e., we have
a well defined (``finite") real number
$d(b,y)-d(b,y')\in\RR$. In a formal set-theoretic way, one can say that $\on{Hor}_b$
{\em consists of} horocycles with center $b$. The following obvious fact will be useful.

\begin{prop} Let $b\in\partial\HH$. Then for any $b'\in\partial\HH$ different from $b$
the geodesic $(b,b')$, considered as an $\RR$-torsor, is canonically identified with $\on{Hor}_b$.

\end{prop}

\begin{proof} Every horocycle with center $b$ meets $(b,b')$ in a unique point. \end{proof}

Let $n\geq 1$. By a {\em decorated ideal $(n+1)$-gon} we will mean, following \cite[Ch. 2, \S 1.1]{penner},
a datum consisting of an ideal $(n+1)$-gon $P=P(b_0, ..., b_n)$ and a choice of a horocycle
$\xi_i\in\on{Hor}_{b_i}$ around each vertex. Two decorated ideal $(n+1)$-gons
\[
\bigl(P(b_0, ..., b_n), \xi_0, ..., \xi_n)\bigr) \quad\text{and}\quad (P(b'_0, ..., b'_n), \xi'_0, ..., \xi'_n\bigr)
\]
 are called {\em similar}, if there exist $g\in SL_2(\RR)$ and $a\in\RR$ such that 
 \[
 g(b_i) = b'_i, g(\xi_i) = \xi'_i + a, \quad i=0, ...,n, 
 \]
 where addition with $a$ is understood in the sense of the $\RR$-torsor structure on $\on{Hor}_{b'_i}$. 
 We denote by $\widetilde\TT_n$ the set of similarity classes of decorated $(n+1)$-gons. 
 For $n=0$ we put $\widetilde\TT_1=\pt$.
 
 \begin{ex}\label{ex:decorated-geodesics}
 (a) Consider an ideal 2-gon, i.e., an oriented geodesic $(b_0, b_1)$. A decoration of $(\xi_9,\xi_1)$,
 i.e., a choice of horocycles $\xi_i\in\on{Hor}_{b_i}$,
produces two intersection points $x_i = \xi_i\cap (b_0, b_1)$ and their midpoint $\omega( \xi_0, \xi_1)$. 
As $(b_0, b_1)$ is naturally an $\RR$-torsor, the choice of $\omega( \xi_0, \xi_1)$ as the origin
defines an identification of $(b_0, b_1)$ with $\RR$, i.e., a {\em coordinate system} on $(b_0, b_1)$. 
Note that after addition of any $a\in\RR$, we have
\[
\omega(\xi_0+a, \xi_1+a) =\omega(\xi_0, \xi_1),
\]
so the corresponding coordinate remains unchanged. Thus $\widetilde \TT_2$ is the set of congruence
classes of geodesics together with a choice of affine coordinate, and so $\widetilde\TT_2=\pt$ as well.

(b) We conclude that for a decorated polygon $\bigl(P(b_0, ..., b_n), \xi_0, ..., \xi_n\bigr)$ each ``diagonal",
i.e., each geodesic $(b_i, b_j)$ has a canonical coordinate, and similarity transformations
preserve these canonical coordinates. 
\end{ex}

If $\phi: [m]\to [n]$ is a monotone embedding, then we have a map $\phi^*: \widetilde\TT_n\to\widetilde\TT_m$
associating to a decorated polygon on vertices $(b_0, ..., b_n)$ its subpolygon
on vertices $b_{\phi(0)}, ..., b_{\phi(m)}$ with the corresponding horocycles. This makes
$\widetilde\TT = (\widetilde\TT_n)_{n\geq 0}$ into a semi-simplicial space.

\begin{prop}
$\widetilde\TT$ is isomorphic to the semi-simplicial space 
$\on{Conf}^+$
from Example \ref{ex:conf-plus}. 
\end{prop}

\begin{proof} The space of all horocycles in $\HH$ is
\[
\begin{pmatrix} 1&t\\0&1
 \end{pmatrix}\bigg\backslash SL_2(\RR)
 \quad = \quad \RR^2-\{0\}. 
  \]
   Under this identification, the action of $a\in \RR$ on the horocycle torsors corresponds to the action of
   the scalar matrix
  $e^a\cdot\1\in GL_2(\RR)$ on $\RR^2-\{0\}$. 
  Thus the set of all decorated $(n+1)$-gons is $(\RR^2)^{\oplus(n+1)}_\circlearrowleft$, and the group of
  similarity transformations acting on this set is
  \[
  (\RR_+^\times\cdot \1) \cdot SL_2(\RR) = GL_2^+(\RR),
  \]
  making the comparison with $\on{Conf}^+_n$ immediate. \end{proof}
 
\begin{rem}\label{rem:TT-2-segal}
Note that the $2$-Segal condition for $\widetilde\TT$ is much more obvious than for $\on{Conf}^+$.
Indeed, suppose we have a triangulation $\T$ of the standard polygon $P_n$, and decorated
hyperbolic structures on all the triangles $P\in\Tc$, compatible (up to similarity) on their common sides.
By Example \ref{ex:decorated-geodesics}, each of
these common sides acquires a
  well-defined coordinate 
(identification with $\RR$). 
This allows us to glue the hyperbolic structures together in a unique way. 
\end{rem}

Let now $M$ be an almost complex manifold, i.e., a $C^\infty$-manifold of even dimension $2d$ with
a smooth field $J$ of complex structures in the fibers of the tangent bundle $T M$, see
\cite{macduff-salamon}. In particular, $M$ can be a complex manifold in the usual sense,
in which case $J$ is called {\em integrable}.

Morphisms of almost complex manifolds
are called {\em pseudo-holomorphic maps}.
In particular, by a {\em pseudo-holomorphic function}, resp. {\em pseudo-holomorphic curve} on $M$ one means a
(locally defined) morphism of almost complex manifolds $M\to\CC$, resp. $\CC\to M$. 
If $J$ is non-integrable, $M$ may have very few pseudo-holomorphic functions but still has a
large supply of pseudo-holomorphic curves. 
In particular, an ideal polygon $P$ can be regarded as a 1--dimensional complex manifold with
boundary and so can be takes as a source of pseudo-holomorphic maps into $M$.

\begin{defi} Let $M$ be an almost complex manifold. We put $\widetilde\TT_0(M)=M$. For $n\geq 1$
we define $\widetilde\TT_n(M)$ to be the set of equivalence of the data consisting of:
\begin{enumerate}
\item[(1)] A decorated ideal $(n+1)$-gon 
 $\bigl(P=P(b_0, ..., b_n), \xi_0, ..., \xi_n\bigr)$.
 
 \item[(2)] A   continuous map $\gamma: P\to M$ which is pseudo-holomorphic on the interior of $P$.
   Here we assume that 
 $P$ is compact, so it contains all the ideal vertices. 
  
\end{enumerate}
These data are considered up to similarity of decorated ideal $(n+1)$-gons in the source. 
\end{defi}

For example, $\widetilde\TT_1(M)$ is simply the path space of $M$. More precisely, it is
the space of continuous maps $[-\infty, +\infty] \to M$, where $[-\infty, +\infty]$ is the interval obtained
by compactifying $\RR$ by two points at infinity.

We make $\widetilde\TT(M) = (\widetilde\TT_n(M))_{n\geq 0}$ into a semi-simplicial set
as follows. The maps
$\partial_0, \partial_1: \widetilde\TT_1(M)\to \widetilde\TT_0(M)=M$ are defined to be the
evaluation maps of paths as above at $(+\infty)$ and $(-\infty)$ respectively. 
For $m\geq 1$ and a monotone embedding $\phi: [m]\to[n]$ the map
$\phi^*: \widetilde\TT_n(M)\to\widetilde\TT_m(M)$
is defined by forming the subpolygon in $P(b_0, ..., b_n)$ on the
vertices $b_{\phi(i)}$, with the induced decoration and map into $M$. 
We call $\widetilde \TT(M)$
the {\em space of pseudo-holomorphic polygons} on $M$.

\begin{prop}
For any almost complex manifold $M$, the semi-simplicial set $\widetilde\TT(M)$ is $2$-Segal. 
\end{prop}

\begin{proof} Suppose given a triangulation $\T$ of the standard polygon $P_n$.
Let us prove that the $2$-Segal map
\[
f_\T: \widetilde\TT_n(M) \lra\widetilde\TT_\T(M)
\]
is a bijection. An element of the target of $f_\T$ is a system $\Sigma$ of 
 decorated
hyperbolic structures on all the triangles $P\in\Tc$ and maps $\gamma_P: P\to M$,
compatible on the common sides. As in Remark \ref{rem:TT-2-segal}, these common
sides acquire canonical coordinates and so can be identified with each other,
thus producing an identification of $P_n$ with a decorated ideal $(n+1)$-gon
$P= P(b_0, ..., b_n)$ for some $b_0, ..., b_n$. We can then view $\T$ as a triangulation of  
 $P$ into ideal triangles. Further, the maps $\gamma_P$,
 being compatible on the sides of these triangles, define a continuous
 map $\gamma: P\to M$ which is pseudo-holomorphic in the interior of
 each triangle of the triangulation. Now, it is a fundamental property of
 the Cauchy-Riemann equations defining pseudo-holomorphic curves
that such a map is pseudo-holomorphic everywhere in the interior of $P$.
We therefore obtain a (necessarily unique) datum $(P, \xi_0, ..., \xi_n, \gamma)
\in \widetilde\TT_n(M)$ lifting $\Sigma$. \end{proof}

\begin{rem}
We have therefore two large classes of $2$-Segal spaces: 
\begin{enumerate}
\item[(a)] Waldhausen spaces,
 encoding homological algebra data in exact categories (and, more generally,
dg- and $\infty$-categorical enhancements of
 triangulated categories, see \S \ref{subsec:waldhausen-exact-infty} below).
 
 \item[(b)] Spaces $\widetilde \TT(M)$ encoding geometry of pseudo-holomorphic polygons. 
 \end{enumerate}
 
  \noindent It is tempting to conjecture some kind of ``homological mirror symmetry"
  relation between these two classes of spaces,

\end{rem}

   \vfill\eject

\subsection{Birationally 1- and 2-Segal semi-simplicial schemes}
\label{subsec:bir-segal}

Let $\FF$ be a field. By a {\em scheme} in this section we will mean a $\FF$-scheme. 
Let $\Sc ch$ be the category of such schemes.
This category has finite limits, so for any semi-simplicial scheme $X\in \Sc ch_{\Delta_{\inj}}$
and any triangulation $\Tc$ of the polygon $P_n$ we have the scheme $X_\T$
and the morphism of schemes
$f_\Tc: X_n\to X_\Tc$.

A morphism $g: S\to S'$ in $\Sc ch$ will be called {\em birational}, if there are open,
Zariski dense subschemes $U\subset S, U'\subset S'$ such that $g$ induces an
isomorphism $U\to U'$. 

\begin{defi} Let $X\in \Sc ch_{\Delta_{\inj}}$ be a semi-simplicial scheme. 

(a) 
We say that $X$ is {\em birationally $1$-Segal}, if for any $n\geq 2$ the morphism of
schemes
\[
f_n: X_n\lra X_1\times_{X_0} X_1 \times_{X_0} \cdots \times_{X_0} X_1\quad \text{($n$ times)}
\]
is birational.

(b) We say that $X$ is {\em birationally $2$-Segal}, if for any $n\geq 2$
and any triangulation $\T$ of $P_n$, the morphism $f_\T$ is birational. 
\end{defi}

\begin{rem} For a birationally $2$-Segal scheme $X$ and any two triangulations $\Tc, \Tc'$
of $P_n$ we get not a regular, but a rational map of schemes
\[
f_{\T,\T'} = f_{\T'}\circ f_\T^{-1}: X_\T \lra X_{\T'}
\]
which form a transitive system of birational isomorphisms.
Such transitive systems appear in the theory of cluster algebras (see \cite{fomin-zelevinsky} 
\cite{fock-goncharov} \cite{gekhtman}).
If, in addition, $X_0=X_1=\pt$ , then, taking $n=3$ and $\T, \T'$ to be the two triangulations of
$P_3$, we get a birational solution of the pentagon equation, $\alpha=f_{\T, \T'}: X_2^2\to
X_2^2$. Such birational solutions  are important in applications \cite{kashaev-sergeev} and are
somewhat more abundant than solutions that are everywhere defined (regular).  This motivates the
study of birationally $2$-Segal schemes.  

\end{rem}

One can get examples of birationally 1- and $2$-Segal
semi-simplicial schemes by modifying the construction of the Hecke-Waldhausen
space from \S \ref{subsec:hecke-waldhausen}. 

Let $E$ be an irreducible quasi-projective variety and $G$ an algebraic group acting on $E$.
We say that the $G$-action on $E$ is {\em generically free} is there is dense Zariski open
$G$-invariant subset $U\subset E$ on which the action in free. 
In this case we have the variety $G\backslash U$. 
Note that if the diagonal $G$-action
on some $E^m$ is generically free, then the action on each $E^{m'}, m'\geq m$, is generically free
as well.

\begin{thm}\label{thm;birationally-segal}
(a) Suppose that the $G$-action on $E$ is generically free. Then there are $G$-invariant open sets
$U_n\subset E^{n+1}$, $n\geq 0$, with free $G$-action such that putting 
$X_n = G_n\backslash U_n$ defines a birationally $1$-Segal semi-simplicial
scheme $X$.

(b) Suppose that the $G$-action on $E^2$ is generically free. Then there are
$G$-invariant open sets
$U_n\subset E^{n+1}$, $n\geq 1$, with free $G$-action such that putting 
\[
X_0=pt, \quad X_n = G_n\backslash U_n, n\geq 1, 
\]
defines a birationally $2$-Segal semi-simplicial
scheme $X$.

\end{thm}

\begin{proof} The conceptually easiest proof is by  
lifting of the Hecke-Waldhausen construction into the setting of algebraic stacks,
see \cite{laumon-stacks}.
That is, for each $n\geq 0$ we consider the {\em quotient stack}
$[\Sc_n(G,E)] = [G\bbs E^{n+1}]$ of the scheme $E^{n+1}$ by the action of $G$,
For example, if $\FF$ is algebraically closed, the groupoid
of $\FF$-points of this stack is the quotient groupoid 
$\Sc_n(G(\FF), E(\FF))$. Taken together, these stacks form
a simplicial stack $[\Sc_\bullet(G,E)]$.  
Proposition \ref{prop:hecke-waldhausen-1-segal}
applied to various groupoids of points implies that 
$[\Sc_\bullet(G,E)]$ is $1$-Segal in the sense of stacks.
This means that for each $n$ the
morphism of stacks
\[
f^{[\Sc]}_n: [\Sc_n(G,E)] \lra [\Sc_1(G,E)]\times^{(2)}_{[\Sc_0(G,E)]}
\cdots \times^{(2)}_{[\Sc_0(G,E)]}
[\Sc_1(G,E)],
\]
(with $\times^{(2)}$ being the fiber product of stacks), is an equivalence of stacks. 
Now, if we are in the situation of part (a) of the theorem, we can choose (inductively)
open dense $G$-invariant subsets $U_n\subset E^{n+1}$, $n\geq 0$, with free
$G$-action such that the face maps (coordinate projections) take each $U_n$ inside $U_{n-1}$.
Then for each $n$ we get a scheme $X_n= G\backslash U_n$ which is an open sense subscheme
in the stack $[\Sc_n(G,E)]$ so that $X=(X_n)$ is a semi-simplicial scheme.
So the morphism of schemes
\[
f^X_n: X_n \lra X_1\times_{X_0} \cdots \times_{X_0} X_1
\]
becomes an open dense sub-morphism of the equivalence of stacks $f_n^{[\Sc]}$.
Therefore it is birational.

Suppose now that we are in the situation of part (b) of the theorem. We then use
Proposition \ref{prop.1segal2segal} which implies 
(either directly, by applying it to various groupoids of points,
or by imitating the proof) that $[\Sc_\bullet(G,E)]$ is $2$-Segal as a stack.
In other words, for any triangulation $\Tc$ of the polytope $P_n$ the morphism of stacks
\[
f_\Tc^{[\Sc]}: [\Sc_n(G,E)] \lra [\Sc_\T(G,E)] = 
2\pro_{\{\Delta^p\hookrightarrow \Delta^{\Tc}\}_{p=1,2}}
[\Sc_p(G,E)]
\]
is an equivalence of stacks. Here $2\pro$ is the projective 2-limit of stacks, and it is
enough to take this limit only over the embeddings of edges and triangles of $\Tc$. 
On the other hand, if the $G$-action on $E^{n+1}, n\geq 1$ is free, we can choose
as before, 
open dense $G$-invariant subsets $U_n\subset E^{n+1}$, $n\geq 1$ with free
$G$-action such that the face maps take each $U_n$, $n\geq 2$, inside $U_{n-1}$.
Then for each $n\geq 1$ we get a scheme $X_n= G\backslash U_n$ which is an open dense subscheme
in the stack $[\Sc_n(G,E)]$, and augmenting this by $X_0=\pt$, we get a semi-simplicial scheme.
To see now that $X$ is birationally $2$-Segal, we notice, as before, that the $2$-Segal morphism for $X$
\[
f_\Tc^X: X_n \lra X_\T =\pro_{\{\Delta^p\hookrightarrow \Delta^{\Tc}\}_{p=1,2}} X_p
\]
is an open dense sub-morphism of the equivalence of stacks $f_\Tc^{[\Sc]}$, so it
is birational. \end{proof}

\begin{exas}
(a) Let $G$ be a split semisimple algebraic group, $T\subset B\subset G$ a
maximal torus and a Borel subgroup,
and $N = [B,B]$ the unipotent radical. Let also $W$ be the Weyl group
of $T$. Take $E=G/N$, so we have a principal $T$-bundle
$p: E\to G/B$. As well known (Bruhat decomposition),
$G$-orbits on $(G/B)^2$ are parametrized by elements of $W$
and we denote by $(G/B)^2_{\on{gen}}$ the unique open orbit. We say that two points $b, b'\in G/B$
are {\em in general position}, if $(b,b')\in (G/B)^2_{\on{gen}}$.
In this case the stabilizer of $(b,b')$
in $G$ is a conjugate of $T$. 

We say that $x, x'\in E$ are in
general position, if $p(x), p(x')\in G/B$ are in general position.  
Let $E^{n+1}_{\on{gen}}\subset E^{n+1}$
be the open subvariety formed by $(x_0, ..., x_n)$
which are pairwise in general position. It follows that for $n\geq 1$ the $G$-action
on $E^{n+1}_{\on{gen}}$ is free, so we are in the situation of part (b) of Theorem
\ref{thm;birationally-segal} and the semi-simplicial algebraic variety $X$
defined by
\[
X_0=\pt, \quad X_n = G\backslash E^{n+1}_{\on{gen}}
\]
is birationally $2$-Segal. Note that $X_1=T$ is identified with the torus. 
The birational transformations
$f_{\T'}\circ f_\T^{-1}: X_\T\to X_{\T'}$ for different pairs of triangulations $\Tc, \Tc'$
of $P_n$ are in this case, cluster coordinate transformations
studied by Fock and Goncharov \cite{fock-goncharov}.

(b) Let $V$ be a 2-dimensional $\FF$-vector space 
and $G=GL(V)$ considered as an algebraic
group. Put $E=V$ considered as an algebraic variety and let $V^{\oplus(n+1)}_{\on{gen}}$ be the 
open part formed by $(v_0, ..., v_n)$ such that each subset of cardinality $\leq 2$ is linearly
independent. For $n\geq 1$ the group $G$ acts on $V^{\oplus( n+1)}_{\on{gen}}$ freely,
so the semi-simplicial variety $\on{Conf}$ defined by
$\on{Conf}_n = GL(V) \backslash V^{\oplus( n+1)}_{\on{gen}}$,
is birationally $2$-Segal. Note that both $\on{Conf}_0$ and $\on{Conf}_1$
reduce to one point,
while $\on{Conf}_2$ is identified with the 2-dimensional algebraic torus $\GG_m^2$,
by associating to $(v_0, v_1, v_2)$ the coefficients of the expansion
$v_1=\lambda_0v_0+\lambda_2 v_2$, similarly to Example
\ref{ex:conf-plus}. 
Therefore $\on{Conf}$
gives rise to a birational solution of the pentagon equation
\[
\alpha: \GG_m^2\times\GG_m^2 \lra \GG_m^2\times\GG_m^2.
\]
It is given by the same formulas as in \eqref{eq:simplest-cluster}, see \cite{kashaev-dilog}
as well as \cite{doliwa-sergeev} which considers a more general situation allowing
$\FF$ to be a noncommutative division ring.    

(c) More generally, a symmetric factorization of an algebraic group $G$ in the sense
of \cite {kashaev-reshitikhin} gives a closed subgroup $K$ such that the diagonal action
of $G$ on $(G/K)^2$ contains an open orbit isomorphic to $G$. Therefore taking $E=G/K$
we get, by Theorem \ref{thm;birationally-segal}(b), a birationally $2$-Segal semi-simplicial set
$X$ with $X_0=X_1=\pt$.
It corresponds to the birational solution of the pentagon equation found in {\em loc. cit}. 

\end{exas}

\begin{exa} Completely different classes of examples of birationally $2$-Segal
simplicial schemes
can be extracted from the theory of ``$N$-valued groups"
as studied in \cite{buchstaber-rees, buchstaber-dragovic}. Let us  
express, in our language, one such class: that of orbit spaces.

Let $G$ be an algebraic group, and $\Gamma\subset \Aut(G)$ be a finite subgroup, $|\Gamma|=N$.
The orbit space $\Gamma\backslash G$ is then a (typically singular) algebraic
variety. It is not a group, but the group structure on $G$ gives rise to
an ``$N$-valued composition law" on $\Gamma\backslash G$ which is
represented by the span
\[
\mu =\bigl\{ (\Gamma\backslash G) \times (\Gamma\backslash G)
\buildrel s \over \lla \Gamma\backslash(G\times G) \buildrel 
m\over \lra \Gamma\backslash G\bigr\},
\]
where $m$ is induced by the multiplication in $G$, and $s$ is generically $N$-to-1. 
The associativity of $G$ implies then that the two spans
\[
\xymatrix{
\mu\circ(\mu\times \Id), \mu\circ(\Id\times\mu): (\Gamma\backslash G)^3 
\ar@{~>}[r]&\Gamma\backslash G
}
\]
are identified over the generic point, i.e., are
connected by a birational isomorphism $\alpha$ satisfying the pentagon condition.
Alternatively, $\Gamma$ acts by automorphisms of $\N G$, the nerve of $G$ considered
as a simplicial algebraic variety and so gives rise to the quotient simplicial variety
$X= \Gamma\backslash (\N G) = (\Gamma\backslash G^n)_{n\geq 0}$ with the simplicial maps
induced by those in $\N G$. This simplicial variety is birationally $2$-Segal. 

\end{exa}

\begin{rem} An interesting particular case is when $G$ is an abelian surface over $\FF$,
and $\Gamma = \{\Id, \sigma\}$, where $\sigma$ is the involution $a\mapsto (-a)$. 
In this case $\Gamma\backslash G$ is a (singular)  {\em Kummer surface}, a special 
type of  a K3 surface.
While K3 surfaces are often regarded as ``quaternionic"  analogs of elliptic curves,
they do not carry any group operation in the usual sense. The above example
shows that at least in the Kummer case they carry a 2-valued operation. 
 We do not know whether such operations exist for more general K3 surfaces.

 \end{rem}

\vfill\eject

\section{Model categories and Bousfield localization} 
\label{sec:model-cat}

\subsection{Concepts from model category theory}

For a systematic study of
$2$-Segal spaces it is convenient to work in the more general framework of
model categories. In this section we summarize its main features, referring for more details to 
\cite{hovey} as well as \cite[Appendix 2]{lurie.htt}.

Let $\mC$ be a category and $f: A\to B$, $g: C\to D$ morphisms in $\mC$.
We write $f \perp g$, if for any commutative square
\[
\xymatrix{
A\ar[d]_f \ar[r]&C\ar[d]^g
\\
B\ar@{.>}[ur]\ar[r]&D
}
\]
there exists a dotted arrow making the two triangles commutative. The standard terminology is
that $f$ has the {\em left lifting property} with respect to $g$, and $g$ has the {\em right lifting property}
with respect to $f$. For a class of morphisms $S\subset\Mor(\mC)$ we denote
\begin{footnote}{
This convention, naturally suggested by the notation $f\perp g$, is opposite to that of
\cite[A.1.2]{lurie.htt} where this notation is not used. Note that orthogonality connotation
suggested by $f\perp g$ is quite in line with categorical interpretation of orthogonality
as absense of nontrivial morphisms. Indeed, viewing $f$ and $g$ as two-term chain complexes,
the lifting property can be read as
``each morphism from $f$ to $g$ is null-homotopic". }
\end{footnote}
 by $_\perp S$, resp. $S_\perp$
the classes formed by morphisms $g$ such that $g\perp f$, resp. $f\perp g$ for any $f\in S$. 
 
Recall that a {\em model structure}
 on a category $\mC$ is given by specifying three classes of morphisms: $\Wen$ (weak equivalences),
 $\Cen$ (cofibrations), and $\Fen$ (fibrations),
 satifying the axioms of Quillen \cite[Def. 1.1.3]{hovey}, in particular the {\em lifting axioms}:
\[
 \Fen = (\Wen\cap\Cen)_\perp, \quad \Cen =  _\perp(\Wen\cap\Fen). 
\]
Note that a category can have
several model structures.
Morphisms in $\Wen\cap\Cen$ (resp. $\Wen\cap\Fen$) are called {\em trivial cofibrations} (resp. {\em trivial fibrations}). 
 
A category with a model structure is called a {\em model category},
if it has small limits and colimits. 
By $\h\mC$ we denote the {\em homotopy category} of $\mC$ obtained by
formally inverting weak equivalences. 
The following examples will be important for us.

 \begin{exa}[(Trivial model structures)] Any category $\mC$ with 
 small limits and
	 colimits becomes a model category
	 with respect to the {\em trivial model structure} for which $\Cen=\Fen=\Mor(\mC)$,
	 and $\Wen$ consists of all isomorphisms. 
 
 \end{exa}
 
 \begin{ex}[(Topological spaces)] The category $\Top$ of compactly generated
 Hausdorff spaces is a model category with respect to the model structure described in 
  \cite[\S 2.4]{hovey}. For this structure, $\Wen$ consists of weak equivalences
  as defined in \S\ref{subsec:homotopy-limits-spaces}, $\Fen$ consists of Serre
  fibrations, and $\Cen$ contains all embeddings of CW-subcomplexes into CW-complexes. 
   
  \end{ex}
  
  \begin{ex}[(Simplicial sets)] 
 The category $\Sp$ of simplicial sets is equipped with the classical {\em Kan model structure},
see, e.g., \cite{lurie.htt}, \S A.2.7, which is given by the following data.
\begin{itemize}
	\item[(W)] A morphism $f: X \to Y$ of simplicial sets is a weak equivalence if the induced
	 map $|f| : |X| \to |Y|$ of geometric realizations is a homotopy equivalence of
	 topological spaces.
	\item[(C)] $f$ is a cofibration if the induced maps of sets
	 $f_n: X_n \to Y_n$ are injective for all $n \ge 0$. In
	 particular, every object is cofibrant.
	\item[(F)] $f$ is a Kan fibration if it has the right
	 lifting property with respect to the maps $\Lambda_i^n \to
	 \Delta^n$, $i=0, ..., n$. Here $\Lambda_i^n$ 
	 denotes the $i$th horn of $\Delta^n$.  
	 \end{itemize}
		
\end{ex}

\begin{ex}[(Groupoids and categories)] \label{exa:groupoidsmodel}
The category $\Gc r$ has the {\em Bousfield model structure}
\cite{bousfield-groupoids}, which is given by the following data.
\begin{itemize}
\item[(W)] A functor $F: \Gc \to\Gc'$ is a weak equivalence if it is an equivalence of 
categories.

\item[(C)] $F$ is a cofibration if it induces an injection of sets $\Ob(\Gc)\to\Ob(\Gc')$.

\item[(F)] $F$ is a fibration if, for every object $x\in\Gc$ and every isomorphism $h: F(x)\to y$
in $\Gc'$, there exists an isomorphism $g: x\to x'$ in $\Gc$ such that $F(g)=h$.
\end{itemize}
The Bousfield model structure on $\Gc r$ can be extended to a model structure on the category $\Cat$ of small
categories. This is explained in detail in \cite{rezk:cat-model}.
\end{ex}

We will freely use the basic concepts of 
{\em Quillen adjunctions}, {\em left and right Quillen functors}, and 
{\em Quillen equivalences} of model categories, see e.g., \cite{hovey}. 

\begin{exa} The category $\Sp$ of simplicial sets equipped with the Kan model structure is
	Quillen equivalent to the model category $\Top$.
\end{exa}

\begin{exa}
The model category $\Gc r$ is Quillen equivalent to the full subcategory in $\Sp$ formed by
simplicial sets $X$ with $\pi_{\geq 2}(|X|,x)=0$ for every $x\in X_0$.
\end{exa}

We will further use the concept of a {\em combinatorial model category} due to J. Smith, 
which intuitively means ``a model category of algebraic nature". See \cite[A.2.6]{lurie.htt}
for more details. 

\begin{defi}\label{def:comb-model-cat}
 A model category $\mC$ is called {\em combinatorial} if 
	\begin{enumerate}[label=(\arabic{*})]
	 \item The category $\mC$ is presentable, i.e. there is a set (not a class) of objects
	 $C \subset \Ob(\mC)$ such that every object of $\mC$ is a colimit of a diagram formed
	 by objects in $C$. 
	 \item There exists a set $I \subset \Cen$ such that $\Cen =  _\perp (I_\perp)$.
	 \item There exists a set $J \subset (\Wen\cap\Cen)$ such that $\Wen\cap\Cen =  _\perp (J_\perp)$.
	\end{enumerate}
Elements of $I$ (resp. $J$) are called {\em generating cofibrations} (resp. {\em generating trivial
cofibrations}). 	
\end{defi}
	 	
For example, the model categories $\Sp$ and $\Gc r$ are combinatorial. 
One of the main advantages of combinatorial model categories is the existence of natural model
structures on diagram categories.

\begin{prop}[\cite{lurie.htt}, Prop. A.2.8.2]\label{prop:model-category-diagrams}
Let $A$ be a small category and $\mC$ a combinatorial model category. 
Then the following data define a combinatorial model structure on the category $\mC^A$ of
$A$-indexed diagrams in $\mC$,
called the {\em injective model structure}. 
\begin{itemize}
	 \item[(W)] The class $\Wen$ consists of 
	 $f: (X_a)_{a\in A} \to (Y_a)_{a\in A}$ such that, for every $a\in A$, the map $f_a:
	 X_a\to Y_a$ is a weak equivalence in $\mC$.
	 
	 \item[(C)] The class $\Cen$ consists of morphisms $f$ such that, for every $a\in A$, $f_a$ is
	 a cofibration in $\mC$.
	 
	 \item[(F)] The class $\Fen$ is defined as $(\Wen\cap\Cen)_\perp$. 
\end{itemize}
\end{prop}

 \vfill\eject

\subsection{Enriched model categories}
\label{subsec:enriched}

We recall basic definitions of enriched model categories. For detailed expositions see
\cite{hovey}, \cite[A.3.1]{lurie.htt}.

\begin{defi}\label{defi:symmon} We define a 
{\em symmetric monoidal model category} to be a symmetric monoidal category $\mC$ 
which carries a model structure satisfying the
following compatibility conditions:
\begin{enumerate}[label=(\arabic{*})]
	\item The tensor product functor $\otimes: \mC \times \mC \to \mC$ is a left Quillen bifunctor.
	\item The unit $\1 \in \mC$ is cofibrant.
	\item The monoidal structure on $\mC$ is closed, i.e. for each $C, C'\in\mC$ 
	there is an object $\Map_\mC(C, C')$
	 with natural isomorphisms
	\[
	\Hom_\mC(C''\otimes C, C') \cong \Hom_\mC(C'', \Map_\mC(C,C')).
	\]
\end{enumerate}
\end{defi}

Let $\mC$ be a symmetric monoidal model category and $\mD$ a 
$\mC$-enriched category, so we have objects $\Map_\mD(D,D')\in\mC$ for any $D, D'\in\mD$
together with the usual composition and unit morphisms among them. Then
$\mD$ can be considered as a category in the usual sense via
\[
\Hom_\mD(D, D') = \Hom_\mC(\1, \Map_\mD(D, D')). 
\]

\begin{defi}\label{def:C-enriched-model} 
A $\mC$-enriched model category is a $\mC$-enriched category $\mD$ whose underlying category
carries a model structure satisfying:
\begin{enumerate}[label=(\arabic{*})]
	 \item The category $\mD$ is tensored and cotensored over $\mC$, i.e. for any
	 $C\in\mC$, $D\in\mD$ there are objects $D \otimes C$ and $D^C \in \mD$
	 together with isomorphisms
	 \[
	 \Map_\mD(D', D^C) \cong \Map_\mC(C, \Map_\mD(D',D))
	 \] 
	 and
	 \[
	 \Map_\mD(D\otimes C, D') \cong \Map_\mC(C, \Map_\mD(D, D'))
	 \]
	 which are natural in $D'$.
	 \item\label{item:quillen-bifunctor} The resulting functor $\otimes: \mD \times \mC \to \mD$ is a left Quillen
	 bifunctor.
\end{enumerate}
\end{defi}

Note that \ref{item:quillen-bifunctor} implies that the functor
\[
	\Map_\mD: \mD^\op \times \mD \lra \mC
\]
is a right Quillen functor in each variable separately. Here, we equip $\mD^{\op}$ with the opposite
model structure. For objects $D$, $D'$ of $\mD$, we define the derived mapping object
\begin{equation}\label{eq:derivedmappingdef}
	\RMap_\mD(D, D') = \Map_\mD(Q(D), F(D')) \in \h\mC
\end{equation}
where $Q$ and $F$ denote the cofibrant and fibrant replacement functors of $\mD$, respectively. An
analogous statement holds for the functor given by the association $(C,D) \mapsto D^C$.

\begin{exa}[(Simplicial model categories and the Dwyer-Kan localization)]
\label{ex:dwyer-kan-localization}
	\hfill
\begin{exaenumerate}
\item The category $\Sp$ of simplicial sets equipped with the Kan model structure and the
	Cartesian monoidal structure is a symmetric monoidal model category. An $\Sp$-enriched model
	category $\mD$ is called {\em simplicial model category}. 
	By a result of Dugger \cite{dugger:simp-comb}, any combinatorial model category is Quillen
	equivalent to a combinatorial simplicial model category. 
	
\item On the other hand, for any category $\Ec$ and any set (not class of morphisms 
	$\s\subset\Mor(\Ec)$ one can form the {\em classical localization}
	$\Ec[\s^{-1}]$ of $\Ec$ along $\s$. This 
	 is a category 
	 with the same objects as $\Ec$, and morphisms
obtained from those in $\Ec$
  by formally adding the inverses
   of morphisms from $\s$ and their iterated compositions with
   morphisms of $\Ec$, modulo obvious relations, see, e.g.,
   \cite[\S 19.1]{schubert}. A morphism in $\Ec[\s^{-1}]$
   from $x$ to $y$ is thus an equivalence class of 
   ``zig-zags", i.e., diagrams
  \[
  x \lla a_1\lra a_2 \lla a_3 \lra \cdots
    \lla a_n\lra y, \quad n\geq 0,
  \]
  with left-going arrows belonging to $\s$. 
  Note that $\s$ is not required to satisfy any Ore-type condition.
  
	\item If $\Ec$ is itself small, Dwyer and Kan
   \cite{dwyerkan} constructed a  category
 $L_\s(\Ec)$ enriched in $\Sp$,
   with the same objects as $\Ec$ such that
   \[
   \pi_0\Map_{L_\s(\Ec)} (x,y)=\Hom_{\Ec[\s^{-1}]}(x,y).
   \]
   The space $\Map_{L_\s(\Ec)} (x,y)$  
   can be seen as a kind of nerve of the category of
   zigzags, so $L_\s(\Ec)$ serves
   as a non-Abelian derived functor of the classical localization. 
   
\item Let now $\mD$ be a simplicial model category. 
   It was shown in \cite{dwyerkan} that the simplicial set $R\Map_\mD(D,D')$   
   is weakly equivalent to 
 $\Map_{L_\s(\Ec)} (D, D')$ where $\Ec\subset \mD$
 is any sufficiently large small full subcategory and $\s=\Mor(\Ec)\cap\Wen$. 
 This provides more canonical models for the derived mapping spaces $R\Map$ 
 and shows that they depend only on the class $\Wen$ of weak equivalences,
 not on the full model structure. 
\end{exaenumerate}
\end{exa}

\begin{exa} \label{exa:groupoidsimplicial} The category $\Gr$ of
	groupoids admits the structure of a simplicial model category as follows. The 
	$\Sp$-enrichment is given by
	\[
	\Map_{\Gr}(\Gc, \Gc') =\N(\Fun(\Gc, \Gc')) \in \Sp.
	\]	
	The actions of a simplicial set 
	$C\in \Sp$ are given by
	\[
	\Gc\otimes C = \Gc\times\Pi_1(C), \quad \Map(C, \Gc) = \N(\Fun(\Pi_1(C), \Gc))
	\]	
	where $\Pi_1(C)$ is
	the combinatorial fundamental groupoid of $C$, with the set of objects $C_0$
	and morphisms being homotopy classes of edge paths. This $\Sp$-enrichment is compatible with the model structure from
	Example \ref{exa:groupoidsmodel}. Again, this $\Sp$-enrichment can be extended to the model
	category of small categories $\Cat$.
\end{exa}

\begin{exa}[(Category of diagrams: homotopical enrichment)] \label{exa:enricheddia} 
Let $\mC$ be a symmetric monoidal model category,
and $A$ be a small category. 
We assume that $\mC$ is combinatorial and equip $\mC^A$ with the injective model structure. Then
  $\mC^A$ can be equipped with the structure of a
	$\mC$-enriched model category as follows.
	For objects $X, Y \in \mC^A$, we put
	\[
	\Map_{\mC^A}(X,Y) = \int_{a\in A} \Map_{\mC}(X(a),Y(a))
	\]
	where the end on the right-hand side exists since $\mC$ admits small limits. 
	For $C\in\mC$ we denote by $\const{C}\in \mC^A $ the constant functor with value $C$.
	The category $\mC^A$ provided with this enrichment is tensored and cotensored where, for $X
	\in \mC^A$ and $C \in \mC$, we have 
	\[
	(X \otimes C) (a) =  X(a) \otimes C, \quad a\in A, 
	\]
	and
	\[
	X^C = \Map_{\mC^A}(\const{C}, X) 
	\cong \Map_\mC\bigl(C, \pro_{a\in A}^\mC X(a)\bigr). 
	\]
	We will call this enrichment of $\mC^A$ the {\em homotopical enrichment}. 
\end{exa}

\begin{exa}[(Category of diagrams: structural enrichment)]
\label{ex:structural-enrichment}
For any category $\mC$, we construct an enrichment of $\mC^A$ over the symmetric monoidal category
$\Set^A$, which we call the {\em structural enrichment}. 
Given objects $X,Y$ of $\mC^A$, we define $\mathsf{Map}_{\mC^A}(X,Y)\in\Set^A$ by
\begin{equation} \label{eq.ehom}
	{\mathsf {Map}}_{\mC^A}(X,Y)(a) := \Hom_{\mC^{(a\backslash A)}}
	(X|_{a\backslash A},Y|_{a\backslash A})\text{,}
\end{equation}
where $a\backslash A$ is the undercategory of $a$ whose objects are arrows $a\to b$ with source 
$a$. 
If $\mC$ has products and coproducts, then $\mC$, equipped with the structural enrichment, is tensored
and cotensored over $\Set^A$.
For $X \in \mC^A$ and $S \in \Set^A$ we have 
\[
(X \itimes S) (a) = \coprod_{S(a)} X(a), \quad a\in A, 
\]
and 
\[
X^S(a) = \int_{\{a\to b\}\in (a\backslash A)} \prod_{S(b)} X(b), \quad a\in A. 
\]
Here we use the notation $\itimes$ for the tensor product to distinguish the structural enrichment
from the homotopical enrichment.
\end{exa}

Let $\mC$ be a combinatorial symmetric monoidal model category. Then $\mC^A$ has two enrichments:
the homotopical and the structural one. Note that, even though the category $\mC$ carries a model
structure, the structural enrichment does not make any reference to it.

For $S \in \Set^A$, we introduce
\begin{equation}
	 \label{eq:discrete-boxtimes-1}
	 \disc{S} := \const{\1} \itimes S \in \mC^A. 
\end{equation}
Note that, in comparison, for $C\in \mC$, we have $\const{C} \cong \const{\1} \otimes C\in \mC^A$. 
Further, for an object $a \in A$, we define the representable functor
\[
h_a: A \lra \Set, \; a' \mapsto \Hom_{A}(a, a') \text{.}
\]
We will need the following enriched version of the Yoneda lemma.

\begin{lem}\label{lem:tcx} For objects $a \in A$, $X \in \mC^A$, we have a natural isomorphism in $\mC$
\begin{equation*}
	\Map_{\mC^A}(\disc{h_a}, X)\cong X(a). \quad
\end{equation*}
\end{lem}
\begin{proof}	
Equivalently, we show that for objects $a \in A$,
$C \in \mC$ and $X \in \mC^A$, there
exists a natural bijection
	\[
	\Hom_{\mC^A}(\disc{h_a} \otimes C , X) \cong \Hom_{\mC}(C,X(a)) \text{.}
	\]
	Using the ordinary Yoneda lemma and the formula
 \eqref{eq.ehom}, we obtain natural bijections
	\begin{align*}
	 \Hom_{\mC^A}(\disc{h_a} \otimes C , X)
& \cong \Hom_{\mC^A}(\underline{C} \itimes h_a,
	X)\\ & \cong \Hom_{\Set^A}(h_a, \EHom_{\mC^A}
(\underline{C},X))\\
	 & \cong \EHom_{\mC^A}(\underline{C},X)(a)\\
	 & \cong \Hom_{\mC}(C,X(a))\text{.}
	\end{align*}
\end{proof}

\begin{exa}[(Combinatorial simplicial spaces)] Consider $\mC = \Sp$, equipped with the Kan model
	structure and the Cartesian monoidal (model) structure. Setting $A = \Dop$, we have $\mC^A =
	\sS$ which can be identified with the category of bisimplicial sets.
	We will refer to
	the objects of $\mC^A$ as {\em combinatorial simplicial spaces}
	and often drop the adjective ``combinatorial". 
	The injective model structure on $\sS$ coincides with the {\em Reedy model structure}. 
	 
	The homotopical and structural enrichments
	of the category $\sS$ both provide enrichments over the category $\Sp$ and correspond to two ways of slicing a bisimplicial set as a simplicial object
	in the category $\Sp$. 	
	For a simplicial set $D \in \Sp$, the object $\disc{D}\in \sS$ 
	is called the {\em discrete simplicial space}, while $\const{D}\in \sS$ is 
	called the {\em constant simplicial space} corresponding to $D$. 
	Thus, viewing a simplicial space as a bisimplicial set $X_{\bb}$, the two simplicial directions 
	have very different significance for us. The first direction is ``structural" 
	(we are interested in the structural relevance of the face and degeneracy maps), 
	while the second direction is purely homotopical (each $X_{n\bullet}$ is thought of, primarily, in terms 
	of its geometric realization). 
	This is, essentially, the point of view of Rezk \cite{rezk} and Joyal-Tierney \cite{joyal-tierney}
	in their work on $1$-Segal spaces. 
\end{exa}

\vfill\eject

\subsection{Enriched Bousfield localization}
\label{subsec:enrichedbousfield}

Let $\mC$ be a symmetric monoidal model category and let $\mD$ be a $\mC$-enriched model category,
so, for objects $X,Y$ of $\mD$, we have a derived mapping object
$\RMap_{\mD}(X,Y)$ in $\h \mC$. Enriched Bousfield localization theory, developed in \cite{barwick},
starts with the following definition. 

\begin{defi} Let $\s$ be a set of morphisms in $\mD$.
	\begin{enumerate}[label=(\roman{*})]
	\item An object $Z \in \mD$ is called {\em $\s$-local} if, for every morphism $f:\; X \to Y$
	 in $\s$, the induced morphism 
	 \[
	 \RMap_{\mD}(Y,Z) \lra \RMap_{\mD}(X,Z)
	 \]
	 is an isomorphism in $\h \mC$.
	\item A morphism $f:\; X \to Y$ in $\mD$ is called {\em $\s$-equivalence} if, for every
	 $\s$-local object $Z$ in $\mD$, the induced morphism
	 \[
	 \RMap_{\mD}(Y,Z) \lra \RMap_{\mD}(X,Z)
	 \]
	 is an isomorphism in $\h \mC$.
\end{enumerate}
\end{defi}

Note that all weak equivalences in $\mD$ are $\s$-equivalences. The goal is to introduce a new model
structure on $\mD$ with weak equivalences given by all $\s$-equivalences.
This is possible under additional assumptions on $\mC$ and $\mD$ which we now recall.

\begin{defi} Let $\mC$ be a model category.
	\begin{enumerate}[label=(\alph{*})]
\item (\cite{barwick}, Def. 1.21) We say $\mC$ is {\em tractable} if $\mC$ is combinatorial, and the 
morphisms in the sets $I$ and $J$ in Definition \ref{def:comb-model-cat} can be chosen to have cofibrant domain.

\item We say $\mC$ is {\em left proper} if weak equivalences are stable under pushout along cofibrations.
\end{enumerate}
\end{defi}

\begin{exa} A combinatorial model category in which all objects are cofibrant,
 is tractable by definition and left proper by 
 \cite[Proposition A 2.4.2]{lurie.htt}.
 This is the case for the model categories $\Sp$ and $\Gc r$.
	\end{exa}
 
\begin{prop}\label{prop:tract}
	Let $\mC$ be a tractable model category, and let $A$ be a small category. Then the
	category $\mC^A$, equipped with the injective model structure, is a tractable model category.
\end{prop}
\begin{proof}
	The proof of the fact that $\mC^A$ is combinatorial in \cite[Proposition A.2.8.2]{lurie.htt}.
	(specifically in the proof of Lemma A.3.3.3 of {\em loc. cit.}) also implies
	that $\mC^A$ is tractable: The set of generating cofibrations for $\mC^A$ can be chosen to  
	consist of morphisms $X \to Y$ in $\mC^A$ such that, for
	each $a \in A$, the morphism $X(a) \to Y(a)$ is a generating cofibration of $\mC$. The
	analogous statement is true for the trivial cofibrations. 
\end{proof}

We recall the main result of \cite{barwick} (Th. 3.18). Here, we leave the choice of a Grothendieck universe $U$ implicit and assume
that all categories and sets involved are $U$-small. We denote by $(\Wen, \Fen, \Cen)$
the model structure on $\mD$.

\begin{thm}\label{thm:barwick} Let $\s$ be
	a set of morphisms in $\mD$ and assume that 
	\begin{enumerate}[label=(\arabic{*})] 
	 \item $\mC$ is tractable. 
	 \item $\mD$ is left proper and tractable. 
	\end{enumerate}	
	Then there exists a unique combinatorial model 
	structure $(\Wen_\s, \Fen_\s, \Cen_\s)$ on 
	the category underlying
	 $\mD$ with the following properties:
	\begin{itemize}
	 \item[(W)] The class of weak equivalences $\Wen_\s$ is given by the class of $\s$-equivalences.
	 \item[(C)] $\Cen_\s=\Cen$, i.e., the class of cofibrations remains unchanged. 
	 \item[(F)] The fibrant objects are the $\s$-local objects which are fibrant w.r.t. $\Fen$.
	\end{itemize} 
	 The model category $(\mD, \Wen_\s, \Fen_\s, \Cen_\s)$
	 together with the given $\mC$-enrichment of $\mD$,
	 is a $\mC$-enriched model category. 
\end{thm}
 
We give several examples of $\Sp$-enriched Bousfield localization.
  
\begin{exas}[(Stacks of groupoids)]\label{ex:stacks-of-groupoids}
	\hfill
\begin{exaenumerate}
	 \item Let $\Uc$ be a small Grothendieck site.
  The category $\underline{\Gc r}_\Uc$ of stacks of (small) groupoids on $\Uc$ has the
{\em Joyal-Tierney model structure} \cite{joyal-tierney:stacks}. 
With respect to this structure, weak equivalences are equivalences
of stacks, cofibrations are functors injective on objects (and fibrations
are defined by $\Fen=(\Wen\cap\Cen)_\perp$). The simplicial structure
is given by a pointwise variant of Example \ref{exa:groupoidsimplicial}.

On the other hand, pre-stacks of
groupoids on $\Uc$, understood as 
 contravariant functors $\Uc\to\Gc r$, form
 a simplicial model category $\Gc r_\Uc$, which is combinatorial 
 by Proposition \ref{prop:model-category-diagrams}. 
  It is a particular case of results of \cite{toen-vezzosi-1}, \S 3.4, 
 that $\underline{\Gc r}_\Uc$ is Quillen equivalent to
 an $\Sp$-enriched Bousfield localization of $\Gc r_\Uc$
 with respect to an appropriate set $\s$ of morphisms. 
 In particular, $\Fen_\s$-fibrant objects of $\Gr_\Uc$ 
 lie in $\underline\Gr_\Uc$, i.e., are stacks. 
 An important corollary is that $\underline{\Gc r}_\Uc$ is
 a combinatorial model category. 
 
An explicit choice of $\s$ can be obtained by considering hypercoverings in $\Uc$.
A hypercovering can be viewed as a morphism
 $U_\bullet\to U$ from a simplicial object $U_\bullet\in \Uc_\Delta$
to an ordinary (=constant simpicial) object $U\in\Uc$. By passing to
representable functors, a hypercovering gives rise to a morphism $h_{U_\bullet} \to h_U$
 of contravariant functors $\Uc\to\Set_\Delta$. By passing to fundamental groupoids, we obtain a morphism
 of prestacks of groupoids 
 \[
 \Pi(h_{U_\bullet}) \lra \Pi (h_U) = h_U,
 \]
 the prestack on the right being discrete. We take $\s$ to consist of such morphisms
 for a sufficiently representative set of hypercoverings $U_\bullet\to U$. 
 Then a morphism of prestacks will be an $\s$-equivalence, iff it induces an
 equivalence of associated stacks. 
 
\item Let $\kk$ is a field, $\Ac lg^{\aleph_0}_\kk$ be the category of
 at most countably generated commutative $\kk$-algebras,
 and $\Aff_\kk$ the opposite category (affine $\kk$-schemes of countable
 type). 
   Then
$\Uc$ is essentially small, so the constructions from (a)
  apply, and 
we get a combinatorial simplicial model 
category containing the algebro-geometric category of Artin stacks over $\kk$,
see 
\cite{laumon-stacks}. 
\end{exaenumerate}
 \end{exas} 
 
\begin{ex}[($\infty$-Stacks)]\label{ex:simplicial sheaves}
 For a small Grothendieck site $\Uc$
 let $\Sp_\Uc$ be the category of presheaves of simplicial sets on $\Uc$.
 The Kan model structure on $\Sp$ gives rise to the injective
 model structure on $\Sp_\Uc$.

A presheaf $\Sc\in\Sc_\Uc$ is called an {\em $\infty$-stack}
(or ``a sheaf up to homotopy"), if for any hypercovering $U_\bullet\to U$
as above, the induced morphism of simplicial sets
\[
 \Sc(U) \lra \hopro^\Sp_{\Dop} \Sc(U_\bullet)
\]
is a weak equivalence, see \cite{goerss-jardine}.
By an $n$-stack we will mean an $\infty$-stack taking values
in simplicial sets with $\pi_{>n} = 0$. Thus, a $0$-stack
is the same 
as a sheaf of sets, and a 1-stack is essentially the same
as a stack of groupoids in the usual sense.

 Similarly to Example \ref{ex:stacks-of-groupoids}, Bousfield localization
 allows one to construct a new combinatorial simplicial model structure 
 $(\Wen_\s, \Fen_\s, \Cen_\s)$ on $\Sp_\Uc$
  whose fibrant objects are $\infty$-stacks. 
  Explicitly, $\s$ can be chosen to consists of
    morphisms $h_{U_\bullet}\to h_U$ for a sufficiently large set of hypercoverings
    $U_\bullet\to U$. 
    See \cite{toen-vezzosi-1}, Thm. 4.6.1. We denote this localized model category
    \[
    \underline\Sp_\Uc = (\Sp_\Uc, \Wen_\s, \Fen_\s, \Cen_\s).
    \]

    When $\Uc = \Ac ff_\kk$, the category $\underline\Sp_\Uc$
    will be denoted by 
    $\underline\Sp_\kk$. 
      In this case   
    there are important classes 
    of $\infty$-stacks on $\Uc$ (and their morphisms) of algebro-geometric nature,
     of which we note the following,
    referring to \cite[Ch. 2.1]{toen-vezzosi} and \cite{toen-stacks} for more details:
    \begin{enumerate}
    \item[$\bullet$] $m$-{\em geometric stacks} and $m$-representable morphisms
    of stacks, concepts defined inductively in $m$, 
     starting
    from $(-1)$-geometric stacks being affine schemes (representable functors
    $\Uc\to\Set$). In particular, an $m$-geometric $\infty$-stack $\Gc$ has an {\em atlas}
    which is an $(m-1)$-representable morphism of stacks
    $\prod_i S_i \to \Gc$, where each $S_i$ is an affine scheme
    in $\Uc$ (identified with the corresponding representable sheaf of sets). 
    
    \item[$\bullet$] {\em Artin $n$-stacks} which are $n$-stacks which are $m$-geometric
    for some $m$.
    
    \item[$\bullet$] Artin $n$-stacks
    {\em locally of finite presentation} defined by the condition that each $S_i$ above
    is an affine scheme of finite type over $\kk$.
    
    \item[$\bullet$]  
      Artin $n$-stacks {\em of finite presentation} defined by an additional condition
     of quasi-compactness.

    \end{enumerate}

\end{ex}
 
\begin{ex}[(Derived stacks)]\label{ex:derived-stacks}
	Let $\kk$ be a field.  The {\em category of derived stacks} over $\kk$, introduced by Toen-Vezzossi
	\cite{toen-vezzosi} and denoted by $D^-\Aff_\kk^{\sim, \text{\'et}}$, is constructed by a Bousfield
	localization procedure similar to Example \ref{ex:simplicial sheaves}. In particular, it is a
	combinatorial simplicial model category.
   
   More precisely,\footnote
   {We are grateful to B. To\"en for indicating this elementary way of handling
   the set-theoretical issues arising in this and the previous examples,
    instead of using universes as in 
    \cite{toen-vezzosi}.}
    let $\Uc = D^-\Aff_\kk$ be the opposite category to the category
   $(\Ac lg^{\aleph_0}_\kk)_\Delta$ of simplicial objects in $\Ac lg^{\aleph_0}_\kk$
   (so objects of $\Uc$ can be thought of as affine cosimplicial schemes of countable
   type). Then $\Uc$ is essentially small, has a natural 
  model structure and a model analog of a Grothendieck
   topology (\'etale coverings of affine dg-schemes),
   see \cite{toen-vezzosi} \S 1.3.1 and 2.2.2. The model category
   $D^-\Aff^{\sim, \text{\'et}}_\kk$ is the Bousfield localization of $\Sp_\Uc$
   with respect to an appropriate set $\s$ of morphisms
   (homotopy hypercoverings).

   While the entire model category $D^-\Aff^{\sim, \text{\'et}}_\kk$
   (whose objects are thus arbitrary simplicial presheaves on $\Uc$) is referred to as 
   ``the category of derived stacks", the term {\em derived stack} 
   is usually reserved for fibrant objects of this category (w.r.t. the Bousfield
   localized model structure) or, what is the same, objects in the essential
   image of the localization (fibrant replacement) functor. See
   \cite[Def. 1.3.2.1]{toen-vezzosi}.

   Each derived stack $\Sc$ has the {\em classical truncation}
   $\tau_{\leq 0}\Sc$ which is the $\infty$-stack on $\Ac ff_\kk$
   obtained by restricting $\Sc$ to constant simplicial algebras
   (corresponding to usual commutative $\kk$-algebras). 
    For the definition of {\em geometric derived stacks}
    we refer to 
  \cite[\S 1.3.3]{toen-vezzosi} and note that the classical truncation of
  a geometric derived stack is a geometric $\infty$-stack. 
\end{ex}

\vfill\eject

\subsection{Homotopy limits in model categories}
\label{subsec:holim-model}

In \S \ref{subsec:homotopy-limits-spaces}, we introduced homotopy limits of diagrams of spaces and
$2$-limits of diagrams of categories by ad hoc constructions.
In fact, these constructions are instances of
the general notion of a homotopy limit in a simplicial model category which we introduce now. For
details, we refer the reader to \cite{dhks,shulman} and references therein.

Let $(\mC, \Wen, \Fen, \Cen)$ be a model category and $A$ a small
category. Since $\mC$ admits small limits, we have a limit functor
\[
\pro: \mC^A \lra \mC.
\]
In general, the diagram category $\mC^A$ may not admit a natural model structure, but it is always
equipped with a class of weak equivalences given by morphisms $X \to Y$ in $\mC^A$ such that, for each $a \in A$, the induced map $F(a)
\to G(a)$ in $\mC$ is a weak equivalence. The functor $\pro$ does not generally preserve weak equivalences. 

\begin{defi} 
Consider the localization functor $l: \mC \to \Ho(\mC)$. We define the {\em derived limit functor}
$(R\pro,\delta)$ 
to be an initial object of the category of pairs $(f, \eta)$ consisting of 
\begin{itemize}
	\item a functor $f:\mC^A \to \Ho(\mC)$ which maps weak equivalences to isomorphisms, 
        \item a natural transformation $\eta: l \circ \pro \to f$.
\end{itemize}
\end{defi}

Informally, the derived limit functor is the best possible
approximation to $\pro$ which {\em does} preserve weak equivalences. Note that, by construction, a
derived limit functor is unique up to canonical isomorphism if it exists. 

\begin{exa}\label{exa:combinatorial} Let $\mC$ be a model category and assume that $\mC^A$ 
	admits a model structure such that the functor $\pro$ is a right Quillen functor. We
	can construct a derived limit functor by setting $R\pro= l \circ \pro \circ F$ where $F$ is a
	fibrant replacement functor of $\mC^A$. For example, if $\mC$ carries a combinatorial model structure,
	then we can always use the injective model structure on the diagram category $\mC^A$ to derive
	the limit functor. 
	However, as shown in \cite{dhks}, derived limit functors always exist: any model category is 
	homotopically complete (and cocomplete). 
\end{exa}

As shown in \cite{dhks}, derived limits can be explicitly calculated as {\em homotopy limits}. 
The formalism of homotopy limits is greatly simplified if the category $\mC$ can be
equipped with a {\em simplicial} model structure. Since all examples of our interest are
simplicially enriched, we will work in the context of simplicial model categories. 
Note that a simplicial model category $\mC$ is in
particular cotensored over $\Sp$ (see 
Definition \ref {def:C-enriched-model}): for objects $K \in \Sp$ and $Y \in \mC$, we have an object
$Y^K \in \mC$ and a natural isomorphism
\[
\Map_{\Sp}(K,\Map_{\mC}(-,Y)) \cong \Map_{\mC}(-,Y^K).
\]
Given a diagram $X \in \mC^A$ we define the {\em
cosimplicial cobar construction} $\Omega^{\bullet}(\pt,A,X)$ in $\mC^{\Delta}$ via 
\[
\Omega^n(\pt,A,X) := \prod_{\alpha: [n] \to A} X(\alpha_n)
\]
with the apparent coface and codegeneracy maps.
Further, we define the {\em cobar construction} $\Omega(\pt,A,X)$ of $X$ as the end of the functor
\[
\Dop \times \Delta \to \mC,\; ([n],[m]) \mapsto \Omega^m(\pt,A,X)^{\Delta^n}
\]
so that 
\[
\Omega(\pt,A,X) = \pro\left\{ \xymatrix{ \displaystyle \prod\limits_{[n] \in \Delta} \Omega^n(\pt,A,X)^{\Delta^n}
\ar@<+.5ex>[r] \ar@<-.5ex>[r]&
\displaystyle \prod\limits_{[n]\to[m] \in \Delta} \Omega^m(*,A,X)^{\Delta^n} } \right\}.
\]
The {\em homotopy limit of $X$} is defined to be the object
\[
	\hopro^{\mC} X := \Omega(\pt,A,FX)
\]
of $\Ho(\mC)$, where $F$ denotes the fibrant replacement functor of the model category $\mC$ which we apply pointwise to
the diagram $X$. 

\begin{thm}[\cite{dhks}]\label{thm:dhks}
Let $\mC$ be a simplicial model category and $A$ a small category. Then the functor
$\hopro^\mC$ is a derived limit functor. 
\end{thm}

\begin{exa} Consider the category $\Set$ of sets equipped with the trivial model structure, such
	that the weak equivalences are given by isomorphisms and every morphism is both a fibration
	and a cofibration. The category $\Set$ is enriched over $\Sp$ by regarding the set of maps
	between two sets as a discrete simplicial set.
	The homotopy limit recovers the ordinary limit of sets. This example generalizes to any
	category $\mC$ which admits small limits and colimits, equipped with the trivial model
	structure.
\end{exa}

\begin{exa} Consider the category $\Top$ of compactly generated Hausdorff topological spaces
	equipped with the Quillen simplicial model structure. 
	The homotopy limit recovers precisely the homotopy limit of spaces 
	introduced in \S \ref{subsec:homotopy-limits-spaces}. Note that, since any topological space is
	fibrant, the definition of the homotopy limit does not involve the model structure on
	$\Top$; it only depends on the simplicial enrichment. 
\end{exa}

\begin{exa} Consider the subcategory $\Gr \subset \Cat$ of small groupoids 
  	with its simplicial model structure defined in Example \ref{exa:groupoidsimplicial}.
  	Comparing the bar-construction in this case with Definition
  	\ref{def:2lim}, we conclude that 
      	homotopy limits in $\Gr$ coincide with 
    	the $2$-limits as defined there. 
	More generally, let $\Uc$ be a small Grothendieck
  	site. We then have the concept of the 2-limit of a diagram of stacks
  	of groupoids on $\Uc$,
	defined in a similar way. As before, it is identified with the 
	homotopy limit in the simplicial combinatorial
	model category $\underline{\Gc r}_\Uc$ of stacks. 
\end{exa}

\begin{rem} Let $\mC$ be a combinatorial simplicial model category. Then can compute 
	derived limit functors in two ways. On the one hand, we can
	utilize the injective model structure on $\mC^A$ to derive the limit functor as explained in
	Example \ref{exa:combinatorial}. On the other hand, we can express the derived limit functor
	as a homotopy limit. Starting from \S \ref{sec:2-segal-model} we will utilize this
	additional flexibility. Since the category $\Top$ is not combinatorial, it has to be
	replaced by the Quillen equivalent category $\S$ of simplicial sets, equipped with the Kan
	model structure. 
\end{rem}

We collect some consequences of the above, to be used below.

\begin{prop}\label{prop:holim-simp-top}
	\begin{enumerate}
	 \item[(a)] If $(X_a \to Y_a)_{a\in A}$ is a weak equivalence of $A$-diagrams in $\Top$, 
then the induced map
\[
  \hopro _{a\in A} X_a \lra \hopro_{a\in A} Y_a
\]
is a weak equivalence in $\Top$.

	\item[(b)] Let $(D_a)_{a\in A}$ be an $A$-diagram of simplicial sets. Then we have
  a weak equivalence in $\Top$
  \[
  | R\pro_{a\in A}^\Sp D_a | \simeq \hopro_{a\in A} |D_a|, 
  \]
  where $R\pro^\Sp$ is the derived limit functor constructed
  using the injective model structure on $\Sp^A$ as in Example 
  \ref{exa:combinatorial}. 
\end{enumerate}
\end{prop} 
  
\begin{proof} Assertion (a) follows since a derived limit functor takes
  weak equivalence of diagrams to isomorphisms in the homotopy category. 
  Part (b) follows from the Quillen equivalence between the model categories $\Sp$ and $\Top$.
\end{proof}
  
\begin{rem} Following the general custom, we will usually write $\hopro$ for the derived
limit functor even if the underlying model category does not carry a simplicial structure.
\end{rem}

\vfill\eject

\section{The 1-Segal and 2-Segal model structures}
\label{sec:2-segal-model}

In this chapter, we introduce the notions of $1$-Segal and $2$-Segal objects in a combinatorial model category
$\mC$. If further $\mC$ admits the structure of a left proper, tractable, symmetric monoidal model category,
then we introduce model structures for $1$-Segal and $2$-Segal objects which arise as enriched
Bousfield localizations of the injective model structure on $\sC$. For $\mC = \Sp$, the model
structure for $1$-Segal objects in $\Sp$ recovers the Rezk model structure for $1$-Segal spaces
introduced in \cite{rezk}.

\subsection{Yoneda extensions and membrane spaces}
\label{subsec:yonedaext} 

The construction of membrane spaces from \S \ref{subsec:membranes}, can be viewed as an instance of the general Kan extension formalism. 
In this section we summarize some aspects of this formalism, to be used later.
 
Let $A$ be a small category and $\mC$ a category with small limits and colimits.
Consider the category $\P(A) = \Fun(A^{\op}, \Set)$ of presheaves on $A$ and the corresponding Yoneda embedding 
\[
	\yoneda: A \to \P(A), \; a \mapsto h_a.
\]
Since $\mC$ admits small limits, we have an adjunction
\begin{equation}\label{eq.rightkan}
	\yoneda^*: \mC_{\P(A)} \longleftrightarrow \mC_{A}: \yoneda_*,
\end{equation}
where $\yoneda^*$ denotes the pullback functor and $\yoneda_*$ the functor of right Kan extension along $\yoneda$. We call $\yoneda_*$ the {\em Yoneda extension functor}.  
For an object $X \in \mC_A$ and $K \in \P(A)$, the value of $\yoneda_*X$ on $K$ will 
be denoted by $(K,X)$ and called the {\em space of $K$-membranes in $X$}. 
The general formula for Kan extensions in terms of limits implies
\begin{equation}\label{eq:pointwise}
(K,X) = \yoneda_* X(K) \cong \pro_{\{h_a \to K\}}^{\mC} X_a \text{.}
\end{equation}

\begin{ex} The previous definition of the membrane spaces in 
	\eqref{eq:membranes-top} is recovered when $A=\Dop$ and $\mC=\Top$, with $h_\{[n]\}$ being the standard simplex $\Delta^n$. 
\end{ex}
 
We recall the following standard result (\cite{kashiwara-schapira}). 

\begin{prop} \label{prop:yonedacont} The functor $\Upsilon_*: \mC_A \to \mC_{\P(A)}$ establishes an 
equivalence between $\mC_A$ and the full subcategory of $\mC_{\P(A)}$ consisting of functors which map
	colimits in $\P(A)$ to limits in $\mC$. The inverse of this equivalence is given by 
	 $\Upsilon^*$. 
\end{prop}

Let $f: A \to A'$ be a functor of small categories. We consider the pullback along $f$
of both $\mC$-valued and $\Set$-valued functors on $A$ and $A'$ and, in each case,
the corresponding {\em left} Kan extension functor, so that we have adjunctions
\[
f_! : \mC_A \longleftrightarrow \mC_{A'} : f^*, \quad 
f_! : \P(A) \longleftrightarrow \P(A') : f^*
\]
We use the notations $f^*, f_!$ in both cases,
since it will be clear from the context which functor is meant.
Note that we have a $2$-commutative square
\[
\xymatrix{
A \ar[r]^f\ar[d]_{\Upsilon^A} & A' \ar[d]^{\Upsilon^{A'}}\\
\P(A) \ar[r]^{f_!} & \P(A').
}
\]

\begin{prop}\label{prop:yonedaadj}
	For $X \in \mC_{A'}$ and $K \in \P(A)$, we have a natural isomorphism in $\mC$
	\[
	(K, f^*X) \cong (f_! K, X). 
	\] 
\end{prop}
\begin{proof} We show that there is an isomorphism in $\mC_{\P(A)}$ between the functors $(\Upsilon^{A'}_*X) \circ
	f_!$ and $\Upsilon^A_*(f^* X)$. Both functors map colimits in $\P(A)$ to limits in $\mC$. The pullbacks of both functors under $\Upsilon^A$ are isomorphic to $f^*X$,
	thus, by Proposition \ref{prop:yonedacont}, we conclude that the functors themselves are isomorphic. 
\end{proof}

Assume now that $\mC$ be a combinatorial model category.
We equip the functor categories $\mC_{A}$ and $\mC_{\P(A)}$
with the injective model structures so that the adjunction \eqref{eq.rightkan} becomes a Quillen
adjunction. We then introduce the {\em homotopy Yoneda extension} functor
$R\yoneda_*$ as the right derived functor of $\yoneda_*$. The value of
$R\yoneda_*X$ at $K\in\P(A)$ will be denoted by $(K,X)_R$ and called
the {\em derived space of $K$-membranes} in $X$. Thus we have $(K,X)_R = (K, F(X))$ where $F(X)$ is
an injectively fibrant replacement of $X$. Further, we have the identification
\begin{equation}\label{eq.dpointwise}
(K,X)_R \simeq R\pro_{\{h_a \to K\}}^{\mC} X_a, 
\end{equation}
obtained from the pointwise formula for homotopy Kan extensions (see \cite[A.2.8.9]{lurie.htt}).   

\begin{rem} Let $\mC$ be a simplicial combinatorial model category. We can compute the derived limit
	$Y = R\pro_{\{h_a \to K\}}^{\mC} X_a$ in two ways. By the above discussion, we have the formula
	\[
	R\pro_{\{h_a \to K\}}^{\mC} X_a \simeq \pro_{\{h_a \to K\}}^{\mC} F(X)_a,
	\]
	where $F(X)$ is an injectively fibrant replacement of $X$. Alternatively, we can utilize the
	simplicial enrichment to compute 
	\[
	R\pro_{\{h_a \to K\}}^{\mC} X_a \simeq \hopro_{\{h_a \to K\}}^{\mC} X_a,
	\]
	where the right-hand side denotes the homotopy limit introduced in \S
	\ref{subsec:holim-model}. In view of Proposition \ref{prop:holim-simp-top}, this shows that
	the formalism introduced in this section is compatible with the notion of membrane spaces
	introduced in \S \ref{subsec:membranes}. 
\end{rem}

\begin{prop}\label{prop:derivedyonedaadj}
Assume that the functor $f^*: \mC_{A'}\to\mC_A$ preserves injectively
fibrant objects. Then, for $X\in\mC_{A'}$ and $K\in\P(A)$, we have a natural weak equivalence
\[
(K, f^*X)_R \simeq (f_!K, X)_R. 
\]
\end{prop}
\begin{proof} The statement follows immediately from Proposition \ref{prop:yonedaadj}.\end{proof}

Let $\mC$ be a symmetric monoidal model category in the sense of Definition \ref{defi:symmon}. 
We equip the model category $\mC_A$ with the homotopical enrichment from Example \ref{exa:enricheddia}. In this situation, we have the
following formula for Yoneda extensions in terms of $\mC$-enriched mapping spaces.

\begin{prop}\label{prop:dkanmap} Let $X \in \mC_A$. 
	\begin{enumerate}[label=(\alph{*})]
	 \item There exists a natural isomorphism 
	 \[
	 \yoneda_* X \cong \Map_{\mC_A}(\disc{-},X)
	 \]
	 of functors $\P(A)^{\op} \to \mC$.
	 \item There exists a natural weak equivalence
	 \[
	 R\yoneda_* X  \simeq \RMap_{\mC_A}(\disc{-},X) 
	 \]
	 of functors $\P(A)^{\op} \to \mC$.
	\end{enumerate}
\end{prop}
\begin{proof}
	(a) Both functors commute with colimits in $A$ (or, more precisely, limits in $A^{\op}$). 
	Since any object $D \in \P(A)$ can be expressed as a colimit of representable
	functors $h_a$, it suffices to check that their restrictions to $A^{\op}$ are
	naturally isomorphic (Proposition \ref{prop:yonedacont}). This follows from Lemma \ref{lem:tcx} and 
	formula \eqref{eq:pointwise} for the Yoneda extension.

	(b) The derived mapping space is obtained by forming the ordinary mapping space of an injectively fibrant
	replacement $F(X)$ of $X$. Indeed, for any $D \in \P(A)$, the object $\disc{D}$ is cofibrant and does not have to
	be replaced. This follows since $\disc{D}(a) = \coprod_{D_a} \1$ and $\1 \in \mC$ is by
	definition cofibrant. On the other hand, the homotopy Kan extension can be
	calculated by applying the functor $\yoneda_*$ to an injectively fibrant replacement of $X$. The
	statement thus follows from (a).
\end{proof}

Let $A, B$ be small categories. Recall that a diagram $K: B \to \Set_A$ is called {\em acyclic} if, for
every $a \in A$, the natural map
\[
	\hoind_b K(a) \lra \ind_b K(a)
\]
is a weak homotopy equivalence of spaces. Here, $K(a): B \to \Set$ denotes the diagram obtained from $K$ by evaluating
at $a$, interpreted as a diagram of discrete topological spaces.

\begin{prop}\label{prop:acyclic}
Let $A$, $T$ be small categories, $(K_b)_{b \in B}$ a $B$-indexed diagram in the category
$\P(A)$, and $X \in \mC_A$. Then the following hold:
	\begin{enumerate}[label=(\alph{*})]
	 \item We have a natural isomorphism in $\mC$
	 \[
	 \bigl( \ind_{b\in B}^{\P(A) } K_b, X \bigr) \cong \pro_{b\in B}^\mC (K_b, X).
	 \]
	 \item If the diagram $(K_b)_{b\in B}$ is acyclic, then we have a natural weak equivalence
	 \[
	 \bigl( \ind_{b\in B}^{\P(A)} K_b, X \bigr)_R \simeq \hopro_{b\in B}^\mC (K_b, X)_R.
	 \]
	\end{enumerate}
\end{prop}
\begin{proof} Let $S = \ind_{b \in B} K_t$. Consider the diagram of categories
	\begin{equation}\label{eq:slicediag}
	 A/S \overset{f}{\lra} \Set_A/S \overset{g}{\longleftarrow} B,
	\end{equation}
	where $f$ is induced by the Yoneda embedding and $g$ maps an object $b \in
	B$ to the canonical map $(K_b \to S)$.
	Let $f/g$ denote the comma category associated to \eqref{eq:slicediag}. An object of $f/g$
	is given by a triple $(x,b,\alpha)$ where $x$ and $b$ are objects of $A/S$ and $B$,
	respectively, and $\alpha: f(x) \to g(b)$ is a morphism in $\Set_A/S$.
	We consider the functors
	\begin{align*}
	 F_1: & \;B^{\op} \to \mC,\; b \mapsto (K_b, X)\\
	 F_2: & \;(f/g)^{\op} \to \mC,\; (h_a \to S, b, \alpha) \mapsto X_a\\
	 F_3: & \;(A/S)^{\op} \to \mC,\; (h_a \to S) \mapsto X_a.
	\end{align*}
	We claim that we have natural isomorphisms in $\mC$
	\[
	 \pro F_1 \cong \pro F_2 \cong \pro F_3.
	\]
	We consider the natural projection functor $q: (f/g)^{\op} \to B^{\op}$.
	Note that any limit functor is a right Kan extension
	along the constant functor, and hence, by the functoriality of Kan extensions, we have an isomorphism of functors
	\[
	 \pro_{\{B^{\op}\}} \circ q_* \cong \pro_{\{(f/g)^{\op}\}}.
	\]
	This implies the identification $\pro F_1 \cong \pro F_2$ since, by definition, the functor $F_1$ is a right
	Kan extension of $F_2$ along $q$. 
	The isomorphism $\pro F_2 \cong \pro F_3$ is obtained by noting that
	$F_2$ is the pullback of $F_3$ along the initial functor $p: (f/g)^{\op} \to (A/S)^{\op}$.
	This proves (1). 
	
	To show (2), we replace the functor $F_1$ by $b \mapsto (K_b, X)_R$. We claim to have a chain of
	natural isomorphisms in $\Ho(\mC)$
	\[
	 R\pro F_1 \simeq R\pro F_2 \simeq R\pro F_3.
	\]
	Again, from the definition of the derived membrane space, the functor $F_1$ is a right
	homotopy Kan extension of $F_2$ along the functor $q: (f/g)^{\op} \to B^{\op}$, which implies
	the identification $R\pro F_1 \simeq R\pro F_2$. To obtain the weak equivalence $R\pro F_2
	\simeq R\pro F_3$, it suffices to show that the functor $p: (f/g)^{\op} \to (A/S)^{\op}$ is {\em
	homotopy} initial (\cite[19.6]{hirschhorn}), i.e., $p$ preserves homotopy limits. We have to show that, for every object $h_a \to S$ of $(A/S)^{\op}$,
	the overcategory $p/(h_a \to S)$ has a weakly contractible nerve. 
	But this statement is easily seen to be equivalent to the assumption that the diagram
	$(K_b)$ is acyclic. Here, we use the fact that $p/(h_a \to S)$ is weakly equivalent to the
	strict fiber $p^{-1}(h_a \to S)$ since the map $p^{\op}$ is a Grothendieck fibration.
\end{proof}
 
\begin{rem} 
In the case when $\mC$ is a simplicial combinatorial model category, we can alternatively prove Proposition
\ref{prop:acyclic} by the exact argument of Proposition \ref{prop:acyclic-topological}, utilizing
the cotensor structure of $\sC$ over $\sS$. 	
\end{rem}
     
Note that, by Proposition \ref{prop:acyclic}, all formulas regarding manipulations of derived membrane
spaces proven in \S \ref{subsec:membranes} for (semi-)simplicial topological spaces 
extend to the context of (semi-)simplicial objects in the combinatorial model category
$\mC$ by setting $A = \Delta$ ($A = \Delta_\inj$).
In what follows, we will use these statements freely. 
    
\vfill\eject

\subsection{1-Segal and 2-Segal objects}

Consider the classes of morphisms in $\Sp$
\begin{equation}
	\label{eq:sets-s-d}
	\begin{split}
	 \s_1 & = \left\{ \Delta^{\Ic_n} \hra \Delta^n |\; n \ge 2 \right\},
	 \\
	 \s_2 & = \left\{ \Delta^{\T} \hra \Delta^n |\; \text{$n \ge 3$, $\T$
	 is a triangulation of the polygon $P_n$} \right\}.
	\end{split}
\end{equation}
Here $\Ic_n\subset 2^{[n]}$ denotes the collection of subsets from Example \ref{I.1segal},
so that $\Delta^{\Ic_n}$ is the union of $n$ composable oriented edges. 
The morphisms in $\s_d$ will be called {\em $d$-Segal coverings}. We 
apply the formalism of \S \ref{subsec:yonedaext} in the case $A=\Delta$. In particular, we 
consider the Yoneda embedding $\yoneda: \Delta \to \Sp$ and the corresponding 
derived Yoneda extension functor $R\yoneda_*: \mC_\Delta \to \mC_\Sp$.

\begin{defi} 
\label{defi:12segal} Let $\mC$ be a combinatorial model category and
$X$ a simplicial object in $\mC$.
We say that $X$ is a {\em $d$-Segal object in $\mC$} if its
homotopy Yoneda extension $R \yoneda_* X \in \mC_{\Sp}$ maps $d$-Segal coverings to weak
equivalences in $\mC$.
\end{defi}

\begin{rem} \label{rem:dsheaves} It is convenient to think of $\s_d$
as defining the rudiment of a Grothendieck topology on $\Sp$.
In this context, the $d$-Segal condition on $X$ is analogous to a (homotopy) descent condition 
for the $\mC$-valued presheaf $R\yoneda_* X$ on $\Sp$.
\end{rem}

\begin{rems}\label{rems:2-segal-top-model-identical}
	\begin{exaenumerate} 
		 \item As in Chapter \ref{sec:topsegal},
		 Definition \ref{defi:12segal} can be modified to define $d$-Segal semi-simplicial
		 objects in a combinatorial model category $\mC$.
		 We leave the details to the reader. 
	 
		 \item Similarly, the definition of a {\em unital $2$-Segal object}
		 in $\mC$ is identical to Definition \ref{def:unital-2-segal-top}.
	\end{exaenumerate}
\end{rems}

\begin{rem} As announced in the introduction, there is a natural way to extend Definition
	\ref{defi:12segal} to $d \in \mathbb N$, using an analogous descent condition
	involving triangulations of $d$-dimensional cyclic polytopes. These higher
	Segal spaces will be the subject of future work.
\end{rem}

\begin{prop}\label{prop:1seg2seg} 
	Every $1$-Segal object in $\mC$ is a $2$-Segal object.
\end{prop}
\begin{proof}
	 Completely analogous to the argument of Proposition \ref{prop.1segal2segal}.
\end{proof}

\begin{exas} 
	\begin{exaenumerate}
	 \item Let $\mC = \Set$ with the trivial model structure. The $d$-Segal objects in $\mC$ are 
	 the discrete $d$-Segal spaces studied in Chapter \ref{sec:discsegal}. More
	 generally, if $\mC$ is any category with limits and colimits equipped with the trivial
	 model structure, we recover the concept of non-homotopical $d$-Segal
	 objects from Chapter \ref{sec:discsegal}. In fact, the existence of
	 colimits is not necessary to formulate the $d$-Segal condition in this context. 
		
	 \item Let $\mC = \Sp$ equipped with the Kan model structure. We call the $d$-Segal 
	 objects in $\Sp$ {\em combinatorial $d$-Segal spaces}. All examples of topological $d$-Segal spaces
	 studied in Chapters \ref{sec:topsegal} and \ref{sec:discsegal} are in fact obtained from combinatorial $d$-Segal
	 spaces by levelwise application of geometric realization. 
	 By Proposition \ref{prop:holim-simp-top} and the fact that the model categories $\Top$ and $\Sp$ are
	 Quillen equivalent, the theory of combinatorial $d$-Segal spaces is essentially equivalent to
	 the theory of topological $d$-Segal spaces. However, since the model category $\Sp$ is
	 combinatorial, it has technical advantages.
	\end{exaenumerate}
\end{exas}

\begin{ex}\label{ex:2-segal-objects}
	\begin{exaenumerate}
	 \item Let $\Ec$ be a proto-exact category. Then the Waldhausen construction gives
	 a $2$-Segal simplicial object $\Sc\Ec$ in $\Gc r$ and its nerve
	 $\N(\Sc\Ec)$ is a $2$-Segal simplicial object in $\Sp$. 

	 \item Similarly, let $G$ be a group acting on a set $E$. Then $\Sc_\bullet(G,E)$ is
	 a $1$-Segal object in $\Gc r$.
	\end{exaenumerate}

\end{ex}
 
\begin{ex}[(Waldhausen stacks)]\label{ex:waldhausen-stacks}
	\begin{exaenumerate}
	 \item Let $\Uc$ be a small Grothendieck site. A {\em stack of proto-exact categories} on $\Uc$
	 is a stack $\Ec$ of categories $\Ec(U), U\in\Ob(\Uc)$ such that each $\Ec(U)$
	 is made into a proto-exact category with classes $\Men(U),\Een(U)$, and these
	 classes are of local nature, i.e., closed under restrictions as well as
	 under gluing in coverings forming the
	 Grothendieck topology. Then $U\mapsto \Sc_n(\Ec(U))$ is a stack of groupoids
	 on $\Uc$, so we obtain a simpicial object $\Sc(\Ec)$ in the category $\underline {\Gc r}_\Uc$
	 if stacks of groupoids over $\Uc$.
	 This simplicial object is $2$-Segal. The proof is the same as 
	 in Proposition \ref{prop:waldhausen-proto-segal}. 

	 In particular, let $\FF$ be a field and $\Uc=\Aff_\FF$ be the \'etale site of affine 
	 $\FF$-schemes of at most countable type, as in Example
	 \ref {ex:stacks-of-groupoids}. 
	 The following Waldhausen stacks on $\Aff_\FF$
	 are important, since they provide examples of $2$-Segal objects
	 of algebro-geometric nature.

	 \item Let 
	 $R$ be a finitely generated associative $\FF$-algebra.
	 We then have the stack of exact categories $\alg{R-\Mod}$ on $\Aff_\FF$.
	 By definition, $\alg{R-\Mod}(U)$ is
	 the category of sheaves
	 of left $\Oc_U\otimes_\FF R$-modules which are locally free of finite rank
	 as $\Oc$-modules. In particular, for $U=\operatorname{Spec}(\FF)$
	 we recover the abelian category $R-\Mod$
	 of finite-dimensional $R$-modules.
	 The Waldhausen stack $\Sc(\alg{R-\Mod})$
	 is thus an algebro-geometric extension of the single simplicial groupoid
	 $\Sc(R-\Mod)$, the Waldhausen space of the category of finite-dimensional
	 $R$-modules.

	 \item Let $V$ be a projective algebraic variety over $\FF$. Then
	 we have the abelian category
	 $\Coh(V)$ of coherent sheaves and the exact category $\Bun(V)$ of vector
	 bundles on $V$. They extend in a standard way to stacks of exact
	 categories $\alg\Coh(V)$ and $\alg\Bun(V)$ on $\Aff_\FF$. For instance,
	 $\alg\Coh(V)(U)$ is formed by quasi-coherent sheaves on $V\times U$,
	 flat with respect to the projection $V\times U\to V$ and whose
	 restriction to each geometric fiber of this projection is coherent. 
	 Therefore we get $2$-Segal simplicial stacks of groupoids $\Sc(\alg\Coh(V))$
	 and $\Sc(\alg\Bun(V))$. 

	 \item Let $G$ be an algebraic group and $E$ be an algebraic variety, both over $\FF$,
	 with $G$ acting on $E$. 
	 Then the stack quotients $\Sc_n(G, E) = [G\bbs E^{n+1}]_{n\geq 0}$ form a simplicial object
	 $\Sc_\bullet(G, E)$
	 in the category of Artin stacks over $\FF$ (which is a subcategory in the category
	 $\underline{\Gr}_{\Aff_\FF}$). This simplicial object, which is
	 the algebro-geometric version of the Hecke-Waldhausen space from
	 \S \ref{subsec:hecke-waldhausen} is $1$-Segal. 
	 The proof is the same as given in that section. 

\end{exaenumerate}
\end{ex}
 
\eject

\subsection{1-Segal and 2-Segal model structures}
\label{subsec:12segalmodel}

In Remark \ref{rem:dsheaves} we expressed the $d$-Segal
condition as a descent condition with respect to $d$-Segal coverings. 
In this section, we use Proposition \ref{prop:dkanmap} to reinterpret these descent
conditions as locality conditions: a simplicial object $X \in \mC_{\Delta}$ is a $d$-Segal object, if and
only if it is $\seg_d$-local in the $\mC$-enriched sense. 
This enables us to apply the general theory of enriched Bousfield localization to
introduce model structures for $1$-Segal and $2$-Segal objects. 

Let $\mC$ be a left proper, tractable, symmetric monoidal model category and $d \in \{ 1,2 \}$.
We consider the category $\sC$ with its injective model
structure and the homotopical $\mC$-enrichment, see
Example \ref{exa:enricheddia}. We further use the notation $\RMap_{\sC}(X,Y)$ to denote the
corresponding $\h \mC$-enriched derived mapping spaces as defined in \eqref{eq:derivedmappingdef}. 
Let $\disc{\s_d}\subset\Mor(\mC_\Delta)$ be the image of $\s_d$ under the discrete object functor
defined in \eqref{eq:discrete-boxtimes-1}.

\begin{prop} \label{prop:12segallocal} A simplicial object $X \in \sC$ is $d$-Segal if and
	only if it is $\disc{\seg_d}$-local in the $\mC$-enriched sense.
\end{prop}
\begin{proof} This is immediate from Proposition \ref{prop:dkanmap}.
\end{proof}

\begin{thm}\label{thm:S-d-localization}
There exists a $\mC$-enriched combinatorial model structure $\Seg_d$
on $\sC$ with the following properties:
	\begin{itemize}
	 \item[(W)] The weak equivalences are given by the
	 $\disc{\seg_d}$-equivalences.
	 \item[(C)] The cofibrations are the injective cofibrations.
	 \item[(F)] The fibrant objects are the injectively fibrant $d$-Segal objects.
	\end{itemize}
\end{thm}
\begin{proof}
	This follows from Theorem \ref{thm:barwick}, Proposition \ref{prop:tract} and Proposition \ref{prop:12segallocal}.
\end{proof}

We call $\Seg_d$ the {\em model structure for $d$-Segal
objects in $\mC$}. We further denote the injective model structure on $\mC_\Delta$ by
$\I$. 

\begin{cor}\label{thm:12segaladjunction}
We have inclusions
\[
\begin{gathered}
\Wen_\I \subset \Wen_{\Seg_2} \subset\Wen_{\Seg_1}, \\
\Fen_\I \supset \Fen_{\Seg_2} \supset\Fen_{\Seg_1}, \\
\Cen_\I = \Cen_{\Seg_2} =\Cen_{\Seg_1}, 
\end{gathered} 
\]
so that the identity functors induce Quillen adjunctions
\[
(\mC_\Delta, \I) \buildrel\Id\over\longleftrightarrow 
(\mC_\Delta, \Seg_{2}) \buildrel\Id\over\longleftrightarrow 
 (\mC_\Delta, \Seg_{1}) \text{.}
\]
 \end{cor}

\begin{proof}
The equality of the $\Cen$-classes 
is clear from Theorem \ref{thm:S-d-localization}. 
 The inclusion of the
$\Wen$-classes follow from the theorem together with the fact
that, by Proposition \ref{prop:1seg2seg}, every $1$-Segal object is $2$-Segal. The opposite inclusion
of the $\Fen$-classes follows from the axiom $\Fen = (\Wen\cap\Cen)_\perp$
of model categories. 
\end{proof}

Note that, as part of the model structure $\Seg_d$, we have a functorial fibrant replacement functor: 
for every $X$ in $\mC_\Delta$, we obtain a $d$-Segal object $\mathfrak S_d(X)$ and a
canonical $d$-Segal weak equivalence
\[
	X \lra {\mathfrak S_d}(X),
\]
functorial in $X$. We refer to this map as the {\em $d$-Segal envelope of $X$}.

Passing to homotopy categories, Corollary 
\ref{thm:12segaladjunction} gives a chain of inclusions of full subcategories
\begin{equation}
	\label{eq:homotopy-d-segal-subcategories}
	\Ho(\mC_\Delta, \I) \supset \Ho(\mC_\Delta, \Seg_2) \supset \Ho(\mC_\Delta, \Seg_1). 
\end{equation}
That is, $\Ho(\mC_\Delta, \Seg_d)$ is identified with the full subcategory in
$\Ho(\mC_\Delta, \I)$ formed by injectively fibrant $d$-Segal objects in $\sC$. 
Further, the $d$-Segal envelope functors induce left adjoint functors to the 
inclusions of homotopy categories:
\[
\xymatrix{
 \Ho(\mC_\Delta, \I)
 \ar@/^1pc/[rr]^{\mathfrak S_1}
  \ar[r]_-{\Sen_2}&\Ho(\mC_\Delta, \Seg_2)
 \ar[r]_{\Sen_{2,1}}&\Ho(\mC_\Delta, \Seg_1), 
 }
\]
 with $\Sen_{2,1}$ being the restriction of $\Sen_1$ to the subcategory
 of $2$-Segal objects.

  \begin{exas}[(Free categories)]\label{ex:1-segal-envelope-discrete}
	 \begin{exaenumerate}
	 \item Let $\mC=\Set$ with trivial model structure. Then the injective model structure on $\Set_\Delta$
    is also trivial, so $\Ho(\Set_\Delta, \I) = \Set_\Delta$ is the category of simplicial sets. 
    This means that the subcategories in
    \eqref{eq:homotopy-d-segal-subcategories} are simply the full subcategories
    formed by $d$-Segal simplicial sets, $d=1,2$:
    \[
    \Set_\Delta \supset \Set_\Delta^{2-\Seg} \supset \Set_\Delta^{1-\Seg}. 
    \] 
      In particular, since
    any $1$-Segal simplicial set is isomorphic to the nerve of a small category, the 
    composite embedding
    $\Ho(\Set_\Delta, \Seg_1)\subset \Ho(\Set_\Delta, \I)$
    is identified with the nerve functor $\N: \Cat\to \Set_\Delta$. 
    Therefore the functor of $1$-Segal envelope $\Sen_1$ is, in this case
      the left adjoint of the $\N$ in the ordinary sense. This is the functor
     \[
     \FC: \Set_\Delta\lra\Cat, \quad D\mapsto \FC(D),
     \]
     where $\FC(D)$ is the {\em free category generated by} $D$. 
     Explicitly, $\Ob({\FC}(D))=D_0$ is the set of
	vertices of the simplicial set $D$, while $\Hom_{{\FC}(D)}(x,y)$
	is the set of oriented edge paths from $x$ to $y$ modulo identifications
	given by the 2-simplices. 

	\item Let $\mC=\Sp$ with the Kan model structure. We then have an embedding
$\Set_\Delta\to\Sp_\Delta$ takind a simplicial set $D$ to the discrete simplicial space
$\disc {D}$. The functors of $d$-Segal envelope in $\Sp_\Delta$, denote them
$\Sen_d^\Sp$, can be compared with the corresponding functors in $\Set_\Delta$
from (a), denote then $\Sen_d^{\Set}$. In fact, they are compactible:
\[
\Sen_d^\Sp(\disc {D}) \simeq \disc{ \Sen_d^{\Set}(D) }. 
\]
To see this, note that 
  model structure $\Seg_d$ on $\sS$ is combinatorial and therefore
	cofibrantly generated. 	
	Thus,
	by the small object argument (e.g., \cite[A.1.2]{lurie.htt}), we may build fibrant replacements by forming iterated (transfinite)
	compositions of pushouts along generating trivial cofibrations. The generating trivial cofibrations consist of two types of
	morphisms. First, the injective model structure itself is generated by a certain set of embeddings of simplicial spaces giving weak equivalences at
  	each level. Second, we have the embeddings from the set 
	$\disc{\s_d}$ from \eqref{eq:sets-s-d}.
	Since the discrete simplicial space $\disc{D}$
	is already injectively fibrant, it suffices to form pushouts only along maps in
 $\disc{\s_d}$. Doing so will produce $\disc{\Sen_d^\Set(D)}$. 
 	\end{exaenumerate}
	
\end{exas}

\eject

\section{The path space criterion for 2-Segal spaces}
\label{sec:pathspace}

The main result of this chapter is Theorem \ref{thm.crit} which
expresses the $2$-Segal condition for a simplicial object $X$ 
in terms of $1$-Segal conditions for simplicial analogs of the path space of $X$, as defined by Illusie.

\subsection{Augmented simplicial objects}
\label{subsec:augmented}

We define the category $\Dp$ to be the category of all finite ordinals, including the empty set. An
{\em augmented simplicial object} of a category $\mC$ is a functor $X: \Dpop \to \mC$. We denote by
$\mC_{\Dp} =  \Fun(\Dpop, \Sp)$ the category of such objects. Explicitly, an augmented simplicial
object is the same as an ordinary simplicial object $X_\bullet\in\mC_\Delta$ together with an object
$X_{-1}=X(\emptyset)$ and an augmentation morphism $\partial: X_0\to X_{-1}$ such that
$\partial\partial_0=\partial\partial_1: X_1\to X_{-1}$. 

Let $\mC$ be a category with finite limits and colimits. The inclusion functor $j: \Delta \to
\Dp$ induces two adjunctions
\[
	\begin{aligned}
		j_!& : \mC_{\Delta} \longleftrightarrow \mC_{\Dp}: j^* \\ 
		j^*& : \mC_{\Dp} \longleftrightarrow \mC_{\Delta}: j_*.
	\end{aligned}
\]
While the pullback functor $j^*$ simply forgets the augmentation,
its left and right adjoints $j_!, j_*$ provide two natural
ways to equip a simplicial object with an augmentation. 
For $X\in\mC_\Delta$ we will use the abbreviations
\[
		X^{\clubsuit} := j_!(X), \quad X^+ := j_*(X).
\]
Explicitly, we have
\begin{equation}\label{eq:X_+-clubsuit}
 X^\clubsuit_{-1} =  \Pi_0(X), \quad 
 X^+_{-1} =  \pt, \quad \text{and} \quad X^\clubsuit_n = X^+_n= X_n, \quad n\geq 0.
\end{equation}
Here $\pt$ denotes the final object of $\mC$ and
\[
\Pi_0(X) := \ind^\mC 
\left\{ 
\xymatrix{ X_1 \ar@<.5ex>[r]^-{\partial_0}\ar@<-.5ex>[r]_-{\partial_1}&X_0}
\right\}
\]
denotes the {\em internal space of connected components}.
 
\begin{rem} An {\em augmented semi-simplicial object} in $\mC$ is a functor $X: \Delta^{+\op}_\inj\to \mC$,
where $\Delta^+_\inj\subset\Delta^+$ is the subcategory formed by injective morphisms of all finite
ordinals. As in the simplicial case, we have the embedding $\bar j:\Delta_\inj\to\Delta^{+\op}_\inj $ 
which gives rise to the pullback functor $\bar j^*$ and its two adjoints
$\bar j_!: X\mapsto X^\clubsuit$, $\bar j_*: X\mapsto X^+$, which are again given by the formulas of \eqref{eq:X_+-clubsuit}.
\end{rem}

\vfill\eject 

\subsection{Path space adjunctions}
\label{subsec:simppath}

We now recall the construction of simplicial path spaces, due originally to Illusie
(\cite{illusie}, Ch. VI) who calls them ``les d\'ecal\'es d'un objet simplicial". 

Given two ordinals $I, I'$, their {\em join} is defined to be the ordinal
$I \star I' := I \cop I'$
where each element of $I$ is declared to be smaller than each element of $I'$. 
Let $\mC$ be a category with small limits and colimits. 
The functors
\[
	\begin{aligned}
		i & : \Dp \lra \Delta,\; I \mapsto [0] \star I \\
		f & : \Dp \lra \Delta,\; I \mapsto I \star [0]
	\end{aligned}
\]
induce adjunctions 
\[
	\begin{aligned}
		i_! & : \mC_{\Dp} \longleftrightarrow \mC_{\Delta}: i^*\\
		f_! & : \mC_{\Dp} \longleftrightarrow \mC_{\Delta}: f^*.
	\end{aligned}
\]
We further consider the inclusion functor $j: \Delta \to \Dp$ and the induced adjunction
\[
	j_! : \mC_{\Delta} \longleftrightarrow \mC_{\Dp}: j^*
\]
from \S \ref{subsec:augmented}. 
We call the functors $j^* \circ i^*$ and $j^* \circ f^*$ the {\em initial} and {\em final path space}
functors, and $i_! \circ j_!$ and $f_! \circ j_!$ the
{\em left} and {\em right cone} functors, respectively. To
emphasize this terminology, we will use the notation
\[
\PI = j^* \circ i^*, \quad \PF = j^* \circ f^*\quad \CI =
i_! \circ j_!, \quad \CF = f_! \circ j_!. 
\]

We give explicit descriptions of the path space and cone functors.
For the path space functors, note that $\PI(X)_n =\PF(X)_n = X_{n+1}$, 
$n\geq 0$, the face morphisms are given by 
\begin{equation}
	\label{eq:faces-path-space}
	\begin{gathered}
	 \bigl\{\partial_i^{n}: \PIp(X)_n\to \PIp(X)_{n-1} \bigr\} =
	 \bigl\{ \partial_{i+1}^{n+1}: X_{n+1}\to X_n\bigr\}, \quad i=0, ..., n; \\
	 \bigl\{\partial_i^{n}: \PFp(X)_n\to \PFp(X)_{n-1} \bigr\} =
	 \bigl\{ \partial_{i}^{n+1}: X_{n+1}\to X_n\bigr\}, \quad i=0, ..., n,
	\end{gathered} 
\end{equation}
and similarly for the degeneracies.

For augmented simplicial objects $X', X''\in \mC_{\Dp}$, we define
the {\em join} $X'\star X'' \in \mC_{\Delta}$ by the formula 
\begin{equation}
	\label{eq:join}
	(X' \star X'') (J) := \coprod_{\substack{I' \cup I'' = J\\I' < I''}} X'(I') \times X''(I'') \text{,}
\end{equation}
for a finite nonempty ordinal $J$. Here the coproduct is
taken over all ordered pairs $(I', I'')$ of possibly empty subsets of $J$ satisfying the conditions
as stated. For instance, $(J,\emptyset)$ and $(\emptyset, J)$ give two different summands. 
This formula is then extended to morphisms by taking preimages of disjoint decompositions.
 
\begin{prop} \label{prop:coneformula}
	Let $\pt$ denote the final object of $\mC_{\Dp}$ so that $\pt$ assigns the
	final object of $\mC$ to every finite ordinal.
	For an augmented simplicial object $X \in \mC_{\Dp}$ we have the formulas
 	\[
	 i_!(X) \cong \pt \star X, \quad
	 f_!(X) \cong X \star \pt.
	\]
\end{prop}
\begin{proof}
	We treat the statement for the functor $i_!$, the argument for $f_!$ is analogous. 
	By the pointwise formula for left Kan extensions, we have
	\[
	 i_! X (J) \cong \ind_{\{J \to [0] \star I\} \in (J \backslash \Dp)^\op} \; X_I.
	\]
	The objects of the category $(J\backslash \Dp)^\op$ are maps $\a: J \to [0] \star I$ in
	$\Delta$ while a morphism from such $\a$ to $\a': J\to [0]\star I'$ is a monotone
	map $I'\to I$ making the triangle commute. Now notice that $(J\backslash \Dp)^\op$
	is the disjoint union of subcategories each having a final object. These subcategories
	are labelled by disjoint decompositions $J=I'\sqcup I''$ as in \eqref{eq:join}.
	The subcategory corresponding to $(I', I'')$ consists of maps $\a$ such that $I'$
	is the preimage of the new minimal element $\overline 0\in [0]*I$, and the final
	object is given by the map $I\to I/I'$. Therefore, we have an isomorphism
	\[
 	 \ind_{\{J \to [0] \star I\} \in 
	 J\backslash \Dp} \; X_I \cong 
	\coprod_{\substack{I' \sqcup I'' = J\\I' < I''}} X(I''), 
	\]
	which implies the claimed formula.
\end{proof}

\begin{rem}\label{rem:simplicial-cone}
	Let $K$ be a simplicial set and set $K_{-1} = \Pi_0(K)$. We deduce from the proposition the
	explicit formula 
	\[
	\CF(K)_n \cong K_n \amalg K_{n-1} \amalg \dots \amalg K_{-1}, \quad n\geq 0,
	\]
	with the summand $K_m$ corresponding to $I' = \{0, 1, ..., m\}$ in \eqref{eq:join}.
	Denoting by $(x,m)_n$ the element of $\CF(X)_n$ corresponding to $x\in X_m$, we find the face
	maps by the formula
	\[
	\partial^n_i (x,m)_n =
	\begin{cases}
	(\partial^m_i(x), m-1)_{n-1}, \quad \text{if}\quad i \leq m,\\
	(x,m)_{n-1}, \quad\quad\quad \text{if $i>m$.}
	\end{cases}
	\]
	and similarly for degeneracies.  

	The right cone $\CF(K)$ is obtained by adding a new vertex
	$\overline v$ for each connected component $v \in \Pi_0(K)$, then an oriented edge from each $\overline v$
	to each vertex of the connected component $v$, then a triangle with one vertex $\overline v$
	for each 1-simplex in $v$ and so on. So the geometric
	realization $|\CF(K)|$ is the disjoint union of the geometric cones
	over the connected components of $|K|$. The left cone can be understood
	similarly but with different orientation of the edges, triangles, etc.
\end{rem}

\begin{ex} \label{ex.cone}  
	Let $\I$ be a collection of subsets of $[n]$
	and $\Delta^\I$ be the corresponding simplicial subset of $\Delta^n$, see \eqref{eq.DI}. We
	assume that $\Delta^\I$ is connected, i.e., $\Pi_0(\Delta^{\I}) \cong \pt$.
	Define the collection $\I^{\triangleleft}$ of subsets of $\{\overline{0},0,\dots,n\}
	\cong [n+1]$ by appending the element $\overline{0}$ to each set in $\I$. Then we have a
	natural isomorphism
	\[
	 \CI(\DI) \cong \Delta^{\I^{\triangleleft}}.
	\]
	We apply this observation to the collection $\I_n = \{ \{0,1\}, \{1,2\},\cdots,\{n-1,n\}\}$.
	The collection $\I_n^{\triangleleft}$, considered as a subset of $2^{[n+1]}$, corresponds 
	to the special triangulation of a convex $(n+2)$-gon in which all triangles have the common vertex $\{0\}$ . Using analogous
	definitions for the right cone, we obtain that $\I_n^{\triangleright}$ corresponds to the
	special triangulation in which all triangles have the common vertex $\{n+1\}$. 
\end{ex}

The following proposition, which is the central result of this section, 
tells us how the path space and cone functors interact with the $\mC$-enriched membrane spaces
defined in \S \ref{subsec:yonedaext}.

\begin{prop}
\label{prop:cofinal} Let $\mC$ be a combinatorial model category. Let $X \in \sC$ and $K
\in \Set_{\Delta}$. 
	\begin{enumerate}[label=(\alph*)]
	 \item We have natural isomorphisms in $\mC$
	 \[
	 (K, \PI(X)) \cong (\CI(K), X),\quad (K, \PF(X)) \cong (\CF(K), X).
	 \]
	 \item Assume further that each connected component of $|K|$ is weakly contractible. 
	 Then we have natural isomorphisms in $\h\mC$
	 \[
	 (K, \PI(X))_R \simeq (\CI(K), X)_R, \quad  
	 (K, \PF(X))_R \simeq (\CF(K), X)_R.
	 \]
	\end{enumerate}
\end{prop}
\begin{proof} Assertion (a) follows immediately from Proposition \ref{prop:yonedaadj}. Part (b) is a consequence of Lemma \ref{lem:adjn1} and Lemma \ref{lem:adjn2}
	below.
\end{proof}

\begin{lem} \label{lem:adjn1} 
	Let $\mC$ be a combinatorial model category. For every $X \in \sC$ and $M
	\in \Set_{\Dp}$ we have natural isomorphisms
	\[
	 (M, i^*X)_R \simeq (i_!M, X)_R, \quad  
	 (M, f^*X)_R \simeq (f_!M, X)_R
	\]
 	in $\Ho(\mC)$.
\end{lem}
\begin{proof} 
	We reduce the statement to Proposition \ref{prop:derivedyonedaadj} by showing that $i^*$ and
	$f^*$ preserve injective fibrations. Equivalently, we can show that the left adjoints $i_!$
	and $f_!$ preserve trivial injective cofibrations. This follows from the formulas of
	Proposition \ref{prop:coneformula}, since trivial injective cofibrations are defined
	pointwise, and trivial cofibrations are stable under coproducts.
\end{proof}

\begin{lem} \label{lem:adjn2} 
	Let $K$ be a weakly contractible simplicial set. Then, for every $Y
	\in \C_{\Dp}$, there is a natural weak equivalence
	\[
	(K, j^*Y)_R \simeq (j_! K, Y)_R 
	\]
	in $\Ho(\C)$.
\end{lem}
\begin{proof}
	We have the formulas
	\[
		(j_! K, Y)_R = \hopro_{\{\Dp/j_! K\}} Y'
	\]
	and
	\[
		(K, j^*Y)_R = \hopro_{\{\Delta/K\}} Y''
	\]
	where $Y'$ and $Y''$ denote the functors induced by $Y$ on the categories 
	$(\Dp/j_!K)^{\op}$ and $(\Delta/K)^{\op}$, respectively. 
	The functor $j_!$ induces a natural embedding 
	\[
	 k: \Delta/K \lra \Dp/j_! K
	\]
	such that $(k^{\op})^*Y' = Y''$. Thus, it suffices to show that $k$ is homotopy final
	(\cite[19.6]{hirschhorn}).  Since the functor $k$ is fully faithful, we only have to verify
	the contractibility of the undercategories of objects in $\Dp/j_! K$ which are not in the
	essential image of $k$. The only such objects are given by maps of the form
	\[
	 c: h_{\emptyset} \to j_! K,
	\]
	which are in natural bijective correspondence with the set $\Pi_0(K)$ of connected
	components of $|K|$. The slice category $c/k$ is isomorphic to the category $\Delta/K_c$
	where $K_c \subset K$ denotes the connected component classified by $c$. By assumption, each
	connected component $K_c$ is weakly contractible. Therefore, it suffices to show that, for
	any weakly contractible simplicial set $S$, the simplicial set $\N(\Delta/S)$ is weakly
	contractible. Consider the inclusion
	\[
	 i: \N(\Delta_{\inj}/S) \subset \N(\Delta/S), 
	\]
	where $\Delta_{\inj} \subset \Delta$ denotes the subcategory of monomorphisms. Using
	Quillen's Theorem A (\cite{quillen-ktheory}), it is easy to verify that $i$ is a weak
	homotopy equivalence. The geometric realization $|\N(\Delta_{\inj}/S)|$ can be identified
	with the barycentric subdivision of $|S|$. Hence we have a natural homeomorphism
	$|\N(\Delta_{\inj}/S)| \cong |S|$ which concludes our argument since $|S|$ is by assumption
	weakly contractible.
\end{proof}

\begin{rem} 
Assume that $|K|$ is connected. In this case, we have $j_! K \cong j_* K$ 
and therefore the cones $\CI(K)$ and $\CF(K)$ are the ``usual" cones over $K$
obtained by adding a single initial or final vertex, respectively. 
\end{rem}

\vfill\eject
   
\subsection{The path space criterion}

Let $\mC$ be a combinatorial simplicial model category. We equip the category $\mC_\Delta$ with the
injective model structure. Given a map $f: K \to K'$ of simplicial sets, we say that an object $X
\in \mC_{\Delta}$ is {\em $f$-local}, if the map
\[
(R\Upsilon_*X)(f): (K',X)_R \lra (K,X)_R
\]
induced by $f$ is a weak equivalence in $\mC$.

\begin{prop}\label{prop.contractible}
	Let $X \in \mC_\Delta$ be a simplicial object and $f: K \to K'$ a morphism of weakly contractible 
	simplicial sets. Then $\PI X$ (resp. $\PF X$) is
	$f$-local if and only if $X$ is $\CI(f)$-local (resp. 
	$\CF(f)$-local).
\end{prop}
\begin{proof} This follows from Proposition \ref{prop:cofinal}.
\end{proof}

Note that applying $\CI$ and $\CF$ to the $1$-Segal coverings $\Delta^{\I_n}\to\Delta^n$ we get some
particular $2$-Segal coverings (Example \ref{ex.cone}(b)). This suggests that the path space
constructions mediate between $1$-Segal and $2$-Segal conditions. Indeed, we have the following
result.
	 
\begin{thm}[Path Space Criterion]\label{thm.crit} 
	 Let $X$ be a simplicial object in $\mC$. Then the following conditions are equivalent:
	\begin{enumerate}
	 \item[(i)] $X$ is a $2$-Segal object.
	 \item[(ii)] Both path spaces $\PI X$ and $\PF X$ are $1$-Segal objects.
	\end{enumerate}
\end{thm}
\begin{proof}
	The implication (i)$\Rightarrow$(ii) follows by applying Proposition \ref{prop.contractible} to Example
	\ref{ex.cone} (b).
	To obtain (ii)$\Rightarrow$(i), let $X$ be a simplicial object with $\PI X$ and
	$\PF X$ being $1$-Segal objects. 
	By Proposition \ref{prop:1-segal-basic}, a $1$-Segal object is in fact local with respect to
	maps $\DI \to \Delta^n$ with $\I$ being any collection of the form
	\[
	\I = \bigl\{ \{0,1,\dots,i_1\},
	\{i_1,i_1+1,\dots, i_2\}, \dots, \{i_{k},i_{k}+1,\dots, n\} \bigr\}\text{.}
	\]
	We now argue by induction on $n$. 
	Assume that for each $n' < n$ and for each triangulation $\T'$ of the $(n'+1)$-gon, 
	the $2$-Segal map
	$f_{\T'}: X_{n'} \to RX_{\T'}$, is a weak equivalence.
	Let $\T$ be a  triangulation of the $(n+1)$-gon $P_n$.  
	Note, that at least one of the following cases must hold:
	\begin{enumerate}[label=(\arabic*)]
	 \item \label{crit.case.1} The triangulation $\T$ contains an internal edge with vertices 
	 $\{0,i\}$ where $1 < i < n$.  
	 \item \label{crit.case.2} The triangulation $\T$ contains an internal edge with vertices
	 $\{i,n\}$ where $0 < i < n-1$.
	\end{enumerate}
	Assume \ref{crit.case.1} holds. Let $\I'= \{ \{0,1,\dots,i\},\{0,i,i+1,\dots,n\}\}$ and
	note that $\I'$ is obtained as the left cone of the collection $\{
	\{1,\dots,i\},\{i,i+1,\dots,n\}\}$. Since by assumption the initial path space $\PI X$ is a
	$1$-Segal object, we apply Proposition \ref{prop.contractible} to deduce that the map
	\[
	g:\;X_n \lra RX_{\I'} =  X_{\{0,1,\dots,i\}} \times^R_{X_{\{0,i\}}} X_{\{0,i,i+1,\dots,n\}}
	\]
	is a weak equivalence. The edge $\{0,i\}$ decomposes the polygon $P_n$ into two
	subpolygons: the ($i+1$)-gon $P^{(1)}$ with vertices $\{0,1,\dots,i\}$ and the ($n-i+2$)-gon
	$P^{(2)}$ with vertices $\{0,i,i+1,\dots,n\}$. Since $\{0,i\}$ is an internal edge of the
	triangulation $\T$, we obtain induced triangulations $\T_1$ of $P^{(1)}$ and $\T_2$ of
	$P^{(2)}$ . By induction, both $2$-Segal maps $f_{\T_1}$ and $f_{\T_2}$ corresponding to
	these triangulations are weak equivalences. Further, by Proposition \ref{prop.colim}, we
	have a natural weak equivalence
	\[ RX_{\T }
	\stackrel{\simeq}{\lra} RX_{\T_1} \times^R_{ X_{{\{ 0 , i \}}} } RX_{\T_2}.
	\]
	We assemble the constructed maps to form the commutative diagram
	\[
	\xymatrix@C=15ex{
	X_n \ar[d]_{f} \ar[r]^-{g} & X_{\{0,1,\dots,i\}} \times^R_{X_{\{0,i\}}}
	X_{\{0,i,i+1,\dots,n\}}\ar[d]^{(f_{\T_1}, f_{\T_2})}\\
	RX_{\T} \ar[r]^-{\simeq} &R X_{\T_1} \times^R_{ X_{{\{ 0 , i \}}} } RX_{\T_2} 
	}
	\]
	from which we deduce, by the two out of three property, that the $2$-Segal map $f$ is a weak equivalence. 
	In the case \ref{crit.case.2}, we argue similarly using that $\PF X$ is a $1$-Segal object.
\end{proof}

\begin{exa} Let $\Ec$ be a proto-exact category (Definition \ref{def:proto-exact}), with
the classes $\Men, \Een$ of admissible mono- and epi-morphisms, which we consider
as subcategories in $\Ec$. Let $\Sc_\bullet(\Ec)$ be the Waldhausen simplicial groupoid of 
$\Ec$. Lemma \ref{leq:waldhausen-filtration-abelian} identifies both $\PI\Sc_\bullet(\Ec)$ and
$\PF\Sc_\bullet(\Ec)$. More precisely, $\PI\Sc_\bullet(\Ec)$ is equivalent, as a simplicial
groupoid, to the categorified nerve of $\Men$, while $\PF\Sc_\bullet(\Ec)$ is equivalent
to the categorified nerve of $\Een$. As the categorified nerve of any category is
a $1$-Segal simplicial groupoid, invoking Theorem \ref{thm.crit} provides an alternative proof of the fact that
$\Sc_\bullet(\Ec)$ is $2$-Segal. 
\end{exa}

\vfill\eject

\subsection{The path space criterion: semi-simplicial case}
\label{sec:semisimp-paths} 

Since many interesting examples of $2$-Segal spaces live in the
semi-simplicial world, we briefly discuss the corresponding
modification of the path space criterion. 
 
We denote by $\Delta^+_\inj\subset\Delta^+$ the category of
all (possibly empty) finite ordinals and monotone injective maps. 
An augmented semi-simplicial object in a category $\mC$ is a contravariant
functor $X: \Delta^+_\inj\to\mC$. The category of such functors will be
denoted $\mC_{\Delta^+_\inj}$. 
For $n\geq -1$, we have the 
$n$th ``augmented semi-simplex"
$\Delta^{+n}_{\inj}$ which is 
the functor represented by $[n]$ on $\Delta^+_\inj$. 
Note that unlike the simplicial case, $\Delta^{+0}_\inj$ is not the final object
of $\Set_{\Delta^+_\inj}$:
\[
(\Delta^{+0}_\inj)_n = 
\begin{cases}
\pt, \quad n= -1, 0,\\
\emptyset, \quad n>0. 
\end{cases}
\]
The final object is the augmented semi-simplicial set $F$ with $F_n=\pt$ for all $n\geq -1$. 
The join $X\star Y$ of two augmented semi-simplicial sets $X$ and $Y$ is
defined in the same way as in \eqref{eq:join}. 

As before, we have the functors
\begin{align*}
 \bar j & : \Delta_\inj \hookrightarrow \Delta^+_\inj, \; I\mapsto I \quad \text{(embedding)}\\
 \bar i & : \Delta_\inj^+ \hookrightarrow \Delta_\inj, \; I\mapsto [0] * I \\
  \bar f & : \Delta_\inj^+ \hookrightarrow \Delta_\inj, \; I\mapsto I * [0]. 
\end{align*}
The adjoint functors 
\[
\bar j_*: X\mapsto X^+, \quad \bar j_!: X\mapsto X^\clubsuit
\]
to $\bar j$ are given by the same formulas as in
\eqref{eq:X_+-clubsuit}. Similarly, the pullback functors
$\bar i^*$, $\bar f^*$ are given by the same formulas as
\eqref{eq:faces-path-space} and we set
\[
\overline \PI = \bar j^* \circ\bar i^*, \quad\overline\PF =\bar j^* \circ\bar f^*\quad\overline\CI =
\bar i_! \circ\bar j_!, \quad\overline\CF =\bar f_! \circ\bar j_!. 
\]
We have the following modification of Proposition \ref{prop:coneformula}.

\begin{prop} We have
\[
\bar i_!(X) = \Delta^{+ 0}_\inj \star X, \quad \bar f_!(X) = 
 X\star \Delta^{+ 0}_\inj.
\]
\end{prop}

\begin{ex}
Comparing to Remark \ref{rem:simplicial-cone}, in the semi-simplicial case we have
\[
\overline\CF(X)_n = X_n \sqcup X_{n-1}, \quad n\geq 0
\]
with faces given by the same formula as in in that example, 
but restricted to $m\in\{n, n-1\}$. Similarly for $\overline\CI$. 
\end{ex}

We have the following semi-simplicial variant of the path space criterion.

\begin{thm}[Path Space Criterion]\label{thm.crit:semisimp} 
	 Let $\mC$ be a combinatorial model category and 
	 $X$ a semi-simplicial object in $\mC$. Then the following conditions are equivalent:
	 \begin{enumerate}[label=(\roman{*})]
	 \item $X$ is $2$-Segal  
	 \item Both $\PI X$ and $\PF X$ are $1$-Segal. 
	\end{enumerate}
	
\end{thm}
\begin{proof} The proof is analogous to that of Theorem \ref{thm.crit} and is left to the reader. \end{proof}

\begin{ex}\label{ex:pentagon-path-criterion}
Let $\mC=\Set$ (with trivial model and simplicial structures) and let $X$ be a
$2$-Segal semi-simplicial object in $\Set$ with $X_0=X_1=\pt$. 
We denote $C = X_2$. By
Corollary \ref{cor:2-segal-pentagon-point}, $X$ corresponds to a 
set-theoretic solution
\[
\alpha: C^2\lra C^2, \quad \alpha (x,y) = (x\bullet y, x*y)
\]
of the pentagon equation. 
The initial path space $\overline\PI(X)$ is, by Theorem 
\ref{thm.crit:semisimp}, a $1$-Segal semi-simplicial set. Since
$\overline\PI(X)_0=X_1=\pt$, we see that $\overline\PI(X)$
must be the nerve of a semigroup. This semigroup is nothing
but $C$ with operation $\bullet$ which is associative by 
\eqref{eq:three-identities-pentagon}. The path space criterion
therefore provides a conceptual explanation of the
suprising fact 
(observed in \cite{kashaev-sergeev, kashaev-reshitikhin})
that the first component of a pentagon solution gives an
associative operation. 

For any semi-simplicial set $Z$ let $Z^\op$ be the 
 semi-simplicial set induced from $Z$
 by the self-equivalence
\[
\Delta_\inj\lra \Delta_\inj, \quad I\mapsto I^\op 
\]
Then 
 the final path space $\overline\PF(X)$ can be identified 
with $(\overline\PI(X^\op))^\op$.
  If $X$ corresponds to a solution $\alpha$ of the pentagon equation, 
  then $X^\op$ corresponds to the new solution
\[
\alpha^* = P_{12}\circ\alpha^{-1}\circ P_{12}, 
\]
 where $P_{12}: C^2\to C^2$ is the permutation. Therefore $\overline\PF(X)$
 is the nerve of the semigroup opposite to that given by the first component of $\alpha^*$. 
\end{ex}

\begin{ex}[(Semi-simplicial suspension)] In the semi-simplicial case
(unlike the simplicial one) the path space functors have right inverses.

For a nonempty finite ordinal $I$ let $I^-\subset I$ be the subset obtained
by removing the maximal element. Note that any monotone {\em injection}
$I\to J$ defines a monotone injection $I^- \to J^-$. Indeed, no element of $I$ other than
$\max(I)$ can possibly map into $\max(J)$. Similarly for $^- I\subset I$, the subset 
obtained by removing the minimal element. 
We have therefore the functors
\[
I \longmapsto {}^-I, I^-, \quad \Delta_\inj \lra\Delta^+_\inj. 
\]
 The induced pullback functors on semi-simplicial objects will be called the
 (augmented) {\em semi-simplicial suspension} functors
 \[
 \begin{gathered}
 \Sigma^\triangleleft_+, \Sigma^\triangleright_+: \mC_{\Delta^+_\inj} \lra \mC_{\Delta_\inj},\\
 \Sigma^\triangleleft_+(X)_I = X_{^- I}, \quad \Sigma^{\triangleright}_+(X) = X_{I^-}.
 \end{gathered}
 \]
 Thus, for instance, 
 \[
 \Sigma^\triangleleft_+ (X)_n = X_{n-1}, \quad n\geq 0,
 \]
  while the face operators are given by
 \[
 \partial^{n, \Sigma^\triangleleft_+ (X)}_i =
 \begin{cases}
 \partial^{n-1, X}_0, \quad i=0; \\
    \partial^{n-1, X}_{i-1}, \quad i\geq 1. 
 \end{cases}
 \]
 Thus the operator $\partial^{n-1, X}_0$ is repeated twice. Similarly for $\Sigma^\triangleright_+$,
 where $\partial^{n-1, X}_{n-1}$ is repeated twice.

 We also define the {\em unaugmented suspensions} of $X$ by applying the above to
 the one-point augmentation $X^+$ of $X$:
 \[
 \Sigma^\triangleleft(X) = \Sigma^\triangleleft_+(X^+), 
 \quad 
 \Sigma^\triangleright(X) = \Sigma^\triangleright_+(X^+). 
 \]
\end{ex}

\begin{prop}
 We have isomorphisms
\[
\PI \Sigma^\triangleleft (X) = X = \PF\Sigma^\triangleright (X). 
\]
\end{prop}

\begin{proof} Follows from the canonical identifications
\[
{}^-([0]*I ) = I = (I*[0])^-. 
\]
\end{proof}

Therefore, if $X$ is $1$-Segal, then $\Sigma^\triangleleft (X)$ (as well as 
$\Sigma^\triangleright (X)$)
automatically satisfies
one half of the conditions needed for it to be $2$-Segal.

\begin{defi}
Let $\mC$ be a semi-category (i.e., possibly without unit morphisms).
We say that $\mC$ is {\em left divisible}, if for any objects $x,y,z\in\mC$
and moprhisms $f: y\to z$, $h: x\to z$ there is a unique morphism $g: x\to y$
such that $h=fg$. We say that $\mC$ is {\em right divisible}, if $\mC^\op$
is left divisible.
\end{defi}

Thus a category (with unit morphisms) is left or right divisible, if and only if
it is a groupoid. 

\begin{prop}\label{prop:suspension-$2$-Segal} 
Let $\mC$ be a small semi-category. Then $\Sigma^\triangleleft (\N\mC)$ is
$2$-Segal if and only if $\mC$ is left divisible. 
\end{prop}

It follows that $\Sigma^\triangleright(\N\mC)$ is $2$-Segal if and only if
$\mC$ is right divisible. 

\begin{proof} Let $X=\NC$. We first prove the ``if" part. 
 Suppose that $\mC$
is left divisible. To prove that $\Sigma^\triangleleft (X)$ is $2$-Segal, it suffices,
by the above, to verify that $\PF\Sigma^\triangleleft(X)$ is $1$-Segal. 
By definition, $\PF\Sigma^\triangleleft(X)_I = X_{{}^-(I*[0])}$.
Identifying $^-([n]*[0])$ with $[n]$, we can write that
$\PF\Sigma^\triangleleft(X)$ has the same components
$\PF\Sigma^\triangleleft(X)_n=X_n$ as $X$, 
but equipped with new face operators $\partial'_i: X_n\to X_{n-1}$, $i=0, ..., n$
given by
\[
\partial'_0=\partial_0, \quad \partial'_1 = \partial_0, 
\quad \partial'_2 =\partial_1,\quad ..., \quad
\partial'_n =\partial_{n-1}. 
\]
Let us view $X_n=\N_n\mC$ as the set of commutative $n$-simplices in $\mC$, i.e.,
of systems of objects and morphisms
\[
 (x_i, u_{ij}: x_i\to x_j)_{0\leq i<j\leq n}, \quad u_{ik}=u_{jk}u_{ij}, i<j<k. 
 \]  
 The $n$-fold fiber product $X_1\times_{X_0} \cdots \times_{X_0} X_1$
 defined with respect to the new face operators $\partial'_i$,
 consists of systems of objects and morphisms
 \[
 x_0, ..., x_n, v_{in}: x_i\to x_n, i=0, ..., n-1.
 \]
 The $1$-Segal map (for the new face operators)
 \[
 f'_n: X_n \lra X_1\times_{X_0} \cdots \times_{X_0} X_1
 \]
  sends a system $(x_i, u_{ij})$ as above, to the subset formed by morphisms
 $u_{in}: x_i\to x_n$ for $i=0, ..., n-1$. If $\mC$ is left divisible, then we can uniquely
 complete any given system 
 of morphisms $(u_{in}: x_i\to x_n)$ to a full commutative simplex $(u_{ij})$ by succesive
 left divisions. This proves the ``if" part. 
  The ``only if" part follows from considering the particular case $n=2$: bijectivity
  of $f'_2$ is precisely the left division property. \end{proof}
  
\begin{ex}
When $\mC=G$ a group, the proposition claims that $\Sigma^\triangleleft(\N G)$ is $2$-Segal. This
$2$-Segal semi-simplicial set corresponds to the solution of the pentagon equation from Example
\ref{ex:group-pentagon-solution}. The $1$-Segal semi-simplicial set $\PF\Sigma^\triangleleft(\N G)$
is isomorphic to the nerve of the semigroup  formed by $G$ with the operation $*$ defined by
$g*h=h$.  This operation is associative but has no unit.  Of course, $\Sigma^\triangleright(\N G)$
is $2$-Segal as well. 
\end{ex}

\vfill\eject

\section{2-Segal spaces from higher categories}
\label{sec:2-segal-from-higher}

All simplicial spaces in this section will be combinatorial, i.e., objects of $\sS$.

\subsection{Quasi-categories vs. complete 1-Segal spaces}
\label{subsec:quasicat-1-segal}
Quasi-categories of and complete $1$-Segal spaces provide two equivalent approaches to formalizing
the intuitive concept of $\inftyone$-categories.  In this section we recall a correspondence between
the two models, as given in \cite{joyal-tierney}. 

Let $X \in \sS$ be a $1$-Segal space. For vertices $x,y$ of the simplicial set $X_0$, we have a natural map
\[
\{x\} \times_{X_0} X_1 \times_{X_0} \{y\} \to \{x\} \times_{X_0}^R X_1 \times_{X_0}^R \{y\} =
\map_X(x,y)
\]
Recall that $\pi_0 \map_X(x,y)$ forms the set of morphisms of the homotopy category $\h X$ of $X$.
Suppose $f \in X_1$ with $\partial_1(f) = x$ and $\partial_0(f) = y$. Its image $[f] \in \pi_0
\Map_X(x,y)$ is a morphism in $\h X$. We call $f$ an {\em equivalence} if $[f]$ is an isomorphism in $\h X$. Denote by $\delta: X_0 \to X_1$ the degeneracy map
corresponding to the unique map of ordinals $[1] \to [0]$. For a vertex $x \in X_0$, the
vertex $\delta(x) = \id_x$ is an equivalence.

\begin{defi} Let $X$ be a $1$-Segal space and let $X_1^{\on{equiv}} \subset X_1$ denote the
	simplicial subset spanned by those vertices which are equivalences. We say $X$ is complete if
	the map $\delta: X_0 \to X_1^{\on{equiv}}$ is a weak homotopy equivalence of simplicial sets.
\end{defi}

\begin{ex} Let $\C$ be a small category, and $\C_\bullet$ be the categorified nerve 
of $\C$, which is the simplicial groupoid defined in Example \ref{ex.1segclass}(a). The 
simplicial space 
\[
\N(\C_\bullet) = (\N(\C_n))_{n\geq 0},
\]
obtained by taking the nerve of each $\C_n$,
is a complete $1$-Segal space (see \cite{rezk}).
Note that the discrete nerve $\disc{\N(\C)}$ from 
Example \ref{ex.1segclass}(b) is $1$-Segal but generally not complete. 
\end{ex}

We recall the following result of \cite{rezk}.

\begin{thm}[Rezk]
	There exists a left proper combinatorial simplicial model structure on $\sS$ with the following properties:
	\begin{itemize}
	 \item[(W)] The weak equivalences are the maps $f$ such that $\RMap_{\sS}(f,X)$ is a
	 weak equivalence of simplicial sets for any complete $1$-Segal space $X$.
	 \item[(C)] The cofibrations are the monomorphisms.
	 \item[(F)] The fibrant objects are the Reedy fibrant complete $1$-Segal spaces.
	\end{itemize}
\end{thm}
\begin{proof} Consider the fat $1$-simplex $(\Delta^1)'$ from Example \ref{ex:nerve,fat-simplex}(a), given by the nerve of the groupoid completion
	of the category $[1]$. By Theorem 6.2 of \cite{rezk}, a Reedy fibrant $1$-Segal space $X$ is complete if and only if it is
	local with respect to the unique map $\disc{(\Delta^1)'} \to \disc{\Delta^0}$. Thus the
	statement follows from the general formalism of simplicial Bousfield localization (e.g., Theorem \ref{thm:barwick} with $\C = \Sp$).
\end{proof}
   
On the other hand, Joyal constructed a model structure $\Joy$ on $\Sp$ whose fibrant objects
are precisely the quasi-categories (\cite{joyal.kan}, also \cite[2.2.5]{lurie.htt}).
In \cite{joyal-tierney}, the authors construct a Quillen equivalence
of model categories
\begin{equation}\label{eq.totalization}
t_! :\; (\sS, \Rc) \longleftrightarrow (\Sp, \Joy) \;: t^!,
\end{equation}
which we call the {\em Joyal-Tierney equivalence}. Here, 
the {\em totalization functor} $t_!:\sS \to \Sp$ is
uniquely described by the formula 
\[
	t_!(\disc{\Delta^n} \times \Delta^m) = \Delta^n \times (\Delta^m)'
\]
and the requirement that it commutes with colimits.
The functor $t^!$ is defined as the right adjoint of $t_!$. Consequently, for a simplicial set $K$, we have  
\[
	(t^!K)_{mn} = \Hom_{\Sp}(\Delta^n \times (\Delta^m)', K)\text{.}
\]
We point out a few aspects of the Joyal-Tierney equivalence which are relevant for
our discussion.
An immediate consequence of \eqref{eq.totalization} is that any complete $1$-Segal space $X$ is weakly
equivalent to a space of the form $t^! \C$ where $\C$ is a quasi-category. 
In the examples below we fix a quasi-category $\C$ and set $X = t^!\C$. Note that, since $\C$ is
Joyal fibrant, $X$ is a Reedy fibrant complete $1$-Segal space.

\begin{exa}[(Homotopy coherent diagrams)]
\label{comp.diag} Consider a simplicial set $D \in \Sp$. A {\em
	 $D$-diagram}
	 in $X$ is defined to be a map of simplicial spaces $p: \disc{D} \to X$. We
	 further define the {\em classifying space of $D$-diagrams} in $X$ to be the simplicial set
	 $\Map_{\sS}(\disc{D}, X)$, which in the terminology of Sections
	 \ref{subsec:membranes} and \ref{subsec:yonedaext}, 
	 is the space $(D,X)$ of membranes in $X$ of type $D$.
	 By a $D$-diagram in $\C$ we will mean a morphism
	 of simplicial sets $D\to\C$.
	 Using the equivalence \eqref{eq.totalization} and \cite[1.20]{joyal-tierney}, we obtain a natural weak equivalence
	 \[
	 \Map_{\sS}(\disc{D}, X) \stackrel{\simeq}{\lra} \Map_{\Sp}(D, \C)_{\Kan}
	 \]
	 where $(-)_{\Kan}$ is the functor from \cite[1.16]{joyal-tierney} which maps a
	 quasi-category $\C$ to its largest Kan subcomplex $\C_{\Kan} \subset \C$. In
	 particular, we obtain, for each $n \ge 0$, a weak equivalence
	 \[
	 X_n \stackrel{\simeq}{\lra} \Map_{\Sp}(\Delta^n, \C)_{\Kan} \text{.}
	 \]
	 If $D=\N(A)$ is the nerve of a small category $A$, then a $D$-diagram in $X$
	 (resp. in $\C$) will
	 also be called a {\em homotopy coherent $A$-diagram} in $X$ (resp.{ in} $\C$). 
\end{exa}

\begin{exa}[(Mapping spaces; limits and colimits)]
\label{comp.map} Fix elements $x,y \in \C_0 (=
	X_{00})$. By Example \ref{comp.diag}, we have a weak
	 equivalence 
	 \[
	 X_1 \stackrel{\simeq}{\lra} 
	 \Map_{\Sp}(\Delta^1, \C)_{\Kan} \text{,}
	 \]
	 which induces a weak equivalence
	 \[
	 \{ x \} \times_{X_0} X_1 \times_{X_0} \{y\} \stackrel{\simeq}{\lra} 
	 \{ x \} \times_{\C} \Map_{\Sp}(\Delta^1, \C) \times_{\C} \{y\} \text{,}
	 \]
	 where the expression $\{x\}$ denotes the simplicial set $\Delta^0$ with vertex
	 labeled by $x$.
	 As explained in Section \ref{subsec:1-segal}, the simplicial set on the left hand side represents the mapping space
	 $\map_X(x,y)$ of the $\inftyone$-category modelled by $X$. 
	 By \cite[4.2.1.8]{lurie.htt}, the simplicial
	 set on the right hand side represents the corresponding mapping space of the
	 $\inftyone$-category modelled by $\C$. Consequently, any concept in the theory of
	 $\inftyone$-categories which can be expressed in terms of mapping spaces will lead
	 to equivalent concepts in both models. In particular:
		
	 \begin{itemize}
	 \item The homotopy categories
	 associated to $\C$ and $X$ are equivalent. 
		
	 \item We have a theory of homotopy limits and
	 colimits of homotopy coherent diagrams in $X$,
	 and the matching theory of quasi-categorical limits and colimits
	 in $\C$. This includes, in particular, the quasi-categorial
	 concepts of initial and final objects, Cartesian and coCartesian squares
	 etc.

	 \end{itemize}
\end{exa}

\eject

\subsection{Exact \texorpdfstring{$\infty$}{infty}-categories}\label{subsec:exact-quasi}
In this and the following sections we generalize the formalism of 
\S \ref{subsec:waldhausen-1} from ordinary categories to quasi-categories. 
We will follow \cite{lurie.htt} and use the term $\infty$-category for a quasi-category.

We recall some basic definitions. 
An {\em equivalence} in an $\infty$-category $\C$ is a morphism 
(1-simplex) in $\C$ which becomes an isomorphism in $\h\C$.
An $\infty$-category $\C$ is called {\em pointed} if it has a zero object $0$, i.e. an
object (0-simplex) which is both initial and final. 
Consider a square
\[
\xymatrix{
x \ar[r]^g \ar[d]
& y \ar[d]^{f} \\
0 \ar[r] & z 
}
\]
in $\C$ where $0 \in \C$ is a zero object. If the square is Cartesian, then we call $g$ a {\em kernel} of $f$.
If it is coCartesion, we call $f$ a {\em cokernel} of $g$.
In the following definition, we use the notion of a subcategory $\C'$ of an $\infty$-category $\C$
as defined in \cite[1.2.11]{lurie.htt}. Namely $\C'$ is part of a Cartesian square
of simplicial sets
\[
\xymatrix{
\C' \ar[r] \ar[d] & \C \ar[d]\\
\N(\h \C') \ar[r] & \N(\h \C),
}
\]
where $\h \C' \subset \h \C$ is a subcategory.

\begin{defi}
\label{def:exact-infinity}
	Let $\C$ be a pointed $\infty$-category. A pair $(\M,\E)$ of subcategories of
	$\C$ is called an {\em exact structure on $\C$}, if the following
	conditions hold:
	\begin{enumerate}[label=(E\arabic{*})]
	 \item \label{e1} The subcategories $\M$ and $\E$ contain all equivalences of $\C$. In
	 particular, $\M$ and $\E$ contain all objects of $\C$.
	 \item \label{e2}
	 \begin{enumerate}[label=(\roman*),labelindent=0cm]
	 \item Morphisms in $\M$ admit pushouts along arbitrary morphisms in $\C$ and $\M$ is stable under
	 pushouts. 	
	 \item Morphisms in $\E$ admit pullbacks along arbitrary morphisms in $\C$ and $\E$ is stable under
	 pullbacks. 	
	 \end{enumerate}
	 \item \label{e3} For any square of the form
	 \[
	 \xymatrix{
	 x \ar[r]^g \ar[d]
	 & y \ar[d]^{f} \\
	 0 \ar[r] & z 
	 }
	 \]
	 in $\C$, we have
	 \begin{enumerate}[label=(\roman*),labelindent=0cm]
	 \item If $g \in \M_1$ and the square is coCartesian, then $f \in \E_1$ and the
	 square is Cartesian.
	 \item If $f \in \E_1$ and the square is Cartesian, then $g \in \M_1$ and the
	 square is coCartesian.
	 \end{enumerate}
	\end{enumerate}
	A triple $(\C,\M,\E)$ satisfying these conditions is called an {\em exact
	$\infty$-category}. We often leave the choice of subcategories implicit, referring to $\C$
	as an exact $\infty$-category.
\end{defi}

\begin{rem} Let $\C$ be an exact $\infty$-category. Then for any object $x$ of $\C$, all morphisms $0
	\to x$ are contained in $\M$, and all morphisms $x \to 0$ are contained in $\E$. This
	follows since we can obtain these morphisms as kernel and cokernel of the identity morphism $x \to
	x$.
\end{rem}

\begin{ex} 
	Let $(\Ec, \Men, \Een)$ be a proto-exact category (Definition \ref{def:proto-exact}).
	Passing to nerves, we obtain an exact $\infty$-category $(\N(\Ec), \N(\Men), \N(\Een))$, as
	in this case the $\infty$-categorical concepts of (co)Cartesian squares reduce to the
	ordinary categorical ones. 
\end{ex}

\begin{ex}\label{ex:stable-infty-exact}
  	Let $\C$ be a stable $\infty$-category (\cite[\S 1.1]{lurie.algebra}). Then $(\C,\C,\C)$ forms an exact $\infty$-category. For
	example, the derived categories of abelian categories (\cite[\S 1.3]{lurie.algebra}) and the
	$\infty$-category of spectra (\cite[1.4.3]{lurie.algebra}) can be considered as exact
	$\infty$-categories in this way.
\end{ex}

\begin{prop}\label{prop:exact} Let $\C$ be an exact $\infty$-category. Consider a square
	\begin{equation}\label{p.origsquare}
	\xymatrix{
	a \ar[r]^g \ar[d]_f
	& b \ar[d]^{f'} \\
	c \ar[r]^{g'} & d.
	}
\end{equation}
Then the following hold:
\begin{enumerate}[label=(\alph*)]
	\item \label{prop:exact1} Assume that $g \in \M_1$, $f \in \E_1$ and the square is coCartesian. Then $g' \in \M_1$,
	 $f' \in \E_1$ and the square is Cartesian.
	\item \label{prop:exact2} Assume that $g' \in \M_1$, $f' \in \E_1$ and the square is Cartesian. Then $g \in \M_1$,
	 $f \in \E_1$ and the square is coCartesian.
\end{enumerate}
\end{prop}
\begin{proof} We provide a proof of \ref{prop:exact2}. Since $f' \in \E_1$, we have $f \in \E_1$
	by \ref{e2}. By \ref{e2}, there exists a Cartesian square of the form
	\begin{equation}\label{p.leftsquare}
	\xymatrix{
	k \ar[r]^i \ar[d]
	& a \ar[d]^{f} \\
	0 \ar[r] & c.
	}
	\end{equation}
	By \ref{e3}, we have $i \in \M_1$ and the square
	\eqref{p.leftsquare} is coCartesian. We obtain a square
	\begin{equation}\label{p.bigsquare}
	 \xymatrix{
	 k \ar[r]^{g \circ i} \ar[d] & b \ar[d]^{f'}\\
	 0 \ar[r] & d.
	 }
	\end{equation}
	which is Cartesian by the dual statement of \cite[Proposition 4.4.2.1]{lurie.htt}.
	Here, $g \circ i$ is a chosen composition of $g$ and $i$
	(which is unique up to homotopy). Therefore, by \ref{e3}, we have $g \circ i \in \M_1$ and the
	square is coCartesian. Since the squares \eqref{p.leftsquare} and \eqref{p.bigsquare} are coCartesian,
	we conclude that the square \eqref{p.origsquare} is coCartesian as well by \cite[Proposition
	4.4.2.1]{lurie.htt}. It remains to show that $g \in \M_1$. This follows from considering the
	square which is obtained by pasting \eqref{p.origsquare} with a coCartesian square
	\begin{equation*}
	 \xymatrix{
	 c \ar[r]^{g'} \ar[d] & d \ar[d]\\
	 0 \ar[r] & e,
	 }
	\end{equation*}
	which exists by \ref{e2}.
\end{proof}

\eject
\subsection{The Waldhausen S-construction of an exact \texorpdfstring{$\infty$}{infty}-category}
\label{subsec:waldhausen-exact-infty}

We recall our notation for the category $T_n = \Fun([1],[n])$
from \S \ref{subsec:waldhausen-1}.
Thus, objects of $T_n$ can be identified with pairs of integers $(i,j)$ satisfying $0 \le i \le j \le n$.  
Given an $\infty$-category $\C$ we will consider homotopy coherent $T_n$-diagrams in $\C$
by which we mean, following Example \ref{comp.diag}, morphisms of simplicial sets $\N(T_n)\to\C$. 

For $0\leq a\leq b\leq n$, we consider the interval $[a,b]\subset [n]$ as 
a poset and therefore as a category. For each $a \in [n]$ we have the
embeddings
\[
	h_a: [a,n]\hookrightarrow T_n, \; j\longmapsto (a,j), \quad v_a: [0,a]\hookrightarrow T_n,\; 
	i\longmapsto (i,a),
\]
which we call the {\em horizontal} and {\em vertical} embeddings corresponding to $a$. 
Given a $T_n$-diagram $F:\N(T_n) \to \C$, the induced $[a,n]$-diagrams $F\circ h_a$
will be called the {\em rows} of $F$, while the $[0,a]$-diagrams $F\circ v_a$
will be called the {\em columns} of $F$. 

\begin{defi} \label{defi:exactinftywald} Let $(\C,\M,\E)$ be an exact $\infty$-category.
	We define 
	\[
	\SW_n \C \subset \Map_\Sp(\N(T_n), \C)_{\Kan}
	\]
 	to be the simplicial subset given by
	those simplices whose vertices are $T_n$-diagrams $F$ satisfying the following conditions:
	\begin{enumerate}[label=(WS\arabic{*})]
	\item \label{exact.wc1} For all $0 \le i \le n$, the object $F(i,i)$ is a zero object in $\C$.
	\item \label{exact.wc2} All rows of $F$ take values in $\M \subset \C$, all 
	columns of $F$ take values in $\E \subset \C$.
	\item \label{exact.wc3} For any $0 \le j \le k \le n$, the square 
	 \[
	 \xymatrix{
	 F(0,j) \ar[r] \ar[d] & F(0,k) \ar[d]\\
	 F(j,j) \ar[r] & F(j,k)
	 }
	 \] 
	in $\C$ is coCartesian.
\end{enumerate}
By construction, $\SW_n \C$ is functorial in $[n]$ and defines a simplicial
space $\SW \C$, which we call the {\em Waldhausen S-construction} or {\em Waldhausen space} of $\C$.
\end{defi}

\begin{rem} By \cite[Proposition 4.4.2.1]{lurie.htt}, condition \ref{exact.wc3} implies that, for
	any commutative square
	 \[
	 \xymatrix{
	 (i,j) \ar[r]\ar[d] & (i,l) \ar[d]\\
	 (k,j) \ar[r] & (k,l)
	 }
	 \]
	 in $T_n$, the corresponding square 
	 \[
	 \xymatrix{
	 F(i,j) \ar[r]\ar[d] & F(i,l) \ar[d]\\
	 F(k,j) \ar[r] & F(k,l)
	 }
	 \]
	 in $\C$ is coCartesian. Further, by \ref{exact.wc2} and Proposition \ref{prop:exact}, this square
	 is in fact biCartesian.
\end{rem}

\begin{thm}\label{thm:waldhausen-infty}
	Let $(\C, \M, \E)$ be an exact $\infty$-category. Then the Waldhausen S-construction 
	$\SW \C$ of $\C$ is a unital $2$-Segal space.
\end{thm}
\begin{proof}
	We first show that $\SW\C$ is a $2$-Segal space.  To this end, we use the Joyal-Tierney
	equivalence $t^!$ to introduce the complete $1$-Segal spaces $X$, $M$ and $E$, corresponding
	to the quasi-categories $\C$, $\M$ and $\E$, respectively.  For $n \ge 1$, consider the
	shifted embeddings
	\[
		 f_h:\; [n-1] \hookrightarrow T_n,\; j \longmapsto (0,j+1), \quad
		 f_v:\; [n-1] \hookrightarrow T_n,\; i \longmapsto (i,n) \text{,} 
	\]
	which, after passing to nerves, induce the pullback maps
	\begin{align*}
		f_h^*:\; & \Fun(\N(T_n), \C)_{\Kan} \lra \Fun(\Delta^{n-1}, \C)_{\Kan} \simeq X_{n-1}, \\ 
		f_v^*:\; & \Fun(\N(T_n), \C)_{\Kan} \lra \Fun(\Delta^{n-1}, \C)_{\Kan} \simeq X_{n-1}\text{,}
	\end{align*}
	where the weak equivalence
	\[
		\Fun(\Delta^{n-1}, \C)_{\Kan} \simeq X_{n-1} 
	\]
	is explained in Example \ref{comp.diag}.
	By \ref{exact.wc2}, the functor $f_h^*$ takes values in $M_{n-1}$, while the functor $f_v^*$
	takes values in $E_{n-1}$. In fact, we obtain maps of simplicial spaces
	\[
		\PI(\SW \C)\buildrel\simeq\over\lra M, \quad \PF(\SW \C)\buildrel\simeq\over\lra E
	\]
	which, by Proposition \ref{prop.inftyequiv} below, are weak equivalences.
	Since both $M$ and $E$ are $1$-Segal spaces, the Waldhausen space $\SW \C$ is a $2$-Segal
	space by the path space criterion (Theorem \ref{thm.crit}).

	It remains to show that $\SW\C$ is unital.
	Given $n \ge 2$ and $0 \le i \le n-1$, we have to show that the square
	\begin{equation}\label{eq:unitalwald}
		\xymatrix{
			\SW_{n-1}\C \ar[r] \ar[d] & \ar[d] \SW_{\{i\}}\C\\
			\SW_{n}\C \ar[r] & \SW_{\{i, i+1\}}\C
		}
	\end{equation}
	is homotopy Cartesian. We assume $i >0$. The restriction map $\rho: \Fun(\N(T_n), \C)_{\on{Kan}} \to \Fun(\N(T_{\{i,i+1\}}),
	\C)_{\on{Kan}}$ induced by $\{i,i+1\} \to [n]$ can be identified with the map 	
	\[
		\rho: \Map^{\sharp}(\N(T_n)^{\flat}, \C^{\natural}) \to \Map^{\sharp}(\N(T_{\{i,i+1\}})^{\flat}, \C^{\natural})
	\]
	of simplicial mapping spaces of marked simplicial sets. Hence, by \cite[3.1.3.6]{lurie.htt},
	the map $\rho$ is a Kan fibration. Since the conditions \ref{exact.wc1}, \ref{exact.wc2} and
	\ref{exact.wc3} are stable under equivalences, it follows that the map $\SW_{n}\C \to
	\SW_{\{i,i+1\}}\C$, which is obtained by restricting $\rho$, is a Kan fibration as well.
	Therefore, the statement that the square \eqref{eq:unitalwald} is homotopy Cartesian is
	equivalent to the assertion that the map
	\begin{equation}\label{eq:cartwald}
		\SW_{n-1}\C \lra \SW_{n}\C \times_{\SW_{\{i, i+1\}}\C} \SW_{\{i\}}\C
	\end{equation}
	is a weak equivalence, where the right-hand side is an ordinary fiber product in the
	category $\sSet$. Analyzing the equivalence 
	\[
		\SW_n \C \simeq \Fun(\Delta^{n-1}, \M)_{\on{Kan}}
	\] 
	of the proof of Proposition \ref{prop.inftyequiv} below, we observe that the subspace 
	\[
		\SW_{n}\C \times_{\SW_{\{i,i+1\}}\C} \SW_{\{i\}}\C \subset \SW_n \C
	\]
	gets identified with the full simplicial subset $K \subset \Fun(\Delta^{n-1},
	\M)_{\on{Kan}}$ spanned by those functors $f$ such that the edge $f(\{i\}) \to f(\{i+1\})$
	in $\M$ is an equivalence. Using Proposition \ref{prop.inftyequiv}, the assertion that the
	map \eqref{eq:cartwald} is a weak equivalence is therefore equivalent to the assertion that
	the $i$th degeneracy map induces a weak equivalence
	\[
		\Fun(\Delta^{n-2},\M)_{\on{Kan}} \overset{\simeq}{\lra} K \subset
		\Fun(\Delta^{n-1},\M)_{\on{Kan}}.
	\]
	Using that $M$ is a $1$-Segal space, we reduce to the statement that
	the degeneracy map
	\[
		\Fun(\Delta^{0},\M)_{\on{Kan}} \lra \Fun(\Delta^1,\M)_{\on{Kan}}
	\]
	induces a weak equivalence onto the full simplicial subset of $\Fun(\Delta^1,
	\M)_{\on{Kan}}$ spanned by the equivalences in $\M$. But this follows from the completeness
	of the $1$-Segal space $M$. The case $i = 0$, follows from a similar argument involving the
	complete $1$-Segal space $E$ instead of $M$.
\end{proof}

\begin{prop}\label{prop.inftyequiv} Let $(\C,\M,\E)$ be an exact $\infty$-category. Let $M = t^!\M$ and
	$E = t^! \E$ denote the complete $1$-Segal spaces corresponding to $\M$ and $\E$. Then
	\begin{enumerate}[label=(\arabic{*})]
	 \item \label{prop.inftyequiv.1} For each $n \ge 1$, the restriction of the functor
	 $f_h^*$ to $\SW_n \C$ induces a weak equivalence
	 \[
	 \SW_n \C \stackrel{\simeq}{\lra} M_{n-1}\text{.}
	 \]
	 \item \label{prop.inftyequiv.2} For each $n \ge 1$, the restriction of the functor
	 $f_v^*$ to $\SW_n \C$ induces a weak equivalence
	 \[
	 \SW_n \C \stackrel{\simeq}{\lra} E_{n-1}\text{.}
	 \]
	\end{enumerate}
\end{prop}
\begin{proof}
	The proof of \ref{prop.inftyequiv.1} is essentially the argument of \cite[Lemma
	1.2.2.4]{lurie.algebra}. 
	We decompose the functor $f_h:\; [n-1] \to T_n$ into 
	\[
	U_1 \stackrel{f_1}{\lra} U_2 \stackrel{f_2}{\lra} U_3 \stackrel{f_3}{\lra} T_n
	\]
	where the functors $f_i$ are inclusions of full subcategories $U_i \subset T_n$ defined as
	follows:
	\begin{itemize}
	 \item $\ob U_1 = \{ (0,j) |\; 1 \le j \le n \}$ 
	 \item $\ob U_2 = \{ (0,j) |\; 0 \le j \le n \}$ 	
	 \item $\ob U_3 = \{ (i,j) |\; 0 \le i \le j \le n \text{ and ($i=0$ or $i=j$)} \}$ 	
	\end{itemize}
	Using the pointwise criterion for $\infty$-categorical
	Kan extensions (see \cite[4.3.2]{lurie.htt}), one easily verifies the following statements.
	\begin{itemize}
	 \item A functor $F: \N(U_2) \to \C$ satisfies the condition that $F(0,0)$ is an initial
	 object (hence zero object) of $\C$ if and only if $F$ is a left Kan
	 extension of $F_{|\N(U_1)}$.
	 \item A functor $F: \N(U_3) \to \C$ satisfies the condition that, for every $1 \le i
	 \le n$, the
	 object $F(i,i)$ is a final object (hence zero object) of $\C$ if and only if $F$ is a right Kan
	 extension of $F_{|\N(U_1)}$.
	 \item A functor $F: \N(T_n) \to \C$ satisfies condition \ref{exact.wc3} if and only if $F$ is a left Kan extension of
	 $F_{|\N(U_2)}$.
	 \item Let $F: \N(T_n) \to \C$ be a functor which satisfies conditions
	 \ref{exact.wc1} and \ref{exact.wc3}. Then $F$ satisfies
	 the conditions \ref{exact.wc2} if and only 
	 if $F_{|\N(U_1)}$ factors through $\M \subset \C$.
	\end{itemize}
	The result now follows from \cite[4.3.2.15]{lurie.htt}.
	The proof of \ref{prop.inftyequiv.2} is obtained by a dual
	argument.
\end{proof}

\begin{rem} In \cite{waldhausen}, Waldhausen defines the S-construction for categories which are
	nowadays called {\em Waldhausen categories}. While part \ref{prop.inftyequiv.1} of Proposition
	\ref{prop.inftyequiv} still holds in this generality, the map of \ref{prop.inftyequiv.2}
	is, in general, {\em not} a weak equivalence. In particular, Theorem \ref{thm:waldhausen-infty}
	does not hold for general Waldhausen categories.
\end{rem}

\vfill\eject

\subsection{Application: Derived Waldhausen stacks}
\label{subsec:der-wald-stacks}

In this section, we use Theorem \ref{thm:waldhausen-infty} to construct $2$-Segal simplicial objects
in model categories of algebro-geometric nature. More precisely, these objects are given by certain
derived moduli spaces of objects in dg categories as constructed by To\"en and Vaqui\'e
\cite{toen-vaquie}. We recall their formalism from a point of view convenient for us. 

Let $\FF$ be a field. We denote by $C(\FF)$, resp. $C^{\leq 0}(\FF)$,  the category of all,
resp. non-positively graded,
 cochain complexes of vector
spaces over $\FF$, equipped with the usual symmetric monoidal structure (tensor product
of complexes).  A morphism $f: M \to N$ of complexes is called a {\em quasi-isomorphism} if, for every
$i \in \ZZ$, the induced map $H^i(M) \to H^i(N)$ is an isomorphism of vector spaces. 
A {\em $\FF$-linear differential graded category}, often abbreviated to dg category, is defined to
be a $C(\FF)$-enriched category. 
Let $\A$ be a dg category. For objects $a,a'$, we denote by 
$\A(a,a')$ the cochain
complex of maps between $a$ and $a'$ given by the enriched $\Hom$-object in $C(\FF)$.
As in any enriched category, the spaces 
\[
\Hom_{\A}(a,a')  = \Hom_{C(\FF)}(\FF, \A(a,a'))  
\]
 define a usual $\FF$-linear category with the same objects as $\A$,
which we call the
  {\em underlying ordinary category} of $\A$  and denote  by   $\underline \A$. 
  Explicitly,  $\Hom_{\A}(a,a') $ consists of 
  0-cocycles in the complex $\A(a,a')$.
  
Further, we define the {\em homotopy category} of $\A$, denoted by $H^0 \A$, to be the $\FF$-linear
category with the same objects as $\A$ and morphisms given by
\[
\Hom_{H^0 \A}(a, a') = H^0 \A(a,a').
\]
Thus we have a functor $\underline \A\to H^0\A$, identical on objects. A morphism  in $\underline\A$
is called a {\em homotopy equivalence}, if it is taken into an isomorphism in $H^0\A$.
We denote by $\Hen=\Hen_\A$ the class of homotopy equivalences in $\A$.

A {\em dg functor} $F: \A \to \B$ between dg categories is defined to be a $C(\FF)$-enriched functor.
We denote by $\dgcat$ the category given by small dg categories with dg functors as morphisms.
A dg functor $F: \A \to \B$ is called a {\em quasi-equivalence} if
\begin{enumerate}
	\item the induced functor of homotopy categories $H^0 \A \to H^0 \B$ is an equivalence of ordinary categories, and
	\item for every pair $a,a'$ of objects in $\A$, the induced map 
	\[
		\A(a,a') \lra \B(F(a),F(a')) 
	\]
	is a quasi-isomorphism of complexes.
\end{enumerate}
 Recall  \cite{tabuada} that 
$\dgcat$   carries a combinatorial model structure in which
weak equivalences are quasi-equivalences and fibrations
 are objectwise surjective dg transformations
of dg functors.  Note that $\dgcat$ contains the category $\dgalg$ of associative
dg algebras over $\FF$, understood as dg categories with one object.

In order to apply Theorem \ref{thm:waldhausen-infty}, it will be useful for us to understand dg categories
from the $\infty$-categorical point of view. To this end, we use the following
construction introduced in \cite[\S 1.3.1]{lurie.algebra}, to which we refer the reader for details.
We associate to the $n$-simplex $\Delta^n$ a dg category $\dg(\Delta^n)$ with objects given by
the set $[n]$. The graded $\FF$-linear category underlying $\dg(\Delta^n)$ is freely generated by morphisms
\[
	f_I \in \dg(\Delta^n)(i_-, i_+)^{-m}
\]
where $I$ runs over the subsets $\{i_- < i_m < i_{m-1} < \dots < i_1 < i_+ \} \subset [n]$,
$m \ge 0$. The differential  is given  on generators by the formula
\[
	d f_I = \sum_{1 \le j \le m} (-1)^j 
	(f_{I \setminus \{i_j\}}  - 
	f_{\{i_j < \dots < i_m < i_+\}} \circ f_{\{i_- < i_1 < \dots < i_j\}})
\]
 and extended by the $\FF$-linear Leibniz rule to compositions of the generators.
 One verifies that $d^2 = 0$ on generators and therefore on all morphisms. Further, it is
straightforward to see that the dg categories $\dg(\Delta^n)$, $n \ge 0$, assemble to form a
cosimplicial object in $\dgcat$.

\begin{defi} \label{defi:dgnerve} We define the dg nerve $\Ndg(\A)$ of a small dg category $\A$ to be
	the simplicial set with $n$-simplices given by
	\[
		\Ndg(\A)_n = \Hom_{\dgcat}(\dg(\Delta^n), \A),
	\]
	and simplicial maps obtained from the cosimplicial structure of $\dg(\Delta^{\bullet})$.
\end{defi}

Thus, vertices of $\Ndg(\A)$ are given by the objects of $\A$ and edges 
by morphisms of $\underline\A$. 
 A triangle in $\Ndg(\A)$ is given by objects $a,a',a''$ of $\A$,  morphisms $f_1: a \to
a'$, $f_2: a' \to a''$, $g: a \to a''$ in $\underline\A$, and a homotopy
$h \in {\A}(a,a'')^{-1}$ such that 
\[
	d h = g - f_2 \circ f_1.
\]
It is shown in \cite[1.3.1.10]{lurie.algebra} that $\Ndg(\A)$ is in fact an $\infty$-category.
Let us list some consequences of this fact. 

\begin{cor}\label{cor:pi-dgnerve} 
\begin{enumerate}[label=(\alph{*})]
	\item We have a natural equivalence of homotopy categories $H^0 \A \simeq \h \Ndg(\A)$.
	
	\item The simplicial subset $\Ndg(\A)_{\Kan}\subset\Ndg(\A)$ is given by those simplices
	  for which all edges lie in $\Hen_\A$.

	\item
	For every object $a$ of $\A$, we have  identifications
	\[
		\pi_1(|\Ndg(\A)_{\Kan}|, a) \cong \Aut_{H^0\A}(a), \quad 
	 \pi_i(|\Ndg(\A)_{\Kan}|,a) \cong H^{1-i}(\A(a,a)), \,\,\, i\geq 2.
	\]
	
\end{enumerate}
\end{cor}
\begin{proof} 
	Part (a) is obvious, and (b) follows from the general properties of $\infty$-categories:
	the Kan subcomplex is given by those edges which become isomorphisms in the homotopy
	category. Let us prove (c) for $i\geq 2$ (the case $i=1$ follows from (a)). As with any Kan
	complex, $\pi_i(|\Ndg(\A)_{\Kan}|,a)$ can be identified with the set of of {\em elementary
	$i$-spheres}  ($i$-simplices  with all faces given by degenerations of the $0$-simplex $a$)
	modulo the relation of {\em elementary homotopy} (see, e.g.,\cite[\S 8.3]{weibel} for
	details).  By the very definition of $\Ndg(\A)$, an elementary $i$-sphere corresponds to an
	element $f \in \A(a,a)^{1-i}$ such that $df = 0$. Further, the elementary homotopy
	relation translates into the cohomology relation: $f-f'=d(h)$ where $f,f' \in
	\A(a,a')^{1-i}$ and $h\in \A(a,a')^{-i}$. 
\end{proof}

\begin{rem} The cosimplicial dg category $\dg(\Delta^{\bullet})$ from Definition \ref{defi:dgnerve}
  is a Reedy-cofibrant replacement (with respect to Tabuada's model structure) of the cosimplicial
  dg category formed by the $\kk$-linear envelopes of the categories $[n]$, $n \ge 0$.  This
  parallels the construction of the simplicial nerve of a simplicial category
  (\cite[1.1.5.5]{lurie.htt}) where the cosimplicial simplicial category ${\mathfrak
  C}[\Delta^\bullet]$ is a Reedy-cofibrant replacement of the cosimplicial object given by the
  discrete categories $[n]$, $n \ge 0$.
\end{rem}

For any simplicial set $K$ we define a dg category $\dg(K)$ by amalgamation:
\[
\dg(K) = \ind^{\dgcat}_{\{\sigma: \Delta^p\to K\}} \dg(\Delta^p).
\]
This gives an  adjunction
	\be\label{eq:dg-nerve-adjunction}
		\dg: \sSet \longleftrightarrow \dgcat : \Ndg
	\ee
	which, as is shown in \cite[1.3.1.20]{lurie.algebra}, is in fact a Quillen adjunction
	with respect to Tabuada's model structure on $\dgcat$ and Joyal's model structure on
	$\sSet$. In particular, since any object in $\dgcat$ is fibrant, the functor $\Ndg$ maps
	quasi-equivalences of dg categories to equivalences of $\infty$-categories.

For dg categories $\A, \B$, we define their {\em tensor product} $\A
\otimes \B$ to be the dg category with 
\begin{align*}
	& \Ob(\A\otimes B) =\Ob(\A)\times\Ob(\B),\\ 
	&(\A\otimes\B)\bigl( (a,b), (a',b')\bigr) 
	=\A(a,a') \otimes_\FF \B(b, b'). 
\end{align*}
Let $\A$ be a small dg category. The category $\Mod_\A$ of dg functors $\A^{\op} \to C(\FF)$, which we
also call {\em $\A^{\op}$-modules}, has a natural $C(\FF)$-enrichment and can hence itself be considered as
a dg category. Note that $C(\FF)$ itself is recovered as $\Mod_\FF$ where $\FF$ is considered
as  the final (one object) dg category. 

\begin{exa}
	Let $\A$ be a dg category. Given an object $a$ of $\A$, we define the $\A^{\op}$-module
	\[
		\underline{h}_a: \A^{\op} \to C(\FF),\; a' \mapsto \A(a',a).
	\]
	This construction can be promoted to a dg functor
	\[
		\underline{h}: \A \lra \Mod_{\A}, \; a \mapsto \underline{h}_a,
	\]
	called the {\em $C(\FF)$-enriched Yoneda embedding}. The dg functor $\underline{h}$ is fully faithful in the
	$C(\FF)$-enriched sense. Those $\A^{\op}$-modules which lie in the essential image of the
	induced functor $H^0 \underline{h}: H^0 \A \to H^0 \Mod_{\A}$ are called {\em
	quasi-representable}.
\end{exa}

\begin{exa}
	Let $\A$ be a dg category. We have the {\em diagonal module} ${\A}^\delta$ in $\Mod_{\A \otimes
	\A^{\op}}$ given by
	\[
		\A^\delta: \A^{\op} \otimes \A \lra C(\FF), \; (a,a') \mapsto \A(a,a'), 
	\]
	This module is important in the derived Morita theory of dg categories \cite{toen.morita}
	where it represents the identity functor on $\A$.
\end{exa}

Recall \cite[\S 2.2]{hinich} that $\Mod_\A$ carries the {\em projective} model structure in which
fibrations are pointwise surjective morphisms and weak equivalences are pointwise
quasi-isomorphisms. This model structure is compatible with the projective model structure on $C(\FF)
\cong \Mod_{\FF}$, making $\Mod_{\A}$ a $C(\FF)$-enriched model category (\S \ref{subsec:enriched}).
Further, the model category $\Mod_\A$ is stable and therefore the homotopy category $\Ho(\Mod_\A)$ is
triangulated. 

Let $\Mod_\A^\circ \subset \Mod_\A$ denote the full dg subcategory spanned by objects which are
cofibrant (all objects of $\Mod_\A$ are fibrant). We have a natural equivalence of categories
\[
	H^0(\Mod_{\A}^\circ) \simeq \Ho(\Mod_{\A}).
\]
Let $\Perf_\A$ be the full subcategory in $\Mod_\A$, whose objects are {\em perfect
$\A^{\op}$-modules}, i.e., objects which are {\em homotopically finitely presented} in the model
category $\Mod_\A$, see \cite[\S 2.1]{toen-vaquie}. We denote by $\Perf_{\A}^\circ \subset
\Perf_{\A}$ the full subcategory spanned by those objects which are cofibrant in $\Mod_\A$.
The categories $\Perf_{\A}$ and $\Perf_{\A}^{\circ}$ inherit $C(\k)$-enrichments from $\Mod_\A$ and can hence be
considered as dg categories. 
Note that $H^0(\Perf^\circ_\A) \simeq  \Ho(\Perf_\A)$ is a triangulated subcategory in 
 $\Ho(\Mod_\A)$. 

\begin{exa}  
	 A complex $M$ in $C(\FF)=\Mod_{\FF}$ is
	perfect, if and only if the total cohomology $H^\bullet(M)$ is a finite dimensional
	$\FF$-vector space.
\end{exa}

\begin{rem} Let $\P$ be one of the dg-categories $\Mod^\circ_\A$, $\Perf^\circ_\A$. 
It can be shown by arguments similar to \cite[1.3.2]{lurie.algebra} that the
	dg nerve $\Ndg(\P)$ of $\P$ is a stable $\infty$-category. As explained in
	\cite[1.1.2]{lurie.algebra}, the homotopy category of any stable $\infty$-category carries a
	natural triangulated structure. This gives an alternative construction of the triangulated
	structure on $H^0 \P$ via the identification
	\[
		\h \Ndg(\P) \simeq H^0 \P
	\]
	from Corollary \ref{cor:pi-dgnerve}. Therefore, we can say that the dg nerve $\Ndg(\P)$
	provides an {$\infty$-categorical enhancement} of the triangulated category $H^0 \P$. 
	It seems likely that a similar comparison holds for any
	 pre-triangulated dg category $\P$ in the sense of \cite{bondal-kapranov}. 
	 \end{rem}

\begin{defi} 
	A dg category $\A$ is called {\em smooth}, if the diagonal $\A^{\op} \otimes
	\A$-module $\A^\delta$ is perfect.  $\A$ is called {\em proper} if, for all
	objects $a,a'$, the mapping complex $\A(a,a')$ is perfect in $\Mod_\FF$, and the
	triangulated category $\Ho(\Mod_\A)$ has a compact generator. 
\end{defi}
 
\begin{rem}\label{rem:NC-smooth-proper}
	As emphasized in \cite{KS-A-infty}, smooth and proper dg categories can be seen as
	noncommutative analogs of smooth and proper varieties over $\FF$. In particular, let $V$ be
	a smooth and proper varierty over $\FF$, and let $\A_V$ the dg category formed by
	finite complexes of injective quasi-coherent $\Oc_V$-modules with coherent
	cohomology sheaves. Then $\A_V$ is smooth and proper, and $H^0\Ac_V$
	is equivalent to $D^b(\Coh(V))$,
	 	derived category  of coherent sheaves on $V$. 
	 \end{rem}
 
Let $\Uc$ be the model site of simplicial commutative $\FF$-algebras and $\mC=D^-\Ac
ff_\FF^{\sim,\text{\'et}}$ be the model category of derived stacks over $\FF$, obtained by
localizing the model category of simplicial presheaves on $\Uc$, see Example
\ref{ex:derived-stacks}. For a simplicial commutative algebra $\Lambda\in\Uc$ we denote by
$\NCC(\Lambda)$ the normalized chain complex of $\Lambda$, which is an associative dg algebra with
grading situated in degrees $\leq 0$.
Here is the main result of \cite{toen-vaquie} (stated in a lesser generality, sufficient for our
purposes). 

\begin{thm}\label{thm:toenvaquie}
Let $\A$ be a smooth and proper dg category. For a simplicial commutative $\FF$-algebra $\Lambda$
define 
\[
	\Mc_\A(\Lambda) = \RMap(\A^\op, \Perf_{\NCC(\Lambda)}). 
\]
Here $\RMap$ is the derived mapping space in the model category $\dgcat$. Then
$\Lambda\mapsto\Mc_\A(\Lambda)$, considered as a simplicial presheaf $\Mc_\A$ on $\Uc$, is a derived
stack, i.e., a fibrant object in $\mC$. This derived stack is locally geometric and locally of
finite presentation.
\end{thm}

We need a general comparison statement. 

\begin{prop} \label{prop:dgnervedklocal}
	Let $\mM$ be a $C(\FF)$-enriched model category, and let $\mM^{\circ}$ be the subcategory of
	fibrant and cofibrant objects. Let $\Ndg(\mM^{\circ})$ be the dg nerve of Definition
	\ref{defi:dgnerve}. There is a weak homotopy equivalence of simplicial sets 
	\begin{equation}\label{eq:hom-mod-space} 
		\Ndg(\mM^\circ)_{\on{Kan}} \simeq \N(\Wen),
	\end{equation} 
	where $\Wen$ denotes the subcategory in $\mM$ formed by weak equivalences. 
\end{prop}

We call the
	homotopy type in \eqref{eq:hom-mod-space} the {\em classifying space of objects} in $\mM$. 
	
\begin{proof} Let $\A$ denote the dg category $\mM^{\circ}$. We will deduce the statement from a
	more general comparison between the mapping spaces of the $\infty$-category $\Ndg(\A)$ with
	the mapping spaces of the Dwyer-Kan simplicial localization $\on{L}_{\Wen}(\M)$ of $\mM$
	along its weak equivalences $\Wen$ (see Example \ref{ex:dwyer-kan-localization}
	and  \cite{dwyerkan-simplicial}). 
	First note that, by \cite[1.3.1.12]{lurie.algebra}, for objects $a,b$ of $\A$, we
	have a weak equivalence
	\begin{equation}\label{eq:luriedg}
		\Hom^R_{\Ndg(\A)}(a,b) \simeq \on{DK}(\tau^{\le 0} \A(a,b))
	\end{equation}
	where the functor $\tau^{\le 0}$ is the cohomological truncation of complexes in
	degrees $\leq 0$ (the    right adjoint to the inclusion $C^{\le 0}(\FF)
	\subset C(\FF)$) and $\on{DK}$ denotes the Dold-Kan correspondence.
	Note that we use  cohomological  grading and so $\on{DK}$ is defined
	on $C^{\le 0}(\FF)$.

	  	On the other hand, using \cite[4.4]{dwyerkan}, we may compute the mapping spaces of the
	simplicial localization in terms of a cosimplicial resolution of the object $a$, i.e., a
	Reedy cofibrant replacement of the constant cosimplicial object $a$ in $\mM^{\Delta}$.  To
	this end, we define
	\[
		a^{\bullet}: \Delta \lra \mM, \; [n] \mapsto \NCC(\FF [\Delta^n]) \otimes a
	\]
	where $\NCC(\FF[\Delta^n])$ denotes the normalized cochain complex of the $\FF$-linear envelope
	of $\Delta^n$, and $\otimes$ denotes the $C(\FF)$-action on $\mM$. The object $a^{\bullet}$ is
	easily verified to define a cosimplicial resolution of $a$, and hence, by
	\cite[4.4]{dwyerkan}, we obtain a weak equivalence of simplicial sets
	\[
		\Map_{\on{L}_{\Wen}(\M)}(a,b) \simeq \Hom_{\A}(a^{\bullet}, b).
	\]
	But by the defining adjunctions of the involved functors, we have, for each $n \ge 0$, an
	isomorphism
	\[
		\Hom_{\A}(\NCC(\FF[\Delta^n]) \otimes a, b) \cong \on{DK}(\tau^{\le 0} \A(a,b)),
	\]
	natural in $[n]$. Therefore, combining with \eqref{eq:luriedg}, we obtain the desired weak
	equivalence
	\begin{equation}\label{eq:mappingspaces}
		\Hom^R_{\Ndg(\mM^{\circ})}(a,b) \simeq \Map_{\on{L}_{\Wen}(\M)}(a,b).
	\end{equation}
	We now deduce \eqref{eq:hom-mod-space} by passing on both sides of \eqref{eq:mappingspaces}
	to the connected components given by equivalences, and applying
	\cite[6.4]{dwyerkan-calculating} and \cite[5.5]{dwyerkan-simplicial}.
\end{proof}

\begin{cor}
	Let $\A$ be any dg category, and $\Hen_\A$ be the class of homotopy equivalences in $\underline\A$.
	Then we have a weak equivalence
	\[
		\Ndg(\A)_{\on{Kan}} \simeq \N(\Hen_\A).
	\]
\end{cor}

In particular, the homotopy groups of $\N(\Hen_\A)$ are found by Corollary  \ref{cor:pi-dgnerve}(c).
For example, given a $C(\FF)$-model category $\mM$ with underlying dg category $\A = \mM^{\circ}$,
we have $\Wen \cap \A = \Hen_\A$, so Corollary \ref{cor:pi-dgnerve}(c) describes the homotopy groups
of $\N(\Wen)$, thus recovering the formulas of \cite{toen-derived}.

 \begin{proof} The Yoneda embedding $\A \to \Mod_{\A}$ provides a $C(\FF)$-fully faithful embedding
	 of $\A$ into a $C(\FF)$-model category. Since further, the images of objects of $\A$ in
	 $\Mod_{\A}$ are cofibrant and fibrant, we can apply Proposition \ref{prop:dgnervedklocal}
	 to obtain the desired weak equivalence.
 \end{proof}
	 
We now proceed to realize the derived stack $\Mc_\A$ from Theorem \ref{thm:toenvaquie} as the first
level of a Waldhausen-type simplicial object in the category $\mC$. Using Proposition
\ref{prop:dgnervedklocal} and the computation of the derived mapping spaces of dg categories in
\cite{toen.morita}, we obtain a weak equivalence
\begin{equation}\label{eq:M-A-Kan}
	\Mc_\A(\Lambda) \simeq \N(\Wen_{\Perf_{\A\otimes\NCC(\Lambda)^\op}}) 
	\simeq \Ndg(\Perf^\circ_{\A\otimes\NCC(\Lambda)^\op})_{\on{Kan}}. 
\end{equation}
Therefore, $\Mc_\A(\Lambda)$ can be identified with the space of $1$-simplices in the Waldhausen
S-construction of the stable $\infty$-category $\Ndg(\Perf^\circ_{\A\otimes\NCC(\Lambda)^\op})$.
Varying the dg algebra $\Lambda$, we obtain, for each $n\geq 0$, a simplicial presheaf $\Sc_n(\underline\Perf_\A)$
on $\Uc$ by defining
\[
	\Sc_n(\underline\Perf_\A)(\Lambda) =\Sc_n( \Ndg(\Perf^\circ_{\A\otimes\NCC(\Lambda)^\op}) )
\]
to be the $n$th component of the Waldhausen space of the stable $\infty$-category
$\Perf^\circ_{\A\otimes\NCC(\Lambda)^\op}$. The simplicial presheaves $\Sc_n(\underline\Perf_\A)$,
$n \ge 0$, assemble to define a simplicial object $\Sc(\underline\Perf_\A)$ which we call the {\em
derived Waldhausen stack of perfect $\A$-modules}. In particular, when $\A=\A_V$  
for a smooth and proper $\FF$-variety $V$,
see Remark \ref{rem:NC-smooth-proper},  then
$\Sc(\underline\Perf_\A)$ can be seen as the derived Waldhausen stack of objects of $D^b(\Coh(V))$. 

\begin{prop}\label{cor:wald-derived-stacks}
Let $\A$ be a smooth and proper dg category. Then:
\begin{enumerate}[label=(\alph{*})]

\item Each $\Sc_n(\underline\Perf_\A)$ is a derived stack, locally geometric
and locally of finite presentation.

\item $\Sc(\underline\Perf_\A)$ is a $2$-Segal object in the model category
$\mC$ of derived stacks. 

\end{enumerate}
\end{prop}

\begin{proof} (a) The case $n=0$ is obvious and, for $n=1$, the statement follows from 
Theorem \ref{thm:toenvaquie} 
and the identification
\[
	\Sc_1(\underline\Perf_\A)(\Lambda) \simeq \Mc_\A(\Lambda)
\]
of \eqref{eq:M-A-Kan}. The case $n>1$, reduces to $n=1$ by  Lemma
\ref{lem:trick}  below because $\dg(\Delta^{n-1})$
is smooth and proper 
and therefore so is $\A\otimes \dg(\Delta^{n-1})$.  This reduction 
is analogous to the trick used in \cite{toen-derived}, Proof of
Lemma 3.2.

(b) This follows from Theorem \ref{thm:waldhausen-infty} applied to each stable 
$\infty$-category $\Perf^\circ_{\A\otimes\NCC(\Lambda)^\op}$. Indeed, 
homotopy limits in $\mC$, appearing in the $2$-Segal conditions
can be calculated object-wise for each object $\Lambda\in\Uc$. 
\end{proof}

\begin{lem}\label{lem:trick}
We have an equivalence of derived stacks
$\Sc_n(\underline\Perf_\A) \simeq \Sc_1(\underline\Perf_{\A\otimes \dg(\Delta^{n-1})})$.
\end{lem}

\begin{proof}
 Proposition \ref{prop.inftyequiv} and the weak equivalence \eqref{eq:hom-mod-space} 
give a weak equivalence of simplicial sets
\[
\Sc_n(\underline\Perf_\A)(\Lambda)   \simeq (\Fun(\Delta^{n-1},
\Ndg(\Perf^\circ_{\A\otimes\NCC(\Lambda)^{\on{op}}}))_{\on{Kan}}.
\]
Here, $\Fun$ stands for the $\infty$-category of functors between two $\infty$-categories given by
the mapping simplicial set between two simplicial sets. Now, it follows from \cite[Proposition
1.20]{joyal-tierney} that for any $\infty$-category $\C$ the simplicial set $\C_{\Kan}$ is weakly
equivalent to the simplicial set $\C'$ defined by
\[
\C'_p = \Hom_\Sp((\Delta^p)', \C)
\]
where $(\Delta^p)'$ is the fat simplex from Example \ref{ex:nerve,fat-simplex}. 
Using the adjunction \eqref{eq:dg-nerve-adjunction}, we deduce that
$\Sc_n(\underline\Perf_\A)(\Lambda) $ is weakly equivalent to the simplicial set
\begin{equation}
	 \label{eq:hom-fat-simplices}
	 \Hom_{\dgcat} \bigl( \dg(\Delta^{n-1}\times(\Delta^\bullet)'), \Ndg(\Perf^\circ_{\A\otimes\NCC(\Lambda)^{\on{op}}})
	 \bigr). 
\end{equation}
Note that the cosimplicial object $\Delta^{n-1} \times (\Delta^\bullet)'$ is a Reedy-cofibrant
replacement of the constant object $\Delta^{n-1}$ with respect to the Joyal model structure on
$\sSet$. Since \eqref{eq:dg-nerve-adjunction} is a Quillen adjunction, the object $\dg(\Delta^{n-1} \times (\Delta^\bullet)')$
is a Reedy-cofibrant replacement of $\dg(\Delta^{n-1})$, and hence a cosimplicial resolution 
in the sense of \cite[\S 4]{dwyerkan}. Therefore, the simplicial set \eqref{eq:hom-fat-simplices} is
weakly equivalent to the simplicial mapping space
$\Map_{\dgcat} (\dg(\Delta^{n-1}), \Perf^\circ_{\A\otimes\NCC(\Lambda)^{\on{op}}})$
defined via the Dwyer-Kan localization of $\dgcat$ with respect to quasi-equivalences of dg
categories.
Let us first analyze the bigger space 
$\Map_{\dgcat} (\dg(\Delta^{n-1}), \Mod^\circ_{\A\otimes\NCC(\Lambda)^{\on{op}}})$
consisting of maps into the category of all dg modules, perfect or not. 
By \cite[Th. 1.1]{toen.morita} this space is equivalent to the nerve
of the category of weak equivalences in the model category of
dg functors between $\dg(\Delta^{n-1})$ and $\Mod^\circ_{\A\otimes\NCC(\Lambda)^{\on{op}}}$.
Such dg functors can be identified with dg modules over 
$\A\otimes\dg(\Delta^{n-1})\otimes\NCC(\Lambda)^{\on{op}}$.
We finally notice that dg functors taking values in the subcategory of perfect modules correspond
correspond to perfect (compact) objects of the category of dg-modules over
$\A\otimes\dg(\Delta^{n-1})\otimes\NCC(\Lambda)^{\on{op}}$. In the terminology of \cite{toen-vaquie}
this corresponds to the statement that, since the dg category $\dg(\Delta^{n-1})$ is smooth and
proper, the notions of pseudo-perfect and perfect modules coincide (\cite[Lemma 2.8]{toen-vaquie}).
\end{proof}

\vfill\eject

\subsection{The cyclic bar construction of an \texorpdfstring{$\infty$}{infty}-category}
\label{subsec:cycnerve}

In this section we define the cyclic bar construction of an $\infty$-category $\C$ and show that it is a
$2$-Segal space. 

Recall the adjunction 
\begin{equation} \label{eq:adjnerv1}
	\FC: \Sp \longleftrightarrow {\mathcal Cat}: \N, 
\end{equation}
from Example \ref{ex:1-segal-envelope-discrete}(a), given by the nerve $\N$ and its left adjoint
$\FC$, which associates to a simplicial set $D$ the free category $\FC(D)$ generated by $D$.
For $n \ge 0$, we define the simplicial set  
\[
K^n := 
\biggl(\Delta^{\{0,1\}} \coprod_{\{1\}} \Delta^{\{1,2\}}\coprod_{\{2\}} \quad \cdots
\quad \coprod_{\{n-1\}} \Delta^{\{n-1,n\}} \biggr)
\coprod_{\{n\} \amalg \{0\}} \Delta^{\{n,0\}}\text{.}
\]
Let $C^n = \FC(K^n)$. The geometric realization of $K^n$ is a closed chain of
$n+1$ oriented intervals and hence homeomorphic to the unit circle $S^1 = \{|z|=1\} \subset \CC$. 

\begin{rem} Let $\mu_{n+1}\subset S^1$ be the set of $(n+1)$st roots of unity. Then the category
$C^n$ can be identified with the subcategory in the fundamental
groupoid $\Pi_1(S^1, \mu_{n+1})$ with the same set of objects $\mu_{n+1}$ and morphisms being homotopy classes
of counterclockwise oriented paths (cf. \cite[\S 2]{drinfeld}).
\end{rem}

The system $(C^n)_{n\geq 0}$
forms a cosimplicial category. The face maps are
given by composition of morphisms, the degeneracies by filling in identity maps. 
By Example \ref{ex:1-segal-envelope-discrete}(b), 
the unit of the adjunction \eqref{eq:adjnerv1} provides us with a canonical map 
$K^n \to \N(C^n)$
exhibiting $\disc{\N(C^n)}$ as a $1$-Segal replacement of $\disc{K^n}$ in $\sS$. Further,
since the category $C^n$ has no nontrivial isomorphisms, the $1$-Segal space $\disc{\N(C^n)}$
is complete.

\begin{defi}\label{defi:icn} 
	\begin{enumerate}[label=(\alph{*})]
\item Let $X$ be a Reedy fibrant $1$-Segal space. We
	define the \emph{cyclic bar construction of $X$} to be the simplicial space 
	\[
	\NC(X): \Dop \lra \Sp,\quad [n] \mapsto \Map(\disc{\N(C^n)}, X). 
	\]
 \item Let $\C$ be a $\infty$-category. We define the {\em cyclic bar construction of $\C$} as
	the cyclic nerve of the complete $1$-Segal space $t^!\C$ and denote it by $\NCI(\C)$. Explicitly.
	this amounts to the formula
	\[
	\NCI (\C): \Dop \lra \Sp, \quad [n] \mapsto \Fun(\N(C^n), \C)_{\Kan}. 
	\]
	\end{enumerate}
\end{defi}

\begin{ex} 
  	Applying the cyclic bar construction of Definition \ref{defi:icn} to the discrete
	$1$-Segal space $X$ given by the discrete nerve of an ordinary category $\C$, we recover the cyclic nerve
	from Section \ref{subsec:twisted-cyclic-nerve}. More precisely, we have $\NC(X) = \disc{\NC(\C)}$. 
	 
	Note, however, that $X$ is, in general, not complete. Therefore, $X$ is equivalent to $t^!\C$ and consequently
	$\NCI(\C)$ and $\NC(\C)$ are different. In fact, $\NCI(\C)$ is expressed
	in terms of the categorified nerve $\C_\bullet$ from Example 
	\ref{ex.1segclass}, which is the complete $1$-Segal space corresponding to $\C$. 
\end{ex}

\begin{thm} Let $X$ be a Reedy fibrant $1$-Segal space. Then the cyclic bar construction $\NCI(X)$
	is a $2$-Segal space.
\end{thm}
\begin{proof}
	Using the path space criterion (Theorem \ref{thm.crit}), it suffices to show that both path
	spaces associated to $\NCI(X)$ are $1$-Segal spaces. We provide a proof for the initial path space
	$Y = \PI \NCI(X)$, the argument for the final path space is analogous.

	Recall, that, for every $n \ge 0$, we have $Y_n = \NCI(X)_{n+1}$ and the face maps are given
	by omitting $\partial_0$. Consider the simplicial set
	\[
	\widetilde{K^n} = \Delta^{\{0,\dots,n+1\}} \coprod_{\Delta^{\{0\}} \coprod \Delta^{\{n+1\}}}
	\Delta^{\{0\}} \text{.}
	\]
	Note that $\widetilde{K^n}$ can be obtained from $K^n$ by attaching the simplex
	$\Delta^{\{0,\dots,n+1\}}$. This simplex can be identified with the simplex in $N(C^n)$
	corresonding to the chain of $n+1$ composable morphisms in $C^n$ given by $\Delta^{\{0,1\}},
	\dots, \Delta^{\{n-1,n\}}, \Delta^{\{n,0\}}$. Thus, the canonical map 
	$K^n \to N(C^n)$
	factors over the inclusion $K^n \hra \widetilde{K^n}$ providing a commutative diagram 
	\[
	\xymatrix{
	\widetilde{K^n} \ar[r]^-{g_n} & \N(C^n)\\
	K^n, \ar[u] \ar[ur]
	}
	\]
	in which, by Example \ref{ex:1-segal-envelope-discrete}(b), 
	the map $g_n$ exhibits $\disc{\N(C^n)}$ as a $1$-Segal replacement of $\disc{\widetilde{K^n}}$.
	Therefore, pulling back along $g_{n+1}$, we obtain a weak equivalence of mapping spaces
	\[
	Y_n = \Map(\disc{\N(C^{n+1})}, X) \stackrel{\simeq}{\lra}
	\Map(\disc{\widetilde{K^{n+1}}},
	X) \cong X_{n+2} \times_{X_0 \times X_0} X_0 \text{.}
	\]
	The pullback maps along $\{g_n\;|\; n \ge 0\}$ assemble to provide a weak equivalence of simplicial spaces
	\[
	g^*: Y \to X_{\bullet+2} \times_{X_0 \times X_0} X_0 \text{.}
	\]
	Here, the simplicial structure on the right-hand side is provided by identifying
	$X_{\bullet+2}$ with the pullback of $X$ along the functor
	\[
	\varphi: \Dop \to \Dop, [n] \mapsto [0] \star [n] \star [0] \text{.}
	\]
	Using the terminology of Section \ref{subsec:simppath}, we have $\varphi^*X = \PI \PF X$. 	
	Further, we have a commutative square
	\begin{equation}\label{eq:compsquare}
	 \xymatrix{ Y_n \ar[d]^{f_n} \ar[r]^-{g_n^*}_-{\cong} & X_{n+2} \times_{X_0 \times X_0} X_0 \ar[d]^{\widetilde{f_n} \times \id} \\
	 Y_1 \times_{Y_0} Y_1 \times_{Y_0} \dots \times_{Y_0} Y_1 \ar[r]_-{\cong} & (X_{3} \times_{X_2} X_{3} \times_{X_2} \dots \times_{X_2} X_{3}) \times_{X_0 \times X_0} X_0, 
	}
	\end{equation}
	where both horizontal maps are isomorphisms, and $f_n$ and $\widetilde{f_n}$ denote the
	$n$th $1$-Segal map associated with the simplicial spaces $Y$ and $\varphi^*X$, respectively.
	Using that $X$ is, by assumption, Reedy fibrant, it follows that all fiber products in \eqref{eq:compsquare} are
	in fact homotopy fiber products. Hence, to show that $f_n$ is a weak equivalence, it
	suffices to show that $\widetilde{f_n}$ is a weak equivalence.
	By Proposition \ref{prop.1segal2segal}, the $1$-Segal space $X$ is a $2$-Segal space. Thus,
	by Theorem \ref{thm.crit}, the simplicial space $\PF X$ is $1$-Segal. Reiterating this
	argument once implies that $\varphi^* X$ is $1$-Segal and hence $\widetilde{f_n}$ is a weak
	equivalence.
\end{proof}

\vfill\eject

\section{Hall algebras associated to 2-Segal spaces}
\label{sec:hall-alg}

In this chapter, we explain how to extract associative algebras from $2$-Segal objects by means of
{\em theories with transfer}. This procedure, applied to Waldhausen spaces, recovers various
variants of Hall algebras, such as classical Hall algebras, derived Hall algebras, and motivic Hall
algebras.  Applying a theory with transfer to other $2$-Segal spaces, we obtain classically known
algebras, such as Hecke algebras, but also new algebras, such as the ones associated to the cyclic
nerve of a category.

\subsection{Theories with transfer and associated Hall algebras}  
We introduce an abstraction of basic functoriality properties of a ``cohomology theory",
motivated by \cite{fulton-macpherson, voevodsky}.  Usually, a cohomology theory has contravariant
functoriality with respect to most maps and a covariant (Gysin, or transfer) functoriality with
respect to some other, typically more restricted, class of maps. We axiomatize this situation as
follows. 

\begin{defi}\label{def:transfer-structure}
Let $\mC$ be a model category. A {\em transfer structure} on $\mC$ is a datum of two classes of
morphisms $\Sc, \Pc\subset\Mor(\mC)$, called {\em smooth}
and {\em proper} morphisms, respectively, which satisfy the following conditions: 
\begin{enumerate}[label=(TS\arabic{*})]
  \item \label{ts1} The classes $\Sc, \Pc$ are closed under composition.

\item \label{ts2} Let
	\[
		\xymatrix{X \ar[r]^{s} \ar[d]_{p} &Y\ar[d]^{q}\\
		X' \ar[r]^{s'}&Y'
		}
	\]
	be a homotopy Cartesian square in $\mC$ with $q\in\Pc$ and $s'\in \Sc$. Then
	$p\in\Pc$ and $s\in\Sc$. 
\end{enumerate}
Note that, in view of \ref{ts1}, we can identify the classes $\Sc$ and $\Pc$ with subcategories of $\mC$ containing all
objects.
\end{defi}

Let $(\Vc, \otimes, \1)$ be a monoidal category. For convenience, we will use the term {\em
``associative algebra in $\Vc$"} to signify a semigroup object in $\Vc$, i.e., an object $A$
together with a morphism $\mu: A\otimes A\to A$ satisfying associativity.  A {\em unital associative
algebra} is a monoid object, i.e., $A$ as above together with a morphism $e: \1\to A$ satisfying the
unit axioms with respect to $\mu$. Given an associative algebra $A$ in $\Vc$, there is a natural
notion of (left, right, and bi) {\em $A$-modules} in $\Vc$ . 
By a \emph{lax monoidal functor} $F: (\Wc, \boxtimes, \1_\Wc) \to 
(\Vc, \otimes, \1_\Vc)$ between two monoidal categories
we mean a functor $F: \Wc\to\Vc$, equipped with
\begin{itemize}
	\item a morphism $\1_{\Vc} \to F(\1_{\Wc})$,
	\item for any objects $x,y \in \Wc$, a morphism $F(x) \otimes F(y) \to F(x \boxtimes y)$, natural
	 in $x$ and $y$,
\end{itemize}
satisfying the standard associativity and unitality constraints. Here, the adjective {\em lax}
means that these morphisms are not required to be isomorphisms. 

Note that a lax monoidal functor $F$ transfers algebra structures: if $A$ is an
associative algebra in $\Wc$, then $F(A)$ is an associative algebra in $\Vc$ which will be unital,
if $A$ is unital. 

\begin{defi}\label{def:TT}
Let $\mC$ be a combinatorial model category, and $(\Vc,\otimes, \1)$ a monoidal category.
A $\Vc$-valued {\em theory with transfer} on $\mC$ is a datum $\hen$ consisting of: 

\begin{enumerate}[label=(TT\arabic{*})]
	\item \label{tt1} A transfer structure $(\Sc, \Pc)$ on $\mC$. 

	\item \label{tt2} A covariant functor $\Pc\to\Vc$ and a contravariant functor $\Sc\to\Vc$, coinciding on objects
	and both denoted by $\hen$. The value of $\hen$ on $s: X\to Y$ from $\Sc$ is denoted by $s^*:
	\hen(Y)\to\hen(X)$. The value of $\hen$ on $p: Z\to W$ from $\Pc$ is denoted by $p_*:
	\hen(Z)\to\hen(W)$. Both functors are required to take weak equivalences in $\mC$ to isomorphisms
	in $\Vc$. 

	\item \label{tt3} Multiplicativity data on $\hen$, i.e., morphisms $m_{X,Y}: \hen(X)\otimes \hen(Y)\to
	\hen(X\times Y)$, natural with respect to morphisms in $\Sc$, as well as an isomorphism
	$\hen(\pt) \cong \1$. These morphisms are required to satisfy the usual associativity and unit conditions.
\end{enumerate}
These data are required to satisfy the following base change property: 
\begin{enumerate}[label=(TT\arabic{*}),resume]
	\item \label{tt4} For any homotopy Cartesian square as in \ref{ts2}, we have an equality
		$p_* \circ s^* = (s')^* \circ q_* $ as morphisms from $\hen(Y)$ to $\hen(X')$.  
\end{enumerate}
\end{defi}

\begin{rem}\label{rem-hall-algebra-model} 
	Assume that $\Vc$ is a {\em symmetric} monoidal category and $\hen$ respects the symmetry. 
	If $X\in\mC$ is such that the canonical morphisms $X\to X\times X$, $X\to\pt$ belong to $\Sc$, then
	the object $\hen(X)$ has a structure of a unital commutative algebra in $\Vc$. This structure is
	obtained by applying the contravariant functoriality of $\hen$ to these morphisms. 
\end{rem}
 
\begin{ex}
	The simplest example of a theory with transfer is obtained as follows.  Let $\mC=\Set$ be the
	category of sets with the trivial model structure where weak equivalences are isomorphisms. Let
	$\k$ be a field and $\Vc=\Vect_\k$ the category of $\k$-vector spaces. For a set $S$, let
	$\Fen(S)$ be the space of all functions $\phi: S\to\k$, and $\Fen_0(S)$ the subspace of
	functions with finite support. A map $f:S'\to S$ induces the inverse and direct image maps
		\begin{align*}
		&f^*: \Fen(S)\lra \Fen(S'), & &f_*: \Fen_0(S')\lra \Fen_0(S),\\
		& (f^*\phi)(x')=\phi(f(x')), & &(f_*\psi)(x) =\sum_{x'\in f^{-1}(x)} \psi(x').
		\end{align*}
	We say $f$ is proper if, for any $x\in S$, the fiber $f^{-1}(x)$ is a finite set. Let $\Pc$
	denote the class of all proper maps of sets. For $f\in \Pc$, we have 
	\begin{align*}
		 &f^*: \Fen_0(S)\lra \Fen_0(S') & &f_*:\Fen(S')\lra \Fen(S),
	\end{align*}
	defined as above. The data provided makes $\Fen$ a $\Vect_\k$-valued theory with transfer on
	$\Set$ with respect to the transfer structure $(\Mor(\Set), \Pc)$.  Similarly, $\Fen_0$ is
	theory with transfer with respect to the structure $(\Pc, \Mor(\Set))$. 
\end{ex}

\begin{ex}[(Universal theory with transfer)]\label{ex:universaltransfer}
	A universal example can be obtained in the spirit of Grothendieck's construction of the category
	of motives by using correspondences. 
	Let $\mC$ be a combinatorial model category, and let $(\Sc, \Pc)$ be a transfer structure on
	$\mC$.  Recall (\S \ref{subsec:mult-cat}) the bicategory $\Spanl_\mC$ with the same objects
	as $\mC$, $1$-morphisms being span diagrams 
	\[
		\sigma =\bigl\{ Z\buildrel s\over \lla W\buildrel p \over\lra Z'\bigr\}
	\]
	and composition of $1$-morphisms given by forming the fiber product.  To keep the notation
	straight, we denote by $[Z]$ the object of $\Spanl_\mC$ corresponding to the object
	$Z\in\mC$.  We call $\sigma$ an {\em $(\Sc, \Pc)$-span}, if $s\in\Sc$ and $p\in\Pc$.  Axiom
	\ref{ts2} implies that the class of $(\Sc, \Pc)$-spans is closed under composition and thus
	gives rise to a sub-bicategory $\Spanl_\mC(\Sc, \Pc)$ in $\Spanl_\mC$.  Two $1$-morphisms in
	$\Spanl_\mC(\Sc, \Pc)$ from $[Z]$ to $[Z']$ are equivalent, if there exists a diagram
	 \[
		 \xymatrix{
		 W_1 \ar[rr]\ar[dd]& & Z'\\
		 & \ar[ul]_{\simeq} \ar[dr]^{\simeq} Y & \\
		 Z & & W_2 \ar[ll]\ar[uu]
		 } \text{,} 
	 \]
	 which is commutative in the homotopy category $\on{Ho}(\mC)$ and in which the diagonal maps
	 are weak equivalences.  Let $\h\Spanl_\mC(\Sc, \Pc)$ be the ordinary category with the same
	 objects $[Z]$ as the bicategory $\Spanl_\mC(\Sc, \Pc)$, and morphisms being equivalence
	 classes of 1-morphisms in $\Spanl_\mC(\Sc, \Pc)$. The Cartesian product on $\mC$ makes
	 $\h\Spanl_{\mC}(\Sc, \Pc)$ a monoidal category, by defining $[X] \otimes [Y]
	 := [X\times Y]$. By construction, we have a contravariant functor
	 \[
	   \hen_{\on{un}}: \Sc \lra \Spanl(\Sc, \Pc), 
	   \left\{ \begin{array}{l} X \mapsto [X], \\
	 	X\buildrel s\over\to Y \; \mapsto \; 
		s^* = \bigl\{ Y\buildrel s\over\lla X\buildrel \Id\over\lra X\bigr\}
	   	\end{array} \right.
	 \]
	 Further, if $p \in \Pc$, we have the morphism
	 \[
		 p_* = \bigl\{ X\buildrel\Id\over\lla X\buildrel p\over\lra Y\bigr\}: [X]\to [Y]. 
	 \]
	 The association $X \mapsto \hen_{\on{un}}(X)$, equipped with the contravariant and covariant functoriality 
	 specified above, defines a $\h\Spanl_{\mC}(\Sc, \Pc)$-valued theory with transfer on $\mC$
	 with respect to the transfer structure $(\Sc, \Pc)$. 
\end{ex}
        
\begin{prop} \label{prop:universal} Let $\mC$ be a combinatorial model category with transfer
  	structure $(\Sc,\Pc)$, and let $\Vc$ be a monoidal category. Then
      $\Vc$-valued theories with transfer on $\mC$ with respect to $(\Sc, \Pc)$
      are in bijective correspondence with lax monoidal functors $\h\Spanl_{\mC}(\Sc, \Pc) \to \Vc$. 
\end{prop}

In light of Proposition \ref{prop:universal}, we call the theory $\hen_{\on{un}}$ from Example \ref{ex:universaltransfer} the {\em universal
theory with transfer} associated to the transfer structure $(\Sc, \Pc)$ on the model category $\mC$.

Let $\mC$ be a combinatorial model category with a transfer structure $(\Sc, \Pc)$, and let
$X\in\mC_\Delta$ be a $2$-Segal object. Consider the spans
\begin{align*}
	\mult &=\bigl\{ X_1\times X_1 \buildrel (\partial_2, \partial_0)\over\lla X_2\buildrel
	\partial_1\over\lra X_1\bigr\}, \\
	\epsilon &=\bigl\{\pt\lla X_0\buildrel s_0\over\lra X_1\bigr\},
\end{align*}
where $\partial_i$ and $s_0$ denote face and degeneration maps of $X$.		
We say that $X$ is $(\Sc,\Pc)$-{\em admissible} if $\mult$ is an $(\Sc, \Pc)$-span. We
say that $X$ is $(\Sc,\Pc)$-{\em unital} if $X$ is further a unital $2$-Segal object (Definition
\ref{def:unital-2-segal-top}, Remark \ref{rems:2-segal-top-model-identical}) and
$\epsilon$ is an $(\Sc, \Pc)$-span.

\begin{prop} Let $X$ be an $(\Sc, \Pc)$-admissible $2$-Segal object, and let
	$\Hall(X)=[X_1]$ be the object of $\Spanl(\Sc, \Pc)$ represented by $X_1$.  The morphism $m:
	\Hall(X) \otimes \Hall(X) \to \Hall(X)$, represented by the span $\mult$, makes $\Hall(X)$
	an associative algebra in $\Spanl(\Sc, \Pc)$. If $X$ is $(\Sc,
	\Pc)$-unital, then the morphism $e: \1\to X$, represented by the span $\epsilon$, is a unit
	for $\Hall(X)$. 
\end{prop} 
\begin{proof} Similar argument to that in Theorem \ref{thm:2-segal-mu-cat}, using homotopy Cartesian
	squares instead of ordinary Cartesian squares. We leave the details to the reader.
\end{proof}
 
We call $\Hall(X)$ the {\em universal Hall algebra of $X$} with respect to the transfer structure $(\Sc, \Pc)$.

\begin{defi}\label{def:hall-algebra-admissible}
Let $\mC$ be a model category with a transfer structure $(\Sc, \Pc)$, and let $\hen$
be a $\Vc$-valued theory with transfer on $\mC$. Let $F_\hen: \Spanl(\Sc, \Pc) \to \Vc$ be the
lax monoidal functor from Proposition \ref{prop:universal} that represents $\hen$. For any $(\Sc, \Pc)$-admissible $2$-Segal object
$X \in \mC$, the {\em Hall algebra} of $X$ with coefficients in $\hen$ is defined as
\[
	 \Hall(X, \hen) := F_\hen(\Hall(X)) = \hen(X_1)\in\Vc
\]
with the associative algebra structure transferred from the universal Hall algebra $\Hall(X)$ along
$F_\hen$. 
\end{defi}
 
Note that the Hall algebra $\Hall(X, \hen)$ is unital if the $2$-Segal object $X$ is $(\Sc, \Pc)$-unital. 
Explicitly, the multiplication on $\Hall(X, \hen)$ is obtained as the composite
\[
	  \hen(X_1)\otimes\hen(X_1)\buildrel m_{X_1, X_1}\over \lra \hen(X_1\times X_1)
	  \buildrel (\partial_0, \partial_2)^*\over\lra \hen (X_2)
	  \buildrel (\partial_1)_*\over\lra \hen (X_1).
\]

\begin{rem} When the object $X_0$ is not a final object in $\mC$, then we can refine the
	construction of the universal Hall algebra to give a monad in a certain $(3,2)$-category of
	bispans. We will not make this statement precise here as it will reappear in the context of
	$\inftytwo$-categories in \S \ref{section:higher-bicat}.
\end{rem}

An alternative construction which takes into account $X_0$ is given as follows. Suppose we are in
the situation of Remark \ref{rem-hall-algebra-model}, so that $\hen(X_0)$ is a commutative algebra
in $\Vc$. Suppose also that the boundary morphisms
\[
X_0 \buildrel \partial_0 \over\lla X_1\buildrel \partial_1\over\lra X_0
\]
belong to $\Sc$. In this case they
endow $\Hall (X,\hen)=\hen(X_1)$ with two (commuting) 
structures of an $\hen(X_0)$-module,
i.e., make it into an $(\hen(X_0), \hen(X_0))$-bimodule. 
Thus the left $\hen(X_0)$-action is induced by $\partial_0$,
while the right action is induced by $\partial_1$. 

\begin{prop}
Under the above assumptions,
the multiplication $m$ on $\Hall(X, \hen)$ is $\hen(X_0)$-bilinear, i.e., 
we have a commutative diagram in $\Vc$:
\[
\xymatrix{
\hen(X_1)\otimes\hen(X_0)\otimes\hen(X_1)\ar[r]^{\hskip 7mm \Id\otimes\lambda}
\ar[d]^{\rho\otimes\Id}&
\hen(X_1)\otimes\hen(X_1)\ar[d]^m 
\cr
\hen(X_1)\otimes\hen(X_1)\ar[r]^m &\hen(X_1)
}
\]
Here $\lambda$ and $\rho$ are the left and right action maps of $\hen(X_0)$ on
$\hen(X_1)$.
\end{prop}

\begin{proof} Straightforward, left to the reader. \end{proof}

 \vfill\eject
 
\subsection{Groupoids: Classical Hall and Hecke algebras}\label{subsec:groupoids-classical-hall}

Let $\mC=\Grp$ be the category of small groupoids with the Bousfield model structure from Example
\ref{exa:groupoidsmodel}.  Recall that, for a groupoid $\Gc\in\Grp$, the set of isomorphism classes
of objects in $\Gc$ is denoted $\pi_0(\Gc)$. The concept of a theory with transfer on various
subcategories of $\Grp$ is closely related with that of a global Mackey functor, cf.
\cite{webb}. We start with some examples.
  
Fix a field $\k$, and let $\Vc=\Vect_\k$ be the category of $\k$-vector spaces. For a groupoid $\Gc
\in\Grp$, we denote by $\Fen(\Gc)$ the space of $\k$-valued functions on $\pi_0(\Gc)$. In other
words, $\Fen(\Gc)$ consists of functions $\phi: \Ob(\Gc)\to\k$ such that $\phi(x)=\phi(y)$ whenever
$x$ is isomorphic to $y$. 
A functor $f: \Gc'\to \Gc$ of groupoids defines the pullback map $f^*: \Fen(\Gc)\to\Fen(\Gc')$.
This contravariant functoriality, together with the obvious multiplicativity maps
$\Fen(\Gc)\otimes\Fen(\Gc')\to\Fen(\Gc\times\Gc')$, extends, in various ways, to the structure of a theory with transfer
on $\Fen$, which we now describe. 
 
We say that a groupoid $\Gc$ is {\em locally finite} (resp. {\em discrete}) if, for any $x \in \Gc$, the group
$\Aut_\Gc(x)$ is finite (resp. trivial). A groupoid $\Gc$ is called {\em finite}, if $\Gc$ is
locally finite and $\pi_0(\Gc)$ is a finite set. If $\Gc$ is finite and
$\ch(\k)=0$, then we have the {\em orbifold integral map}
\begin{equation}\label{eq:orbifold-integral}
	\int_\Gc: \Fen(\Gc)\lra \k, \quad 
	\int_\Gc \phi =\sum_{[x]\in\pi_0(\Gc)} \frac{\phi(x)}{|\Aut_\Gc(x)|}\in \k.
\end{equation}
Here $x$ is any object in the isomorphism class $[x]$. If $\Gc$ is finite and discrete,
then $\int_\Gc$ is defined without any assumptions on $\k$.
   
For a functor $f: \Gc' \to \Gc$ of groupoids, we recall the definiton of the $2$-fiber of $f$ over an object 
$x \in \Gc$ (Definition \ref{def:2lim}), given by
\[
	Rf^{-1}(x) =\twopro \bigl\{
	\{x\}\lra \Gc\buildrel f\over\lla \Gc'
	\bigr\}, \quad x\in\Ob(\Gc).
\]
We introduce several classes of functors.
 
\begin{defi} A functor $f: \Gc'\to \Gc$ of groupoids is called
\begin{itemize}
\item {\em weakly proper} if the map
$\pi_0(\Gc')\to\pi_0(\Gc)$ is finite-to-one, 

\item $\pi_0$-{\em proper} if each $2$-fiber of $f$ has finitely many isomorphism classes,

\item {\em proper} if each $2$-fiber of $f$ is finite,

\item {\em absolutely proper} if each $2$-fiber of $f$ is finite and discrete. 
\end{itemize}
\end{defi}
The last three classes, being defined in terms of $2$-fibers, are stable under arbitrary
$2$-pullbacks and therefore each of them forms, together with $\Mor(\Grp)$, a transfer structure.

\begin{prop}\label{prop:Z-proper} A functor $f: \Gc'\to \Gc$ of groupoids is
\begin{enumerate}
	\item $\pi_0$-proper, if and only if $f$ is
		weakly proper and, for every $x' \in \Gc'$, the homomorphism of groups
		\[
			f_{x'}: \Aut_{\Gc'}(x')\lra\Aut_\Gc(f(x')), \quad x'\in\Ob(\Gc'),
		\]
		has finite cokernel.
  
	\item proper, if and only if $f$ is $\pi_0$-proper and, for every $x' \in
		\Gc'$, the homomorphism $f_{x'}$ has finite kernel.
   
	\item absolutely proper, if and only if $f$ is $\pi_0$-proper and,
		for every $x' \in \Gc'$, the homomorphism $f_{x'}$ is injective. 
\end{enumerate}
\end{prop}
\begin{proof}  
The statements reduce to the case when both $\Gc'$ and $\Gc$ have one object, which we denote
$\bullet'$ and $\bullet$, respectively. Then $f$ reduces to a homomorphism of groups $f: G'\to G$. In this
situation $G'$ acts on $G$ on the left via $(g',g)\mapsto f(g')g$, and we find that the $2$-fiber of $f$
\[
	  Rf^{-1}(\bullet) = G'\bbs G
\]
is the corresponding action groupoid. Isomorphism classes of objects of this groupoid correspond to 
right cosets of $G$ by $\on{Im}(f)$, and the automorphism group of any object is $\Ker(f)$. The
statements follow directly from these observations. 
\end{proof}
  
Given an absolutely proper functor $f: \Gc'\to \Gc$ of small groupoids, we define the
{\em orbifold direct image map}
\[
	f_*: \Fen(\Gc')\lra\Fen(\Gc), \quad (f_*\phi)(x) =\int_{Rf^{-1}(x)} \phi_{|Rf^{-1}(x)}. 
\]
If $\ch(\k)=0$, then $f_*$ is defined for any proper functor. 
  
\begin{ex}\label{ex:orbifold-image-groups}
Suppose $\Gc'$ and $\Gc$ have one object each, so $f$ reduces to a homomorphism of groups 
$f: G'\to G$. By the above, $f$ being proper means that $\Ker(f)$ and $\Coker(f)$ are finite. In
this case, denoting $1_{\Gc'}$ the element of $\Fen(\Gc')= \k$ corresponding to $1\in \k$, and
similarly for $\Gc$, we have
\[
	f_*(1_{\Gc'})=\frac{|\Coker(f)|}{|\Ker(f)|} \cdot 1_{\Gc}. 
\]
\end{ex}
  
\begin{prop}\label{prop:F-theory-transfer} 
	\begin{enumerate}[label=(\alph{*})]
		\item Let $\k$ be any field. Then the orbifold direct image makes $\Fen$ into a
			theory with transfer on $\Grp$, contravariant with respect to all functors
			and covariant with respect to absolutely proper functors. 

		\item If $\k$ is a field of characteristic $0$, then $\Fen$ becomes a theory with
			transfer covariant with respect to all proper functors. 
	\end{enumerate}
\end{prop}

\begin{proof} 
The fact that orbifold direct image is compatible with composition, i.e., $(f\circ
g)_*=f_*\circ g_*$ for (absolutely) proper $f$ and $g$, reduces to the case of functors
between groupoids with one object, in which case it follows from Example
\ref{ex:orbifold-image-groups}. The base change for a $2$-Cartesian square of groupoids
follows, in a standard way, from the identification of $2$-fibers. 
\end{proof}

We say that a functor $f: \Gc'\to\Gc$ is {\em locally proper} (resp. {\em locally absolutely
proper}), if the restriction of $f$ to any isomorphism class in $\Gc'$ is proper (resp. absolutely
proper). Such functors are characterized by the condition that, for every $x' \in \Gc'$, the
homomorphism $f_{x'}$ from Proposition \ref{prop:Z-proper} has finite kernel and cokernel (resp. trivial kernel and finite
cokernel). 
A groupoid $\Gc$ is called an {\em orbifold}, if the constant functor $\Gc \to \pt$ is a locally
proper functor, i.e., for every $x \in \Gc$, the automorphism group $\Aut_\Gc(x)$ is finite. Thus a
functor of groupoids is locally proper if and only if all its $2$-fibers are orbifolds. In
particular, any functor of orbifolds is locally proper. 

Let $\k$ be a field of characteristic $0$. For a groupoid $\Gc$, let $\Fen_0(\Gc)\subset \Fen(\Gc)$
be the subspace formed by functions $\pi_0(\Gc)\to \k$ with finite support. Note that formula
\eqref{eq:orbifold-integral} defines the map $\int_\Gc: \Fen_0(\Gc)\to \k$ for any orbifold $\Gc$,
and thus we can define
\begin{equation}\label{eq:f-0-covariant}
	f_*: \Fen_0(\Gc')\lra \Fen_0(\Gc)
\end{equation}
for any locally proper functor $f: \Gc'\to \Gc$. Note, that we have the contravariant functoriality
  \[
	  f^*: \Fen_0(\Gc)\lra \Fen_0(\Gc')
  \]
for weakly proper functors of orbifolds. 
  
\begin{lem} 
	The classes of weakly proper and locally proper functors form a transfer structure on
	$\Grp$. The same is true for the classes of weakly proper and locally absolutely proper
	functors. 
\end{lem}
    
\begin{proof} Let 
 \[
\xymatrix{\Gc_2 \ar[r]^{f} \ar[d]_{u_2} &\Gc_1\ar[d]^{u_1}\cr
\Gc'_2 \ar[r]^{f'}&\Gc'_1
}
\]
be a 2-Cartesian square of groupoids such that the functor $u_1$ is weakly proper, and $f'$ is
locally (absolutely) proper.  We need to prove that $u_2$ is again weakly proper and $f$ is locally
(absolutely) proper.  The statement about $f$ follows from identification of 2-fibers in a
2-pullback. Let us prove that $u_2$ is weakly proper. 
As the 2-fiber product is additive w.r.t. disjoint union of groupoids in each argument, the
statement about $u_2$ reduces to the case when $\Gc_1, \Gc'_1, \Gc_2$ each have one object, i.e.,
the corresponding part of the above square comes from a diagram of groups and homomorphisms
\[
	G_2\buildrel f'\over\lra G'_1 \buildrel u_1\over\lla G_1
\]
with $u_1$ having finite kernel and cokernel (resp. being injective
with finite cokernel). The groupoid $\Gc_2$ is then equivalent to the action groupoid
\[
	\Gc_2 \simeq (G'_2\times G_1^\op)\bbs G'_1,
\]
where $G'_2\times G_1^\op$ acts on the set $G'_1$ by
\[
	(g'_2, g_1)\cdot g'_1 = f'(g'_2) g'_1 u_1(g_1). 
\]
Since $u_1$ has finite cokernel, the action of $G'_1$ alone already has finitely
many orbits. This implies that the action groupoid above has finite $\pi_0$
and so $u_2$ is weakly proper. 
\end{proof}

\begin{prop}
Let $\k$ be a field of characteristic $0$ (resp. of arbitrary characteristic).  The correspondence
$\Gc\mapsto \Fen_0(\Gc)$ gives rise to a $\Vect_\k$-valued theory with transfer on $\Grp$,
contravariant with respect to weakly proper functors and covariant with respect to locally proper
(resp. locally absolutely proper) functors. 
\end{prop}
   
\begin{proof} 
	Once the required functorialities are in place, the argument is similar to that of
	Proposition \ref{prop:F-theory-transfer}. 
\end{proof}
   
\begin{exa}[(Classical Hall algebras)]\label{ex:classical-hall}
Let $\Ec$ be an exact category in the sense of Quillen, and let $\Sc(\Ec)$ be its Waldhausen space,
considered as a $2$-Segal object in $\Grp$ as in \S \ref{subsec:waldhausen-1}.  We say that
$\Ec$ is {\em finitary}, if 
\begin{enumerate}
	\item the category $\Ec$ is essentially small, 
	\item for all objects $A,B\in\Ec$ and every $i \ge 0$, the groups $\on{Ext}^i_\Ec(A,B)$ are finite and 
	\item for $i\gg 0$, $\on{Ext}^i_\Ec(A,B) \cong 0$. 
\end{enumerate} 
Here the $\on{Ext}$-groups are calculated in the abelian envelope of $\Ec$. An example of a finitary
exact category is provided by the category $\Coh(X/\FF_q)$ of coherent sheaves on a smooth
projective variety $X$ over a finite field. 
  
If $\Ec$ is finitary, then each groupoid $\Sc_n(\Ec)$ is an orbifold and, moreover, the functor
$(\partial_2, \partial_0)$ in the diagram
\begin{equation}\label{eq:hall-waldhausen-main-diagram}
	\Sc_1(\Ec)\times \Sc_1(\Ec) \buildrel (\partial_2, \partial_0)\over\lla 
	\Sc_2(\Ec)\buildrel \partial_1\over \lra \Sc_1(\Ec), 
\end{equation}
is proper. Indeed, $\Sc_2(\Ec)$ is the groupoid formed by admissible short exact sequences
\[
  0\to A'\lra A\lra A''\to 0
\]
in $\Ec$ and their isomorphisms. The functor $(\partial_2, \partial_0)$ associates to such a
sequence its two extreme terms, so it is finite-to-one on $\pi_0$ because of finiteness of
$\Ext^1$. Note that $(\partial_2, \partial_0)$ is in general not absolutely proper. 
The functor $\partial_1$ is always locally absolutely proper. Indeed, it is injective on
morphisms since an automorphism of a short exact sequence is determined by its action on the
middle term.
  
Therefore, we can form the associative $\k$-algebra $\Hall(\SW(\Ec), \Fen_0)$. This is nothing but
the classical {\em Hall algebra} $\on{Hall}(\Ec)$ of $\Ec$ defined as follows
(cf. \cite{schiffmann:hall}). It has a $\k$-basis $\{e_A\}$, where $A$ runs over all isomorphism classes
of objects of $\Ac$. The multiplication, denoted $*$, is given by the formula
\[
e_A*e_B =  \sum_C g_{AB}^C e_C, 
\]
where $g_{AB}^C\in \ZZ_+$ is the number of subobjects $A'\subset C$ such that
$A'\simeq A$ and $C/A'\simeq B$. This number is finite because of
the finiteness of the $\Hom$ and $\Ext^1$-groups in $\Ec$. The identification 
$\on{Hall}(\Ec) \cong \Hall(\SW(\Ec), \Fen_0)$ is obtained by mapping $e_A$ to $\1_A\in\Fen_0(\Sc_1(\Ec))$,
the characteristic function of the isomorphism class of $A$. 
  
We say that $\Ec$ is {\em cofinitary}, if any object has only finitely many subobjects. An example
is provided by the category of $\FF_q$-representations of a finite quiver. 
If $\Ec$ is both finitary and cofinitary, the algebra structure extends to $\widehat{\on{Hall}}(\Ec)$, 
the completion of the vector space $\on{Hall}(\Ec)$ formed by all infinite
formal linear combinations of the $e_A$.
On the other case, in this case the functor $\partial_1$ in \eqref{eq:hall-waldhausen-main-diagram}
is absolutely proper (as its action on $\pi_0$ will be finite-to-one). Therefore,  the algebra
$\Hall(\SW(\Ec), \Fen)$ is defined for any field $\k$. This algebra is isomorphic to 
$\widehat{\on{Hall}}(\Ec)$. 

We have a similar interpretation of the Hall algebras of set-theoretic representations of quivers
and semigroups considered by Szczesny \cite{szczesny:quivers, szczesny:semigroups}. They can be
obtained from the Waldhausen spaces of the (nonlinear) proto-exact categories formed by such representations,
see Example \ref{ex:pointed-sets}. 
\end{exa}
 
\begin{exa}[(Classical Hecke algebras)]\label{exa:heckealg}
Let $G$ be a group, $K \subset G$ a subgroup, and let $\Sc(G, G/K)$ be their Hecke-Waldhausen
simplicial groupoid from \S \ref{subsec:hecke-waldhausen}. It is a $1$-Segal (hence 
$2$-Segal) object in $\Grp$. We say that $K$ is {\em almost normal} if the following condition holds:
\begin{enumerate}[label=(AN)]
	\item \label{l:an} For any $g\in G$ the subgroup $gKg^{-1}$ is commensurate with
	$K$, i.e., the intersection $K\cap (gKg^{-1})$ has finite index in each of them.
\end{enumerate}
For instance, any subgroup of a finite group is almost normal. 
If $K$ is almost normal then, by Proposition \ref{prop:main-diag-hecke-wald} below, we can apply the theory with transfer $\Fen_0$
on $\Grp$, contravariant along weakly proper maps and covariant along locally proper maps, to form
the Hall algebra
\[
	\Hall(\Sc(G, G/K), \Fen_0) = \Fen_0(\Sc_1(G, G/K)) = \Fen_0(K\backslash G/K).
\]
This nothing but the classical {\em Hecke algebra} $\Heck(G, K)$ of the pair $(G,K)$
(see, e.g., \cite[\S 3.1]{shimura}). 
\end{exa}

\begin{prop}\label{prop:main-diag-hecke-wald} 
	Given a group $G$ and a subgroup $K \subset G$, consider the diagram
	\begin{equation}\label{eq:main-diag-hecke-wald}
		 \Sc_1(G, G/K) \times \Sc_1(G, G/K)\buildrel (\partial_2, \partial_0)\over\lla
		  \Sc_2(G, G/K)\buildrel \partial_1\over\lra \Sc_1(G, G/K).
	\end{equation}
	If $K$ is almost normal, then the functor $(\partial_2, \partial_0)$ is weakly proper, and
	$\partial_1$ is locally absolutely proper. If $G$ is finite, then $\partial_1$ is absolutely
	proper. 
\end{prop}
\begin{proof} 
The conjugates of $K$ are precisely the stabilizers of various points of $G/K$. The condition \ref{l:an} 
implies that the intersection of any finite number of such stabilizers has finite index in each of
them. 
For any object of $\Sc_2(G, G/K)$, i.e., an ordered pair of points $(x,y)\in (G/K)^2$, we denote by
$d(x,y)\in G\backslash(G/K)^2= K\backslash G/K$ the corresponding $G$-orbit, i.e., the class of
$(x,y)$ in $\pi_0\Sc_2(G, G/K)$. 

To prove that $(\partial_2, \partial_0)$ is weakly proper means to prove that for any $\alpha,
\beta\in K\backslash G/K$ the set of triples $(x,y,z)\in(G/K)^3$ such that $d(x,y)=\alpha$ and
$d(y,z)=\beta$, splits into finitely many $G$-orbits.  For this, it suffices to fix $x$ and $y$ such
that $d(x,y)=\alpha$, look at all $z$ such that $d(y,z)=\beta$ and prove that the set $Z$ of such
$z$ splits into finitely many orbits of $\on{Stab}(x)\cap\on{Stab}(y)$. But $Z$ is one orbit of
$\on{Stab}(y)$, and \ref{l:an} implies that $\on{Stab}(x)\cap\on{Stab}(y)$ is a finite index subgroup
there, whence the statement. 

The statement that $\partial_1$ is locally absolutely proper, means that for any $(x,y,z)\in(G/K)^3$
the homomorphism $\on{Stab}(x,y,z)\to \on{Stab}(x,z)$ is an embedding of a subgroup of finite index.
But $\on{Stab}(x,z)$ is the intersection of $\on{Stab}(x)$ and $\on{Stab}(z)$, and
$\on{Stab}(x,y,z)$ is the triple intersection. So the ``embedding" part is obvious, and the ``finite
index" part follows from \ref{l:an}. 
\end{proof}

\vfill\eject

\subsection{Groupoids: Generalized Hall and Hecke algebras}
We now survey some other theories with transfer on the category of groupoids. 
Each such theory gives rise to a generalization of classical Hall and Hecke algebras. 

\paragraph{A. Groupoid cohomology.} 

Let $\k$ be a field and consider the functor $\Fen: \Grp \to \Vect_\k$ from \S
\ref{subsec:groupoids-classical-hall}. Note that, for a groupoid $\Gc$, the vector space $\Fen(\Gc)$
can be identified with the $0$th cohomology group $H^0(B\Gc, \k)$, where $B\Gc$ denotes the
classifying space of $\Gc$. In this paragraph, we show that the transfer theories of \S
\ref{subsec:groupoids-classical-hall} can be extended to full cohomology functors. To this end,
we will use an explicit model for the cohomology of $B\Gc$ given by {\em groupoid cohomology}.

We consider the functor 
\[
	\pro: \Fun(\Gc, \Vectk) \lra \Vectk,\; F \mapsto \pro F
\]
mapping a $\Gc$-indexed diagram in $\Vectk$ to its projective limit. For convenience, we will simply
write $\lim$ for this functor. Given a diagram $F \in  \Fun(\Gc,
\Vectk)$, we can explicitly describe $\lim F$ as the subspace of $\prod_{x \in \Gc} F(x)$ given
by those sequences $(v_x)_{x \in \Gc}$ such that, for every morphism $f: x \to y$ in
$\Gc$, we have $v_y = F(f)(v_x)$.

\begin{exa} Let $G$ be a group considered as a groupoid $\Gc$. Then a $\Gc$-diagram in $\Vectk$
  	corresponds to a representation of the group $G$ and the functor $\lim$ takes a representation $V$ to the the space $V^G$ of
	$G$-invariants.
\end{exa}

\begin{exa} Let $\Gc$ be a groupoid and consider the constant diagram $\k$. Then $\lim \k$ can be
	identified with the space $\Fen(\Gc)$ of $\k$-valued functions on $\pi_0(\Gc)$.
\end{exa}

As a right adjoint, the functor $\lim$ is left exact. For $i \ge 0$, the right derived functor
\[
	R^i \lim: \Fun(\Gc, \Vectk) \lra \Vectk
\]
is called the {\em $i$th groupoid cohomology functor} associated to $\Gc$. Given a $\Gc$-diagram
$F$, we will also write $H^i(\Gc, F)$ for $R^i \lim(F)$. 

\begin{exa} Let $G$ be a group considered as a groupoid $\Gc$. Then groupoid cohomology coincides
	with group cohomology. The groupoid cohomology of a general groupoid $\Gc$ can always be identified with a
	direct sum of group cohomology groups associated to the various automorphism groups of
	objects in $\Gc$. 
\end{exa}

Let $\varphi: \Hc \to \Gc$ be a functor of groupoids. Note that, for formal reasons, we have a
canonical natural transformation
\begin{equation}\label{eq:formalreason}
  \lim_{\Gc}{} \lra \lim_{\Hc}{} \circ \varphi^*.
\end{equation}
Assume now, that $\varphi$ is absolutely proper, so that the $2$-fibers of $\varphi$ are finite and
discrete. Then, we have a {\em transfer map}
\begin{equation}
  \tau_{\varphi}: \lim_{\Hc}{} \circ \varphi^* \lra \lim_{\Gc}{},
\end{equation}
which, given a diagram $F \in \Fun(\Gc, \Vectk)$, is defined as follows. As explained above, we may
identify $\lim \varphi^* F$ and $\lim F$ with subspaces of $\prod_{y \in \Hc} F(\varphi(y))$
and $\prod_{x \in \Gc} F(x)$, respectively. The map $\tau_{\varphi}$ is then
obtained by sending a sequence $(w_y)_{y \in \Hc}$ to the sequence $(v_x)_{x \in \Gc}$ given by the
formula
\[
	v_x = \sum_{[(y, f: \varphi(y) \to x)] \in \pi_0(R\varphi^{-1}(x))} F(f)(w_y) \in F(x).
\]
Here the sum is taken over isomorphism classes of objects of the $2$-fiber of $\varphi$ over $x$, and
one easily verifies that the summand $F(f)(w_y)$ does not depend on the choice of a representative
of the class $[(y, f: \varphi(y) \to x)] \in \pi_0(R\varphi^{-1}(x))$. Note that, due to the
assumption that $\varphi$ is absolutely proper, the sum on the right-hand side is actually finite.

\begin{exa} Let $H \subset G$ be a subgroup of finite index. Then the functor of corresponding
	groupoids $\varphi: \Hc \to \Gc$ is absolutely proper. Given a representation $V$ of $G$, 
	the (well-known) transfer map $\tau_{\varphi}(V)$ corresponds to the map between invariant subspaces given by
	\[
		V^H \lra V^G,\; v \mapsto \sum_{gH \in [G:H]} g v.
	\]
\end{exa}

\begin{exa} Let $\varphi: \Hc \to \Gc$ be an absolutely proper map of groupoids. Let $\k$ be the
	constant $\Gc$-diagram. Then the transfer map $\tau_{\varphi}(\k)$ corresponds to a map
	$\Fen(\Hc) \to \Fen(\Gc)$ which coincides with the orbifold direct image of \S
	\ref{subsec:groupoids-classical-hall}. 
\end{exa}

\begin{rem} We give a more conceptual perspective on the existence of the transfer
	map. Let $\varphi: \Hc \to \Gc$ be an absolutely proper functor of groupoids. Then the pullback
	functor $\varphi^*: \Fun(\Gc, \Vectk) \to \Fun(\Hc, \Vectk)$ admits left and right adjoints
	$\varphi_!$
	and $\varphi_*$, given by left and right Kan extensions, respectively. Remarkably, under our
	assumptions on $\varphi$, the functors $\varphi_!$ and $\varphi_*$ are isomorphic: The pointwise formula for Kan
	extensions, together with the assumption that the $2$-fibers of $\varphi$ are finite and
	discrete, reduces our claim to the statement that, in any abelian category, finite
	coproducts and finite products coincide. 
	Thus, there exists a {\em trace map}
	\[
		\varphi_* \circ \varphi^* \to \id,
	\]
	exhibiting $\varphi_*$ as the {\em left} adjoint of $\varphi^*$. Composing the trace map with the
	pushforward along the constant functor $\Gc \to \pt$, we recover the transfer map.
\end{rem}

By Grothendieck's characterization of derived functors as universal $\delta$-functors (see e.g.
\cite[\S2]{weibel}), the transfer map $\tau_{\varphi}$ induces a unique map of graded vector spaces
\[
	\tau_{\varphi}^{\bullet}(F): H^{\bullet}(\Hc, \varphi^* F) \lra  H^{\bullet}(\Gc, F). 
\]
Let $\varphi: \Hc \to \Gc$ be an absolutely proper functor of groupoids. We denote by
$\k$ the trivial $\Gc$-diagram with value $\k$. Then we obtain a map
\[
	  \tau_{\varphi}^{\bullet}(\k): H^{\bullet}(\Hc, \k) \lra  H^{\bullet}(\Gc, \k),
\]
which we will denote by $\varphi_{\circledast}$. Note that we further have a pullback map
\[
	  \varphi^{\circledast}: H^{\bullet}(\Hc, \k) \lra  H^{\bullet}(\Gc, \k),
\]
obtained by deriving \eqref{eq:formalreason}. 
Let $\Vect_\k^\ZZ$ be the monoidal category of $\ZZ$-graded $\k$-vector
spaces, with the usual graded tensor product.

\begin{prop}	The association 
		\[
		  H^{\bullet}: \Gr \lra \Vectk^{\ZZ},\; \Gc \mapsto H^{\bullet}(\Gc, \k)
		\]
		gives rise to a $\Vect_k^\ZZ$-valued theory with transfer on $\Grp$ contravariant,
		via $\varphi \mapsto \varphi^{\circledast}$, along arbitrary functors and covariant,
		via $\varphi \mapsto \varphi_{\circledast}$, along absolutely proper functors.
\end{prop}  
\begin{proof} 
	The functoriality of the association $\varphi \mapsto \varphi_{\circledast}$ follows from the
	following statement: Given absolutely proper functors $\varphi: \Hc \to \Gc$ and $\psi: \Kc \to \Hc$ of
	groupoids, and let $F$ be a $\Gc$-diagram in $\Vectk$, we have an equality 
	\[
		\tau_{\varphi}(F) \circ \tau_{\psi}(\varphi^* F) = \tau_{\varphi \circ
		\psi}(F)
	\]
	of maps $\lim_{\Kc} \psi^* \varphi^* F \to \lim_{\Gc} F$. This statement follows directly 
	from the definition of the transfer map.
	It remains to verify property \ref{ts2} of Definition \ref{def:transfer-structure}.
	To this end, we claim that, given a $2$-Cartesian square
	\[
		\xymatrix{\Hc \ar[r]^{\varphi} \ar[d]_{p} &\Gc\ar[d]^{q}\\
		\Hc' \ar[r]^{\psi}&\Gc'
		}
	\]
	and a $\Gc'$-diagram $F$, the two natural maps $\lim_{\Hc'} \psi^* F \to \lim_{\Gc} q^* F$
	given by the composites of
	\[
		\lim_{\Hc'} \psi^* F \overset{\tau_{\psi}}{\lra} \lim_{\Gc'} F
		\lra \lim_{\Gc} q^* F
	\]
	and
	\[
		\lim_{\Hc'} \psi^* F \lra \lim_{\Hc} p^* \psi^* F
		\overset{\cong}{\lra} \lim_{\Hc} \varphi^* q^* F
		 \overset{\tau_{\varphi}}{\lra} \lim_{\Gc} q^* F,
	\]
	respectively. This claim can easily be reduced to the case when the $2$-Cartesian square is
	a $2$-fiber square. In this case, the statement follows directly from the definition of
	$\tau$.
\end{proof}

\begin{rem} We can vary the construction of the theory with transfer $H^{\bullet}$ to provide
	a {\em homological} theory with transfer $H_{\bullet}$ which is covariant along arbitrary functors and
	contravariant along absolutely proper functors. This theory admits an explicit description
	in terms of {\em groupoid homology} which is obtained by deriving the inductive limit
	functor. Further, we can define a theory with transfer
	$H^{\bullet}_c$ of {\em compactly supported cohomology} which is contravariant along weakly proper
	functors and covariant along locally absolutely proper functors. We leave the details of
	these constructions to the reader.
\end{rem}

\begin{ex}[(Group-cohomological Hall algebras)] 
	Let $\Ec$ be a finitary exact category. Then the functor $(\partial_2, \partial_0)$ in
	\eqref{eq:hall-waldhausen-main-diagram} is weakly proper, and $\partial_1$ is absolutely
	proper, so we can form the Hall algebra with coefficients in $H^\bullet_c$ which is the
	graded vector space 
	\begin{equation}\label{eq:coh-classical-hall}
		\Hall(\SW(\Ec), H^\bullet_c) = H^\bullet_c(B\Sc_1(\Ec),\k) \cong
		\bigoplus_{[A] \in \pi_0(\Ec)} H^\bullet(\Aut(A), \k)
	\end{equation}
	with multiplication given by the map $\partial_{1*} \circ (\partial_2,\partial_0)^*$ obtained from
	\eqref{eq:hall-waldhausen-main-diagram}.  Note that we can not use the theory with transfer
	$H_\bullet$, since $(\partial_2, \partial_0)$ is proper but not absolutely proper.  If $\Ec$ is also
	cofinitary, then we can apply the theory $H^\bullet$ which will give the direct product instead of
	the direct sum in \eqref{eq:coh-classical-hall}. 

	The groups $\Aut(A)$ are all finite, so for $\ch(\k)=0$ their higher cohomology vanishes and the
	above algebra reduces to the completion of the classical Hall algebra from Example
	\ref{ex:classical-hall}. On the other hand, if $\k$ has finite characteristic, then this algebra is
	quite large and potentially very interesting. 
	The simplest example is obtained by taking $\Ec=\Vect^{\on{fd}}_{\FF_q}$ to be the category of
	finite-dimensional vector spaces over a finite field. In this case, we obtain the algebra 
	\[
		\Hall(\SW(\Ec), H^\bullet_c) = \bigoplus_{n\geq 0} H^\bullet(GL_n(\FF_q), \k),
	\]
	with multiplication of the $m$th and $n$th factors
	coming from the diagram of groups
	\[
		GL_m\times GL_n\buildrel \pi_{m,n}\over \lla 
		\begin{pmatrix}
			GL_m&*\\
			0&GL_n
		\end{pmatrix} 
		\buildrel i_{m,n}\over\lra GL_{m+n}
	\]
	obtained by pull back along $\pi_{m,n}$, and transfer along $i_{m,n}$. The algebra
	$\Hall(\SW(\Ec),
	H^\bullet_c)$ resembles an algebra studied by Quillen \cite{quillen} which is given by
	\[
		H_Q =\bigoplus_{n\geq 0} H_\bullet(GL_n(\FF_q), \k)
	\]
	with multiplication induced by the embedding $GL_{m}\times GL_n\to GL_{m+n}$. 
\end{ex}

\begin{ex}[(Group-cohomological Hecke algebras)] 
	Let $G$ be a group, and $K \subset G$ an almost normal subgroup. By Proposition
	\ref{prop:main-diag-hecke-wald}, the functor $(\partial_2,\partial_0)$ in the diagram
	\eqref{eq:main-diag-hecke-wald} is weakly proper, and the functor $\partial_1$ is locally
	absolutely proper.  Therefore we can apply the theory with transfer $H^\bullet_c$, to obtain the algebra
	\[
		 \cong\Heck_H(G, K) = H^\bullet_c(B\Sc_1(G, G/K), \k) \cong 
		\bigoplus_{(KgK)\in K\backslash G/K} H^\bullet( K\cap (gKg^{-1}), \k).
	\]
	We call $\Heck_H(G, K)$ the {\em group-cohomological Hecke algebra} of $G$ with respect to $K$. 

	Restricting to degree 0 cohomology, we recover the classical Hecke algebra $\on{Heck}(G,K)$.
	As in the previous example, if $K$ is finite and $\on{char}(\k)=0$, then $\Heck_H(G, K)=
	\on{Heck}(G,K)$.  A potentially interesting class of examples is provided by pairs of
	arithmetic groups $(G,K)$ where $\on{Heck}(G,K)$ is well known by a version of the Satake
	isomorphism \cite{gross}, for example
	\[
		G= GL_n(\ZZ[1/p]), K= GL_n(\ZZ), \quad \on{Heck}(G,K) \simeq \k[t_1^{\pm 1}, ..., t_n^{\pm 1}]^{S_n}.
	\]  
\end{ex}

\paragraph{B. Generalized cohomology.} 
More generally, let $h^\bullet$ be any multiplicative generalized cohomology theory on the category
of CW-complexes, such as K-theory, cobordism, etc. Then $h^\bullet$ is contravariant with respect to
arbitrary maps and admits transfer with respect to finite unramified coverings, see
\cite{kahn-priddy}, or, for more general transfers, \cite{becker-gottlieb}. We define the functor
$h^\bullet_c$ for disconnected CW-complexes by taking the direct sum.  Then $h^\bullet_c$, like
$H^\bullet_c$, is a theory with transfer covariant with respect to locally absolutely proper
functors and contravariant with respect to weakly proper functors.  This theory takes values in the
monoidal category of $\ZZ$-graded abelian groups. 
 
In particular, for any finitary exact category $\Ec$, we have the Hall algebra with coefficients in
$h^\bullet_c$
\[
	\Hall(\SW(\Ec), h^\bullet_c),=h^\bullet_c(B\Sc_1(\Ec))=
	\bigoplus_{[A]\in\pi_0(\Ec)} h^\bullet(B\Aut(A)).
\]
Similarly, for an almost normal subgroup $K$ in a group $G$ we have the Hecke algebra with
coefficients in $h^\bullet_c$
\[
	\Heck_h(G, K) = h^\bullet_c(B\Sc_1(G, G/K), \k) =
	\bigoplus_{(KgK)\in K\backslash G/K} h^\bullet\bigl( B(K\cap (gKg^{-1})) \bigr).
\]
In several classical examples, applying $h^\bullet$ to the classifying space of a finite groupoid
$\Gc$ gives in fact the completion of a more direct algebraic construction, applicable to $\Gc$
itself. Below we consider two such cases. 
 
\paragraph {C. Representation rings.} 
Let $\Vect_\CC^{\on{fd}}$ be the category of finite-dimensional complex vector spaces.  By a {\em
representation} of a groupoid $\Gc$ we mean a covariant functor $\rho: \Gc\to\Vect_\CC^{\on{fd}}$.
Topologically, a representation is the same as a local system (locally constant sheaf of
finite-dimensional $\CC$-vector spaces) on $B\Gc$.  Representations form an abelian category
$\Rep(\Gc)$, and we denote by $\Re (\Gc)$ the Grothendieck group of this category.  For a finite
group $G$ the topological K-theory of $BG$ is, by Atiyah's theorem \cite{atiyah}, identified with
the completion of $\Re(G)$ by powers of the kernel ideal of the rank homomorphism $\Re(G) \lra \ZZ$. 
 
We denote by $\Re_0(\Gc)$ the Grothendieck group of {\em finitely supported representations}, i.e.,
functors $\rho$ which are zero on all but finitely many isomorphism classes of $\Gc$.  A functor $f:
\Gc'\to\Gc$ gives rise to the pullback functor
\[
f^*: \Rep(\Gc)\lra \Rep(\Gc')
\]
which is exact and therefore gives rise to a pullback functor $[f^*]: \Re(\Gc)\to\Re(\Gc')$.  If $f$
is weakly proper, then we also obtain a functor $[f^*]: \Re_0(\Gc)\to\Re_0(\Gc')$.
   
If $f$ is a $\pi_0$-proper functor, then $f^*$ has a left adjoint $f_*$ and a right
adjoint $f_!$ which can be defined as Kan extensions along $f$ (\S \ref{subsec:kan}). In
particular, for an object $\rho'\in\Rep(\Gc')$, we have the formulas
\begin{equation}\label{eq:kan-ext-reps}
(f_*\rho)(x) =\varprojlim_{\{f(x')\to x\}} \rho(x'), \quad
(f_!\rho)(x) =\varinjlim_{\{x\to f(x')\}} \rho(x'), \quad x\in\Ob(\Gc).
\end{equation}
Note that since $\Gc$ is a groupoid, both comma categories are identified with $Rf^{-1}(x)$. Since
$f$ is $\pi_0$-proper, each $Rf^{-1}(x)$ is equivalent to a groupoid with finitely many objects, so
the limits above (taken in the category of all $\CC$-vector spaces) result in finite-dimensional
vector spaces.
  
\begin{ex} 
	If $f: G'\to G$ is an embedding of a subgroup of finite index, then $f_*$ is the
	functor of taking the induced representation. If $f: G'\to \{1\}$, then $f_*$ is the
	functor of taking invariants. Similarly for $f_!$, we obtain coinduced representation and
	coinvariants. 
\end{ex}

Assume that $f$ is proper, so that each $Rf^{-1}(x)$ is equivalent to a finite groupoid.  Since
higher (co)homology of a finite group with coefficients in a complex representation vanishes, for a
proper $f$ the limits above and hence $f_*$ and $f_!$ are exact functors and therefore induce maps of
Grothendieck groups.  Since for a representation of a finite group the space of
coinvariants can be identified with the space of invariants, the two functors induce the same map $[f_*]: \Re(\Gc')\lra \Re(\Gc)$.
If $f$ is only assumed to be locally proper, then $f_*$ still gives rise to a map $[f_*]:
\Re_0(\Gc')\to\Re_0(\Gc)$. In this context, we have the following general base change property for
Kan extensions.

\begin{prop}\label{prop:base-change-general}
Let $\Cc$ be a category with small inductive and projective limits. Then, 
for any $2$-Cartesian square of small groupoids
\[
	\xymatrix{\Hc \ar[r]^{\varphi} \ar[d]_{p} &\Gc\ar[d]^{q}\\
	\Hc' \ar[r]^{\psi}&\Gc',
	}
\]
we have natural isomorphisms of functors
\[
\begin{split}
	\varphi_! \circ p^* & \simeq q^* \circ \psi_!\\
	\varphi_* \circ p^* & \simeq q^* \circ \psi_*.
\end{split}
\]
\end{prop}
\begin{proof} 
	The statement is easily reduced to the case when the diagram is a $2$-fiber diagram. In this
	case, it follows from the pointwise formula for Kan extensions.
\end{proof}

Therefore, condition \ref{tt3} of Definition \ref{def:TT} is satisfied and we have the following statement.

\begin{prop} 
\begin{enumerate}[label=(\alph{*})]
	\item The functor $\Re: \Grp\to\Ac b$ defines a theory with transfer, contravariant with
		respect to all functors and covariant with respect to proper functors.

	\item The functor $\Re_0: \Grp\to\Ac b$ defines a theory with transfer, contravariant with respect to
		weakly proper functors and covariant with respect to locally proper
		functors.
\end{enumerate}
\end{prop} 

\begin{ex}[(Representation ring version of Hall algebras)] 
	For a finitary exact category $\Ec$, we can define the {\em representation ring Hall
	algebra}
	\[
		\Hall(\SW(\Ec), \Re_0) =\bigoplus_{[A]\in\pi_0(\Ec)} \Re(\Aut(A)). 
	\]
	Here each $\Aut(A)$ is a finite group, so $\Re(\Aut(A))\otimes\QQ$ is identified with the
	ring of $\QQ$-valued class functions on $\Aut(A)$. 
	The simplest example is obtained by taking $\Ec=\Vect^{\on{fd}}_{\FF_q}$. In this case, the
	algebra
	\[
		\Hall(\SW(\Vect^{\on{fd}}_{\FF_q}), \Re_0)\otimes \QQ = \bigoplus_{n\geq 0} \Re(GL_n(\FF_q))\otimes\QQ.
	\]
	was studied by Green \cite{green} and, later, in spirit closer to our approach, by
	Zelevinsky \cite{zelevinsky}. Both authors show that this algebra is commutative and isomorphic to
	a polynomial algebra on infinitely many generators, which correspond to the cuspidal
	representations of all groups $GL_n(\FF_q)$. 
\end{ex}
  
\begin{ex}[(Representation ring version of Hecke algebras)]
Let $K \subset G$ be an almost normal subgroup. We obtain the ring
\[
\Hall(\SW(G,K), \Re) \cong \bigoplus_{(KgK)\in K\backslash G/K} \Re(K\cap (gKg^{-1})).
\]
Note that here the groups $K \cap (gKg^{-1})$ may be infinite. The multiplication involves
induction with respect to finite index embeddings of possibly infinite subgroups. 
\end{ex}

\paragraph {D. Burnside rings.} 
Let $\Fset$ be the category of finite sets. For a groupoid $\Gc$, let $\Act(\Gc)=\Fun(\Gc, \Fset)$
be the category of set-theoretic representations of $\Gc$. This category has objectwise operations
$\sqcup, \times$ of disjoint union and Cartesian product. The set of isomorphism classes of objects
of $\Act(\Gc)$ is a commutative monoid under $\sqcup$, and taking the group completion, we get a
group (in fact a commutative ring under $\times$) called the {\em Burnside ring} of $\Gc$ and
denoted $\Ben(\Gc)$. See \cite{dress, dress-siebeneicher} for more background on Burnside rings of
(pro)finite groups. 

As above, each functor $f: \Gc'\to\Gc$ of groupoids gives rise to the pullback functor
$f^*: \Act(\Gc)\to\Act(\Gc')$, which commutes with disjoint unions and hence
induces a homomorphism $f^*: \Ben(\Gc)\to\Ben(\Gc')$. 
As in the representation-theoretic setting above, for a $\pi_0$-proper $f$ the functor
$f^*$ has left and right adjoints $f_!$ and $f_*$ defined by the same formulas
as in \eqref{eq:kan-ext-reps} but with limits taken in $\Set$. Note that these functors 
commute with disjoint unions, so they induce two homomorphisms
\[
f_*, f_!: \Ben(\Gc')\lra\Ben(\Gc). 
\]
These homomorphisms can be quite different, since for a group $G$ acting on a finite set $E$ the set
of invariants $E^G$ and coinvariants (orbits) $G\backslash E$ are, in general, different.
If, however,  the functor $f$ is absolutely proper, then we have $f_*=f_!$, since in this case the only
procedures involved in forming the Kan extensions are induction and coinduction with respect to
embedding of finite index subgroups, and these procedures coincide. 
As before, Proposition \ref{prop:base-change-general} implies that $(f^*, f_*)$
and $(f^*, f_!)$ satisfy condition \ref{tt3}, hence leading to the following statement.
\begin{prop} 
\begin{enumerate}[label=(\alph{*})]
	\item The data $(f^*, f_*)$ and $(f^*, f_!)$ both extend $\Ben$ to a theory with transfer on
		$\Grp$, contravariant with respect to all functors and covariant with respect to
		$\pi_0$-proper functors. 

	\item Similarly, we can extend $\Ben_0$ to a theory with transfer on $\Grp$,
		contravariant with respect to weakly proper functors and covariant with respect to
		locally $\pi_0$-proper functors.
\end{enumerate}
\end{prop}

\vfill\eject

\subsection{\texorpdfstring{$\infty$}{infty}-groupoids: Derived Hall algebras}
\label{subsec:derivedhall}

Let $\mC=\Top$ be the category of compactly generated Hausdorff topological spaces equipped with the
Quillen model structure, and let $\k$ be a field of characteristic $0$. For $Y \in \Top$, we denote
by $\Fenh(Y)$ the space of functions $\pi_0(Y)\to \k$, which we may identify with locally constant
functions on $Y$.  It is clear that the correspondence $Y \mapsto \Fenh(Y)$ provides a contravariant
functor $\Top \to \Vect_k$. In this section, we extend this functor to a $\Vect_k$-valued theory and
study Hall algebras with coefficients in $\Fenh$. The material of this section is an interpretation
of the results of \cite[\S 2, \S 3]{toen-derived}, using the terminology of theories with transfer.
 
We call a space $Y$ {\em locally homotopy finite} if 
\begin{enumerate}
  \item for every $y \in Y$ and $i \ge 1$, the homotopy group $\pi_i(Y,y)$ is finite, and
  \item $\pi_i(Y,y) = 0$ for $i \gg 0$.
\end{enumerate}
If, in addition, the space $Y$ has finitely many connected components, then we say that $Y$ is {\em
homotopy finite}.
We denote by $\Top^{<\infty}$ the full subcategory in $\Top$ formed by homotopy finite spaces. 
For a homotopy finite space $Y$, we define its {\em homotopy cardinality} to be the
rational number
\begin{equation}\label{eq:chi-BD}
	  |Y|_h = \sum_{[y]\in\pi_0(Y)} \prod_{i\geq 1} |\pi_i(Y,y)|^{(-1)^i} \in \QQ\subset \k.
\end{equation}
As far as we know, formula \eqref{eq:chi-BD}, as well as Proposition
\ref{prop:-chi-BD-properties} below, first appeared in the literature in work of J. Baez and
J. Dolan \cite{baez-dolan}. Similar ideas were earlier proposed (orally) by J.-L. Loday, who was
motivated by constructions of homotopy finite spaces in \cite{loday}. 

\begin{prop}\label{prop:-chi-BD-properties} 
	\begin{enumerate}[label=(\alph{*})]
	\item The category $\Top^{<\infty}$ is closed under disjoint unions and Cartesian products,
		and we have
		\[
			|Y \sqcup Z|_h =|Y|_h + |Z|_h, \quad |Y\times Z|_h
			 =|Y|_h\cdot |Z|_h.
		\]

	\item Let 
		\[
		\xymatrix{
			F \ar[d] \ar[r] & E \ar[d]\\
			\pt \ar[r] & B
		}
		\]
		be a homotopy Cartesian square with $B$ connected.
		If any two of the three spaces $F,E,B$ are homotopy finite, then so is the third. 
		If all three are homotopy finite, then we have $|E|_h=|F|_h\cdot
		|B|_h$.
\end{enumerate}
\end{prop}
\begin{proof} 
	The first part is obvious and (b) follows from the long exact sequence of homotopy groups. 
\end{proof}

We give some examples which illustrate the meaning of homotopy cardinality in various contexts.

\begin{exa}
	Let $X$ be a finite set, interpreted as a discrete topological space. Then $X$ is
		homotopy finite and the homotopy cardinality of $X$ is simply the cardinality of $X$. 
\end{exa}

\begin{exa} \label{ex:groupoidcardinality} Let $\Gc$ be a finite groupoid as introduced in \S
		\ref{subsec:groupoids-classical-hall}. Then the classifying space $B\Gc$
		is homotopy finite and we have the formula
		\[
			|B\Gc|_h = \sum_{[C] \in \pi_0(\Gc)} \frac{1}{|\Aut(C)|}
		\]
		for the homotopy cardinality of $B\Gc$. Hence, we obtain the relation
		\[
			|B\Gc|_h = \int_{\Gc} 1,
		\]
		expressing the homotopy cardinality in terms of the orbifold integral from
		\S\ref{subsec:groupoids-classical-hall}.

		The following simplest example in this context suggests a mysterious relation between the concepts of
		homotopy cardinality and Euler characteristic.  Consider a finite group $G$ of order
		$g$. Then its classifying space $BG$ is homotopy finite and has
		homotopy cardinality $1/g$. On the other hand, the simplicial set $\N G$ has exactly
		$(g-1)^n$ non-degenerate simplices in dimension $n$. So, naively writing the formula
		for the Euler characteristic as the alternating sum of the numbers of non-degenerate
		simplices of all dimensions, we get, by {\em formally} summing the geometric series: 
		\[
			|BG|_h = \sum_{n=0}^\infty (-1)^n (g-1)^n \quad ``=" \quad {1\over 1-(1-g)} = {1\over g}. 
		\]
		This example shows that it is natural to view the homotopy cardinality as a
		``regularization'' of the Euler characteristic. See \cite{berger-leinster} for
		further examples of this kind. 
\end{exa}

\begin{exa} \label{ex:inftycardinality} Let $\C$ be an $\infty$-category. We call $\C$ {\em locally finite} if, for every pair
		of objects $x,y$ of $\C$, and every $i \ge 1$, the topological mapping space
		$|\Map_{\C}(x,y)|$ is homotopy finite. If, in addition, the homotopy category
		$\h\C$ of $\C$ has only finitely many isomorphism classes of objects, then we call
		$\C$ {\em finite}. As above, we denote by $K = \C_{\on{Kan}}$ the largest Kan complex
		contained in $\C$ which, in the language of $\infty$-categories, is the
		$\infty$-groupoid of equivalences in $\C$. We call the topological space $X = |K|$ the {\em
		classifying space of objects in $\C$}. With this notation, we have 
		\begin{itemize}
			\item $\pi_0(X) \cong \pi_0(\h\C)$,

			\item for every object $x$ of $\C$, we have the formula
			\[
				\pi_1(X, x) \cong \Aut_{\h\C}(x) \subset \pi_{0}(\Map_{\C}(x,x)),
			\]

			\item for every object $x$ of $\C$ and $i \ge 2$, we have the formula
			\[
				\pi_i(X, x) \cong \pi_{i-1}(\Map_{\C}(x,x), \id_x).
			\]
		\end{itemize}
		These statements can, for example, be obtained as follows: On the one hand,
		\cite[4.2.1.8]{lurie.htt} implies that the Kan complex $\{x\} \times_K K^{\Delta^1}
		\times_K \{x\}$ is a model for the mapping space $\Map_K(x,x)$. On the other hand,
		the homotopy Cartesian square
		\[
			\xymatrix{
				\{x\} \times_K K^{\Delta^1} \times_K \{x\} \ar[r]\ar[d] &
				K^{\Delta^1} \ar[d]\\
				\{x\} \times \{x\} \ar[r] & K \times K,
			}
		\]
		exhibits $\{x\} \times_K K^{\Delta^1} \times_K \{x\}$ as a simplicial model for the
		space of loops in $X$ based at $x$. This implies the above formulas for the homotopy
		groups of the space $X$.
		In particular, if $\C$ is finite, then its classifying space of objects is homotopy
		finite and we obtain an explicit formula for its homotopy cardinality.  Assume that
		$\C$ is a stable $\infty$-category.  For objects $x,y$ of $\C$, we have an weak
		equivalence 
		\[
			\Map_{\C}(\Sigma x, y) \simeq \Omega \Map_{\C}(x, y),
		\]
		where $\Sigma$ denotes the suspension functor and $\Omega$ signifies the loop space
		based at the zero map (\cite[1.1]{lurie.algebra}). Therefore, for $i \ge  1$, we have 
		\[
			\pi_{i}(\Map_{\C}(x,y),0) \cong \Ext^{-i}_{\h\C}(x,y) :=
			\Hom_{\h\C}(\Sigma^i x,y).
		\]
		Further, since the suspension functor is invertible, the mapping space
		$\Map_{\C}(x,y)$ is an infinite loop space and hence its homotopy groups are
		independent of the choice of basepoint.  Thus, we have the formula
		\[
		  |X|_h = \sum_{[x] \in \pi_0(\h\C)} \frac{\prod_{i\geq 2}
		  |\Ext^{1-i}_{\h\C}(x,x)|^{(-1)^{i}},
			}{|\Aut_{\h\C}(x)|}
		\]
		expressing the homotopy cardinality of $X$ completely in terms of the triangulated structure on
		the homotopy category $\h\C$.
\end{exa}

\begin{rem} In light of Examples \ref{ex:groupoidcardinality} and \ref{ex:inftycardinality}, the theory developed in this section
	can be regarded as a generalization of the transfer theory for ordinary groupoids developed
	in \S \ref{subsec:groupoids-classical-hall} to $\infty$-groupoids, modelled by topological
	spaces. 
\end{rem}
 
\begin{rem} Example \ref{ex:groupoidcardinality} exhibits a connection between homotopy
	cardinality and ordinary Euler characteristic.
	Thus we have two Euler characteristic-type invariants defined on two different subcategories
	of $\Top$:
	\begin{enumerate}
		\item The usual Euler characteristic $\chi$, defined on the category $\Top_{<\infty}$ of
			spaces weakly equivalent to a finite CW-complex and taking values in $\ZZ$. 
		\item The homotopy cardinality, defined on the category $\Top^{<\infty}$ of homotopy finite
			spaces and taking values in $\QQ$.
	\end{enumerate}
	This raises a natural question (posed by J. Baez) of whether one can obtain both invariants
	as restrictions of a single invariant defined on a category containing both $\Top_{<\infty}$
	and $\Top^{<\infty}$ and satisfying both additivity and multiplicativity properties. We are
	not aware of any result in this direction. 
\end{rem}
   
If $Y$ is homotopy finite, and $\phi\in\Fenh(Y)$, we define the {\em homotopy integral of $\phi$} to be
\[
  \int^{\h}_Y \phi = \sum_{[y] \in \pi_0(Y)} \phi([y]) \cdot |C_y|_h \in \k,
\]
where $C_y$ denotes the connected component of $Y$ containing $y \in Y$.
Compare with \cite{viro} which treats similar ``integrals'' over the usual Euler characteristic.

\begin{defi}\label{defi:weaklyproper} 
Let $f: Y'\to Y$ be a morphism in $\Top$. We say that $f$ is:
\begin{itemize}
	\item {\em weakly proper}, if the induced map $\pi_0(Y')\to\pi_0(Y)$ is finite-to-one,

	\item {\em homotopy proper}, if each homotopy fiber $Rf^{-1}(y)$,
		$y\in Y_0$, is a homotopy finite space,

	\item {\em locally homotopy proper}, if the restriction of $f$ to each connected component of $Y'$ is
		homotopy proper. 
\end{itemize}
\end{defi}

It is clear that the class of homotopy proper maps is closed under homotopy pullbacks. Therefore, 
together with the class of all maps, the homotopy proper maps form a transfer structure on $\Top$.
Further, the pair formed by weakly proper and locally homotopy proper maps gives another transfer
structure. 
If $f$ is homotopy proper and $\phi\in\Fenh(Y')$, we define the locally constant function $f_*\phi
\in \Fenh(Y)$ by the formula
\[
  (f_*\phi)(y) = \int^{\h}_{Rf^{-1}(y)} \phi_{|Rf^{-1}(y)}.
\]

\begin{rem} 
	In \cite{fulton-macpherson}, a similar construction based on the usual Euler characteristic
	is applied to constructible functions on complex algebraic varieties. 
\end{rem}

Let $\Fenh_0(Y)\subset\Fenh(Y)$ be the subspace of functions 
supported on finitely many connected components of $Y$. Such functions can be pulled back along weakly proper maps $f: Y'\to Y$,
giving $f^*: \Fenh_0(Y)\to\Fenh_0(Y')$. In a similar way, to form the pushforward
$f_*: \Fenh_0(Y')\to\Fenh_0(Y)$, it suffices that $f$ is locally homotopy proper. 

\begin{prop}\label{prop:homotopytransfer}
\begin{enumerate}[label=(\alph{*})]
	\item The assignment $Y\mapsto\Fenh(Y)$, equipped with the above functorialities, defines a
		$\Vect_k$-valued theory with transfer on $\Top$, contravariant with respect to all
		maps and covariant with respect to homotopy proper maps.

	\item Similarly, the association $Y\mapsto\Fenh_0(Y)$ extends to a theory with transfer,
		contravariant with respect to weakly proper maps and covariant with respect to
		locally homotopy proper maps.
\end{enumerate}
\end{prop}
\begin{proof} Corollary 2.4 and Lemma 2.6 in \cite{toen-derived}.
\end{proof}

\begin{ex} Note that the theories with transfer $\Grp \to \Vect_k$ from \S \ref{subsec:groupoids-classical-hall} are
	recovered from Proposition \ref{prop:homotopytransfer} by precomposing with
	the lax monoidal functor 
	\[
	B: \Grp \to \Top,\; \Gc \mapsto B\Gc
	\]
	given by the classifying space construction. Moreover, for an admissible $2$-Segal groupoid
	$\Gc_{\bullet}$, we have a natural isomorphism
	\[
	  \Hall(\Gc_{\bullet}, \Fen) \cong \Hall(B\Gc_{\bullet}, \Fenh),
	\]
	and similarly for the theory $\Fen_0$.  Thus all Hall algebras with coefficients defined in
	\S \ref{subsec:groupoids-classical-hall} can alternatively be obtained as Hall algebras with coefficients
	in the theories $\Fenh$ and $\Fenh_0$. 
\end{ex}

\begin{prop} \label{prop:finitestable} Let $\C$ be a locally finite stable $\infty$-category, and
	let $\SW(\C)$ denote its Waldhausen S-construction.
	\begin{enumerate}
		\item For every $n \ge 0$, the topological space $|\SW_n(\C)|$ is locally homotopy finite.
		\item The topological $2$-Segal space given by $[n] \mapsto |\SW_n(\C)|$ is admissible for the transfer theory given by
			the pair (weakly proper maps, locally proper maps) on $\Top$.
	\end{enumerate}
\end{prop}
\begin{proof} We first prove (1). We say that a Kan complex is (locally) homotopy finite if its
	geometric realization is (locally) homotopy finite. Similarly, we use the terminology of
	Definition \ref{defi:weaklyproper} for Kan complexes in virtue of the geometric realization
	functor. The Kan complex $\SW_0(\C)$ is contractible, hence homotopy finite. Further, by
	assumption, the Kan complex $\SW_1(\C)$ is locally homotopy finite.  We utilize the marked
	model structure on $\sSet$ of \cite[\S 3.1]{lurie.htt}, and freely use the musical notation
	introduced there. For example, we denote by $\C^{\natural}$ the marked simplicial set
	obtained by marking all edges which are equivalences in the $\infty$-category $\C$.  Recall
	from Proposition \ref{prop.inftyequiv} that, for each $n \ge 0$, we have a weak equivalence
	of Kan complexes
	\[
		\SW_n(\C) \overset{\simeq}{\lra} \Fun(\Delta^n, \C)_{\on{Kan}}.
	\]
	Directly from the definition, we obtain an isomorphism
	\[
		\Fun(\Delta^n, \C)_{\on{Kan}} \cong \Map^{\sharp}( (\Delta^n)^{\flat},
		\C^{\natural}),
	\]
	providing a description in terms of simplicial mapping spaces with respect to the marked
	model structure on $\sSet$. Therefore, it suffices to show that $\Map^{\sharp}(
	(\Delta^n)^{\flat},\C^{\natural})$ is locally homotopy finite. Since the inclusion of
	simplicial sets
	\[
		i: \Delta^{n-1} \coprod_{ \{n-1\} } \Delta^{\{n-1,n\}} \subset \Delta^n
	\]
	is inner anodyne, the corresponding marked map $i^{\flat}$ is marked anodyne, and we obtain
	a weak equivalence
	\[
		\Map^{\sharp}( (\Delta^n)^{\flat}, \C^{\natural}) \overset{\simeq}{\lra}
		\Map^{\sharp}( (\Delta^{n-1})^{\flat}, \C^{\natural}) 
		\times_{\Map^{\sharp}( \{n-1\}^{\flat}, \C^{\natural})}
		\Map^{\sharp}( (\Delta^{\{n-1,n\}})^{\flat}, \C^{\natural}),
	\]
	where the right-hand side fiber product is a homotopy fiber product. Therefore, by an induction
	using Proposition \ref{prop:-chi-BD-properties}, we
	can reduce to showing that the Kan complex $\Map^{\sharp}( (\Delta^1)^{\flat},
	\C^{\natural})$ is locally homotopy finite. To this end, note that, for objects $x,y$ of
	$\C$, we a pullback square
	of simplicial sets
	\begin{equation}\label{eq:mappingsquare}
		\xymatrix{
			\{x\} \times_{\C} (\C^{\Delta^1})_{\on{Kan}} \times_{\C} \{y\} \ar[d] \ar[r] &
			(\C^{\Delta^1})_{\on{Kan}}\ar[d]^{s \times t}\\
			\{x\} \times \{y\} \ar[r] & \C_{\on{Kan}} \times \C_{\on{Kan}}.
		}
	\end{equation}
	This square is in fact homotopy Cartesian, since the projection $s \times t$ is a Kan
	fibration. This can, for example, be deduced by interpreting $s \times t$ as the map of
	simplicial mapping spaces, induced by the inclusion $\{0\}^{\flat}\times\{1\}^{\flat}
	\subset (\Delta^1)^{\flat}$ of marked simplicial sets.
	Since, by \cite[4.2.1.8]{lurie.htt}, the Kan complex $\{x\} \times_{\C}
	(\C^{\Delta^1})_{\on{Kan}} \times_{\C} \{y\} \cong \{x\} \times_{\C} \C^{\Delta^1}
	\times_{\C} \{y\}$ is a model for the mapping space $\Map_{\C}(x,y)$ of $\C$, we can again use Proposition
	\ref{prop:-chi-BD-properties} to reduce to the statement that $\C_{\on{Kan}}$ is locally homotopy
	finite. 

	To show (2), we have to verify that, in the diagram
	\[
	\Sc_1(\C)\times \Sc_1(\C) \buildrel (\partial_2, \partial_0)\over\lla 
	\Sc_2(\C)\buildrel \partial_1\over \lra \Sc_1(\C), 
	\]
	the map $(\partial_2, \partial_0)$ is weakly proper, and $\partial_1$ is locally proper.
	The fact that every connected component of $\SW_1(\C)$ and $\SW_2(\C)$ is a homotopy finite space, implies that
	any morphism $\SW_2(\C) \to \SW_1(\C)$, in particular $\partial_1$, is locally proper. 
	To show that the map $(\partial_2, \partial_0)$ is weakly proper, we have to verify that,
	for objects $a,a'$ of $\C$, the subspace $Y$ of $\SW_2(\C)$, lying above the connected component of $\SW_1(\C)
	\times \SW_1(\C)$ represented by the pair $(a,a')$, has finitely many connected components.
	Using the homotopy Cartesian square \eqref{eq:mappingsquare}, with $x = a'$ and $y=\Sigma
	a$, it is easy to see that we have a surjection $\pi_0(\Map_{\C}(a',\Sigma a)) \to
	\pi_0(Y)$. Hence, the statement follows from the local finiteness of the $\infty$-category $\C$.
\end{proof}

\begin{ex} 
	Let $\C$ be a locally finite stable $\infty$-category. By Proposition
	\ref{prop:finitestable}, we can form the Hall algebra $\Hall(\SW(\C), \Fenh_0)$ with
	coefficients in $\Fenh_0$. This recovers the derived Hall algebra defined in
	\cite{bergner-hall}.
\end{ex}

\begin{ex}
	Let $\FF$ be a finite field, and let $\A$ be a dg category over $\FF$,
	see \S \ref{subsec:der-wald-stacks} for conventions.
	 	 We say that $\A$ is {\em locally finite} if, for every pair
	of objects $a,b$ of $T$, the total cohomology space of the mapping complex $\A(a,b)$ is a
	finite dimensional $\FF$-vector space. Let $\C = \Ndg(\Perf^{\circ}_\A)$ be the
	dg nerve of  the dg category of cofibrant, perfect $\A$-modules.
	  Using \cite[1.3]{lurie.algebra}, one shows that the $\infty$-category
	$\C$ is locally finite and stable.  The Hall algebra of $\SW(\C)$ with coefficients in
	$\Fenh_0$ recovers To\"en's derived Hall algebra associated to $\A$ constructed in
	\cite{toen-derived}.
\end{ex}

\vfill\eject

\subsection{Stacks: Motivic Hall algebras}

Motivic Hall algebras were introduced by Joyce \cite{joyce-II} and Kontsevich-Soibelman
\cite{KS-motivic}, see also \cite{bridgeland} for a transparent introduction using the work of
To\"en \cite{toen-stacks} on Grothendieck groups of Artin stacks. From our point of view, the
existence of these algebras is a reflection of the $2$-Segal property of the Waldhausen stacks from
Example \ref{ex:waldhausen-stacks}. 

More precisely, we consider the situation of Example \ref {ex:stacks-of-groupoids}. That is, let
$\FF$ be a field, and $\Uc=\FF-\Aff$ be the category of affine $\FF$-schemes of at most countable
type, made into a Grothendieck site via the \'etale topology. Let $\mC=\underline{\Grp}_\Uc$ be the
category of stacks of groupoids on $\Uc$, with the Joyal-Street model structure.  Objects of $\mC$
will be simply referred to as {\em stacks}.  Inside $\mC$ we have the subcategory $\Ac rt$ of Artin
stacks. Recall that by a {\em geometric point} of a stack $\Gc$ one means an object of the groupoid
$\Gc(\KK):= \Sc(\on{Spec} \KK)$, where $\KK$ is an algebraically closed field containing $\FF$ (here
assumed at most countably generated over $\FF$ so that $\on{Spec}(\KK)\in\Uc$).  
Following \cite{toen-stacks} (also cf. \cite{bridgeland}), we give the following definition. 
 
\begin{defi}
	A stack $\Sc $ is called {\em special}, if it is an Artin stack of finite type over
		$\FF$, and if the stabilizer of any geometric point is an affine algebraic group.
	A morphism of stacks $\phi: \Gc'\to\Gc$ in $\underline{\Grp}_\Uc$ is called {\em
		special}, if for any morphism of stacks $\psi: \Sc\to\Gc$ with $\Sc$ a special
		stack, the $2$-fiber product $\Sc\times^{(2)}_{\Gc} \Gc'$ is a special stack. We
		denote by $\Sc p$ the class of special morphisms of stacks. 
	A morphism $\phi: \Gc'\to\Gc$ is called a {\em geometric bijection}, if for any
		algebraically closed field $\KK\supset\FF$ as above, the induced functor of
		groupoids $\Gc'(\KK)\to\Gc(\KK)$ is an equivalence.  
\end{defi}
 
For example any morphism of special stacks is special. We then obtain easily: 

\begin{prop}
The pair $(\Sc p, \Mor(\underline{\Grp}_\Uc))$ forms a transfer structure on the model category
$\underline{\Grp}_\Uc$. 
\end{prop}

The following is an adaptation of \cite[Def. 3.10]{bridgeland}.

\begin{defi}\label{def:motivic-fun-stacks}
Let $\Gc$ be a stack. The group $\fen(\Gc)$ of {\em motivic functions} on $\Gc$ is the abelian
group generated by the symbols $[\Sc\buildrel s\over\to\Gc]$ for all special stacks $\Sc$ over
$\Gc$, subject to the following relations:
\begin{enumerate}
	\item Additivity in disjoint unions:
		\[
			[\Sc_1\sqcup \Sc_2 \buildrel s_1\sqcup s_2\over\lra \Gc] =  [\Sc_1\buildrel s_1\over\to\Gc] + 
			[\Sc_2\buildrel s_2\over\to\Gc].
		\]

	\item If $\phi: \Sc_1\to\Sc_2$ is a geometric bijection of special stacks, and $s_i:
		\Sc_i\to\Gc$ are such that $s_1=s_2\circ\phi$, then
		\[
			[\Sc_1\buildrel s_1\over\to\Gc] = [\Sc_2\buildrel s_2\over\to\Gc]. 
		\]

	\item Let $\Sc_i\buildrel s_i\over \to \Sc$, $i=1,2$, be two morphisms of special stacks
		with the same target.  Assume that for any scheme $S$ of finite type over $\FF$ and
		any morphism $p: S\to\Sc$, the pullbacks $\Sc_i\times_\Sc^{(2)}S$ are schemes and
		the projections to $S$ are locally trivial Zariski fibrations with equivalent fiber. Then,
		for any morphism $s:\Sc\to\Gc$, we impose the relation
		\[
			[\Sc_1\buildrel s\circ s_1\over\lra \Gc] = [\Sc_2\buildrel s\circ s_2\over\lra \Gc]. 
		\]
\end{enumerate} 
\end{defi}

\begin{ex} 
	The group $\fen(\FF):= \fen(\on{Spec}(\FF))$ is a ring, known as the {\em Grothendieck ring
	of special $\FF$-stacks}, with multipliciation induced by the Cartesian product. The reason
	for restricting to special stacks in Definition \ref{def:motivic-fun-stacks} is that it
	allows $\fen(\FF)$ to be identified with an explicit localization of a  similar but more
	``elementary" Grothendieck ring $\Lambda$ formed by $\FF$-schemes (not stacks) of finite
	type. More precisely, 
	\[
		\fen(\FF) =  \Lambda\bigl[ \LL^{-1}, \LL^n -1, n\geq 1\bigr],
	\]
	where $\LL$ is the class of the affine line over $\FF$, see Lemma 3.9 of \cite{bridgeland}.
\end{ex} 
 
Let $\phi: \Gc'\to\Gc$ be a morphism of stacks. Then we have the pushforward functor
\[
	\phi_*: \fen(\Gc')\lra\fen(\Gc), \quad [\Sc'\buildrel s'\over\to\Gc'] \mapsto [\Sc'\buildrel 
	\phi\circ s'\over\lra \Gc].
\]
If $\phi$ is a special morphism of stacks, we also have the pullback functor
\[
	\phi^*: \fen(\Gc) \lra\fen(\Gc'), \quad [\Sc\buildrel s\over\to\Gc'] \mapsto 
	[\Sc \times^{(2)}_{\Gc'} \Gc \to \Gc]. 
\]
 
\begin{prop}
The above functorialities make $\fen$ a theory with transfer on the model category $\underline{\Gr}_\Uc$ with respect
to the transfer structure $(\Sc p, \Mor(\underline{\Gr}_\Uc))$. 
\end{prop}
\begin{proof} The multiplicativity $\fen(\Gc)\otimes\fen(\Gc')\to\fen(\Gc\times\Gc')$ is given by
	Cartesian product of stacks. The base change for a 2-Cartesian square as in Definition
	\ref{def:transfer-structure} with $s_1, s_2$ special, is tautological, by definition of the
	functorialities of $\fen$. 
\end{proof}
 
Therefore (Definition \ref{def:hall-algebra-admissible}), each $2$-Segal simplicial object $X$ in
$\underline{\Gr}_\Uc$ which is admissible with respect to $(\Sc p, \Mor(\underline{\Gr}_\Uc))$,
gives rise to the Hall algebra $\Hall(X, \fen) = \fen(X_1)$ which can be called the {\em motivic
Hall algebra} of $X$. This includes the following examples. 
 
\begin{exas} 
\begin{exaenumerate}
	\item Let $R$ be a finitely generated associative $\FF$-algebra, and $X$
		be the Waldhausen stack of finite-dimensional left $R$-modules, see Example
		\ref{ex:waldhausen-stacks}(b). Then each $X_n$ is an Artin stack (locally of finite
		type), and which is, moreover, locally special. Indeed, for any field extension
		$\KK\supset\FF$, the stabilizer (automorphism group) of any finite dimensional
		$R\otimes_\FF \KK$-module is clearly an affine algebraic group over $\KK$.  It
		follows that the morphism $(\partial_2, \partial_0): X_2\to X_1$ is special, so $X$
		is an $(\Sc p, \Mor(\underline{\Gr}_\Uc))$-admissible $2$-Segal simplicial object. 

	\item Let $V$ be a projective algebraic variety over $\FF$.  Then $\Sc(\alg\Coh(V))$ and
		$\Sc(\alg\Bun(V))$ are $2$-Segal simplicial objects in $\underline{\Gr}_\Uc$. As
		before, we see that they are $(\Sc p, \Mor(\underline{\Gr}_\Uc))$-admissible. 
\end{exaenumerate}
\end{exas} 

\begin{rem} 
	In \cite[\S 4.1]{bridgeland}, Bridgeland emphasizes that the reason for associativity of the Hall
	algebra lies in a ``certain duality" between the stacks parametrizing flags of subobjects
	(monomorphisms) and quotient objects (epimorphisms). From our point of view, this corresponds to
	Lemma \ref{leq:waldhausen-filtration-abelian} and Proposition \ref{prop.inftyequiv}: the $n$th
	component of the Waldhausen space is weak equivalent to both types of flag spaces.  This is indeed
	the key element in the proof of the $2$-Segal property for general $\infty$-categorical Waldhausen
	spaces (Theorem \ref{thm:waldhausen-infty}), via the path space criterion (Theorem \ref{thm.crit}).  
\end{rem}

Finally, let us point out that the formalism of this section admits an extension to the model
category $\underline\Sp_\FF$ of $\infty$-stacks on $\Aff_\FF$, see Example \ref{ex:simplicial
sheaves}. This generalization proceeds by generalizing to $\infty$-stacks all the relevant concepts
used to construct the theory $\fen$ (special $\infty$-stacks, geometric equivalences and Zariski
fibrations of $\infty$-stacks). See \cite{toen-stacks} for these generalizations. This leads to a
theory with transfer $\fb$ on $\underline\Sp_\FF$ defined similarly to Definition
\ref{def:motivic-fun-stacks}. 

An important example to which the theory $\fb$ can be applied, is the Waldhausen $\infty$-stack
$\tau_{\leq 0} \Sc(\underline\Perf_\A)$ for a smooth and proper dg-category $\A$. It is defined as
the classical truncation (restriction from simplicial commutative algebras to ordinary commutative
algebras) of the derived Waldhausen stack $\Sc(\underline\Perf_\A)$ of perfect $\A$-modules, see
Proposition \ref{cor:wald-derived-stacks}.  As $\Sc(\underline\Perf_\A)$ is $2$-Segal, $\tau_{\leq 0}
\Sc(\underline\Perf_\A)$ is in turn a $2$-Segal simplicial object in $\underline\Sp_\FF$.  Applying
$\fb$ to $\tau_{\leq 0} \Sc_1(\underline\Perf_\A) = \Mc_\A$ gives then the {\em derived motivic Hall
algebra} of perfect $\A$-modules. Algebras of these type were first considered by Kontsevich and
Soibelman \cite{KS-motivic} by directly introducing the motivic analogs of the Baez-Dolan homotopy
cardinality into the multiplication rules. Their construction applies, in particular, to $\A =
 \A_V$, the dg enhancement of the bounded derived category of a smooth projective variety
$V$, see Remark \ref{rem:NC-smooth-proper}.
 A generalization to the non-smooth projective case was proposed by P. Lowrey \cite{lowrey}.
  
\vfill\eject

\section{Hall \texorpdfstring{$\inftytwo$}{(infty,2)}-categories}
\label{sec:higherhall}

In this chapter we lift the Hall algebra construction to the $\infty$-categorical
level, generalizing the elementary considerations of \S \ref{subsec:2-segal-monoidal}. Our
approach is based on associating to a space $B\in\Sp$ the model category as well as the $\infty$-category
of  all spaces over $B$. Such categories   play the role of the space of functions on a set (or groupoid)  $B$
in the classical approach to Hall algebras.

\subsection{Hall monoidal structures}
\label{subsec:homotopy.monoidal}

Let $X \in \Top_{\Delta}$ be a unital $2$-Segal topological space with weakly contratible space of
$0$-simplices. Replacing, if necessary,
 $X$ by  a weakly equivalent simplicial space, we can and will assume that $X$ is Reedy
fibrant and satisfies $X_0 = \pt$.  

The category $\Top{/X_1}$ of topological spaces over $X_1$ carries a unique model structure such
that the forgetful functor preserves weak equivalences, fibrations and cofibrations. As a first
step, we will construct a monoidal structure on the homotopy category of $\Top{/X_1}$. We denote
the resulting monoidal category by $\hHX$. As the notation suggests, the monoidal category $\hHX$ is in
fact the homotopy category of a monoidal $\infty$-category $\HX$ which will be constructed in \S
\ref{subsec:higherhall}.

We set $\mC = \Top{/X_1}$ and denote an object $A \to X_1$ of $\mC$ by its total space $A$. For each
pair of objects $A,B$ in $\mC$, we choose a pullback square
\begin{equation} \label{eq.deftensor}
	\begin{gathered}
		\xymatrix{
		A \otimes B \ar[r]\ar[d] & X_{\{0,1,2\}} \ar[d]\\
		A \times B \ar[r] & X_{\{0,1\}} \times X_{\{1,2\}},
		}
	\end{gathered}
\end{equation}
and interpret the composition 
\[
\xymatrix{
A \otimes B \ar[r] & X_{\{0,1,2\}} \ar[r] & X_{\{0,2\}}
}
\]
as an object of $\mC$. Note that, since $X$ is Reedy fibrant, the above square is in fact homotopy Cartesian. 
These choices extend to define a functor
\[
\otimes: \Ho(\mC) \times \Ho(\mC) \to \Ho(\mC),\; (A,B) \mapsto A \otimes B.
\]
We define the unit object $\one$ of $\Ho(\mC)$ to be given by the degeneracy map $X_0 \to X_1$.
For $B = \one$, the square \eqref{eq.deftensor} can be refined to a diagram
\begin{equation} 
	\begin{gathered}
		\xymatrix{
		A \otimes \one \ar[r]\ar[d] & X_{\{0,1\}} \ar[r]\ar[d] & X_{\{0,1,2\}} \ar[d]\\
		A \ar[r] & X_{\{0,1\}} \ar[r] & {X_{\{0,1\}}} \times X_{\{1,2\}},
		}
	\end{gathered}
\end{equation}
where the right square is homotopy Cartesian by the unitality of $X$ (Definition
\ref{def:unital-2-segal-top}). This implies that the vertical
map $A \otimes \one \to A$ is a weak equivalence and hence induces a functorial isomorphism
\[
\alpha_A: A \otimes \one \to A
\]
in $\Ho(\mC)$. Similarly, one obtains a functorial isomorphism
\[
\beta_A: \one \otimes A \to A.
\]
For each tripel of objects $A,B,C \in \mC$, we choose a pullback square
\begin{equation}
	\label{eq.tripletensor}
	\begin{gathered}
		\xymatrix{
		A \otimes B \otimes C \ar[r]\ar[d] & X_{\{0,1,2,3\}} \ar[d] \\
		A \times B \times C \ar[r] & X_{\{0,1\}} \times X_{\{1,2\}} \times X_{\{2,3\}},
		}
	\end{gathered}
\end{equation}
and interpret the composite
\[
\xymatrix{
A \otimes B \otimes C \ar[r] & X_{\{0,1,2,3\}} \ar[r] & X_{\{0,3\}}
}
\]
as an object of $\mC$. We claim that these choices uniquely determine a functorial isomorphism
\[
\eta_{A,B,C}: (A \otimes B) \otimes C \to A \otimes (B \otimes C)
\]
in $\Ho(\mC)$. Indeed, from the defining Cartesian squares \eqref{eq.deftensor} of $\otimes$ we obtain a
canonical Cartesian square
\[
\xymatrix{
(A \otimes B) \otimes C \ar[r]\ar[d] & X_{\{0,1,2\}} \times_{X_{\{0,2\}}} X_{\{0,2,3\}}\ar[d]\\
A \times B \times C \ar[r] & X_{\{0,1\}} \times X_{\{1,2\}} \times X_{\{2,3\}}
}
\]
which, using \eqref{eq.tripletensor}, can be extended to the diagram
\[
\xymatrix{
A \otimes B \otimes C \ar[rr] \ar@{-->}[dr] \ar@/_1pc/[ddr] & & X_{\{0,1,2,3\}} \ar[d]^{\simeq}\\
&(A \otimes B) \otimes C \ar[r]\ar[d] & X_{\{0,1,2\}} \times_{X_{\{0,2\}}} X_{\{0,2,3\}}\ar[d]\\
&A \times B \times C \ar[r] & X_{\{0,1\}} \times X_{\{1,2\}} \times X_{\{2,3\}}
}
\]
where the dashed arrow is canonical and further a weak equivalence. By an analogous statement for
$A \otimes (B \otimes C)$, we obtain a canonical diagram 
\[
\xymatrix{
(A \otimes B) \otimes C & \ar[l]_{\simeq} A \otimes B \otimes C \ar[r]^{\simeq} & A \otimes (B \otimes C)
}
\]
which induces the desired isomorphism $\eta_{A,B,C}$ in $\Ho(\mC)$. It is easy to verify that
$\eta_{A,B,C}$ is functorial in its arguments. 

\begin{thm} \label{thm:monoidalhall} The data $(\Ho(\mC), \otimes, \alpha, \beta, \eta)$ forms a monoidal category.
\end{thm}

\begin{proof} We verify MacLane's pentagon relation leaving the remaining compatibilities to the
	reader (they will also follow from \S \ref{subsec:higherhall}). For each quadrupel $A,B,C,D$
	of objects in $\mC$ we choose a pullback square
	\begin{equation}
		\label{eq.quadtensor}
		\begin{gathered}
			\xymatrix{
			A \otimes B \otimes C \otimes D \ar[r]\ar[d] & X_{\{0,1,2,3,4\}} \ar[d] \\
			A \times B \times C \times D \ar[r] & X_{\{0,1\}} \times X_{\{1,2\}} \times X_{\{2,3\}} \times X_{\{3,4\}},
			}
		\end{gathered}
	\end{equation}
	Using the universal properties of the chosen squares \eqref{eq.deftensor},
	\eqref{eq.tripletensor}, \eqref{eq.quadtensor}, we can construct a canonical diagram 
	\begin{equation}
		\begin{gathered}
		\label{eq.onetriangle}
		\xymatrix{
		A \otimes B \otimes C \otimes D \ar[r]\ar[d]^f \ar@/_5pc/[dd]_{g \circ f} &
		X_{\{0,1,2,3,4\}} \ar[d]_{\simeq}^{f'} \\
		(A \otimes B \otimes C) \otimes D \ar[r] \ar[d]^g & X_{\{0,1,2,3\}}
		\times_{X_{\{1,3\}}} X_{\{1,3,4\}} \ar[d]_{\simeq}^{g'}\\
		( (A \otimes B) \otimes C) \otimes D \ar[r] \ar[d] & X_{\{0,1,2\}} \times_{ X_{\{1,2\}}}
		X_{\{0,2,3\}} \times_{X_{\{3,4\}}} X_{\{0,3,4\}} \ar[d]\\
		A \times B \times C \times D \ar[r] & X_{\{0,1\}} \times X_{\{1,2\}} \times
		X_{\{2,3\}} \times X_{\{3,4\}}.
		}
	\end{gathered}
	\end{equation}
	The maps $f'$ and $g'$ are $2$-Segal maps and hence weak equivalences. By Definition, all
	squares in the diagram are Cartesian. 
	Therefore, the maps $f$, $g$ and $g \circ f$ are weak equivalences which are uniquely
	determined by universal properties. The analogous diagrams for all possible bracketings of the expressions $A \otimes B \otimes C
	\otimes D$ assemble to form the diagram depicted in Figure \ref{fig:maclanepentagon} in which all triangles commute (the
	labeled triangle corresponds to \eqref{eq.onetriangle}).
	\begin{figure}[h]
		\[ 
		\begin{gathered}
		\begin{tikzpicture}[scale=1.4,font=\scriptsize]

			\node (origin) at (0,0) (origin) {$A \otimes B \otimes C \otimes D$}; 
			\node (P0) at (0:3.5cm)  { $( (A \otimes B) \otimes C) \otimes D$}; 
			\node (P2) at (1*72:3.5cm) { $( A \otimes (B \otimes C)) \otimes D$}; 
			\node (P4) at (2*72:3.5cm) { $ A \otimes ((B \otimes C) \otimes D)$}; 
			\node (P1) at (3*72:3.5cm) { $ A \otimes (B \otimes (C \otimes D))$}; 
			\node (P3) at (4*72:3.5cm) { $ (A \otimes B) \otimes (C \otimes D)$}; 
			\path (P0) -- (P2)
			node [midway] (P02) {$(A \otimes B \otimes C) \otimes D$}; 
			\path (P2) -- (P4) node [midway] (P24) {$A \otimes (B \otimes C) \otimes D$};
			\path (P4) -- (P1) node [midway] (P14) {$A \otimes (B \otimes C \otimes D)$};
			\path (P1) -- (P3) node [midway] (P13) {$A \otimes B \otimes (C \otimes D)$};
			\path (P0) -- (P3) node [midway] (P03) {$(A \otimes B) \otimes C \otimes D$}; 
			\draw [-latex,  thick] (P02) -- (P2);
			\draw  [-latex, thick] (origin) -- (P2);
			\draw  [-latex,  thick] (origin) -- (P1);
			\draw  [-latex,  thick] (origin) -- (P3);
			\draw  [-latex,  thick] (origin) -- (P4);
			\tikzset{auto}
			\draw  [-latex,  thick] (origin) to node [swap] {$f$} (P02);
			\draw  [-latex,  thick] (origin) to node {$g \circ f$} (P0);
			\draw  [-latex,  thick] (origin) -- (P24);
			\draw  [-latex,  thick] (origin) -- (P14);
			\draw  [-latex,  thick] (origin) -- (P13);
			\draw  [-latex,  thick] (origin) -- (P03);
			\draw   [-latex,  thick] (P02) to node [swap]{$g$} (P0);
			\draw [-latex,  thick] (P24) -- (P2); 
			\draw [-latex,  thick] (P24) -- (P4); 
			\draw [-latex,  thick] (P14) -- (P4); 
			\draw [-latex,  thick] (P14) -- (P1); 
			\draw [-latex,  thick] (P13) -- (P3); 
			\draw [-latex,  thick] (P13) -- (P1); 
			\draw [-latex,  thick] (P03) -- (P0); 
			\draw [-latex,  thick] (P03) -- (P3); 

		\end{tikzpicture}
		\end{gathered}
		\]
		\caption{MacLane's pentagon for $\hHX$}
		\label{fig:maclanepentagon}
	\end{figure}
	Passing to the homotopy category $\Ho(\mC)$ we deduce the commutativity of MacLane's pentagon.
\end{proof}

\begin{rem}
	Note that the construction of the monoidal category $\hHX$ only involves the $4$-skeleton of the
	$2$-Segal space $X$. In \S \ref{subsec:higherhall}, we refine Theorem \ref{thm:monoidalhall} by
	constructing a monoidal $\infty$-category $\HX$ whose homotopy category is given by the monoidal category
	$\hHX$. The construction of $\HX$ will utilize the full simplicial structure of $X$.
\end{rem}

Finally, we describe the relation of the monoidal structure on $\h\HX$ to the derived Hall algebras
of \S \ref{subsec:derivedhall} (whose construction only involves the $3$-skeleton of $X$). 
The following Proposition is immediately verified.

\begin{prop}\label{prop:hf} Let $X$ be a Reedy fibrant, unital $2$-Segal topological space which is
	admissible with respect to the transfer structure (weakly proper maps, locally proper maps).
	Assume further that $X$ satisfies $X_0 = \pt$. Consider the full subcategory $\hHX_{\on{hf}} \subset \hHX$ spanned by those maps $Y
	\to X_1$ such that $Y$ is homotopy finite. Then the monoidal structure on
	$\hHX$ restricts to a monoidal structure on the category $\hHX_{\on{hf}}$.
\end{prop}

In the situation of Proposition \ref{prop:hf}, denote by $M$ the monoid of isomorphism classes of
objects in $\hHX_{\on{hf}}$. We form the semigroup algebra $\QQ[M]$. For $x \in X_1$, let $C_x
\subset X_1$ denote the connected component represented by $x$. The isomorphism class of $C_x
\subset X_1$, considered as an object of $\hHX_{\on{hf}}(X)$ provides an element of $M$ which we denote by $[x]$.  
Then there exists a natural surjective homomorphism of $\QQ$-algebras
\[
	\pi: \QQ[M] \lra \Hall(X, \Fenh_0), (Y \overset{f}{\to} X_1) \mapsto \sum_{[x] \in \pi_0(X_1)}
	|Rf^{-1}(x)|_h [x],
\]
where $\Hall(X, \Fenh_0)$ denotes the Hall algebra from \S \ref{subsec:derivedhall}. This shows that
the derived Hall algebra $\Hall(X, \Fenh_0)$ can be recovered from the monoidal structure on
$\hHX$.

\subsection{Segal fibrations and \texorpdfstring{$\inftytwo$}{(infty,2)}-categories}

In analogy to the situation for $\inftyone$-categories, there are various models for the notion of an
$\inftytwo$-category. To describe the bicategorical structures appearing in this work, we will use
{\em Segal fibrations}. In fact, we will also use the dual notion of a {\em co}Segal fibration.  These
and other models for $\inftytwo$-categories, as well as their relations, are studied in detail in the
comprehensive treatment \cite{lurie.2cat}.

\begin{defi} \label{defi:dinftybi} A map $p: Y \to \ND$ of simplicial sets is called a {\em
	Segal fibration} if it satisfies the following conditions:
	\begin{enumerate}[label=(S\arabic{*})]
		\item \label{dinftybi.1} The map $p$ is a Cartesian fibration.
		\item \label{dinftybi.2} For every $n \ge 2$, the diagram
			
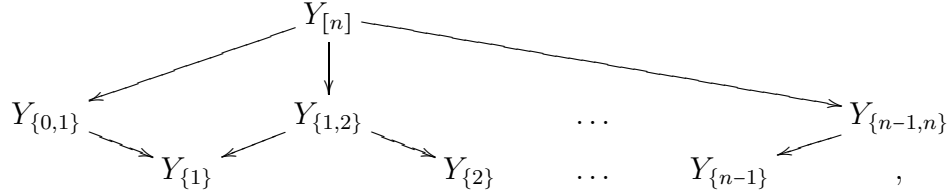
\[
				\xymatrix@R=0.5ex{  & &Y_{[n]} \ar[ddd] \ar[dddrrrr]\ar[dddll]&  &  \\ 
				& & &  &  \\ 
				& & &  &  \\ 
				 Y_{\{0,1\}}\ar[dr] & &\ar[dl]\ar[dr] Y_{\{1,2\}} && \dots&
				&\ar[dl] Y_{\{n-1,n\}}   \\
				& Y_{\{1\}} &&Y_{\{2\}}& \dots &Y_{\{n-1\}}& ,
				}	
			\]
			induced by the functors associated with the Cartesian fibration $p$, 
			is a limit diagram in the $\infty$-category $\Cat_{\infty}$.
		\item \label{dinftybi.3} The $\infty$-category $Y_{[0]}$ is a Kan complex.
	\end{enumerate}
\end{defi}

A Segal fibration $p: Y \to \ND$ models an $\inftytwo$-category $\B$ with
set of objects given by the vertices of $Y_{[0]}$. Given objects $x,y$ of $\B$, the
$\infty$-category $\Map_{\B}(x,y)$ of $1$-morphisms between $x$ and $y$ is defined as the limit of
the diagram of $\infty$-categories
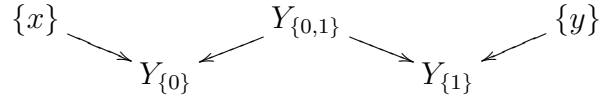
\[
	\xymatrix@R=0.5ex{ 
		\{x\} \ar[rd] & &\ar[dl] Y_{\{0,1\}} \ar[dr] & &\ar[dl] \{y\} \\
		& Y_{\{0\}} & & Y_{\{1\}} & 
	}
\]
involving the functors associated with the Cartesian fibration $p$.  Further, by using similar
arguments as for $1$-Segal spaces, we can use property \ref{dinftybi.2} to obtain coherently
associative composition functors
\[
	\Map_{\B}(x_1,x_2) \times \Map_{\B}(x_2,x_3) \times \dots \times \Map_{\B}(x_{n-1},x_n) \lra
	\Map_{\B}(x_1,x_n)
\]
in $\h \Cat_{\infty}$.

\begin{ex} \label{ex:tau} 
	Let $p: Y \to \ND$ be a Segal fibration. We define $K \subset Y$ to be the simplicial
	subset consisting of those simplices with all edges $p$-Cartesian. Then the restriction
	$p_{|K} : K \to \ND$ is a right fibration (\cite[2.4.2.5]{lurie.htt}) which corresponds
	under the Grothendieck construction \cite[2.2.1.2]{lurie.htt} to a simplicial space
	$\tau^{\le 1}(Y): \Dop \to \sSet$. It is easy to verify that $\tau^{\le 1}(Y)$ is a $1$-Segal space. The
	$\inftyone$-category corresponding to $\tau^{\le 1}(Y)$ is obtained from the $\inftytwo$-category
	modelled by the Segal fibration $p$ by discarding non-invertible $2$-morphisms. Note, however, that the
	$1$-Segal space $\tau^{\le 1}(Y)$ is not necessarily complete. 
\end{ex}

\begin{defi}\label{defi:bicomplete} 
	Let $p: Y \to \ND$ be Segal fibration. We say $p$ is {\em complete} if the $1$-Segal space
	$\tau^{\le 1}(Y)$ from Example \ref{ex:tau} is complete. 
\end{defi}

\begin{rem}
	Under the Grothendieck construction \cite[3.2.0.1]{lurie.htt}, the completeness condition of
	Definition \ref{defi:bicomplete} corresponds to the respective condition for Segal objects
	in $\Cat_{\infty}$ introduced in \cite[\S 1.2]{lurie.2cat}. As shown in loc. cit., complete
	Segal fibrations and complete Segal objects in $\Cat_{\infty}$ provide equivalent models for
	$\inftytwo$-categories. 
\end{rem}

\begin{ex} \label{ex:monoidalcomplete} 
	Let $p: Y \to \ND$ be a Segal fibration and assume that the Kan complex $Y_{[0]}$ is weakly
	contractible. In this case, we say $p$ {\em exhibits a monoidal structure on the
	$\infty$-category $\C = Y_{[1]}$}. As explained in detail in \cite[\S
	1]{lurie.noncommutative}, the Cartesian fibration $p$ equips the homotopy category $\h \C$
	with a monoidal structure. Beyond that, the functors 
	\[
		\C^n \lra \C, \; (y_1, y_2, \dots, y_n) \mapsto y_1 \otimes y_2 \otimes \dots
		\otimes y_n
	\]
	associated to the fibration $p$, where $n \ge 2$, encode a coherently associative system of
	functors of $\infty$-categories. 

	The monoidal structure on $\h \C$ equips the set $\pi_0( \h \C)$ of isomorphism classes of
	objects in $\C$ with the structure of a monoid.  Let $P \subset \C$ be the largest
	simplicial subset such that
	\begin{enumerate}
		\item for every vertex $x$ of $P$, the corresponding class $[x]$ in the monoid
			$\pi_0(\h \C)$ is invertible,
		\item every edge in $P$ is an equivalence in $\C$.
	\end{enumerate}
	The simplicial set $P$ is a Kan complex which models the classifying space of objects in
	$\C$ which are invertible with respect to $\otimes$. The Segal fibration $p$ is complete if
	and only if $P$ is weakly contractible.
\end{ex}

We introduce a notion dual to Definition \ref{defi:dinftybi}.

\begin{defi} \label{defi:inftybi} 
	A map $p: Y \to \NDop$ of simplicial sets is called a {\em coSegal fibration} if the
	opposite map $p^{\op}: Y^{\op} \to \ND$ is a Segal fibration. Explicitly, this amounts to
	the following conditions:
	\begin{enumerate}[label=(CS\arabic{*}$)$]
		\item \label{inftybi.1} The map $p$ is a coCartesian fibration.
		\item \label{inftybi.2} For every $n \ge 2$, the diagram
			
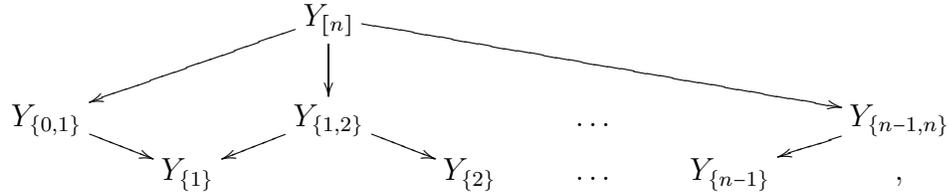
\[
				\xymatrix@R=0.5ex{  & &Y_{[n]} \ar[ddd] \ar[dddrrrr]\ar[dddll]&  &  \\ 
				& & &  &  \\ 
				& & &  &  \\ 
				 Y_{\{0,1\}}\ar[dr] & &\ar[dl]\ar[dr] Y_{\{1,2\}} && \dots&
				&\ar[dl] Y_{\{n-1,n\}}   \\
				& Y_{\{1\}} &&Y_{\{2\}}& \dots &Y_{\{n-1\}}& ,
				}	
			\]
			induced by the functors associated with the coCartesian fibration $p$, 
			is a limit diagram in the $\infty$-category $\Cat_{\infty}$.
		\item \label{inftybi.3} The $\infty$-category $Y_{[0]}$ is a Kan complex.
	\end{enumerate}
	We call a coSegal fibration {\em complete} if its opposite map is a complete Segal fibration.
\end{defi}

\begin{rem}\label{rem:segop}
	As for Segal fibrations, given a coSegal fibration $p: Y \to \NDop$, we obtain an
	$\inftytwo$-category $\B$ with set of objects given by the vertices of $Y_{[0]}$. For objects
	$x,y$, we can define the $\infty$-category of $1$-morphisms between $x$ and $y$ as the limit
	of the diagram of $\infty$-categories
	\[
		\xymatrix@R=0.5ex{ 
			\{x\} \ar[rd] & &\ar[dl] Y_{\{0,1\}} \ar[dr] & &\ar[dl] \{y\} \\
			& Y_{\{0\}} & & Y_{\{1\}} & 
		}
	\]
	where the functors are now obtained from the {\em co}Cartesian fibration $p$. Note, however,
	that the $\inftytwo$-category $\B_p$ corresponding to $p$ is not generally equivalent to the
	$\inftytwo$-category $\B_{p^{\op}}$ modelled by the opposite Segal fibration $p^{\op}$.
	Rather, $\B_{p^{\op}}$ is the $\inftytwo$-category obtained by passing to opposites at all
	levels of morphisms.
	
	Passing from a given Segal fibration, modelling a certain $\inftytwo$-category $\B$, to the
	coSegal fibration which models {\em the same} $\inftytwo$-category is a rather tedious and
	inexplicit process: For example, one can use the functor $\C \mapsto \C^{\op}$ on $\Cat_{\infty}$
	defined in \cite[1.2.16]{lurie.noncommutative} in combination with the Grothendieck
	construction \cite[3.2.0.1]{lurie.htt}.
\end{rem}

While Segal fibrations and coSegal fibrations furnish equivalent models for $\inftytwo$-categories,
the models show differences with regard to different notions of {\em lax} functors which will be defined
now.
A morphism $f: [m] \to [n]$ of ordinals is called {\em convex} if $f$ is injective and the image
$\{f(0),\dots,f(m)\} \subset [n]$ is a convex subset.

\begin{defi} \label{defi:leftlax} Let $p: Y \to \ND$ and $q: Y' \to \ND$ be Segal fibrations. A {\em left lax}
	functor between $p$ and $q$ is a map of simplicial sets $F: Y \to Y'$ such that the diagram
	\[
		\xymatrix@C=1ex{ Y \ar[dr]_p \ar[rr]^F & & Y'\ar[dl]^q \\
			& \ND &
		}
	\]
	commutes and, for every $p$-Cartesian edge $e$ of $Y$ such that $p(e)$ is convex, the edge
	$F(e)$ is $q$-Cartesian. A map between coSegal fibrations is called {\em right lax} if its
	opposite is a left lax functor of Segal fibrations.
\end{defi}

\begin{ex} Let $p: Y \to \ND$ and $q: Y' \to \ND$ be Segal fibrations with contractible $[0]$-fiber. As
	explained in Example \ref{ex:monoidalcomplete} this means that $p$ and $q$ model monoidal $\infty$-categories. 
	Informally, a left lax functor $F$ between $Y$ and $Y'$ corresponds to a functor 
	\[
		f: Y_{[1]} \to Y'_{[1]}
	\]
	of underlying $\infty$-categories together with a coherent system of maps
	\[
		f(y_1 \otimes y_2 \otimes \dots \otimes y_n) \lra f(y_1) \otimes f(y_2) \otimes
		\dots \otimes f(y_n)
	\]
	which are not required to be equivalences. In contrast, assuming the monoidal $\infty$-categories to be
	modelled by coSegal fibrations, a right lax functor corresponds to a coherent system of maps
	\[
		f(y_1) \otimes f(y_2) \otimes \dots \otimes f(y_n) \lra f(y_1 \otimes y_2 \otimes \dots \otimes y_n)
	\]
\end{ex}

\begin{rem} Since both complete Segal and coSegal fibrations are models for the notion of
	an $\inftytwo$-category, Definition \ref{defi:leftlax} provides the collection of
	$\inftytwo$-categories with two different kinds of morphisms: left lax and right lax functors.

  	In view of Remark \ref{defi:inftybi}, it is not clear how to effectively describe right
	lax functors using Segal fibrations, and, vice versa, left lax functors using coSegal
	fibrations. 
	
	We expect the association $X \mapsto H(X)$ to be functorial with respect to left lax
	functors. Since we constructed $H(X)$ as a coSegal fibration, it is unclear how to express
	this functoriality in the current context. This problem will be resolved in \S
	\ref{section:higher-bicat} where we use structures which are self-dual, so that they can be
	described both in terms of Segal and coSegal fibrations.
\end{rem}

\subsection{The Hall \texorpdfstring{$\inftytwo$}{(infty,2)}-category of a $2$-Segal space}
\label{subsec:higherhall}

In this section, we associate to a $2$-Segal space $X$ a coSegal fibration $\HX \to \NDop$,
in the sense of Definition \ref{defi:inftybi}. We call the corresponding $\inftytwo$-category the {\em Hall $\inftytwo$-category of $X$}. 
To construct $\HX$, we will first associate to any simplicial space $X$ a coCartesian fibration $\pSX \to \NDop$.
We will then show that, if $X$ is a $2$-Segal space, we can restrict to a coSegal fibration 
\[
\xymatrix{
\SX \ar@{^{(}->}[r] \ar[dr] & \ar[d] \pSX\\
& \NDop
}.
\]

The starting point of our construction is analogous to the construction of the Cartesian monoidal
structure associated to an $\infty$-category with finite limits (see \cite[\S 1.2]{lurie.noncommutative}). 
We define a category $\Delta^{\times}$ as follows:
\begin{itemize}
	\item The objects $\Delta^{\times}$ are given by pairs $([n], \{i,\dots,j\} )$ where $[n]$ is
		an object of $\Delta$ and $\{i,\dots,j\}$ is an interval in $\{0,\dots,n\}$. 
	\item A morphism $f: ([n], \{i,\dots,j\}) \to ([m],\{i,\dots,j\})$ is given by a morphism $f : [n] \to [m]$ in $\Delta$
		satisfying $f(\{i,\dots,j\}) \subset \{i',\dots,j'\}$.
\end{itemize}
The forgetful functor $(\Delta^{\times})^{\op} \to \Dop$ is a Grothendieck fibration which
implies that the induced functor $\N(\Delta^{\times})^{\op} \to \NDop$ is a Cartesian fibration
of $\infty$-categories. 

Consider the $\infty$-category $\inftyS$ of spaces, defined as the simplicial nerve of the full simplicial
subcategory in $\sSet$ spanned by the Kan complexes.
We define a map of simplicial sets 
\[
	q: \pS^{\times} \to \NDop
\] 
via the following universal property: For all maps $K \to \NDop$, we have a natural bijection
\begin{equation} \label{eq.cocartesianfib}
	\Hom_{\sSet{/\NDop}}(K, \pS^{\times}) \cong \Hom_{\sSet}(K \times_{\NDop}
	\N(\Delta^{\times})^{op}, \inftyS).
\end{equation}
For $n \ge 0$, the fiber $\pS^{\times}_{[n]}$ of $q$ over $[n]$ can be identified with the $\infty$-category of 
functors
\[
	Y: \N(I_{[n]})^{\op} \to \inftyS,
\]
where $I_{[n]}$ denotes the poset of nonempty intervals of $[n]$. For example, the objects of
$\pS^{\times}_{[2]}$ correspond to homotopy coherent diagrams in $\inftyS$ of the form
\[
	\xymatrix@C=1ex@R=1ex{
	& &\ar[dl]\ar[dr]  Y_{\{ 0,1,2 \}} & & \\
	& \ar[dl]\ar[dr] Y_{\{ 0,1 \}} & &\ar[dl]\ar[dr]  Y_{\{ 1,2 \}} &\\
	Y_{\{ 0 \}} & & Y_{\{ 1 \}} &  &  Y_{\{ 2 \}}.
	}
\]

\begin{prop} The projection map $q: \pS^{\times} \to \NDop$ is a coCartesian fibration.  The
	$q$-coCartesian edges of $\pS^{\times}$ covering $f: [n] \to [m]$ are those edges $Y \to Y'$
	such that, for all $0 \le i \le j \le n$, the induced map $Y_{\{f(i), \dots, f(j)\}} \to
	{Y'}_{\{i,\dots,j\}}$ is an equivalence of spaces. 
\end{prop}
\begin{proof}\cite[3.2.2.12]{lurie.htt}
\end{proof}

Let $X$ be a Reedy fibrant simplicial space. The functor of categories
\[
	(\Delta^{\times})^\op \to (\sSet)^{\circ}, \; ([n], \{i,\dots,j\}) \mapsto X_{\{i,\dots,j\}}
\]
induces a functor $\N(\Delta^{\times})^{\op} \to \inftyS$
of $\infty$-categories. Evaluating the adjunction \eqref{eq.cocartesianfib} for $K = \NDop$, we obtain
a section
\[
  \xymatrix@C=1ex{
    \NDop \ar[rr]^-{s_X} \ar[dr]_-{\id} & & \ar[dl]^-{q} \pS^{\times}\\
    	&	\NDop &
	}
\]
of the map $q$. 
We define the simplicial set $\pSX$ to be the overcategory $(\pS^{\times})^{/s_X}$ relative to
$\NDop$ (see \cite[4.2.2]{lurie.htt}). 

\begin{exa} The objects of $\pSX_{[2]}$ can be
identified with edges in $\Fun(\N(I_{[2]})^{\op}, \inftyS)$
\[
	\xymatrix@C=1ex@R=1ex{
	& &\ar[dl]\ar[dr]  Y_{\{ 0,1,2 \}} & & \\
	& \ar[dl]\ar[dr] Y_{\{ 0,1 \}} & &\ar[dl]\ar[dr]  Y_{\{ 1,2 \}} &\\
	Y_{\{ 0 \}} & & Y_{\{ 1 \}} &  &  Y_{\{ 2 \}}
	} 
	\xymatrix@C=1ex@R=1ex{\\
	\quad\lra\quad\\
	}
	\xymatrix@C=1ex@R=1ex{
	& &\ar[dl]\ar[dr]  X_{\{ 0,1,2 \}} & & \\
	& \ar[dl]\ar[dr] X_{\{ 0,1 \}} & &\ar[dl]\ar[dr]  X_{\{ 1,2 \}} &\\
	X_{\{ 0 \}} & & X_{\{ 1 \}} &  &  X_{\{ 2 \}}.
	}
\]
Note, that $X_{\{i,\dots,j\}} = X_{j-i}$ while, for example, the spaces $Y_{\{i\}}$ and $Y_{\{j\}}$ are
generally unrelated for $i \ne j$.
\end{exa}

By the dual statement of \cite[4.2.2.4]{lurie.htt}, the natural map $q_X: \pSX \to
\NDop$ is a coCartesian fibration and an edge of $\pSX$ is $q_X$-coCartesian if and only
if its image in $\pS^{\times}$ is $q$-coCartesian. 

\begin{defi}\label{defi.SX}
We define $\HX \subset \pSX$ to be the full simplicial subset spanned by those vertices $Y \in
\pSX_{[n]}$, $n \ge 0$, which satisfy the following conditions: 
\begin{enumerate}[label=(H\arabic*)]
	\item \label{m1} For all $0 \le i \le n$, the space $Y_{\{i\}}$ is contractible. 
	\item \label{m2} If $n > 0$ then, for all $0 \le i \le j \le n$, the square 
		\[
		\xymatrix{
		Y_{\{i, \dots, j\}} \ar[r] \ar[d]& \ar[d]Y_{\{i, i+1\}} \times_{Y_{\{i+1\}}} \dots
		\times_{Y_{\{j-1\}}} Y_{\{j-1,j\}} \\ 
		X_{\{i, \dots, j\}} \ar[r] & X_{\{i, i+1\}} \times_{X_{\{i+1\}}} \dots
		\times_{X_{\{j-1\}}} X_{\{j-1,j\}}  
		}
		\]
	      is Cartesian. Here, the fiber products are to be understood as $\infty$-categorical
	      limits.
\end{enumerate}
\end{defi}

\begin{rem} In the context of Definition \ref{defi.SX}, if $Y \in \pSX_{[n]}$
	satisfies condition \ref{m1}, then requiring the square in \ref{m2} to be
	Cartesian is equivalent to requiring the square
	\[
		\xymatrix{
		Y_{\{i, \dots, j\}} \ar[r] \ar[d]& \ar[d]Y_{\{i, i+1\}} \times \dots
		\times Y_{\{j-1,j\}} \\ 
		X_{\{i, \dots, j\}} \ar[r] & X_{\{i, i+1\}} \times_{X_{\{i+1\}}} \dots
		\times_{X_{\{j-1\}}} X_{\{j-1,j\}}  
		}
	\]
	to be Cartesian.
\end{rem}

\begin{thm} \label{thm:2segalcosegal} Let $X \in \sS$ be a Reedy fibrant unital $2$-Segal space. Then the map $p: \HX \to \NDop$, 
	obtained by restricting $q_X$, is a coSegal fibration. 
\end{thm}

\begin{proof} We have to verify the conditions of Definition \ref{defi:inftybi}. First, note that
	\ref{inftybi.3} immediately follows from condition \ref{m1}: the $\infty$-category
	$\HX_{[0]}$ is equivalent to the $\infty$-groupoid represented by the Kan complex $X_0$. 
	
	Since $\HX \subset \pSX$ is a full simplicial subset, it is obvious that $p$ is an inner
	fibration. To demonstrate condition \ref{inftybi.1}, we thus have to prove that every edge $f$
	of $\NDop$, corresponding to a map $f:[n] \to [m]$ in $\Delta$, can be lifted to a $p$-coCartesian edge in $\HX$ with prescribed initial vertex
	$Y \in \HX_{[m]}$. Since the projection $q_X: \pSX \to \NDop$ is a coCartesian fibration, there
	exists a $q_X$-coCartesian edge $e: Y \to Y'$ in $\pSX$ which covers $f$. It suffices to
	verify that $Y'$ lies in $\HX_{[n]} \subset \pSX_{[n]}$. As above, 
	we use the notation
	\[
	  {Y'}_{\{i,\dots,j\}} := Y'([n],(i,j))
	\]
	and the analogous notation for $Y$. 
	Since $e$ is $q_X$-coCartesian, for all $0 \le i \le j \le n$, the associated maps
	\[
	Y_{\{f(i),\dots,f(j)\}} \to Y'_{\{i,\dots,j\}}
	\]
	are equivalences of spaces. In particular, choosing $i = j$, we deduce that, for each $i$, the space
	$Y'_{\{i\}}$ is contractible. It remains to show that $Y'$ satisfies condition \ref{m2} of Definition
	\ref{defi.SX}. For each interval $\{i,\dots,j\}$ in $[n]$, we have 
	a square
	\[
	\xymatrix{
	Y_{\{f(i),\dots,f(j)\}} \ar[r]^{\simeq}\ar[d] & Y'_{\{i,\dots,j\}} \ar[d]\\
	X_{\{f(i),\dots,f(j)\}} \ar[r] & X_{\{i,\dots,j\}}
	}
	\]
	associated to $e$. Therefore, to show that $Y'$ satisfies condition \ref{m2}, it suffices to show
	that, for each interval $\{i,\dots,j\}$, the square
	\begin{equation}\label{eq:supersquare}
	    \xymatrix{
	      Y_{\{f(i),\dots,f(j)\}} \ar[r]\ar[d] & Y_{\{f(i),\dots,f(i+1)\}} \times \dots \times
	      Y_{\{f(j-1),\dots,f(j)\}} \ar[d]\\
	      X_{\{i, \dots, j\}} \ar[r] & X_{\{i, i+1\}} \times_{X_{\{i+1\}}} \dots
	      \times_{X_{\{j-1\}}} X_{\{j-1,j\}}  
	    }
	\end{equation}
	is Cartesian. This square can be obtained as the vertical rectangle in the diagram
	\[
		\xymatrix@!C=10pc{
		Y_{\{f(i),\dots,f(j)\}} \ar[r]\ar[d] & 
		Z_1 \ar[d]\ar[r] & \ar[d] 
		Y_{\{f(i),f(i)+1\}}\times \dots \times Y_{\{f(j)-1,f(j)\}}\\
		\ar[d] X_{\{f(i), \dots, f(j)\}} \ar[r] & 
		Z_2 \ar[d]\ar[r] & 
		X_{\{f(i),f(i)+1\}} \times_{X_{\{f(i)+1\}}} \dots
		\times_{X_{\{f(j)-1\}}} X_{\{f(j)-1,f(j)\}}\\
		X_{\{i, \dots, j\}} \ar[r] & 
		X_{\{i, i+1\}} \times_{X_{\{i+1\}}} \dots
		\times_{X_{\{j-1\}}} X_{\{j-1,j\}} &
		}
	\]
	with 
	\[
		Z_1 = Y_{\{f(i),\dots,f(i+1)\}}\times \dots \times Y_{\{f(j-1),\dots,f(j)\}}
	\]
	and
	\[
		Z_2 = X_{\{f(i),\dots, f(i+1)\}} \times_{X_{\{f(i+1)\}}} \dots
		\times_{X_{\{f(j-1)\}}} X_{\{f(j-1),\dots,f(j)\}}.
	\]
	By condition \ref{m2} for $Y$, the top right square and the horizontal rectangle are Cartesian. Using
	\cite[4.4.2.1]{lurie.htt}, we deduce that the top left square is Cartesian. Since, by
	assumption, the simplicial space $X$ is a unital $2$-Segal space, Proposition \ref{prop.key} below implies that
	the lower square and thus the vertical rectangle is Cartesian.

	It remains to verify condition \ref{inftybi.2} of Definition \ref{defi:inftybi}. From
	\cite[2.4.7.12]{lurie.htt}, we conclude that, for each $i$, the functors
	$\HX_{\{i,i+1\}} \to \HX_{\{i\}}$ and $\HX_{\{i,i+1\}} \to \HX_{\{i+1\}}$ are
	coCartesian fibrations of $\infty$-categories. In particular, combining \cite[2.4.1.5,2.4.6.5]{lurie.htt}, these maps are
	categorical fibrations, i.e., fibrations with respect to the Joyal model structure on $\sSet$.
	Since this model structure models the $\infty$-category $\Cat_{\infty}$ of
	$\infty$-categories, we can utilize it to calculate limits in $\Cat_{\infty}$. This implies
	that the ordinary fiber product of simplicial sets
	\[
		\C = \HX_{\{i,i+1\}}
		\times_{ \HX_{\{i+1\}}} \dots
		\times_{ \HX_{\{n-1\}}} \HX_{\{n-1,n\}}
	\]
	is in fact a homotopy fiber product. As above, we use the notation $I_{[n]} = [n] \times_{\Delta}
	\Delta^{\times}$ such that the $\infty$-category $\pSX_{[n]}$ is by definition the
	$\infty$-category $\Fun(\N(I_{[n]})^\op, \inftyS)^{/F}$, where we set $F = s_X([n])$.
	Let $I^{\le 1}_{[n]} \subset I_{[n]}$ denote the full subcategory spanned by the intervals of length
	$\le 1$. Then the $\infty$-category
	$\C$ can be identified with the $\infty$-category $\Fun(\N(I^{\le 1}_{[n]})^{\op},
	\inftyS)^{/F^0}$, where $F^0$ is defined to be the restriction of $F$ to $\N(I^{\le
	1}_{[n]})^{\op}$. Now let $Y \in \Fun(\N(I_{[n]})^\op, \inftyS)^{/F}$, satisfying
	condition \ref{m1}, and define $Y^0$ to be the restriction of $Y$ to $\N(I^{\le
	1}_{[n]})^{\op}$.
	These functors can be assembled into the diagram of $\infty$-categories
	\[
	\xymatrix{
	\N(I^{\le 1}_{[n]})^{\op} \ar[r]^{Y^0} \ar@{^{(}->}[d] & \ar[d]^{t} \inftyS^{\Delta^1} \\
	\N(I_{[n]})^{\op} \ar[r]_{F} \ar[ur]^{Y} & \inftyS, 
	}
	\]
	where $t$ denotes pullback along the inclusion $\Delta^{\{1\}} \to \Delta^1$. Note that, by
	the argument already used above, the map $t$ is a coCartesian fibration and thus a
	categorical fibration. Now we observe
	that $Y$ satisfies condition \ref{m2} if and only if $Y$ is a $t$-right Kan extension of
	$Y^0$ in the sense of \cite[4.3.2.2]{lurie.htt}. Indeed, for an object $c = \{i,\dots,j\}$ of
	$\N(I_{[n]})$, the square 
	\[
	\xymatrix{
	\N(I^{\le 1}_{[n]})^{\op}_{c/} \ar[r]^{Y^0_c} \ar@{^{(}->}[d] & \ar[d]^{t} \inftyS^{\Delta^1} \\
	(\N(I^{\le 1}_{[n]})^{\op}_{c/})^{\triangleleft} \ar[r] \ar[ur] & \inftyS, 
	}
	\]
	exhibits $Y(c)$ as a $t$-limit of $Y^0_c$ if and only if the corresponding square in
	condition \ref{m2} is Cartesian. Note, that we use the assumption that $Y$ satisfies condition
	\ref{m1}. Now we can apply \cite[4.3.2.13]{lurie.htt} and \cite[4.3.2.15]{lurie.htt} to deduce that
	the restriction functor
	\[
	\Fun(\N(I_{[n]})^\op, \inftyS)^{/F} \to  \Fun(\N(I^{\le 1}_{[n]})^{\op}, \inftyS)^{/F^0}
	\]
	induces an equivalence of $\infty$-categories
	$\HX_{[n]} \stackrel{\simeq}{\to} \C$.
\end{proof}

\begin{prop} \label{prop.key} Let $X \in \sS$ be a Reedy fibrant simplicial space. 
	For every morphism $f: [n] \to [m]$ in the ordinal category $\Delta$, satisfying $f(0) = 0$
	and $f(n) = m$, we obtain a commutative square
	\begin{equation}\label{eq.charsquare}
		\xymatrix{
		X_m \ar[r] \ar[d] & X_{\{f(0),\dots,f(1)\}} \times_{X_{f(1)}} X_{\{f(1),\dots,f(2)\}}
		\times \dots \times X_{\{f(n-1),\dots,f(n)\}} \ar[d]\\
		X_n \ar[r] & X_{\{0,1\}} \times_{X_{\{1\}}} X_{\{1,2\}} \times \dots
		\times_{X_{\{n-1\}}} X_{\{n-1,n\}}
		}
	\end{equation}
	where the maps are induced by pullback along the respective maps in the commutative diagram
	\[
		\xymatrix{
			[m] & \ar[l] \{f(i),\dots,f(i+1)\}\\
			[n] \ar[u]^p & \ar[l] \ar[u] \{i,i+1\}.
		}
	\]
	in $\Delta$.
	Then the following hold.
	\begin{enumerate}[label=(\arabic*)]
		\item The simplicial space $X$ is a $2$-Segal space if and only if, for every
			injective morphism $p$ as above, the square \eqref{eq.charsquare} is a pullback square.
		\item The simplicial space $X$ is a unital $2$-Segal space if and only if, for every 
			morphisms $p$ as above, the square \eqref{eq.charsquare} is a pullback square.
	\end{enumerate}
\end{prop}

\begin{proof} 
	We express the morphism $f: [n] \to [m]$ as a composition
	\[
		[n] \overset{g}{\lra} [l] \overset{h}{\lra} [m]
	\]
	with $g$ surjective and $h$ injective. In the situation of part (1), the map $p$ is already injective, and so $g = \id$. 
	We obtain a corresponding diagram
	\begin{equation}\label{eq:decomp2seg1}
		\xymatrix{ 
			X_m \ar[r] \ar[d] & X_{\{h(0),\dots,h(1)\}} \times_{X_{\{h(1)\}}} X_{\{h(1),\dots, h(2)\}}
			\times \dots \times X_{ \{ h(l-1),\dots,h(l)\} } \ar[d]\\
			X_l \ar[d]\ar[r] & X_{\{0,1\}} \times_{X_{\{1\}}} X_{\{1, 2\}} \times_{X_{\{2\}}} \dots
			\times_{X_{ \{l-1\} }} X_{ \{ l-1,l\} } \ar[d]  \\
			X_n \ar[r] & X_{\{0,1\}} \times_{X_{\{1\}}} X_{\{1, 2\}} \times_{X_{\{2\}}} \dots \times_{X_{ \{n-1\} }} X_{ \{ n-1,n\} } 
		}
	\end{equation}
	where the outer rectangle is isomorphic to the sqare \ref{eq.charsquare}. The $2$-Segal
	condition for the polygonal subdivision of the convex ($m+1$)-gon described by the collection
	of subsets of $[m]$
	\[
		\T = \lbrace \{h(0),
		\dots, h(1)\}, \{h(1), \dots, h(2)\}, \dots, \{h(l-1), \dots, h(l)\},
		\{h(0),h(1),\dots,h(l)\} \rbrace 
	\]
	implies that the upper square in \eqref{eq:decomp2seg1} is a homotopy pullback square if $X$
	is a $2$-Segal space. This implies the ``only if'' direction of (1).
	To show the ``only if'' direction of (2), we express the surjective map $g$ as a composition of degeneracy maps. If $X$ is
	a unital $2$-Segal space, we conclude that the lower square in \eqref{eq:decomp2seg1} is
	a homotopy pullback square by iteratively using the homotopy pullback square 
	\eqref{eq:unitalsquare} from Definition \ref{def:unital-2-segal-top}. The ``if'' directions
	are easily obtained by making suitable choices for the morphism $f$. 
\end{proof}

\begin{cor} Let $X$ be a Reedy fibrant unital $2$-Segal space. Assume the space $X_0$ is
	contractible. Then the map $p: \HX \to \NDop$ is a coSegal fibration which models a 
	monoidal $\infty$-category.
\end{cor}

\begin{exa} Let $X$ be the Waldhausen S-construction of an exact $\infty$-category $\C$. Then
	$X$ is a unital $2$-Segal space and the space $X_0$, being the classifying space of zero
	objects in $\C$, is contractible. Therefore, in this case, we obtain a monoidal $\infty$-category $\HX$
	which we call the {\em Hall monoidal $\infty$-category associated to $\C$}. It is easy to verify that
	the homotopy category of $\HX$ can be identified with the monoidal category from
	Theorem \ref{thm:monoidalhall}.
\end{exa}

\begin{rem} Theorem \ref{thm:2segalcosegal} allows us to introduce a completeness condition for
	$2$-Segal spaces. A Reedy fibrant unital $2$-Segal space $X$ is called {\em complete} if the
	associated coSegal fibration $\HX \to \NDop$ is complete. 
\end{rem}

\begin{ex} Every complete $1$-Segal space is complete as a $2$-Segal space. The Waldhausen
	S-construction of a stable $\infty$-category is complete. 
\end{ex}

\begin{rem} The construction of this section has a drawback: It is not
	clear to how promote the association $X \mapsto \HX$ to a functor. We expect the functoriality to be given by
	left lax functors, which are most naturally described in terms of Segal fibrations. In \S
	\ref{sec:spans} below, we will provide a functorial construction in the context of an
	$\inftytwo$-categorical theory of spans.
\end{rem}

\vfill\eject

\section{An \texorpdfstring{$\inftytwo$}{(infty,2)}-categorical theory of spans}
\label{sec:spans}

Let $\C$ be an $\infty$-category with pullbacks. In this chapter, we will associate to $\C$ a new
$\infty$-category $\Spanl(\C)$, called the {\em $\infty$-category of spans in $\C$}, which is an
$\inftyone$-categorical variant of the span category introduced in \S \ref{subsec:mult-cat}.
\begin{itemize}
	\item The vertices of $\Spanl(\C)$ are given by vertices of $\C$. 
	\item An edge in $\Spanl(\C)$ between vertices
		$x_{\{0\}}$ and $x_{\{1\}}$ corresponds to a diagram in $\C$ of the form
		\[
		      \xymatrix@R=1pc{
			      x_{\{0\}} &\ar[l] x_{\{0,1\}} \ar[r] & x_{\{1\}}
		      }
		\]
		where $x_{\{0,1\}}$ is a vertex of $\C$. 
	\item A $2$-simplex of $\Spanl(\C)$ corresponds to a diagram in $\C$
		\[
			\xymatrix@=1pc{ 
			& & x_{\{1\}} & & \\ 
			&x_{\{0,1\}} \ar[dl] \ar[ur] & & \ar[dr] \ar[ul] x_{\{1,2\}}& \\
			x_{\{0\}} & & \ar[ll] x_{\{0,2\}} \ar[uu] \ar[ul] \ar[ur] \ar[rr] & & x_{\{2\}}
			}
		\]
	      which is a limit diagram with limit vertex
	      $x_{\{0,2\}}$.
	\item In general, we define the {\em asymmetric subdivision} of the standard $n$-simplex
		$\Delta^n$ as
		\[
			\asd(\Delta^n) = \N(I_{[n]})^{\op}
		\]
		where $I_{[n]}$ denotes the poset of nonempty intervals in $[n]$. An $n$-simplex in
		$\Spanl(\C)$ is then given by a diagram $\asd(\Delta^n) \to \C$ such that the
		induced diagram for every $2$-subsimplex $\Delta^2 \subset \Delta^n$ is a limit diagram.
\end{itemize}

Further, we will generalize the above span construction to a relative framework where we study families of
$\infty$-categories varying in a Cartesian fibration. The resulting theory will allow us to enhance the above
span construction in various ways:
\begin{enumerate}
	\item Given a monoidal $\infty$-category $\C$ with pullbacks, the $\infty$-category
		$\Spanl(\C)$ can be equipped with a natural ``pointwise'' monoidal structure.
	\item Given an $\infty$-category $\C$ with pullbacks, we can construct an {\em
		$\inftytwo$-category of bispans in $\C$}, modelled by a complete Segal fibration
		$\bSpan(\C) \to \ND$. This $\inftytwo$-category has {\em horizontal} spans as
		$1$-morphisms and {\em vertical} spans as $2$-morphisms. The construction of
		$\bSpan(\C)$, which is the main result of this section, will proceed in a two-step
		process which introduces horizontal spans (\S \ref{subsec:horizontal}) and vertical
		spans (\S \ref{subsec:vertical}) separately.
\end{enumerate}

The theory developed in this chapter will be used in \S \ref{section:higher-bicat} to describe
various higher bicategorical structures associated to $2$-Segal spaces.

\vfill\eject

\subsection{Spans in Kan complexes}
\label{subsec:span-kan}

We first study the span construction in the context of Kan complexes. The results will later be
applied to Kan complexes given as mapping spaces in $\infty$-categories.

Consider the functor
\begin{equation}\label{eq:P}
	P^{\bullet}: \Delta \lra \sSet, \; [n] \mapsto P^n = \N(I_{[n]})^{\op}
\end{equation}
where $I_{[n]}$ denotes the partially ordered set of nonempty intervals $\{i,j\}$ where $0 \le i \le j
\le n$. By forming a left Kan extension along the Yoneda embedding $\Delta \subset \sSet$,
we can extend $P^{\bullet}$ to a functor 
\[
\asd: \sSet \lra \sSet,
\]
which is the unique extension of $P^{\bullet}$ that commutes with
colimits. We will refer to the functor $\asd$ as {\em asymmetric
subdivision}, suggestive of its geometric significance. The functor $\asd$ admits a right adjoint which
we denote by $\Spanl$. We further define the standard simplex 
\[
\Delta^{\bullet}: \Delta \lra \sSet, \; [n] \mapsto \Delta^n.
\] 
The left Kan extension of $\Delta^{\bullet}$ to $\sSet$ is given by the identity functor on $\sSet$. We have a natural
transformation $P^{\bullet} \to \Delta^{\bullet}$ induced by the functors 
\[
(I_{[n]})^{\op} \lra [n], \; \{i,j\} \mapsto i.
\]
Via the functoriality of Kan extension, we obtain a natural transformation 
\begin{equation}\label{eq:eta} 
	\eta: \asd \lra \id 
\end{equation}
of functors on $\sSet$.

\begin{prop}\label{prop:asd-id} For every simplicial set $K$, the map $\eta(K): \asd(K) \to K$ is a weak
	homotopy equivalence. 
\end{prop}
\begin{proof} By the inductive argument of \cite[2.2.2.7]{lurie.htt} it suffices to verify this for
	$K = \Delta^n$, $n \ge 0$, in which case both $\asd(\Delta^n)$ and $\Delta^n$ are weakly
	contractible.
\end{proof}

\begin{prop}\label{prop:asd-quillen}
	The adjunction 
	\[
	\asd: \sSet \longleftrightarrow \sSet: \Spanl
	\]
	defines a Quillen self equivalence of $\sSet$ equipped with the Kan model structure.
\end{prop}
\begin{proof}
	This follows from Proposition \ref{prop:asd-id} by the argument of \cite[2.2.2.9]{lurie.htt}.
\end{proof}

\begin{cor}\label{cor:asd-quillen} 
	Given a Kan complex $K$, the simplicial set $\Spanl(K)$ is a Kan complex.
	Further, the functor $\Spanl$ preserves weak equivalences between Kan complexes.
\end{cor}

Given a small category $\J$, we define a category $\Mo(\J)$ as follows. The objects of $\Mo(\J)$ are
given by morphisms $x \to y$ in $\J$. A morphism from $x \to y$ to $x' \to y'$ is given by a
commutative diagram
\[
\xymatrix{ x \ar[d] \ar[r] & y\\
x' \ar[r] & y' \ar[u]}
\]
and composition is provided by concatenation of diagrams.

\begin{exa} Considering the ordinal $[n]$ as a category, the category $\Mo([n])$ can be identified with the
	opposite of the category $I_{[n]}$ corresponding to the partially ordered set of nonempty
	intervals in $[n]$. In other words, we have $P^n \cong \N(\Mo([n]))$. Proposition
	\ref{prop:asd-mo} below generalizes this observation.
\end{exa}

\begin{prop}\label{prop:asd-mo} There exists a $2$-commutative square of functors
	\[
	\xymatrix{
	\Cat \ar[d]^{\Mo} \ar[r]^{\N} &  \sSet \ar[d]^{\asd}\\
	\Cat \ar[r]^{\N} & \sSet,
	}
	\]
	where $\Cat$ denotes the category of small categories.
\end{prop}
\begin{proof} We will provide, for each category $\J$, an isomorphism $\asd(\N(\J)) \to
	\N(\Mo(\J))$, natural in $\J$. A $k$-simplex $\sigma$ of $\asd(\N(\J))$ can be represented by an
	$n$-simplex $f: \Delta^n \to \N(\J)$ together with a chain $\{i_0,j_0\} \supset \{i_1,j_1\}
	\supset \dots \supset \{i_k, j_k\}$ of intervals in $[n]$. We associate to $\sigma$ the $k$-simplex in
	$\N(\Mo(\J))$ given by 
	\[
	\xymatrix{
	f(i_0) \ar[d] \ar[r] & f(j_0)\\
	f(i_1) \ar[d] \ar[r] & \ar[u] f(j_1)\\
	 \vdots\ar[d] & \vdots \ar[u]\\
	f(i_k) \ar[r] & f(j_k)\ar[u].
	}
	\]
	It is straightforward to verify that this association descends to a well defined map
	$\asd(\N(\J)) \to \N(\Mo(\J))$ which is an isomorphism of simplicial sets, functorial in
	$\J$.
\end{proof}

\begin{cor}\label{cor:asd-product}
	Let $K,L$ be simplicial sets. The natural map $\gamma_{K,L}: \asd(K \times L) \lra \asd(K) \times \asd(L)$ is an
	isomorphism of simplicial sets. This provides the functor $\asd$ with the structure of a
	monoidal functor with respect to the Cartesian monoidal structure on $\sSet$.
\end{cor}
\begin{proof} Observe that source and target of the map $\gamma_{K,L}$ commute with colimits in both
	variables $K$ and $L$. It therefore suffices to verify the statement in the case $K =
	\Delta^m$, 
	$L = \Delta^n$ for $m,n \ge 0$. We conclude the argument by Proposition \ref{prop:asd-mo},
	noting that we have a natural isomorphism of categories
	\[
	\Mo([m] \times [n]) \lra \Mo([m]) \times \Mo([n]).
	\]
\end{proof}

\vfill\eject
	
\subsection{Vertical Spans}
\label{subsec:vertical}

In this section, we provide a relative span construction which applies to a family of $\infty$-categories parametrized by a Cartesian fibration. 
The main application we have in mind is the following. Suppose $\B$ is an
$\inftytwo$-category modelled by a complete Segal fibration $Y \to \ND$. 
Under suitable assumptions on $\B$, we will define an $\inftytwo$-category $\Spanl(\B)$ of {\em vertical
spans} modelled by a complete Segal fibration $\Spanl(Y) \to \ND$. The $\inftytwo$-category $\Spanl(\B)$ 
can be described informally as follows.
\begin{itemize}
	\item The objects of $\Spanl(\B)$ are given by objects of the $\inftytwo$-category $\B$.
	\item A $1$-morphism between objects $x,y$ of $\Spanl(\B)$ is given by a $1$-morphism
		between $x$ and $y$ in $\B$.
	\item A $2$-morphism between $1$-morphisms $f: x \to y$ and $g: x \to y$ is given by a
		$2$-span diagram of the form
		\[
		\xymatrix@=5pc{
		x \ruppertwocell<10>^f{^} \rlowertwocell<-10>_g \ar[r]^(.35)h & y
		}
		\]
		in $\B$.
	\item The higher morphisms are given by spans, in which both edges are equivalences, of
	  spans of spans of \dots in $\C$.
\end{itemize}

Let $T$ be a simplicial set. Given a map $K \to T$ of simplicial sets, we define $\asd(K) \to T$ to
be the composite of the map $\asd(K) \to \asd(T)$ and the map $\eta(T)$ defined in \eqref{eq:eta}.
This association extends to an adjunction 
\begin{equation}\label{eq:adj-S}
	\asd_T: (\sSet)_{/T} \lra (\sSet)_{/T}: \Span_T.
\end{equation}
Note that, for a map $Y \to T$ of simplicial sets, we have a pullback square
\[
\xymatrix{
\Span_T(Y) \ar[r]\ar[d] & \Span(Y) \ar[d]\\
T \ar[r] & \Span(T),
}
\]
where the inclusion $T \to \Span(T)$ is adjoint to the map $\eta(T): \asd(T) \to T$. 

\begin{defi} \label{defi:segalcone} 
	Let $n \ge 2$. We introduce the notation
	\[
		\J^n = \Delta^{\{0,1\}} \coprod_{\{1\}} \Delta^{\{1,2\}} \coprod_{\{2\}} \dots \coprod_{\{n-1\}}
		\Delta^{\{n-1,n\}} \subset \Delta^n,
	\]
	where, as usual, we will occasionally use the notation $\J^I \subset \Delta^I$ for a finite
	nonempty ordinal $I$ whenever explicit reference to the vertices of $\Delta^I$ is needed.
	We call simplicial set
	\[
		S(\Delta^n) :=  \{0,n\} \star \asd(\J^n) \subset \asd(\Delta^n)
	\]
	the {\em Segal cone in $\Delta^n$}. 
\end{defi}

Let $Y \to T$ be a map of simplicial sets. A simplex $\Delta^n \to \Span_T(Y)$ corresponds by
definition to a map $\asd(\Delta^n) \to Y$ which we may restrict to obtain a Segal cone diagram
$S(\Delta^n) \to Y$. A simplex $\Delta^n \to \Span_T(Y)$ is called {\em Segal simplex} if, for every
subsimplex $\Delta^k \subset \Delta^n$ with $k \ge 2$, the corresponding Segal cone diagram
$S(\Delta^k) \to Y$ is a $p$-limit diagram (see \cite[4.3.1]{lurie.htt}). It is easy to
verify that the collection of all Segal simplices in $Y$ assembles to a simplicial subset 
\[
	\Spanl_T(Y) \subset \Span_T(Y).
\]

%

\begin{exa}\label{exa:spankan} Let $Y$ be a Kan complex. Then every simplex of $\Span_{\pt}(Y) =
	\Span(Y)$ is a Segal simplex. This follows since, by \cite[4.4.4.10]{lurie.htt}, any diagram 
	$K^{\triangleleft} \to Y$ with $K$ weakly contractible is a limit diagram.
\end{exa}

\begin{prop}\label{prop:pullsegal} Let $p: Y \to T$ be a Cartesian fibration which admits relative
	pullbacks, i.e. $K$-indexed $p$-limits where $K = \Delta^1 \coprod_{\{1\}} \Delta^1$. Then
	any diagram $\asd(\J^n) \to Y$ admits a $p$-limit.
\end{prop}
\begin{proof}
	Using \cite[4.3.1.10]{lurie.htt} and \cite[4.3.1.11]{lurie.htt} one reduces to the case $T =
	\pt$. Now the statement follows immediately from the dual statement of
	\cite[4.4.2.2]{lurie.htt}.
\end{proof}

In what follows, we will make use of the Cartesian model structure on the category $(\sSet^{+})_{/T}$ of marked
simplicial sets over $T$ for which we refer the reader to \cite[3.1.3]{lurie.htt}. We will freely
use the notation introduced in loc. cit. In particular, given a simplicial set $K$, we denote by
$K^{\flat}$ the marked simplicial set where only the degenerate edges are marked, and by
$K^{\sharp}$ the marked simplicial set where all edges are marked. 
Given a simplicial subset $K \subset \Delta^n$, we define the marked
simplicial set $K^{\lastedge} = (K, \E)$ where $\E$ denotes the set of all degenerate edges together
with the edge $\{n-1,n\}$ if $\{n-1,n\} \subset K$.
Further, given a Cartesian
fibration $p: Y \to T$, we obtain an object $Y^{\natural} \to T^{\sharp}$ of $(\sSet^{+})_{/T}$ where
$Y^{\sharp}$ is the marked simplicial set obtained by marking all $p$-Cartesian edges. By \cite[3.1.4.1]{lurie.htt}, the objects of
$(\sSet^{+})_{/T}$ arising via this construction are exactly the fibrant objects. 

We promote the adjunction \eqref{eq:adj-S} to an adjunction of marked simplicial sets
\begin{equation}\label{eq:markedadj}
	\asd_T: (\sSet^+)_{/T} \lra (\sSet^+)_{/T}: \Span_T
\end{equation}
by declaring an edge $\Delta^1 \to \Span_T(Y)$ to be marked if both edges of $Y$ determined by the
adjoint map $\asd(\Delta^1) \to Y$ are marked. We do not distinguish this marked adjunction
notationally. Further, we obtain an induced marked structure on $\Spanl_T(Y)$ by declaring an edge
to be marked if its image in $\Span_T(Y)$ is marked.
Given objects $K,Y$ in $(\mSet)_{/T}$, we define
\[
\Hom_T(\asd(K), Y)^l \subset \Hom_T(\asd(K), Y)
\]
to be the subset consisting of those maps whose adjoint map $K \to \Span_T(Y)$ factors through
$\Spanl_T(Y) \subset \Span_T(Y)$. Further, we define 
\[
\Map_T^{\sharp}(\asd(K), Y)^l \subset \Map^{\sharp}_T(\asd(K), Y)
\]
to be the full simplicial subset spanned by those vertices lying in $\Hom_T(\asd(K), Y)^l$. 

The following result will be a corollary of a more general statement for Cartesian fibrations
proven in Theorem \ref{thm:spancartesian}. Nevertheless, we present a proof since it already illustrates some of the
ideas of the technically more involved argument in the relative situation. 

\begin{thm} \label{thm:spancat} Let $\C$ be an $\infty$-category which admits pullbacks. Then the
	simplicial set $\Spanl_{\pt}(\C)$ is an $\infty$-category.
\end{thm}
\begin{proof}
	From Proposition \ref{prop:restmap} below, it follows that the simplicial sets
	\[
		Y_n = \Map^{\sharp}(\asd(\Delta^n)^{\flat}, \C^{\natural})^l,
	\]
	where $n \ge 0$, organize into a Reedy fibrant simplicial space $Y$. Note that the marked
	edges of $\C^{\natural}$ are exactly the equivalences in $\C$, hence the simplicial set
	$\Map^{\sharp}(\asd(\Delta^n)^{\flat}, \C^{\natural})$ coincides with the largest Kan
	complex contained in the $\infty$-category $\Fun(\asd(\Delta^n), \C)$.  By
	\cite[4.3.2.15]{lurie.htt}, the inclusion $\asd(\J^n) \to \asd(\Delta^n)$ induces a trivial
	fibration
	\begin{equation}\label{eq:mapseg1}
			Y_n \overset{\simeq}{\lra} \Map^{\sharp}(\asd(\J^n)^{\flat}, \C^{\natural}).
	\end{equation}
	Further, we have a canonical indentification
	\begin{equation}\label{eq:mapseg2}
			\Map^{\sharp}(\asd(\J^n)^{\flat}, \C^{\natural}) \cong Y_{\{0,1\}} \times_{Y_{\{1\}}}
			Y_{\{1,2\}} \times \dots \times Y_{\{n-1,n\}}.
	\end{equation}
	The composite of the maps in \eqref{eq:mapseg1} and \eqref{eq:mapseg2} coincides with the
	natural map 
	\[
		Y_n \lra Y_{\{0,1\}} \times_{Y_{\{1\}}} Y_{\{1,2\}} \times \dots \times
		Y_{\{n-1,n\}}.
	\]
	Therefore, we deduce that $Y$ is a Reedy fibrant $1$-Segal space. Note that, by the equality 
	\[
		\Spanl_{\pt}(\C)_n = \Hom(\asd(\Delta^n), \C)^l = (Y_n)_0, 
	\]
	the simplicial set $\Spanl_{\pt}(\C)$ coincides with the $0$th row of $Y$. 
	Using Corollary 3.6 in \cite{joyal-tierney}, we conclude that $\Spanl_{\pt}(\C)$ is an
	$\infty$-category.
\end{proof}

\begin{prop} \label{prop:restmap} Let $T$ be a simplicial set, and let $Y$ be a fibrant object of $(\mSet)_{/T}$. Then the following assertions hold:
	\begin{enumerate} 
		\item For any object $K \in (\mSet)_{/T}$, the simplicial set
			$\Map_T^{\sharp}(\asd(K), Y)^l$ is a Kan complex which is a union of
			connected components of $\Map_T^{\sharp}(\asd(K), Y)$.
		\item \label{l:restmap2} For any cofibration $K \to L$ in $(\mSet)_{/T}$, the induced restriction map
			\[
			\Map_T^{\sharp}(\asd(L), Y)^l  \lra \Map_T^{\sharp}(\asd(K), Y)^l 
			\]
			is a Kan fibration.
		\item Consider a pushout diagram in $(\mSet)_{/T}$
			\[
			\xymatrix{
			K \ar[r]^g\ar[d]_f & \ar[d] L\\
			K' \ar[r] & L'
			}
			\]
			where $f$ and $g$ are cofibrations and assume that the restriction map 
			\[
			\Map_T^{\sharp}(\asd(L), Y)^l  \lra \Map_T^{\sharp}(\asd(K), Y)^l 
			\]
			is a weak equivalence. Then the restriction map
			\[
			\Map_T^{\sharp}(\asd(L'), Y)^l  \lra \Map_T^{\sharp}(\asd(K'), Y)^l
			\]
			is a weak equivalence.
	\end{enumerate}
\end{prop}
\begin{proof} By \cite[3.1.4.4]{lurie.htt}, the mapping space $\Map_T^{\sharp}(\asd(K), Y)$ is a Kan
	complex. Assertion (1) now follows immediately from the fact that relative limit cones are
	stable under equivalences (cf. \cite[4.3.1.5(3)]{lurie.htt}).

	By \cite[3.1.4.4]{lurie.htt}, the map
	\[
		\Map_T^{\sharp}(\asd(L), Y) \to \Map_T^{\sharp}(\asd(K), Y)
	\]
	is a Kan fibration. The assertion (2) now follows from (1).

	To show (3) note that, under the given assumptions, the square
		\[
		\xymatrix{
			\Map_T^{\sharp}(\asd(L'), Y)  \ar[r]\ar[d] &  \Map_T^{\sharp}(\asd(K'), Y)
			\ar[d]^{f^*}\\
			\Map_T^{\sharp}(\asd(L), Y)  \ar[r] &  \Map_T^{\sharp}(\asd(K), Y) 
		}
		\]
	is a pullback square. From this we deduce further that the square 
	\[
		\xymatrix{
			\Map_T^{\sharp}(\asd(L'), Y)^l  \ar[r]\ar[d] &  \Map_T^{\sharp}(\asd(K'),Y)^l
			\ar[d]^{f^*}\\
			\Map_T^{\sharp}(\asd(L), Y)^l  \ar[r] &  \Map_T^{\sharp}(\asd(K), Y)^l 
		}
	\]
	is a pullback square. Since the map $f^*$ is a Kan fibration, the statement now follows
	from the fact that the category of simplicial sets equipped with the Kan model structure
	is right proper: pullback along fibrations preserves weak equivalences.
\end{proof}

In the remainder of this section, we will generalize Theorem \ref{thm:spancat} to the following
result. 

\begin{thm}\label{thm:spancartesian} Let $p: Y \to T$ be a Cartesian fibration of
	$\infty$-categories which admits relative pullbacks, and let $Y^{\natural}$ be the
	corresponding fibrant object of $(\mSet)_{/T}$.  Then the marked simplicial set
	$\Spanl_T(Y^{\natural})$ is a fibrant object of $(\mSet)_{/T}$. Equivalently, the underlying
	map $q: \Spanl_T(Y) \to T$ is a Cartesian fibration and the $q$-Cartesian edges of
	$\Spanl_T(Y)$ are precisely the marked edges.
\end{thm}

For the proof of Theorem \ref{thm:spancat}, we need some preparatory results. The first result
concerns a slight elaboration on Lemma 3.5 in \cite{joyal-tierney}, where Lemma
\ref{lem:joyal-tierney}(1) is proven in the case $T = \pt$.  
Let $A$ be a class of morphisms in a category $\C$.  We say $A$ has the {\em right cancellation
property} if, for any pair $f,g$ of morphisms in $\C$, we have: If $g \in A$ and $f \circ g \in A$
then $f \in A$.  The following Lemma is a slight elaboration on Lemma 3.5 in \cite{joyal-tierney},
where assertion (1) is proven in the case $T = \pt$. 

\begin{lem}\label{lem:joyal-tierney} Let $T$ be a simplicial set and let $F:
	(\mSet)_{/T} \to \mSet$ the forgetful functor. Suppose $A$ be a class of cofibrations in
	$(\mSet)_{/T}$ which is closed under composition, pushouts along cofibrations, and satisfies the right
	cancellation property. Consider the following sets of cofibrations in $\mSet$: 
	\begin{align*}
		B_1 & = \{ (\J^n)^{\flat} \subset (\Delta^n)^{\flat} \; | \; \text{where } n \ge 2 \}\\
		M_1 & = \{ (\Lambda^n_i)^{\flat} \subset (\Delta^n)^{\flat} \; | \; \text{where $n \ge 2$, $0 < i < n$} \}\\
		B_2 & = \{ (\Lambda^n_n)^{\lastedge} \subset (\Delta^n)^{\lastedge} \; | \;
		\text{where $1 \le n \le 2$} \}\\
		M_2 & = \{ (\Lambda^n_n)^{\lastedge} \subset (\Delta^n)^{\lastedge} \; | \;
		\text{where $n \ge 1$} \}
	\end{align*}
	\begin{enumerate}
		\item Assume $A$ contains $F^{-1}(B_1)$. Then $A$ contains $F^{-1}(M_1)$. 
		\item Assume $A$ contains $F^{-1}(B_1)$ and $F^{-1}(B_2)$. Then $A$ contains
			$F^{-1}(M_1)$ and $F^{-1}(M_2)$.
	\end{enumerate}
\end{lem}
\begin{proof}
	(1) We first introduce some notation. For a subset $T \subset [n]$, we define
	\[
	\Lambda^n_T := \bigcup_{k \notin T} \partial_k \Delta^n \subset \Delta^n
	\]
	We define another set of cofibrations in $\mSet$
	\[
	\widetilde{M_1} = \{ (\Lambda^n_T)^{\flat} \subset (\Delta^n)^{\flat} \; | \; \text{where $n \ge 2$, $\emptyset \ne T
	\subset \{1,\dots,n-1\}$ } \}
	\]
	such that we have $M_1 \subset \widetilde{M_1}$.
	The set $\widetilde{M_1}$ is better adapted to our inductive argument which will show the stronger statement 
	that if $A$ contains $F^{-1}(B_1)$, then it contains $F^{-1}(\widetilde{M_1})$. We start with some
	preparing comments. We will 
	denote cofibrations in $(\mSet)_{/T}$ by their image in $\mSet$ under the forgetful functor
	$F$. 
	This can be justified by the observation that all constructions appearing in the proof will
	only involve maps in $(\mSet)_{/T}$ whose images in $\mSet$ are inclusions between
	simplicial subsets of $F(\Delta^n \to T)$ for a fixed object $\Delta^n \to T$ of $(\mSet)_{/T}$. Therefore, the compatibility of these
	constructions with the structure map to $T$ is automatically guaranteed.
	Further, to keep the notation light, we will denote marked simplicial sets of the form
	$K^{\flat}$ by their underlying simplicial set $K$.

	We will now prove assertion (1) by induction on $n$. For $n = 2$, the only map in $\widetilde{M_1}$ is
	$\Lambda^2_{\{0,2\}} \subset \Delta^2$ which coincides with the map $J^2 \to \Delta^2$.
	Let $n > 2$ and assume $A$ contains all maps $\Lambda^m_T \to \Delta^m$ in $F^{-1}(\widetilde{M_1})$ with $m <
	n$. Let $f: \Lambda^n_T \to \Delta^n$ a cofibration contained in $F^{-1}(\widetilde{M_1})$. To show that
	$f \in A$, note that the composite 
	\[
	\J^n \overset{h}{\lra} \Lambda^n_{ \{1,\dots,n-1\} } \overset{g}{\lra} \Lambda^n_T \overset{f}{\lra} \Delta^n
	\]
	is contained in $A$, so it suffices to show that both $g$ and $h$ are contained in $A$.

	To show that $h \in A$, note that we have $h = h_1 \circ h_2$ such that $h_1$ and $h_2$ are
	part of the pushout diagrams
	\begin{equation} \label{eq:push1}
		\xymatrix{ \J^{\{1,\dots,n\}} \ar[r]^-{h_2'} \ar[d] & \Delta^{\{1,\dots,n\}} \ar[d]  \\
		\J^n \ar[r]^-{h_2} & \partial_0 \Delta^n \cup \J^n  }
	\end{equation}
	and
	\[
	\xymatrix{ \Delta^{\{1,\dots,n-1\}} \cup \J^{\{0,\dots,n-1\}} \ar[r]^-{h_1'}\ar[d] & \Delta^{\{0,\dots,n-1\}} \ar[d] \\
	\partial_0 \Delta^n \cup \J^n \ar[r]^-{h_1} & \partial_0 \Delta^n \cup \partial_n \Delta^n}.
	\] 
	Since $h_2' \in A$, we immediately deduce $h_2 \in A$. To show that $h_1'$ is contained
	in $A$, we apply the cancellation property to the composition
	\[
	\J^{\{0,\dots,n-1\}} \lra \Delta^{\{1,\dots,n-1\}} \cup \J^{\{0,\dots,n-1\}}
	\overset{h_1'}{\lra} \Delta^{\{0,\dots,n-1\}}
	\]
	and then consider the pushout diagram \eqref{eq:push1} with $\{1,\dots,n\}$ replaced by
	$\{0,\dots,n-1\}$. This proves $h_1 \in A$ and hence also $h \in A$.

	To show that $g \in A$, we choose a descending chain $T_0 \supset T_1 \supset \dots \supset T_k$ of
	subsets of $[n]$ such that $T_0 =\{1,\dots,n-1\}$, $T_k = T$ and, for every $0 \le j \le k-1$, we
	have $T_{j+1} = T_j \setminus \{l_j\}$ where $1 \le l_j \le n-1$. This chain of subsets induces a sequence of inclusions 
	\[
	\Lambda^n_{T_0} \overset{g_0}{\lra} \Lambda^n_{T_1} \overset{g_1}{\lra} \Lambda^n_{T_2}
	\overset{g_2}{\lra} \dots \overset{g_{k-1}}{\lra} \Lambda^n_{T_k}
	\]
	which composes to $g$. Now we observe that, for every $0 \le j \le k-1$, we have a pushout
	square
	\[
	\xymatrix{
	\Lambda^{[n] \setminus \{l_j\}}_{T_j \setminus \{l_j\}} \ar[r]^{g_j'} \ar[d] & \Delta^{[n]
	\setminus \{l_j\}} \ar[d]\\
	\Lambda^n_{T_j} \ar[r]^{g_j}  & \Lambda^n_{T_j \setminus \{l_j\}}.
	}
	\]
	Since, by induction hypothesis, $g_j'$ is contained in $A$, this shows that $g_j$ and
	henceforth $g$ is contained in $A$.\\

	(2) First, we only assume that $A$ contains $F^{-1}(B_1)$.
	We define the set
	\[
	M_2' = \{ (\Lambda^n_T)^{\flat} \subset (\Delta^n)^{\flat} \; | \; \text{where $n \ge 2$, $\{1,n\}
	\subset T \subset \{1,\dots,n-2,n\}$ } \}
	\]
	of cofibrations in $\mSet$.
	An inductive argument, similar to the one given for the proof of assertion (1), applied to
	compositions of the form
	\[
	\J^n \lra \Lambda^n_{ \{1,\dots,n-2,n\} } \lra
	\Lambda^n_T \lra \Delta^n,
	\]
	shows that $A$ contains $F^{-1}(M_2')$. As a preparatory observation, note that if
	 a cofibration $K^{\flat} \subset (\Delta^n)^{\flat}$ is
	contained in $A$ and $K$ contains the final edge $\{n-1,n\}$, then the pushout
	square
	\[
	\xymatrix{ 
	K^{\flat} \ar[r]\ar[d]& (\Delta^n)^{\flat} \ar[d] \\
	K^{\lastedge} \ar[r] & (\Delta^n)^{\lastedge} }
	\] 
	shows that $A$ contains the induced map $K^{\lastedge} \to (\Delta^n)^{\lastedge}$. We will
	use this observation implicitly below.
	We now assume in addition that $A$ contains
	$F^{-1}(B_2)$ and show by induction on $n$ that $A$ contains $F^{-1}(M_2)$. For $n=2$, there
	is nothing to show, since $(\Lambda^2_2)^{\lastedge} \to
	(\Delta^2)^{\lastedge}$ is already contained in $B_2$. Let $n > 2$ and assume that $A$
	contains all maps $(\Lambda^m_m)^{\lastedge} \to (\Delta^m)^{\lastedge}$ in $F^{-1}(M_2)$
	with $m < n$. Let $f: (\Lambda^n_n)^{\lastedge} \to (\Delta^n)^{\lastedge}$ in
	$F^{-1}(M_2)$. We have to show that $f$ is contained in $A$. 
	Because $F^{-1}(M_2')$ is contained in $A$, the composition
	\[
	(\Lambda^n_{\{1,n\}})^{\lastedge} \overset{g}{\lra} (\Lambda^n_n)^{\lastedge} \overset{f}{\lra} (\Delta^n)^{\lastedge}
	\]
	is contained in $A$, and it hence suffices to prove $g \in A$. To this end, consider the
	pushout square 
	\[
	\xymatrix{
	(\Lambda^{\{0,2,\dots,n\}}_{\{n\}})^{\lastedge} \ar[r]^{g'}\ar[d] & (\Delta^{\{0,2,\dots,n\}})^{\lastedge} \ar[d] \\
	(\Lambda^n_{\{1,n\}})^{\lastedge} \ar[r]^{g} & (\Lambda^n_n)^{\lastedge}
	}
	\]
	and note that $g'$ is contained in $A$ by induction hypothesis.
\end{proof}

We will frequently use the following lemma, which provides a means to compare (relative) Segal
simplices in $\Span_T(Y)$ with (absolute) Segal simplices in the fibers $\Span_{s}(Y_s)$, $s \in T$.

\begin{lem}\label{lem:comparelimit}
	Let $p: Y \to T$ be a Cartesian fibration of $\infty$-categories. Consider an $n$-simplex
	$\sigma: \Delta^n \to \Span_T(Y)$ and let $f: \asd(\Delta^n) \to Y$ be the adjoint map. We
	set $L = \asd(\Delta^n)$ and $s_0 := p \circ f(\{0\}) \in T$.
	Then there exists a homotopy 
	\[
	h: (\Delta^1)^{\sharp} \times L^{\flat} \lra Y^{\natural}
	\]
	in $(\mSet)_{/T}$ with the following properties:
	\begin{enumerate}
		\item We have $h|\{1\} \times L = f$ and the diagram $f' := h|\{0\} \times L$ lies
			completely in the fiber $Y_{s_0}$.
		\item The edge $h|\Delta^1 \times L_{s_0}$ in
			$\Fun(L_{s_0}, Y_{s_0})$ is degenerate, in particular, we have $f'|L_{s_0} =
			f|L_{s_0}$.
		\item The following assertions are equivalent:
			\begin{enumerate}
				\item \label{l:pr1} The simplex $\sigma$ is a Segal simplex in $\Span_T(Y)$.
				\item \label{l:pr2} The simplex $\sigma'$ given by the adjoint of
					$f'$ is a Segal simplex in $\Span_T(Y)$. In particular, the
					simplex $\sigma'$ is a Segal simplex in the fiber $\Span_{\{s_0\}}(Y_{s_0})$.
			\end{enumerate}
		\item Let $e$ be an edge in $L$. If $f(e)$ is $p$-Cartesian then $f'(e)$ is
			$p$-Cartesian, and hence an equivalence in the $\infty$-category $Y_{s_0}$.
	\end{enumerate}
\end{lem}
\begin{proof}
	Consider the diagram
	\[
	\xymatrix{
	\displaystyle
	\{1\}^{\sharp} \times L^{\flat} \coprod_{\{1\}^{\sharp} \times (L_{s_0})^{\flat}}
	(\Delta^1)^{\sharp} \times (L_{s_0})^{\flat} \ar@{^{(}->}[d]_{g} \ar[r]^-{f \circ \pi \circ
	g} & Y^{\natural}\ar[d]\\
	(\Delta^1)^{\sharp} \times L^{\flat} \ar@{-->}[ur]^h \ar[r]^{p \circ f \circ \pi} & T^{\sharp}
	}
	\]
	where $\pi$ denotes the nerve of the map of partially ordered sets
	\[
	[1] \times \Mo([n]) \lra \Mo([n]),\; (i, \{j,k\}) \mapsto 
	\left\{ \begin{matrix} \{0,n\} \quad \text{if $i=0$,}\\ \{j,k\} \quad \text{if $i=1$.}
	\end{matrix} \right.
	\]
	Since the map $g$ is marked anodyne and $Y \to T$ is a Cartesian fibration, we can
	 find a map $h$ solving the specified lifting
	problem. Assertions (1) and (2) follow immediately from the construction of $h$.
	The map $h$ satisfies the hypothesis of \cite[4.3.1.9]{lurie.htt}, hence we deduce that $f = h|\{1\} \times L$ is a
	$p$-limit diagram if and only if the map $f':= h|\{0\} \times L$ is a $p$-limit diagram.
	More generally, we given a subsimplex $\Delta^k \to \Delta^n$, we can restrict $h$ to deduce
	the analogous statement for $L$ replaced by $\asd(\Delta^k)$. Therefore, we obtain the
	equivalence of the assertions \ref{l:pr1} and \ref{l:pr2} by
	the definition of a Segal simplex. To prove (4), let $e = I_1 \to I_2$ denote
	the edge in $L$ under consideration and consider the diagram in $Y$
	\[
	\xymatrix{ f'(I_1) \ar[r]^{f'(e)} \ar[d]_{\natural} & f'(I_2) \ar[d]^{\natural}\\
	f(I_1) \ar[r]^{\natural}_{f(e)} & f(I_2) }
	\]
	induced by $h$. Here, we indicated those edges which are $p$-Cartesian by the construction of
	$h$. Applying \cite[2.4.1.7]{lurie.htt} twice implies the statement.
\end{proof}

\begin{proof}[Proof of Theorem \ref{thm:spancartesian}]
	Using \cite[3.1.1.6]{lurie.htt}, we will show that the map $\Spanl_T(Y^{\natural}) \to
	T^{\sharp}$ has the right lifting property with respect to all marked anodyne maps. 	
	For this, it suffices to verify the right lifting property with respect to the following sets of
	cofibrations in $\mSet$:
	\begin{align}
		\tag{1}	M_1 & = \{ (\Lambda^n_i)^{\flat} \subset (\Delta^n)^{\flat} \; | \; \text{where $n
			\ge 2$ and $0 < i < n$} \}\\
		\tag{2} M_2 & = \{ (\Lambda^n_n)^{\lastedge} \subset (\Delta^n)^{\lastedge} \; | \; \text{where $n
			\ge 1$} \}\\
		\tag{3} M_3 & = \{ \textstyle (\Lambda^2_1)^{\sharp} \coprod_{(\Lambda^2_1)^{\flat}} (\Delta^2)^{\flat}
			\subset (\Delta^2)^{\sharp} \}\\
		\tag{4} M_4 & = \{ K^{\flat} \to K^{\sharp} \; | \; \text{where $K$ is a Kan complex} \}
	\end{align}

	\noindent
	(1) For every $0 < i < n$ and every diagram
	\[
	\xymatrix{
	(\Lambda^n_i)^{\flat} \ar[r]\ar[d] & \Spanl_T(Y^{\natural})\ar[d]\\
	(\Delta^n)^{\flat} \ar[r] \ar@{-->}[ur] & T^{\sharp}}
	\]
	in $\mSet$, we have to provide the dashed arrow, rendering the diagram commutative. Passing to adjoints,
	this lifting problem is equivalent to the surjectivity of the restriction map
	\[
	\Hom_T(\asd(\Delta^n)^{\flat}, Y^{\natural})^l \to \Hom_T(\asd(\Lambda^n_i)^{\flat},
	Y^{\natural})^l.
	\]
	We will prove the more general statement that the restriction map of simplicial sets
	\[
		\rho: \Map_T^{\sharp}(\asd(\Delta^n)^{\flat}, Y^{\natural})^l \to
		\Map_T^{\sharp}(\asd(\Lambda^n_i)^{\flat}, Y^{\natural})^l
	\]
	is a trivial Kan fibration. By \ref{prop:restmap}, it suffices to show that $\rho$ is a weak homotopy equivalence. 
	Consider the class $A$ of cofibrations $K \subset L$ in $(\mSet)_{/T}$ such that the
	induced map 
	\[
		\Map_T^{\sharp}(\asd(L), Y^{\natural})^l \to \Map_T^{\sharp}(\asd(K), Y^{\natural})^l
	\]
	is a weak equivalence. The class $A$ is stable under compositions, pushouts along
	cofibrations (Proposition \ref{prop:restmap}), and has the right
	cancellation property. Hence, by Lemma \ref{lem:joyal-tierney}(1), it suffices to show that, for every
	$n \ge 2$, any map 
	\[
	(\J^n)^{\flat} \subset (\Delta^n)^{\flat} \to T^{\sharp}
	\]
	in $(\mSet)_{/T}$ is contained in $A$, where we recall the notation $\J^n = \Delta^{\{0,1\}} \coprod_{\{1\}}
	\dots \coprod_{\{n-1\}} \Delta^{\{n-1,n\}}$.
	Thus, we have to verify that, for $\C^0 = \asd(\J^n)$ and $\C = \asd(\Delta^k)$, the map
	\begin{equation}\label{eq:kanext}
		\Map_T^{\sharp}(\C^{\flat}, Y^{\natural})^l \to \Map_T^{\sharp}((\C^0)^{\flat}, Y^{\natural})^l
	\end{equation}
	is a weak equivalence. Consider the diagram
	\[
	\xymatrix{
	\C^0 \ar@{^{(}->}[r] & \C \ar[r] & T & \ar[l]^p Y.
	}
	\]
	Notice that both $\C^0$ and $\C$ are $\infty$-categories and $\C^0 \subset \C$ is a full
	subcategory. Further, every vertex of
	$\Map_T^{\sharp}(\C^{\flat}, Y^{\natural})^l$ corresponds to a Segal simplex, and hence, to a functor $F: \C \to Y$ which
	is a $p$-right Kan extension of $F|\C^0$. Thus, using \cite[3.1.3.1]{lurie.htt}, 
	\cite[4.3.2.15]{lurie.htt}, and Proposition \ref{prop:pullsegal}, we conclude that the restriction map \eqref{eq:kanext} is a
	trivial Kan fibration.\\

	\noindent
	(2) As in the argument for (1), we will prove the more general statement $M_2 \subset A$.
	To this end, it suffices by Lemma \ref{lem:joyal-tierney}(2) to
	verify the following assertions:
	\begin{enumerate}[label=(\roman*)] 
		\item Any map
		\begin{equation}\label{eq:f1}
			\{1\}^{\sharp} \subset (\Delta^1)^{\sharp} \to T^{\sharp}
		\end{equation}
		in $(\mSet)_{/T}$ belongs to $A$.

		\item Any map
		\begin{equation}\label{eq:f2}
			(\Lambda^2_2)^{\lastedge} \subset (\Delta^2)^{\lastedge} \to T^{\sharp}
		\end{equation}
		in $(\mSet)_{/T}$ belongs to $A$.
	\end{enumerate}
	We say that a cofibration $K \to L$ in $(\mSet)_{/T}$ is {\em $Y$-local} if
	the pullback map
	\[
	\Map_T^{\sharp}(L, Y^{\natural}) \lra \Map_T^{\sharp}(K, Y^{\natural})
	\]
	is a weak equivalence, and hence a trivial Kan fibration. Note that, by
	\cite[3.1.3.4]{lurie.htt}, any marked anodyne map
	is $Y$-local.
	To show (i), it suffices to show that for any map as in \eqref{eq:f1}, the map in
	$(\mSet)_{/T}$
	\[
		\{1\}^{\sharp} \subset \asd((\Delta^1)^{\sharp}) \to T^{\sharp},
	\]
	obtained by applying the functor $\asd_T$, is $Y$-local.
	To show this, we first observe that the inclusion $f_1: \{1\} \subset
	(\Delta^{\{\{0,1\},\{1\}\}})^{\sharp}$ is marked anodyne, hence $Y$-local. Therefore, it suffices to show that restriction along $f_2:
	(\Delta^{\{\{0,1\},\{1\}\}})^{\sharp} \subset \asd(\Delta^1)^{\sharp}$ is $Y$-local. 
	Further, since $f_2$ is a pushout of $f_3: \{0,1\} \to (\Delta^{\{\{0,1\},\{0\}\}})^{\sharp}$
	along a cofibration, it suffices to prove that restriction along $f_3$ is $Y$-local. Note,
	that the image of the edge $\{0,1\}$ in $T$ is degenerate, so we can
	restrict attention to one fiber of $Y \to T$. We reduce to
	the statement that, for every Kan complex $K$, the restriction map
	\[
	\Map(\Delta^1, K) \lra \Map(\{0\}, K)
	\]
	is a weak homotopy equivalence.

	We prove (ii). Let
	\[
	E := \Delta^{\{\{0,2\},\{0,1\}\}} \subset \asd(\Delta^2).
	\]
	We first show that restriction along the cofibration
	\[
	\textstyle
	f: \asd( (\Delta^2)^{\lastedge}) \lra \asd( (\Delta^2)^{\lastedge}) \coprod_{E^{\flat}} E^{\sharp}
	\]
	induces an isomorphism
	\begin{equation}\label{eq:auxequ}
		\Map_T^{\sharp}(\asd( (\Delta^2)^{\lastedge}) \coprod_{E^{\flat}} E^{\sharp}, Y^{\natural}) 
		\lra \Map_T^{\sharp}(\asd( (\Delta^2)^{\lastedge}), Y^{\natural})^l.
	\end{equation}
	This assertion is equivalent to the following statement: A
	$2$-simplex in $\Span_T(Y)$, corresponding to a diagram in $Y$ of the form
	\begin{equation}\label{eq:unm}
		\xymatrix@=1pc{ 
		& & x_{\{1\}} & & \\ 
		&x_{\{0,1\}} \ar[dl] \ar[ur] & &\ar[dr]^{\natural}\ar[ul]_{\simeq} x_{\{1,2\}}& \\
		x_{\{0\}} & & \ar[ll] x_{\{0,2\}} \ar[uu] \ar[ul]_e \ar[ur] \ar[rr] & & x_{\{2\}},
		}
	\end{equation}
	is a Segal simplex if and only if the edge $e: x_{\{0,2\}} \to x_{\{0,1\}}$ is
	$p$-Cartesian. Here, we indicate all $p$-Cartesian edges by $\{\natural, \simeq\}$, distinguishing those $p$-Cartesian
	edges which are equivalences (since they lie over degenerate edges in $T$) by $\{\simeq\}$. 
	Note, that since the edge $e$ lies over a degenerate edge in $T$, this
	statement is in turn equivalent to saying that $e$ is an equivalence in the
	$\infty$-category $Y^0$ given by the fiber of $Y$ over $p(x_{\{0\}})$.
	To show this, we apply Lemma \ref{lem:comparelimit} to the diagram \eqref{eq:unm}, obtaining a
	corresponding diagram 
	\begin{equation}\label{eq:tunm}
		\xymatrix@=1pc{ 
		& & x'_{\{1\}} & & \\ 
		&x_{\{0,1\}} \ar[dl] \ar[ur] & & \ar[dr]^{\simeq}\ar[ul]_{\simeq} x'_{\{1,2\}}& \\
		x_{\{0\}} & & \ar[ll] x_{\{0,2\}} \ar[uu] \ar[ul]_e \ar[ur] \ar[rr] & & x'_{\{2\}},
		}
	\end{equation}
	contained in the fiber of $Y^0$. 
	The edge $e$ is an equivalence in $Y^0$ if and only if 
	the subdiagram of \eqref{eq:tunm} 
	\[
	\xymatrix{ 
		x_{\{0,2\}}\ar[d] \ar[r]^e & x_{\{0,1\}}\ar[d] \\
		x_{\{1,2\}}' \ar[r]^{\simeq} & x_{\{1\}}'
	}
	\]
	is a pullback diagram in $Y^0$. But, since every functor of $\infty$-categories preserves
	equivalences, we deduce from \cite[4.3.1.10]{lurie.htt} that the latter condition is in turn
	equivalent to the statement that the diagram \eqref{eq:tunm} (and hence, by Lemma
	\ref{lem:comparelimit}, the diagram \eqref{eq:unm}) represents a Segal simplex in
	$\Spanl_T(Y)$.
	
	Since the map \eqref{eq:auxequ} is a weak equivalence, statement (ii) is equivalent to the
	assertion that the map
	\[
	\textstyle
		g: \asd((\Lambda^2_2)^{\lastedge}) \subset \asd((\Delta^2)^{\lastedge}) \coprod_{E^{\flat}} E^{\sharp}
	\] 
	is $Y$-local.
	We can express $g$ as a composite 
	\[
	\textstyle
	\asd((\Lambda^2_2)^{\lastedge}) \overset{g_1}{\lra} \N(\Mo([2]) \setminus \{0,1\})^{\lastedge}
		\overset{g_2}{\lra} \asd((\Delta^2)^{\lastedge}) \coprod_{E^{\flat}} E^{\sharp}
	\]
	where $g_1$ is easily seen to be marked anodyne. The cofibration $g_2$ is a pushout of a map
	$g_2'$ whose {\em opposite} is marked anodyne. This implies that the map $g_2$ is $Y$-local, 
	since its only nondegenerate marked edge is $\{0,2\} \to \{0,1\}$ which maps to
	the equivalence, hence {\em $p$-coCartesian} edge, $e$ in $Y$.\\

	\noindent
	(3) We have to solve the lifting problem 
	\[
	\xymatrix{
	\displaystyle
	(\Lambda^2_1)^{\sharp} \coprod_{(\Lambda^2_1)^{\flat}} (\Delta^2)^{\flat} \ar[r] \ar@{^{(}->}[d] & \Spanl_T(Y^{\natural})\ar[d] \\
	(\Delta^2)^{\sharp} \ar[r]\ar@{-->}[ur] & T^{\sharp}.
	}
	\]
	which, passing to adjoints, translates into
	\[
	\xymatrix{
	\displaystyle
	\asd( (\Lambda^2_1)^{\sharp}) \coprod_{\asd(\Lambda^2_1)^{\flat}} \asd(\Delta^2)^{\flat}
	\ar[r]^-t \ar@{^{(}->}[d] & Y^{\natural}\ar[d] \\
	\asd((\Delta^2)^{\sharp}) \ar[r]\ar@{-->}[ur]^{\overline{t}} & T^{\sharp}.
	}
	\]
	The map $t$ corresponds to a diagram in $Y$ of the form
	\begin{equation}\label{eq:original-unmarked}
		\xymatrix@=1pc{ 
		& & x_{\{1\}} & & \\ 
		&x_{\{0,1\}} \ar[dl]_{\simeq} \ar[ur]^{\natural} & &\ar[dr]^{\natural}\ar[ul]_{\simeq} x_{\{1,2\}}& \\
		x_{\{0\}} & & \ar[ll]_{l} x_{\{0,2\}} \ar[uu] \ar[ul] \ar[ur] \ar[rr]^{r} & & x_{\{2\}},
		}
	\end{equation}
	where as above, we mark $p$-Cartesian edges and equivalences.
	Further, we know that \eqref{eq:original-unmarked} is a $p$-limit diagram, since it
	represents a Segal simplex of $\Span_T(Y)$. 
	We have to show that a lift $\overline{t}$ exists, which is equivalent to the assertion that $l$ and $r$ are
	$p$-Cartesian edges.
	We apply Lemma \ref{lem:comparelimit} to the diagram \eqref{eq:original-unmarked}, obtaining
	a corresponding limit diagram
	\begin{equation}\label{eq:transfer-unmarked}
		\xymatrix@=1pc{ 
		& & x'_{\{1\}} & & \\ 
		&x_{\{0,1\}} \ar[dl]_{\simeq} \ar[ur]^{\simeq} &
		&\ar[dr]^{\simeq}\ar[ul]_{\simeq} x'_{\{1,2\}}& \\
		x_{\{0\}} & & \ar[ll]_{l} x_{\{0,2\}} \ar[uu] \ar[ul] \ar[ur] \ar[rr] & & x'_{\{2\}}.
		}
	\end{equation}
	Note that the lower left triangle of \eqref{eq:transfer-unmarked} coincides by construction of $h$ with the
	corresponding triangle of \eqref{eq:original-unmarked}. As above, we conclude that the edge 
	$x_{\{0,2\}} \to x_{\{0,1\}}$ of \eqref{eq:transfer-unmarked} (and hence of \eqref{eq:original-unmarked}) is an equivalence.
	By multiple applications of \cite[2.4.1.7]{lurie.htt}, we deduce that every edge in
	\eqref{eq:original-unmarked} is $p$-Cartesian, in particular $l$ and $r$.\\

	\noindent
	(4) We have to solve any lifting problem of the form
	\begin{equation}\label{eq:lastlift}
		\xymatrix{
		\displaystyle
		K^{\flat} \ar[r] \ar@{^{(}->}[d] & \Spanl_T(Y^{\natural})\ar[d] \\
		K^{\sharp} \ar@{-->}[ur] \ar[r] & T^{\sharp},
		}
	\end{equation}
	where $K$ is a Kan complex. From the proof of \cite[3.1.1.6]{lurie.htt} it follows that
	it suffices to solve the problem for a constant map $K \to T$, hence we may assume
	that $T = \pt$. In this case, we deduce from (1) that $\Spanl(Y)$ is an $\infty$-category,
	and an edge of $\Spanl(Y)$ is $q$-Cartesian if and only if it is
	an equivalence. Therefore, it suffices to show the following statement: 
	If an edge $x \to y$ of $\Spanl(Y)$ is an equivalence, then the two edges $x \leftarrow z
	\rightarrow y$ comprising the adjoint map $\asd(\Delta^1) \to Y$ are equivalences in $Y$.
	An edge in an $\infty$-category is an equivalence if and only if it admits an inverse in the
	associated homotopy category. Assume an edge $f: x \to y$ of $\Spanl(Y)$ is an equivalence
	and let $f': y \to x$ be a representative of the homotopy inverse of $f$.
	There exist triangles $l: \Delta^{\{0,1,2\}} \to \Spanl(Y)$ with $l(\{0,1\}) = f'$, $l(\{1,2\}) = f$
	and $l(\{0,2\}) = \id_x$, as well as $r: \Delta^{\{1,2,3\}} \to \Spanl(Y)$ with $r(\{1,2\}) = f$,
	$r(\{2,3\}) = f'$ and $r(\{1,3\}) = \id_y$. The inclusion
	\[
	\Delta^{\{0,1,2\}} \coprod_{\Delta^{\{1,2\}}} \Delta^{\{1,2,3\}} \subset \Delta^{\{0,1,2,3\}}
	\]
	is inner anodyne, which allows us to extend the pair $(l,r)$ of $2$-simplices to a
	$3$-simplex
	\[
	m: \Delta^{\{0,1,2,3\}} \to \Spanl(Y).
	\]
	The adjoint map $\asd(\Delta^{\{0,1,2,3\}}) \to Y$ corresponds to a diagram 
	\begin{equation}\label{eq:3simplexsegal}
		\xymatrix@=1pc{ 
		&& & w \ar@/_1.5pc/[dldl]_{l_1} \ar@/^1.5pc/[drdr]^{r_1} \ar[dr] \ar[dl] & & &\\ 
		& & y \ar[dr] \ar[dl]^{l_2} & & x \ar[dr]_{r_2} \ar[dl]& & \\
		& z' \ar[dr] \ar[dl]& & z \ar[dr]^{r_3} \ar[dl]_{l_3} & & z' \ar[dr] \ar[dl]& \\
		y &  & x & & y &  &x. \\
		}
	\end{equation}
	in $Y$. Since $m$ is a Segal simplex, all rectangles in \eqref{eq:3simplexsegal} are
	pullback squares. Now we argue as follows: $r_1$ is a pullback of $\id: y \to y$, hence an
	equivalence. But this implies that the class of $r_2$ in the homotopy category has a left
	and right inverse, making $r_2$ itself an equivalence. Analogously, $l_1$ and $l_2$ are
	equivalences. An easy argument now implies that in fact all maps in \eqref{eq:3simplexsegal}
	are equivalences, in particular the maps $l_3$ and $r_3$.
\end{proof}

\begin{rem} Note that, in the case $T = \pt$, Theorem \ref{thm:spancartesian} implies Theorem
	\ref{thm:spancat} and, further, identifies those edges of $\Spanl(\C)$ which are
	equivalences.
\end{rem}

The following Proposition studies how the span construction interacts with marked mapping
spaces. This will be essential for the proof of Theorem \ref{thm:span-bi-mon}.

\begin{prop} \label{prop:mapadj} Let $Y \to T$ be a Cartesian fibration of $\infty$-categories which
	admits relative pullbacks.
	Then the following assertions hold:
	\begin{enumerate}
		\item For every object $L$ of $(\mSet)_{/T}$, we have an isomorphism of Kan
			complexes
			\[
			\Map^{\sharp}_T(L, \Spanl_T(Y^{\natural})) \cong
			\Spanl(\Map^{\sharp}_T(\asd(L), Y^{\natural})^l)
			\] 
			which is functorial in $L$.
		\item For every object $L$ of $(\sSet)_{/T}$, we have an isomorphism of
			$\infty$-categories
			\[
			\Map^{\flat}_T(L^{\sharp}, \Spanl_T(Y^{\natural})) \cong
			\Spanl(\Map^{\flat}_T(\asd(L)^{\sharp}, Y^{\natural}))
			\]
			which is functorial in $L$.
	\end{enumerate}
\end{prop}

The proof of the proposition needs some technical preparation. Let $p: Y \to T$ be a Cartesian fibration of $\infty$-categories which is classified by a diagram $f: T^{\op} \to
\Cat_{\infty}$. Recall from \cite[3.3.3]{lurie.htt} that the limit of $f$ can be identified
with the $\infty$-category $\Map^{\flat}_T(T^{\sharp}, Y^{\natural})$ of Cartesian sections of $p$. The
following proposition gives a pointwise characterization of limits in $\Map^{\flat}_T(T^{\sharp},Y^{\natural})$.

\begin{prop} \label{prop:pointwise} Let $p: Y \to T$ be a Cartesian fibration of $\infty$-categories and $K$ be a simplicial
	set. Assume that $Y$ admits all
	$K$-indexed $p$-limits. Then
	\begin{enumerate}[label=(\roman{*})] 
		\item The $\infty$-category $\Map^{\flat}_T(T^{\sharp}, Y^{\natural})$ admits all
			$K$-indexed limits.
		\item A diagram $f: K^{\triangleleft} \to \Map^{\flat}_T(T^{\sharp}, Y^{\natural})$ is a
			limit diagram if and only if, for every vertex $s$ of $T$, the corresponding map $\{s\} \times
			K^{\triangleleft} \to Y_s$ is a limit diagram.
	\end{enumerate}
\end{prop}
\begin{proof} Let $K \to \Map^{\flat}_T(T^{\sharp}, Y^{\natural})$ and consider the adjoint map $q_T: T
	\times K \to Y$. By \cite[5.2.5.4]{lurie.htt}, the map $p': Y^{/q_T} \to T$ is a Cartesian
	fibration. Since the Cartesian fibration $p: Y \to T$ admits $K$-indexed $p$-limits, we
	deduce by \cite[4.3.1.10]{lurie.htt} that, for every vertex $s$ of $T$, the fiber $Y_s$ admits $K$-indexed
	limits and the functors associated to $p$ preserve $K$-indexed limits in the fibers of $p$.
	This translates into the statement that the fibers of $p'$ admit final objects and the
	functors associated to $p'$ preserve final objects in the fibers of $p'$.
	The claimed assertions now follow immediately from Lemma \ref{lem:pointwise} below.
\end{proof}

\begin{lem}\label{lem:pointwise} Let $p:Y \to T$ be a Cartesian fibration of $\infty$-categories. Assume that for each
	vertex $s$ of $T$, the $\infty$-category $Y_s$ admits a final object. Further, assume that
	the functors associated to $p$ preserve final objects in the fibers of $p$. 
	\begin{enumerate} 
		\item Let $Y' \subset Y$ denote the largest simplicial subset of $Y$ such that each
			vertex $y$ of $Y'$ is a final object in $Y_{p(y)}$ and each edge of $Y'$ is
			$p$-Cartesian. Then $p|Y'$ is a trivial fibration of simplicial sets.
		\item Let $\C = \Map_T(T^{\sharp}, Y^{\natural})$ be the $\infty$-category of
			Cartesian sections of $p$. A Cartesian section $f: T \to Y$ is a final
			object of $\C$ if and only if it factors through $Y'$.
	\end{enumerate}
\end{lem}
\begin{proof} The argument is an adaption of the proof of \cite[2.4.4.9]{lurie.htt}.
	To prove the first assertion, it suffices to show that, for every $n \ge 0$, every lifting problem
	\begin{equation}\label{eq:smalllift}
	\xymatrix{
		\partial \Delta^n \ar[r]^{g_0} \ar[d] & Y \ar[d]^p \\
		\Delta^n \ar[r]^h \ar@{-->}[ur]^{g} & T,
	}
	\end{equation}
	such that $g_0(\{n\})$ is a final object in the $\infty$-category $Y_{h(\{n\})}$, admits a
	solution. To solve this problem, we may replace $T$ by $\Delta^n$. By \cite[4.3.1.10]{lurie.htt}, the vertex $g_0(\{n\})$ is a $p$-final
	object in $Y$. Since $p(g_0(\{n\})) = \{n\}$ is a final object of $T=\Delta^n$, the vertex
	$g_0(\{n\})$ is a final object of $Y$ by \cite[4.3.1.5]{lurie.htt}. This immediately implies
	the existence of a solution $g$ of the above lifting problem. 

	From (1), we deduce the existence of a section $f: T \to Y'$ of $p$. 
	By the uniqueness of final objects in $\infty$-categories, to show (2) it suffices to
	prove that $f$ is a final object in $\C$. To this end, we have to show that any lifting
	problem
	\[
	\xymatrix{
		T \times \partial \Delta^n \ar[r]^{j} \ar[d] & Y\ar[d]^p\\
		T \times \Delta^n \ar[r]^h \ar@{-->}[ur] &  T,
	}
	\]
	with $j|T \times \{n\} = f$ admits a solution. This solution can be found simplex by simplex
	using that \eqref{eq:smalllift} admits a solution.
\end{proof}

The following lemma isolates the most technical part of the proof of Proposition \ref{prop:mapadj}.

\begin{lem} \label{lem:limcom} Let $Y \to T$ be a Cartesian fibration of $\infty$-categories which
	admits relative pullbacks.
\begin{enumerate} 
	\item \label{lem:limcom1} Let $L$ be an object of $(\mSet)_{/T}$ and $K$ a simplicial set. 
		Let $f: \asd(L) \times \asd(K)^{\sharp} \to Y^{\natural}$ be a morphism in
		$(\mSet)_{/T}$. Assume that for each vertex $\{y\}$ of $\asd(K)$, the induced map $f_y: \asd(L) \times \{y\}^{\sharp} \to Y^{\natural}$
		is contained in $\Hom_T(\asd(L) \times \{y\}^{\sharp}, Y^{\natural})^l$.
		Then the map $f$ itself is contained in $\Hom_T(\asd(L) \times \asd(K^{\sharp}),
		Y^{\natural})^l$.
	\item \label{lem:limcom2} Let $L$ be an object of $(\sSet)_{/T}$ and $K$ a simplicial set. 
		Let $f: \asd(L^{\sharp}) \times \asd(K)^{\flat} \to Y^{\natural}$ be a morphism in
		$(\mSet)_{/T}$. Assume that for each vertex $\{y\}$ of $\asd(L)$, the induced map
		$f_y: \{y\} \times \asd(K)^{\flat} \to Y^{\natural}$
		is contained in $\Hom_T(\{y\} \times \asd(K)^{\flat}, Y^{\natural})^l$.
		Then the map $f$ itself is contained in $\Hom_T(\asd(L^{\sharp}) \times
		\asd(K)^{\flat}, Y^{\natural})^l$.
\end{enumerate}
\end{lem}
\begin{proof}
	(1) We have to show that, under the stated assumption, the adjoint 
	$g: L \times K^{\sharp} \to \Span_T(Y^{\natural})$ of the map $f: \asd(L) \times \asd(K^{\sharp}) \to
	Y^{\natural}$ factors over $\Spanl_T(Y) \subset \Span_T(Y)$. In other words, we
	have to show that, for every $n$-simplex $\sigma = (\sigma_L, \sigma_K)$ of $L \times K$,
	the image $g(\sigma)$ is a Segal simplex of $\Span_T(Y)$. It suffices to show that the induced
	composite
	\[
	s_d: T(\Delta^n) \hra \asd(\Delta^n) \overset{\sigma}{\lra} \asd(L) \times \asd(K) \overset{f}{\lra} Y
	\]
	is a $p$-limit diagram. Indeed, the $p$-limit condition on a subsimplex $\sigma' \subset
	\sigma$ will be obtained by repeating the same argument with $\sigma$ replaced by $\sigma'$.
	We consider the composite map 
	\[
	s: \asd(\Delta^n) \times \asd(\Delta^n) \overset{(\sigma_L,\sigma_K)}{\lra} \asd(L)
	\times \asd(K) \overset{f}{\lra} Y
	\]
	on the underlying simplicial sets.
	Note that $s_d$ is obtained from $s$ by restricting along the diagonal embedding
	\[
	S(\Delta^n) \to \Delta^n \times \Delta^n.
	\]
	Further, we obtain another Segal cone $s_0$ by restricting $s$ along the embedding
	\[
	S(\Delta^n) \to \Delta^n \times \{0\}.
	\]
	Note that, by our assumption, the Segal cone $s_0$ is a $p$-limit cone. We define a map 
	\[
	\widetilde{h}: \asd(\Delta^1) \times \asd(\Delta^n) \hra \asd(\Delta^n) \times \asd(\Delta^n)
	\]
	as the nerve of the functor
	\[
	\Mo([1]) \times \Mo([n]) \lra \Mo([n]),\; (I, \{i,j\}) \mapsto 
	\left\{ \begin{matrix}
		(\{i,j\},\{0\}) \quad \text{if $I = \{0\}$,}\\
		(\{i,j\},\{0,j\}) \quad \text{if $I = \{0,1\}$,}\\
		(\{i,j\},\{i,j\}) \quad \text{if $I = \{1\}$.}
	\end{matrix} \right.
	\]
	We let
	\begin{align*}
	E & := \Delta^{\{\{0,1\},\{0\}\}} \subset \asd(\Delta^1) &  F & := \Delta^{\{\{0,1\},\{1\}\}} \subset \asd(\Delta^1) 
	\end{align*}
	denote the two nondegenerate edges of $\asd(\Delta^1)$.
	The map $h = s \circ \widetilde{h}|\asd(\Delta^1) \times S(\Delta^n)$ is a concatenation of
	two homotopies $h_1 = h | E \times S(\Delta^n)$ and $h_2 = h | F \times S(\Delta^n)$ with the following
	properties:
	\begin{enumerate}
		\item The Segal cone $h_1 | \{0\} \times S(\Delta^n)$ factors as a composition
			\[
			S(\Delta^n) \lra \asd(L) \times \{y\} \overset{f_y}{\lra} Y
			\]
			and is therefore, by our assumption, a $p$-limit diagram.
		\item By construction, the Segal cones $h_1 | \{0,1\} \times S(\Delta^n)$ and $h_2|
			\{0,1\} \times S(\Delta^n)$ coincide.
		\item The Segal cone $h_2 | \{1\} \times S(\Delta^n)$ coincides with $s_d$.
		\item For every vertex $\{v\}$ of $S(\Delta^n)$, the edges $h_1 | E \times \{v\}$
			and $h_2 | F \times \{v\}$ are $p$-Cartesian. This follows since in the
			definition of the map $f$ every edge of $\asd(K)$ is marked.
		\item The edges $h_1 | E \times \{0,n\}$ and $h_2 | F \times \{0,n\}$ in $Y$ map to
			degenerate edges in $T$.
	\end{enumerate}
	Hence, we conclude the argument by \cite[4.3.1.9]{lurie.htt} which implies that $s_d$ is a
	$p$-limit diagram.\\

	(2) This follows from an argument similar to the one provided in (1).
\end{proof}

\begin{proof}[Proof of Proposition \ref{prop:mapadj}]
	(1) For every simplicial set $K$, we have a chain of natural isomorphisms
	\begin{align}
		\notag \Hom(K, \Map^{\sharp}_T(L, \Spanl_T(Y^{\natural}))) 
			 & \cong \Hom_T(L \times K^{\sharp}, \Spanl_T(Y^{\natural}))\\
			 \tag{I} \label{a.I} & \cong \Hom_T(\asd(L) \times \asd(K^{\sharp}), Y^{\natural})^l\\
			 \tag{II} \label{a.II} & \cong \Hom_T(\asd(L) \times \asd(K)^{\sharp}, Y^{\natural})^l\\
			 \tag{III} \label{a.III} & \cong \Hom(\asd(K), \Map^{\sharp}_T(\asd(L) , Y^{\natural})^l)\\
		\notag & \cong \Hom(K, \Spanl(\Map^{\sharp}_T(\asd(L) , Y^{\natural})^l))
	\end{align}
	The only nontrivial identifications are \eqref{a.I} $\cong$ \eqref{a.II} and \eqref{a.II} $\cong$ \eqref{a.III}.

	To obtain the identification \eqref{a.I} $\cong$ \eqref{a.II}, we will show that the map $\asd(L) \times \asd(K^{\sharp}) \to
	\asd(L) \times \asd(K)^{\sharp}$ is marked anodyne. By \cite[3.1.2.3]{lurie.htt}, it
	suffices to prove that $\asd(K^{\sharp}) \to \asd(K)^{\sharp}$ is marked anodyne.
	Arguing simplex by simplex, it suffices to show this in the case $K = \Delta^n$, $n \ge 0$. For $n > 3$, every edge of $\asd(\Delta^n)$ is
	contained in $\asd(\Delta^{3}) \subset \asd(\Delta^n)$ for some subsimplex $\Delta^{3} \subset
	\Delta^n$. Therefore, it suffices to prove the statement for $1 \le n \le 3$. For $n=1$, the
	assertion is trivial, while the cofibrations $\asd( (\Delta^2)^{\sharp}) \to
	\asd(\Delta^2)^{\sharp}$ and $\asd( (\Delta^3)^{\sharp}) \to \asd(\Delta^3)^{\sharp}$ are
	easily seen to be iterated pushouts of the marked anodyne morphisms
	\[
	(\Lambda^2_2)^{\sharp} \coprod_{(\Lambda^2_2)^{\flat}} (\Delta^2)^{\flat} \hra (\Delta^2)^{\sharp}
	\]
	and 
	\[
	(\Lambda^2_1)^{\sharp} \coprod_{(\Lambda^2_1)^{\flat}} (\Delta^2)^{\flat} \hra
	(\Delta^2)^{\sharp}.
	\]

	To show the identification \eqref{a.II} $\cong$ \eqref{a.III}, first note that, by adjunction, we have a natural isomorphism
	\[
	  \Hom_T(\asd(L) \times \asd(K)^{\sharp}, Y^{\natural}) \cong \Hom(\asd(K), \Map^{\sharp}_T(\asd(L) , Y^{\natural})).
	\]
	The claim that this identification descends to \eqref{a.II} $\cong$ \eqref{a.III} follows
	immediately from \ref{lem:limcom1} in Lemma \ref{lem:limcom}.\\

	\noindent
	(2) For every simplicial set $K$, we have a chain of natural isomorphisms
	\begin{align}
		\notag \Hom(K, \Map^{\flat}_T(L^{\sharp}, \Spanl_T(Y^{\natural}))) 
		& \cong \Hom_T(L^{\sharp} \times K^{\flat}, \Spanl_T(Y^{\natural}))\\
		\tag{I} \label{b.I} & \cong \Hom_T(\asd(L^{\sharp}) \times \asd(K^{\flat}), Y^{\natural})^l\\
		\tag{II} \label{b.II} & \cong \Hom_T(\asd(L)^{\sharp} \times \asd(K^{\flat}), Y^{\natural})^l\\
		\tag{III} \label{b.III} & \cong \Hom(\asd(K), \Map^{\flat}_T(\asd(L)^{\sharp},Y^{\natural}))^l\\
		\notag & \cong \Hom(K, \Spanl(\Map^{\flat}_T(\asd(L)^{\sharp} , Y^{\natural})))
	\end{align}

	The identification \eqref{b.I} $\cong$ \eqref{b.II} follows as in Part (1) from the fact
	that the map $\asd(L^{\sharp}) \to \asd(L)^{\sharp}$ is marked anodyne. The isomorphism
	\eqref{b.II} $\cong$ \eqref{b.III} follows from \ref{lem:limcom2} in Lemma \ref{lem:limcom}
	and Proposition \ref{prop:pointwise}. 
\end{proof}

\begin{cor} \label{cor:preserve} 
	Let $Y \overset{q}{\lra} Z \overset{p}{\lra} T$ be maps of $\infty$-categories. Assume that
	$p$ and $p \circ q$ are Cartesian fibrations which admit relative pullbacks. Further assume
	that $q$ is a Cartesian equivalence. Then the induced map $\Span_T(Y) \to \Span_T(Z)$
	descends to a Cartesian equivalence $\Spanl_T(q): \Spanl_T(Y) \to \Spanl_T(Z)$.
\end{cor}
\begin{proof}
	Since the objects $Y^{\natural}$ and $Z^{\natural}$ are fibrant objects of $(\mSet)_{/T}$
	equipped with the Cartesian model structure, the map $q$ is a categorical equivalence by
	\cite[3.1.5.3]{lurie.htt}. Hence, by \cite[4.3.1.6]{lurie.htt}, the map $q$ preserves
	relative limits and we obtain a well-defined induced map $\Spanl_T(q): \Spanl_T(Y) \to
	\Spanl_T(Z)$. To show that $\Spanl_T(q)$ is a Cartesian equivalence, it suffices to show
	that, for every object $L \in (\mSet)_{/T}$, the induced map of mapping spaces
	$\Map^{\sharp}_T(L, \Spanl_T(Z^{\natural})) \to \Map^{\sharp}_T(L, \Spanl_T(Y^{\natural}))$
	is a weak equivalence of Kan complexes. Since $\Map^{\sharp}_T(\asd(L), Z^{\natural}) \to
	\Map^{\sharp}_T(\asd(L), Y^{\natural})$ is a weak equivalence, this follows from
	\cite[4.3.1.6]{lurie.htt}, Proposition \ref{prop:mapadj}(1), and Corollary
	\ref{cor:asd-quillen}.
\end{proof}

\begin{thm}\label{thm:span-bi-mon} Let $p: Y \to \N(\Delta)$ be a Segal fibration admitting
	relative pullbacks. Then the
	following assertions hold:
	\begin{enumerate}
		\item The map $\Spanl_{\N(\Delta)}(Y) \to \N(\Delta)$ is a Segal fibration.
		\item Assume $Y \to \ND$ is complete. Then the Segal fibration $\Spanl_{\N(\Delta)}(Y) \to \N(\Delta)$
			is complete.
		\item Assume $Y \to \N(\Delta)$ exhibits a monoidal structure on the
			$\infty$-category $\C = Y_{[0]}$. Then $\Spanl_{\N(\Delta)}(Y) \to \ND$
			exhibits a monoidal structure on the $\infty$-category $\Span(\C)$.
	\end{enumerate}
\end{thm}
\begin{proof} To show part (1), we have to verify the conditions of Definition \ref{defi:dinftybi}.
	Condition \ref{dinftybi.1} follows immediately from Theorem \ref{thm:spancartesian}. 
	To verify condition \ref{dinftybi.2}, let $n \ge 2$ and denote by $L^{\triangleright}$ the opposite Segal cone $S(\Delta^n)^{\op}$.
	By \cite[3.3.3.1]{lurie.htt} it suffices to show that, for every $n \ge 2$, the map 
	\[
	\Map^{\flat}_{\ND}( (L^{\triangleright})^{\sharp}, \Spanl_{\ND}(Y^{\natural})) \lra \Map^{\flat}_{\ND}(
	L^{\sharp}, \Spanl_{\ND}(Y^{\natural}))
	\]
	is an equivalence of $\infty$-categories. Using Corollary \ref{cor:preserve} and Proposition
	\ref{prop:mapadj}(2), we reduce to the statement that the map 
	\[
	\Map^{\flat}_{\ND}( \asd(L^{\triangleright})^{\sharp}, Y^{\natural}) \lra \Map^{\flat}_{\ND}( \asd(L)^{\sharp}, Y^{\natural}) 
	\]
	is an equivalence of $\infty$-categories. 
	Using Lemma \ref{lem:asdsharp} below, we reduce further to the statement that the map
	\[
	\Map^{\flat}_{\ND}( (L^{\triangleright})^{\sharp}, Y^{\natural}) \lra \Map^{\flat}_{\ND}( L^{\sharp}, Y^{\natural}) 
	\]
	is an equivalence of $\infty$-categories which, again by \cite[3.3.3.1]{lurie.htt}, is
	equivalent to condition \ref{dinftybi.2} for the Segal fibration $Y \to \ND$.
	Condition \ref{dinftybi.3} follows immediately from Proposition \ref{prop:asd-quillen} and Example \ref{exa:spankan}. 

	We show assertion (2). Consider the functor 
	\[
		f: Y_{[0]} \to Y_{[1]}
	\]
	associated to the unique edge $[1] \to [0]$ of $\ND$ via the Cartesian fibration $Y \to \ND$.
	The statement that $Y \to \ND$ is complete, means, by definition, that $f$ induces a weak
	equivalence of Kan complexes
	\[
		Y_{[0]} \lra (Y_{[1]})_{\Kan}^{\on{equiv}}
	\]
	where we use the terminology of \S \ref{subsec:quasicat-1-segal}. Using Corollary
	\ref{cor:asd-quillen}, we obtain a weak equivalence of Kan complexes
	\[
		\Spanl(f): \Spanl(Y_{[0]}) \lra \Spanl((Y_{[1]})_{\Kan}^{\on{equiv}}).
	\]
	Using Theorem \ref{thm:spancartesian}, we can naturally identify the Kan complex
	$\Spanl((Y_{[1]})_{\Kan}^{\on{equiv}})$ with $\Spanl(Y_{[1]})_{\Kan}^{\on{equiv}}$ so that
	$\Spanl(f)$ is the functor associated to the edge $[1] \to [0]$ via the Cartesian fibration
	$\Spanl_{\ND}(Y) \to \ND$. Hence the Segal fibration $\Spanl_{\ND}(Y) \to \ND$ is complete.

	It remains to prove assertion (3). Note that, for a Kan complex $K$, we have weak homotopy equivalences
	\[
	\xymatrix{
	\Spanl(K) & \ar[l]_-{f} \asd(\Spanl(K)) \ar[r]^-{g} & K
	}
	\]
	where $f$ is the weak equivalence from Proposition \ref{prop:asd-id}
	and $g$ is the counit map corresponding to the Quillen equivalence of Proposition
	\ref{prop:asd-quillen}. Thus, if $Y_{[0]}$ is contractible then $\Spanl_T(Y)_{[0]}
	\cong \Span(Y_{[0]})$ is contractible as well.
\end{proof}

\begin{lem}\label{lem:asdsharp} Let $L$ be an object of $(\sSet)_{/T}$ and $Y \to T$ a Cartesian
	fibration of $\infty$-categories. Then the natural map $\asd(L) \to L$ induces an equivalence of
	$\infty$-categories
	\[
	\Map^{\flat}_T(L^{\sharp}, Y^{\natural}) \lra
	\Map^{\flat}_T( \asd(L)^{\sharp}, Y^{\natural}).
	\]
\end{lem}
\begin{proof} We argue simplex by simplex as in \cite[2.2.2.7]{lurie.htt} using Remark
	\cite[3.1.4.5]{lurie.htt}. For a simplex $\Delta^n \to T$, we argue as follows. The map
	$\asd(\Delta^n) \to \Delta^n$ admits a section given by the nerve of the functor
	\[
	s: [n] \lra \Mo([n]),\; \{k\} \mapsto \{k,n\}.
	\]
	Note that $\N(s)$ identifies $\Delta^n$ with a full subcategory of $\asd(\Delta^n)$.
	Further, it is easy to see that every vertex of 
	\[
	\Map^{\flat}_T(\asd(\Delta^n)^{\sharp}, Y^{\natural})	
	\]
	is a $p$-left Kan extension of its restriction to $\Delta^n$. Thus, we can apply
	\cite[4.3.2.15]{lurie.htt} to deduce that the restriction map
	\[
	\Map^{\flat}_T(\asd(\Delta^n)^{\sharp}, Y^{\natural}) \overset{s^*}{\lra} \Map^{\flat}_T(
	(\Delta^n)^{\sharp}, Y^{\natural})
	\]
	is a trivial fibration of simplicial sets, in particular, an equivalence of
	$\infty$-categories. The final statement now follows from the $2$-out-of-$3$ property of weak
	equivalences for the Joyal model structure on $\sSet$.
\end{proof}

\vfill\eject

\subsection{Horizontal Spans}
\label{subsec:horizontal}

Let $\C$ be an $\infty$-category which admits pullbacks. In this section, we associate to $\C$ a
complete Segal fibration $\HSpan(\C) \to \N(\Delta)$ which models an $\inftytwo$-category $\B$, referred
to as the {\em $\inftytwo$-category of horizontal spans in $\C$}. Informally, we can describe $\B$ as follows:
\begin{itemize}
	\item The objects of $\B$ are given by objects of $\C$.
	\item A $1$-morphisms between objects $x$ and $y$ of $\B$ is given by a span diagram $x
		\leftarrow z \to y$ in $\C$. Composition of $1$-morphisms is given by forming
		pullbacks.
	\item A $2$-morphism between $1$-morphisms $x \leftarrow z \to y$ and $x \leftarrow z' \to
		y$ is given by a diagram
		
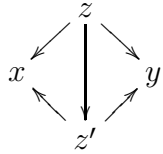
\[
		\xymatrix@=1pc{
		&\ar[dl]\ar[dr] z\ar[dd] & \\
		x &  & y \\
		&\ar[ul]\ar[ur] z' & }
		\]
		in $\C$.
	\item The higher morphisms are given by spans, in which both edges are equivalences, of
	  spans of spans of \dots in $\C$.
\end{itemize}

\begin{defi} We define a category $\Deltac$ as follows. 
\begin{itemize}
	\item The objects of $\Deltac$ are given by pairs $([n], \{i,j\})$, where $[n]$ is a finite
		nonempty ordinal and $0 \le i \le j \le n$.
	\item A morphism between objects $([n], \{i,j\})$ and $([m], \{i',j'\})$ is given by a
		morphism $f: [n] \to [m]$ of underlying ordinals such that $f(i) \le i' \le j' \le
		f(j)$.
\end{itemize}		
\end{defi}
The forgetful functor $\Deltac \to \Delta$ is a Grothendieck opfibration which
implies that the induced functor $\pi: \N(\Deltac) \to \N(\Delta)$ is a coCartesian fibration
of $\infty$-categories. 

\begin{rem} The functor 
	\[
		P^{\bullet}: \Delta \lra \Cat, \;[n] \mapsto I_{[n]}^{\op}
	\]
	defined in \eqref{eq:P} induces a functor $\N(P^{\bullet}): \N(\Delta) \to
	\Cat_{\infty}$. This functor classifies the coCartesian fibration $\pi$ in the sense of
	\cite[3.3.2]{lurie.htt}. In other words, the functor $\pi$ is obtained from $P^{\bullet}$
	via a Grothendieck construction. In comparison, the coCartesian fibration
	$\N(\Delta^{\times}) \to \ND$ from \S \ref{subsec:higherhall} corresponds, via the
	Grothendieck construction, to the functor
	\[
		\Delta \lra \Cat, \;[n] \mapsto I_{[n]}.
	\]
\end{rem}

\begin{rem} The nomenclature for $\Deltac$ is chosen to be compatible with 
	\cite[1.2.8]{lurie.noncommutative} where the Cartesian monoidal structure on an
	$\infty$-category $\C$ with products is constructed. Given an $\infty$-category $\C$ with {\em
	coproducts}, we can construct the {\em coCartesian} monoidal structure along the lines of loc.
	cit, by using the Cartesian fibration $\N(\Delta^{\! \displaystyle \times})^{\op} \to \N(\Delta)^{\op}$
	instead of the Cartesian fibration $\pi^{\op}: \N(\Deltac)^{\op} \to \N(\Delta)^{\op}$.
	This will result in a coSegal fibration $\C^{\!\, \displaystyle \amalg} \to \N(\Delta)^{\op}$ exhibiting
	the coCartesian monoidal structure on $\C$. 
\end{rem}

Let $Y \to \N(\Deltac)$ be a map of simplicial sets. We define a map $\pi_*Y \to \N(\Delta)$
characterized by the universal property
\[
\Hom_{\N(\Delta)}(K, \pi_*Y) \cong \Hom_{\N(\Deltac)}(K \times_{\N(\Delta)} \N(\Deltac), Y).
\]
For an $\infty$-category $\C$, we introduce the notation $\HSpanp(\C) := \pi_*(\N(\Deltac) \times
\C)$.

\begin{defi}\label{defi:hspan}
Let $\C$ be an $\infty$-category and consider the map $p: \HSpanp(\C) \to \N(\Delta)$. 
By the characterizing property of $p$, for every ordinal $[n]$, the fiber $\HSpanp(\C)_{[n]}$ can be identified 
with the $\infty$-category of functors $\Fun(\asd(\Delta^n),\C)$. 
\begin{enumerate}[label=(\Alph*)]
	\item \label{l.vertex} We call a vertex of $\HSpanp(\C)$
	{\em admissible} if the corresponding functor $F: \asd(\Delta^n) \to \C$ 
	satisfies the following condition: 
	\begin{itemize}
		\item For every subsimplex $\Delta^k \to \Delta^n$, with $k \ge 2$, the
			corresponding Segal cone (Definition \ref{defi:segalcone}) given by the
			composite $S(\Delta^k) \to \asd(\Delta^n) \overset{F}{\to} \C$ is a limit diagram in $\C$.
	\end{itemize}
	\item \label{l.edge} An edge $e: F \to G$ of $\HSpanp(\C)$, which lies over an edge $f: [n] \to [m]$, is called {\em admissible} if 
	it satisfies the following condition:
	\begin{itemize} 
		\item For every $0 \le i \le n$, the edge $F(\{i\}) \to G(\{f(i)\})$ in $\C$
			induced by $e$ is an equivalence.
	\end{itemize}
\end{enumerate}
Using this terminology, we define $\HSpan(\C) \subset \HSpanp(\C)$ to be the largest simplicial
subset such that every vertex and every edge is admissible.
\end{defi}

\begin{thm}\label{thm:horispan} Let $\C$ be an $\infty$-category. Then the following hold:
	\begin{enumerate} 
		\item The map $\HSpanp(\C) \to \N(\Delta)$ is a Cartesian fibration.
		\item Assume that $\C$ admits pullbacks. Then the map $\HSpan(\C) \to \N(\Delta)$ is
			a complete Segal fibration which admits relative pullbacks.
	\end{enumerate}
\end{thm}
\begin{proof}
	Assertion (1) follows from the dual statement of \cite[3.2.2.13]{lurie.htt}. We show (2). 
	First note that $\HSpan(\C) \to \N(\Delta)$ is a Cartesian fibration:
	Condition \ref{l.vertex} is preserved under the functors associated to the Cartesian
	fibration $\HSpanp(\C) \to \N(\Delta)$ and condition \ref{l.edge} is immediately checked to
	be compatible with the respective lifting problems. We let $Y = \HSpan(\C)$.
	To verify condition \ref{dinftybi.2} of Definition \ref{defi:dinftybi}, we have to show that,
	for every $n \ge 2$, the Segal cone diagram in $\Cat_\infty$
	\begin{equation}\label{eq:segalcone1}
		\xymatrix{ &&&Y_{[n]} \ar[dlll] \ar[dll] \ar[dl] \ar[d] \ar[drr] \ar[drrr]&  &&\\
		Y_{\{0\}} & \ar[l] Y_{\{0,1\}} \ar[r] & Y_{\{1\}} & \ar[l] Y_{\{1,2\}} & \dots & Y_{\{n-1,n\}} \ar[r]& Y_{\{n\}} 
		}
	\end{equation}
	classifying the Cartesian fibration $Y \times_{\N(\Delta)} S(\Delta^n)^{\op}
	\to S(\Delta^n)^{\op}$ is a limit diagram in $\Cat_{\infty}$. Recall the notation
	\[
	\J^n = \Delta^{\{0,1\}} \coprod_{\{1\}} \dots \coprod_{\{n-1\}} \Delta^{\{n-1,n\}} \subset
	\Delta^n.
	\]
	Consider the inclusion
	$j: \asd(\J^n) \subset \asd(\Delta^n)$ and the corresponding restriction functor
	\[
	j^*: \Fun(\asd(\Delta^n), \C) \lra \Fun(\asd(\J^n), \C).
	\]
	Let $\D \subset \Fun(\asd(\Delta^n), \C)$ denote the full subcategory spanned by the vertices
	satisfying condition \ref{l.vertex}. A vertex $F$ of $\Fun(\asd(\Delta^n),\C)$ lies in
	$\D$ if and
	only if it is a right Kan extension of its restriction $F|\asd(\J^n)$. On the other hand,
	since $\C$ admits pullbacks, we deduce from Proposition \ref{prop:pullsegal} that every vertex of $\Fun(\asd(\J^n), \C)$ admits a right Kan
	extension along $j$. By \cite[4.3.2.15]{lurie.htt}, the induced map
	\[
		\D \lra \Fun(\asd(\J^n), \C)
	\]
	is a trivial fibration of simplicial sets. Further, we have a pullback diagram of
	simplicial sets
	\[
	\xymatrix{
		Y_{[n]} \ar[r]\ar[d] & \ar[d] \D\\
	{ Y_{\{0,1\}} \times_{Y_{\{1\}}} Y_{\{1,2\}} \times \dots \times_{Y_{\{n-1\}}} Y_{\{n-1,n\}}
	}\ar[r] & \Fun(\asd(\J^n), \C)
	}
	\]
	which shows that the restriction functor $j^*$ induces a trivial fibration 
	\[
	f: Y_{[n]} \lra Y_{\{0,1\}} \times_{Y_{\{1\}}}
	Y_{\{1,2\}} \times \dots \times_{Y_{\{n-1\}}} Y_{\{n-1,n\}}.
	\]
	The equivalence $f$ of $\infty$-categories induces an equivalence between the Segal
	cone \eqref{eq:segalcone1} and the Segal cone
	\begin{equation}\label{eq:segalcone2}
		\xymatrix@C=1pc{ &&& Y_{\{0,1\}} \times_{Y_{\{1\}}} Y_{\{1,2\}} \times \dots \times_{Y_{\{n-1\}}} Y_{\{n-1,n\}}\ar[dlll] \ar[dll] \ar[dl] \ar[d] \ar[drr] \ar[drrr]&  &&\\
		Y_{\{0\}} & \ar[l] Y_{\{0,1\}} \ar[r] & Y_{\{1\}}  &\ar[l] Y_{\{1,2\}} & \dots &
		Y_{\{n-1,n\}} \ar[r]& Y_{\{n\}}. 
		}
	\end{equation}
	Hence it suffices to show that \eqref{eq:segalcone2} is a limit cone. This is equivalent
	to the statement that the ordinary fiber product of simplicial sets 
	\[
	Y_{\{0,1\}} \times_{Y_{\{1\}}} Y_{\{1,2\}} \times \dots \times_{Y_{\{n-1\}}} Y_{\{n-1,n\}}
	\]
	is a homotopy fiber product with respect to the Joyal model structure on $\sSet$ (cf.
	\cite[4.2.4.1]{lurie.htt}). To prove
	this, it suffices to show that, for every $0 \le i \le n-1$, the functors $Y_{\{i,i+1\}}
	\to Y_{\{i\}}$ and $Y_{\{i,i+1\}} \to Y_{\{i+1\}}$ are categorical fibrations which follows
	immediately from \cite[2.4.7.12]{lurie.htt}.
	Finally, it is clear that condition \ref{dinftybi.3} of Definition \ref{defi:dinftybi} is
	satisfied in virtue of condition \ref{l.edge}.

	To show that $q: \HSpan(\C) \to \ND$ admits relative pullbacks,
	consider the simplicial set
	\[
	K = \Delta^1 \coprod_{\{1\}} \Delta^1
	\]
	so that $K$-indexed limit diagrams are pullback diagrams. We will apply
	\cite[4.3.1.11]{lurie.htt} to show that $q$ admits relative
	pullbacks, i.e. $K$-indexed $q$-limits. We first show that, for every $n \ge 0$, the
	$\infty$-category $Y_{[n]}$ admits $K$-indexed limits. As above, let $\D \subset \Fun(\asd(\Delta^n),
	\C)$ denote the full subcategory spanned by those vertices satisfying condition
	\ref{l.vertex} from Definition \ref{defi:hspan}. As above, we consider the 
	adjunction of $\infty$-categories
	\[
	j^*: \Fun(\asd(\Delta^n),\C) \longleftrightarrow \Fun(\asd(\J^n), \C): j_*
	\]
	where the right Kan extension functor $j_*$ has essential image $\D$. By
	\cite[5.1.2.3]{lurie.htt}, we conclude that the $\infty$-category $\Fun(\asd(\J^n), \C)$ and
	hence the equivalent $\infty$-category $\D$ admits $K$-indexed limits. Further, $j_*$ is a
	right adjoint which, by \cite[5.2.3.5]{lurie.htt}, preserves limits. Thus, using
	\cite[5.1.2.3]{lurie.htt}(2), we deduce that
	\begin{enumerate}
		\item The $\infty$-category $\D$ admits $K$-indexed limits.
		\item A diagram $K^{\triangleleft} \to \D \subset \Fun(\asd(\Delta^n), \C)$ is a
			limit diagram if and only if, for every vertex of $\asd(\Delta^n)$, the induced
			diagram $K^{\triangleleft} \to \C$ is a limit diagram.
	\end{enumerate}
	Next, we show that the $\infty$-category $Y_{[n]}$ admits $K$-indexed limits and, further,
	the inclusion $i: Y_{[n]} \subset \D$ preserves $K$-indexed limits. Consider $f: K \to Y_{[n]}$ and
	let $K^{\triangleleft} \to \D$ be a limit diagram extending $i \circ f: K \to
	\D$. Then it is
	easy to verify that the limit diagram $K^{\triangleleft} \to \D$ factors through
	$Y_{[n]}$ and is a limit diagram in $Y_{[n]}$. This shows that the $\infty$-category
	$Y_{[n]}$ admits $K$-indexed limits. Further, these limits can be calculated pointwise in $\Fun(\asd(\Delta^n), \C)$.
	To apply \cite[4.3.1.11]{lurie.htt} it remains to verify that the functors associated to the
	Cartesian fibration $q: Y \to \ND$ preserve $K$-indexed limits in the fibers of $q$. But this follows directly from
	the fact established above that, for every $n \ge 0$, $K$-indexed limits in $Y_{[n]}$ can be
	computed pointwise.

	It remains to show that the Segal fibration $Y \to \ND$ is complete. To this end, we have to
	verify that the functor of Kan complexes
	\[
		Y_{[0]} \lra (Y_{[1]})^{\on{equiv}}_{\Kan}
	\] 
	associated to the edge $[1] \to [0]$ of $\ND$ via the Cartesian fibration $Y \to \ND$, is a
	weak equivalence. This map can be explicitly identified with the functor
	\[
		\Fun(\Delta^0, \C_{\Kan}) \lra \Fun(\asd(\Delta^1), \C_{\Kan})
	\]
	obtained by pullback along the constant map $\asd(\Delta^1) \to \Delta^0$. Since
	$\asd(\Delta^1)$ is weakly contractible, this latter map is a weak homotopy equivalence,
	implying our claim.
\end{proof}

Let $\C$ be an $\infty$-category with finite limits. The complete Segal fibration $\HSpan(\C) \to
\ND$ models an $\inftytwo$-category $\B$ which we call the {\em $\inftytwo$-category of horizontal
spans in $\C$}.  Let $\pt$ denote a final object of $\C$, then the $\inftyone$-category
$\Map_{\B}(\pt, \pt)$ carries a natural monoidal structure given by composition of $1$-morphisms. In
fact, the $\inftyone$-category $\Map_{\B}(\pt, \pt)$ is equivalent to $\C$ itself, and the monoidal
structure is the Cartesian monoidal structure on $\C$. This can be seen in the language of Segal
fibrations as follows. Consider the full simplicial subset $\C^{\times} \subset \HSpan(\C)$ spanned
by those vertices such that the corresponding functor $F: \asd(\Delta^n) \to \C$ satisfies the
following condition:
\begin{itemize}
	\item For $0 \le i \le n$, the vertex $F(\{i\})$ of $\C$ is a final object. 
\end{itemize}
With this notation, we have the following statement.

\begin{prop}\label{prop:Cartesianmonoidal} Let $\C$ be an $\infty$-category with finite limits. 
	The map $\C^{\times} \to \ND$, obtained by restricting the fibration $\HSpan(\C) \to \ND$,
	is a complete Segal fibration with contractible $[0]$-fiber. It exhibits the Cartesian
	monoidal structure on the $\infty$-category $\C$.
\end{prop}

\vfill\eject

\subsection{Bispans}
\label{subsec:bispans}

Let $\C$ be an $\infty$-category admitting pullbacks. We introduce the simplicial set
\[
	\bSpan(\C) := \Spanl_{\N(\Delta)}(\HSpan(\C))
\]
which, by Theorem \ref{thm:span-bi-mon} and Theorem \ref{thm:horispan}, comes equipped with a
complete Segal fibration
\[
	q: \bSpan(\C) \lra \ND.
\]
We refer to the $\inftytwo$-category $\B$ modelled by $q$ as the {$\inftytwo$-category of bispans in
$\C$.} We give an informal description of $\B$ allowing for direct comparison with the descriptions
of vertical and horizontal spans. 
\begin{itemize}
	\item The objects of $\bSpan(\C)$ are given by objects of $\C$.
	\item A $1$-morphisms between objects $x$ and $y$ of $\bSpan(\C)$ is given by a span diagram $x
		\leftarrow z \to y$ in $\C$. Composition of $1$-morphisms is given by forming
		pullbacks (hence we require the existence of limits).
	\item A $2$-morphism between $1$-morphisms $x \leftarrow z \to y$ and $x \leftarrow z' \to
		y$ is given by a diagram
		
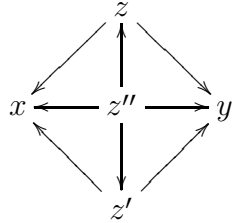
\[
		\xymatrix{
		&\ar[dl]\ar[dr] z & \\
		x & \ar[l] z'' \ar[r] \ar[u]\ar[d] & y \\
		&\ar[ul]\ar[ur] z' & }
		\]
		in $\C$.
	\item The higher morphisms are given by spans, in which both edges are equivalences, of
	  spans of spans of \dots in $\C$.
\end{itemize}

Assume $\C$ admits finite limits and consider the Segal fibration $\C^{\times} \to \ND$ from Proposition
\ref{prop:Cartesianmonoidal}. By Theorem \ref{thm:span-bi-mon}, the Segal fibration $\Spanl_{\ND}(\C^{\times}) \to \ND$
exhibits a monoidal structure on the $\infty$-category $\Span(\C)$ which we call the {\em pointwise
Cartesian monoidal structure on $\Span(\C)$.}

\vfill\eject

\section{2-Segal spaces as monads in bispans}
\label{section:higher-bicat}

We show how $2$-Segal spaces can be naturally interpreted in the context of the
$\inftytwo$-categorical theory of spans developed in \S \ref{sec:spans}. More precisely, we will
functorially associate to a unital $2$-Segal space $X$ a monad in the $\inftytwo$-category of
bispans in spaces, called {\em higher Hall monad of $X$}.

\subsection{The Higher Hall monad}

In this section, we construct a functor which assigns to a unital $2$-Segal space $X$ a {\em monad} in the $\inftytwo$-category of
bispans in the $\infty$-category $\inftyS$ of spaces. When considering $2$-Segal spaces with
contractible space of $0$-simplices, this construction can be simplified to obtain an {\em algebra
object} in the $\infty$-category $\Spanl(\inftyS)$ equipped with the pointwise Cartesian monoidal
structure.
In the context of Segal fibrations, monads and algebra objects can be defined as follows (cf. \cite{lurie.noncommutative}).

\begin{defi} Let $p: Y \to \NDop$ be a coSegal fibration, and let $\NDop \to \NDop$ be the coSegal fibration
	given by the identity map. A {\em monad in $Y$} is defined to be a right lax
	functor $s: \NDop \to Y$, i.e., a section
	\[
	\xymatrix{
		 Y \ar[r]_p & \ar@/_1pc/[l]_s \NDop
	}
	\]
	which maps convex edges in $\NDop$ to $p$-coCartesian edges in $Y$. We also say that $s$ defines
	a {\em monad in the $\inftytwo$-category modeled by $p$.}
	Let $Y \to \NDop$ be a Segal fibration with contractible $[0]$-fiber which, hence, exhibits a monoidal structure on the
	$\infty$-category $\C = Y_{[1]}$. In this situation, a monad in $Y$ is called an {\em
	algebra object in $\C$}.
	Dually, given a Segal fibration $Z \to \ND$, a left lax functor $\ND \to Z$ is called a {\em
	comonad in $Z$} or, if $Z_{[0]}$ is contractible, a {\em coalgebra object in $Z_{[1]}$}.
\end{defi}

Informally, a monad in an $\inftytwo$-category $\B$, modelled by a coSegal fibration $p: Y \to
\NDop$ corresponds, corresponds to the following data:
\begin{itemize}
	\item an object $x$ of $\B$, 
	\item a $1$-morphism $F: x \to x$, 
	\item a coherently associative collection of $2$-morphisms
		\[
			F^n = F \circ F \circ \dots \circ F \lra F
		\]
		where $n \ge 0$.
\end{itemize}
In terms of these data, we can describe the higher Hall monad in the
$\inftytwo$-category $\B$ of bispans in spaces, corresponding to a unital $2$-Segal space $X$, as follows:
\begin{itemize}
	\item the object of $\B$ is the space $X_0$,
	\item the $1$-morphism $F$ is given by the span 
		\[
		      \xymatrix@R=1pc{
			      X_0 &\ar[l]_{\partial_1} X_1 \ar[r]^{\partial_0} & X_0,
		      }
		\]
	\item for every $n \ge 2$, we consider the natural $2$-morphism in $\B$ given by the span
		\[
		      \xymatrix@R=1pc{
			      F^n \simeq X_1 \times_{X_0} X_1 \times_{X_0} \dots \times_{X_0} X_1 &\ar[l] X_n \ar[r] & X_1.
		      }
		\]
\end{itemize}
The $2$-Segal conditions satisfied by $X$ are responsible for the fact that this data is coherently
associative. This statement is made precise in Theorem \ref{thm:ax}. 

\begin{rem} The notion of a monad defined above is a lax variant of the classical concept of a monad which is
	typically studied in the context of the strict $2$-category $\Cat$ of categories: A
	classical monad in $\Cat$ corresponds to the data of
	\begin{itemize}
		\item a category $\C$,
		\item an endofunctor $F: \C \to \C$, 
		\item natural transformations $\mu: F \circ F \to F$ and $\eta: \id_{\C} \to F$,
	\end{itemize}
	such that the diagrams of natural transformations
	\[
		\xymatrix{
			F \circ F \circ F \ar[r]^-{F\mu}\ar[d]_-{\mu F} & F \circ F
			\ar[d]^-{\mu} \\
			F \circ F \ar[r]^-{\mu} & F 
		} \quad \quad \quad
		\xymatrix{
			F \ar[r]^-{F\eta}\ar[dr]_-{\id_F} & F \circ F\ar[d]^-{\mu} & F \ar[dl]^-{\id_F}
			\ar[l]_-{\eta F}\\
			& F & 
		}
	\]
	commute (cf. \cite{street-monads}). 
\end{rem}

\begin{rem} The structure of a multivalued category defined in \S \ref{subsec:mult-cat} can be regarded as a
$(3,2)$-categorical variant of the notion of a monad considered here.
\end{rem}

The following construction lies at the heart of what follows: 
\begin{defi}
We define a functor
\[
\Past: \Mo(\Delta) \times_\Delta \Deltac \to \sSet^{\op}
\]
by associating to an object $([m] \overset{f}{\to} [n],
([m], \{i,j\}))$ the simplicial set
\[
\Delta^{\{f(i),\dots,f(i+1)\}} \coprod_{\{f(i+1)\}} \Delta^{\{f(i+1),\dots,f(i+2)\}}
\coprod_{\{f(i+2)\}} \dots \coprod_{\{f(j-1)\}} \Delta^{\{f(j-1),\dots,f(j)\}} \subset
\Delta^n.
\]
\end{defi}


\begin{rem} Note that, using Proposition \ref{prop:asd-mo}, the nerve of the functor $\Past$ provides a
	map
	\[
	\N(\Past): \asd(\N(\Delta)) \times_{\N(\Delta)} \N(\Deltac) \to \N(\sSet)^{\op}.
	\]
\end{rem}

Let $\mC$ be a simplicial combinatorial model category $\mC$ in which every object is cofibrant. 
For a small category $I$, we equip the functor category $\Fun(I, \mC)$ with the injective
model structure. We denote by $\Fun(I,\mC)^{\circ} \subset \Fun(I,\mC)$ the full simplicial
subcategory of injectively fibrant objects. Recall the Yoneda extension functor
\[
\Upsilon_*: \Fun(\Dop, \mC) \lra \Fun( \sSet^{\op}, \mC) 
\]
from \S \ref{subsec:yonedaext}, defined as the right adjoint of the pullback functor along the
Yoneda embedding $\Dop \to \sSet^{\op}$. The functor $\Upsilon_*$ is a right Quillen functor with
respect to the injective model structures on both functor categories, in particular it preserves injectively
fibrant objects. We obtain a functor of simplicial categories by forming the composite
\[
\Past_{\bullet}: \Mo(\Delta) \times_\Delta \Deltac \times \Fun(\Dop, \mC)^{\circ}
\overset{(\Past,\Upsilon_*)}{\lra} \sSet^{\op} \times
\Fun(\sSet^{\op}, \mC)^{\circ} \overset{\on{ev}}{\lra} \mC^{\circ}.
\]
In particular, for every injectively fibrant object $X$ of $\Fun(\Dop,\mC)$, we obtain, after passing to
simplicial nerves, a functor
\[
\N(\Past_{X}): \asd(\ND) \times_{\ND} \N(\Deltac) \to \C
\]
where $\C = \N(\mC^{\circ})$ denotes the $\infty$-category given as the simplicial nerve of $\mC^{\circ}$.
Via the defining adjunctions of horizontal and vertical spans from \S \ref{sec:spans}, 
the functor $\N(\Past_{X})$ corresponds to a section 
\[
\xymatrix@=1pc{
\ND \ar[dr]^-{\id} \ar[rr]^-{A_X} && \Span_{\ND}(\HSpanp(\C)) \ar[dl]\\
& \ND. &
}
\]

\begin{thm} \label{thm:ax} Let $X$ be an injectively fibrant object of $\Fun(\Dop, \mC)$.
	Then the following are equivalent.
	\begin{enumerate}
		\item The object $X$ is a unital $2$-Segal object.
		\item The section $A_X$ factors through $\bSpan(\C) \subset \Span_{\ND}(\HSpanp(\C))$.
	\end{enumerate}
\end{thm}
\begin{proof}
	Assume $X$ is a unital $2$-Segal object.
	We have to verify that $A_X$ maps every $k$-simplex of $\ND$ to a Segal simplex of
	$\Span_{\ND}(\HSpan(\C))$. Let $p: \HSpan(\C) \to \ND$ be the $\inftytwo$-category of
	horizontal spans in $\C$. We first show that, for every $k$-simplex $\sigma: \Delta^k \to
	\ND$, the corresponding composite
	\[
	f_\sigma: \asd(\Delta^k) \overset{\asd(\sigma)}{\lra} \asd(\ND) \overset{A_X}{\lra} \HSpanp(\C)
	\]
	factors through $\HSpan(\C) \subset \HSpanp(\C)$. To this end we have to show that, for
	$k=0$ and $k=1$, the conditions \ref{l.vertex} and \ref{l.edge} of Definition
	\ref{defi:hspan} are satisfied. The vertex $\{[n]\}$ of $\ND$ gets associated by $f_\sigma$
	to the diagram in $\HSpanp(\C)_{[n]} \subset \Fun(\asd(\Delta^n), \C)$ which is given by the nerve of the functor
	\[
	\Mo([n]) \lra \mC,\; \{i,j\} \mapsto \Upsilon_*X(\J^{\{i,\dots,j\}})
	\]
	Using \cite[4.2.4.1]{lurie.htt}, the limit condition of \ref{l.vertex} can now easily be seen to correspond to the fact
	that, since $X$ is injectively fibrant, the evaluated right Yoneda extension
	$\Upsilon_*X(\J^{\{i,\dots,j\}})$ can be expressed as a homotopy limit indexed by the
	category of simplices of the simplicial set $\J^{\{i,\dots,j\}}$ (add reference).

	For every edge $[n] \to [m]$ of $\N(\Delta)$, the corresponding span diagram in $\HSpanp(\C)$
	evaluated at $\{i\} \subset [n]$ corresponds to a diagram of the form $X_0 \overset{\id}{\leftarrow} X_0
	\overset{\id}{\rightarrow} X_0$. Since both edges in this diagram are trivially equivalences
	in $\C$, we deduce that condition \ref{l.edge} is satisfied. We conclude, that,
	irrespectively of the $2$-Segal condition, the section $A_X$ factors through
	$\Span_{\ND}(\HSpan(\C))$.

	We show next that, for every $k$-simplex $\sigma: \Delta^k \to \ND$, the corresponding
	Segal cone
	\[
	g_\sigma: S(\Delta^k) \lra \asd(\Delta^k) \overset{\asd(\sigma)}{\lra} \asd(\ND)
	\overset{A_X}{\lra} \HSpan(\C)
	\]
	is a $p$-limit diagram. This will imply that $A_X$ factors through $\bSpan(\C)$. First note
	that the simplex $\sigma$ corresponds to a composable chain of maps
	\begin{equation}\label{eq:chainofm}
		[n_0] \overset{f_1}{\lra} [n_1] \overset{f_2}{\lra} \dots \overset{f_k}{\lra} [n_k].
	\end{equation}
	We apply Lemma \ref{lem:comparelimit} to the map $f_\sigma$ to obtain a homotopy $h: \Delta^1 \times \asd(\Delta^k)
	\to \HSpan(\C)$ such that $h|\{1\} \times \asd(\Delta^k) = f_\sigma$ and the diagram $f'_{\sigma} := h|\{0\} \times
	\asd(\Delta^k)$ lies in the fiber $\HSpan(\C)_{[n_0]}$. By Lemma \ref{lem:comparelimit}, the
	diagram $g_\sigma$ is a $p$-limit diagram if and only if the composite
	\[
	g'_\sigma : S(\Delta^k) \lra \asd(\Delta^k) \overset{f'_\sigma}{\lra} \HSpan(\C)
	\]
	is a $p$-limit diagram. In the proof of Theorem \ref{thm:horispan}, we have seen that 
	all functors associated with the Cartesian fibration $p: \HSpan(\C) \to \ND$ preserve
	$S(\Delta^k)$-indexed limit diagrams in the fibers of $p$. By \cite[4.3.1.11]{lurie.htt} it
	hence suffices to show that the diagram $g'_\sigma$ induces a limit diagram in the fiber $\HSpan(\C)_{[n_0]}$.
	
	In the proof of Theorem \ref{thm:horispan}, we have further seen that a
	diagram $S(\Delta^k) \to \HSpan(\C)_{[n_0]}$ is a limit diagram if and only if the
	composite diagram 
	\[
	g''_\sigma: S(\Delta^k) \to \HSpan(\C)_{[n_0]} \subset \Fun(\asd(\Delta^{n_0}), \C)
	\]
	is a limit diagram. By \cite[5.1.2.3]{lurie.htt}, a diagram in $\Fun(\asd(\Delta^{n_0}),
	\C)$ is a limit diagram if and only if, for every vertex $\{i,j\}$ of $\asd(\Delta^{n_0})$,
	the corresponding diagram in $\C$ is a limit diagram. Further, every vertex in
	$\HSpan(\C)_{[n_0]} \subset \Fun(\asd(\Delta^{n_0}), \C)$ is a right Kan extension of its
	restriction along $j: \asd(\J^{n_0}) \to \asd(\Delta^{n_0})$. The right Kan extension
	functor $j_*$ is a right adjoint which, by \cite[5.2.3.5]{lurie.htt}, preserves limits.
	Hence it suffices to check that the evaluation of the diagram $g''_\sigma$ at every vertex of
	$\asd(\J^{n_0}) \subset \asd(\Delta^{n_0})$ is a limit diagram in $\C$. This is easily
	verified for the vertices $\{i\}$ of $\J^{n_0}$ where $0 \le i \le n_0$. It remains to verify the condition for
	vertices of the form $\{i,i+1\}$ where $0 \le i < n_0$. To this end, we introduce the chain of morphisms
	\begin{equation}\label{eq:modchainofm}
		[n'_0] \overset{f'_1}{\lra} [n'_1] \overset{f'_2}{\lra} \dots \overset{f'_k}{\lra} [n'_k]
	\end{equation}
	which is obtained by restricting \eqref{eq:chainofm} where $[n'_0] \cong \{i,i+1\}$ and
	$[n'_j] \cong \{ f_j \circ f_{j-1} \circ \dots \circ f_1 (i), \dots, f_j \circ f_{j-1} \circ
	\dots \circ f_1 (i+1) \}$. Note that, for every $1 \le j \le k$, we have $f'_j(0) = 0$ and
	$f'_j(n'_{j-1}) = n'_j$.
	Unraveling the definitions of the functor $\Past$ and the homotopy $h$, it follows that  
	the evaluation of the diagram $g''_\sigma$ at the vertex $\{i,i+1\}$ is equivalent to 
	the simplicial nerve of the diagram \eqref{eq:2segdiag} in Lemma \ref{lem:2segdiag} below.
	By Lemma \ref{lem:2segdiag} this diagram is a homotopy limit diagram which, using
	\cite[4.2.4.1]{lurie.htt}, concludes our argument for the implication (1) $\Rightarrow$ (2).

	Assume $A_X$ factors through $\bSpan(\C)$. To show that $X$ is a $2$-Segal object it
	suffices to show that, for every $n \ge 2$ and every polygonal subdivision 
	\[
	\T = \{ \{i,\dots,j\},\{0,\dots,i,j,\dots,n\} \}
	\]
	the diagram
	\begin{equation}\label{eq:2seghompull}
		\xymatrix{ X_n \ar[r]\ar[d] & X_{\{i,\dots,j\}} \ar[d] \\
		X_{\{0,\dots,i,j,\dots,n\}} \ar[r] & X_{\{i,j\}}}
	\end{equation}
	is a homotopy pullback square. 
	Consider the $2$-simplex $\sigma$ in $\ND$ given by the chain 
	\[
		\{0,n\} \overset{f_1}{\lra} \{0,\dots,i,j,\dots,n\} \overset{f_2}{\lra} [n].
	\]
	The simplex $A_X(\sigma)$ lies by assumption in $\bSpan(X)$ and is hence a Segal simplex.
	By the argumentation in the proof of the implication (1) $\Rightarrow$ (2) above, the evaluation of the
	corresponding diagram $g''_\sigma$ at the interval $\{0,1\}$ of $[1] \cong \{0,n\}$ is
	equivalent to the simplicial nerve of the diagram
	\begin{equation}\label{eq:transportdiag}
		\xymatrix@!C=5pc{ & & X_n \ar[dr] \ar[dl] & & \\  
		& X_{\{0,\dots,i,j,\dots,n\}}\ar[dr]\ar[dl] & &\ar[dr]\ar[dl] X_{\J^{\{0,\dots,i\}}} \times_{X_{\{i\}}} X_{\{i,\dots,j\}} \times_{X_{\{j\}}} X_{\J^{\{j,\dots,n\}}} &  \\
		X_{\{0,n\}} & & X_{\J^{\{0,\dots,i\}}} \times_{X_{\{i\}}}
		X_{\{i,j\}}\times_{X_{\{j\}}} X_{\J^{\{j,\dots,n\}}} & & X_{\J^{\{0,\dots,n\}}}.}	
	\end{equation}
	Hence, by \cite[4.2.4.1]{lurie.htt}, the diagram \eqref{eq:transportdiag} is a homotopy
	limit diagram which is easily seen to be equivalent to the assertion that the square
	\eqref{eq:2seghompull} is a homotopy pullback square.

	It remains to show that $X$ is unital. Consider the $2$-simplex $\sigma$ in $\ND$ given by the chain 
	\[
	\{0,n\} \lra [n] \overset{\delta_i}{\lra} [n-1]
	\]
	where $n \ge 0$ and $\delta_i$ denotes the $i$th degeneracy map. 
	By an analogous argumentation we conclude that $A_X(\sigma)$ is a Segal simplex if and only
	if the square
	\[
	\xymatrix{ X_{n-1} \ar[r]\ar[d] & X_{\{i\}} \ar[d] \\
		X_{n} \ar[r] & X_{\{i,i+1\}}}
	\]
	is a homotopy pullback square, showing that $X$ is unital.
\end{proof}

\begin{lem}\label{lem:2segdiag} Let $\mC$ be a combinatorial simplicial model category and $X$ an
	injectively fibrant $2$-Segal object in $\Fun(\Dop, \mC)$. 
	Consider a $k$-simplex $\sigma$ in $\ND$ which corresponds to a chain of morphisms
	\[
		[1] \overset{f_1}{\lra} [n_1] \overset{f_2}{\lra} \dots \overset{f_k}{\lra} [n_k].
	\]
	Assume that, for every $1 \le i \le k$, we have $f_{i}(0) = 0$ and $f_{i}(n_{i-1}) = n_i$.
	Consider the collections of subsets 
	\begin{align*}
		\E_i & = \{\{0,1\}, \{1,2\}, \dots, \{n_{i}-1,n_{i}\} \} \subset 2^{[n_i]}\\
		\P_i & = \{\{f_i(0),\dots,f_i(1)\}, \{f_i(1),\dots,f_i(2)\}, \dots,
		\{f_i(n_{i-1}-1),\dots,f_i(n_{i-1})\} \} \subset 2^{[n_i]}
	\end{align*}
	where $1 \le i \le k$ and further $\E_0 = \{\{0,1\}\}$.
	Then the diagram in $\mC$
	\begin{equation}\label{eq:2segdiag}
		\xymatrix{ & & &X_{n_k} \ar[d] \ar[drrrr]\ar[dll]&  & & \\  
		& X_{\P_1}\ar[dr]\ar[dl] & &\ar[dl]\ar[dr] X_{\P_2} && \dots& &\ar[dr]\ar[dl] X_{\P_k} &  \\
		X_{\E_0} & & X_{\E_1} &&X_{\E_2}& \dots &X_{\E_{k-1}}& & X_{\E_k}
		}	
	\end{equation}
	is a homotopy limit diagram with limit vertex $X_{n_k}$.
\end{lem}
\begin{proof}
	Using induction on $k$, it is clear that it suffices to prove the stament for $k = 2$. We
	set $[m] = [n_1]$, $[n] = [n_2]$ and $f = f_2$. If $f_1$ is constant then we have $[m] = [n] = [0]$ and the statement is trivial. Thus we
	may assume that $f_1$ is injective. In this case, the statement is easily seen to be
	equivalent to the assertion that the square
	\begin{equation}\label{eq:total2seg}
		\xymatrix{ 
			X_n \ar[r] \ar[d] & X_{\{f(0),\dots,f(1)\}} \times_{X_{\{f(1)\}}} X_{\{f(1),\dots, f(2)\}}
			\times \dots \times X_{ \{ f(m-1),\dots,f(m)\} } \ar[d]\\
			X_m \ar[r] & X_{\{0,1\}} \times_{X_{\{1\}}} X_{\{1, 2\}} \times_{X_{\{2\}}} \dots \times_{X_{ \{m-1\} }} X_{ \{ m-1,m\} } 
		}
	\end{equation}
	is a homotopy pullback square. We conclude the argument as in the proof of Proposition
	\ref{prop.key}.
\end{proof}

\begin{cor} Let $X$ be an injectively fibrant $2$-Segal object in $\Fun(\ND, \mC)$. Then the section
	$A_X$ defines a comonad in the Segal fibration $\bSpan(\C)$.
\end{cor}
\begin{proof}
	According to Theorem \ref{thm:ax}, it remains to show that $A_X$ maps convex edges in $\ND$
	to Cartesian edges in $\bSpan(\C)$. This becomes apparent after unwinding the definition of
	$\Past$.
\end{proof}
	
Further, the comonad $A_X$ associated to a $2$-Segal space $X$ depends functorially on $X$. More
precisely, the simplicial nerve $\N(\Past_{\bullet})$ corresponds via adjunction to a functor
\begin{equation}\label{eq:A}
	A: \N(\Fun(\Dop, \sSet)^{\circ}_{2-\on{Seg}}) \lra \Fun^{\on{lax}}_{\N(\Delta)}(\N(\Delta),
	\bSpan(\inftyS))
\end{equation}
of $\infty$-categories. The left-hand side is by definition the $\infty$-category of $2$-Segal
spaces.

\begin{rem} \label{rem:comonad-monad} Note that, given an $\infty$-category $\C$ with limits, the
	$\infty$-category $\Spanl(\C)$ can be identified with its opposite category. 
	This implies, that we can equivalently describe the $\inftytwo$-category of bispans as a {\em
	coCartesian} fibration over $\N(\Delta)^{\op}$ by passing to the opposite of the
	functor $\bSpan(\C) \to \N(\Delta)$. The comonad $A_X$ in $\bSpan(\C)$ defines a section
	\[
	\xymatrix@=1pc{ \N(\Delta)^{\op} \ar[rr]^{(A_X)^{\op}} \ar[dr]_{\id} & & \ar[dl]
	\bSpan(\inftyS)^{\op}\\
	& \N(\Delta)^{\op} & }
	\]
	which corresponds to a {\em right} lax functor of $\inftytwo$-categories. Such a functor
	corresponds to a {\em monad} in the $\inftytwo$-category of bispans in $\inftyS$.
\end{rem}

\begin{rem} \label{rem:functoriality} By Remark \ref{rem:comonad-monad} we can associate to a
	$2$-Segal space $X$ both a monad and a comonad in the $\inftytwo$-category of bispans in
	$\inftyS$. However, note that the functorial dependence on $X$ given by the functor $A$ defined
	in \eqref{eq:A} changes when passing from $A_X$ to $A_X^{\op}$. 
\end{rem}

\begin{defi} Given an injectively fibrant $2$-Segal space $X$, we call $A_X$ the
	{\em higher Hall comonad associated to $X$}. Dually, we call $(A_X)^{\op}$ the {\em higher Hall monad
	associated to $X$}.
\end{defi}

Let $X$ be a $2$-Segal space and assume that $X_0 \simeq \pt$. In this case, the
above construction simplifies as follows. 

\begin{thm} Let $X$ be a $2$-Segal space with contractible space of $0$-simplices. Then the functor $A_X$ factors
	through $\Spanl_{\N(\Delta)}(\C^{\times})$ defining a coalgebra object in the
	$\infty$-category $\Spanl(\inftyS)$ equipped with the pointwise Cartesian monoidal
	structure.
\end{thm}

\begin{rem} As in \ref{rem:comonad-monad}, the functor $(A_X)^{\op}$ defines an {\em algebra} object
	in the $\infty$-category $\Spanl(\inftyS)$.
\end{rem}

\vfill\eject

\vfill\eject

\appendix
\numberwithin{equation}{section}

\section{Bicategories}
\label{app.bicategories}
The goal of this appendix is to recall the classical comcepts of bicategories and 2-categories
 and to explain their connection to the concept known as quasi-categories or $\infty$-categories
 \cite{lurie.htt}. 
 
\begin{ex}[(Classical $(2,1)$-categories)] \label{ex:classical-(2,1)}
A {\em(strict) 2-category} $\C$ can be defined as a category enriched in $\Cat$, so for any
$a,b\in\Ob(\Cc)$ we have a small category $\Homc_\Cc(a,b)$ and the composition functors
\[
\otimes: \Homc_\Cc(b,c) \times \Homc_\Cc (a,b)\lra \Homc_\Cc (a,c),
\]
which are strictly associative.  Further, for any $a\in\Ob(\Cc)$ there is an object
$\1_a\in\Homc_\Cc(a,a)$ which is a unit with respect to $\otimes$.  Objects of $\Homc_\Cc(a,b)$ are
called {\em 1-morphisms} in $\Cc$ from $a$ to $b$, and we write $E: a\to b$.  A morphism $u$ in
$\Hom_\Cc(x,y)$ from $E$ to $F$ is called a 2-morphism in $\Cc$, and we write $u: E\Rightarrow F$.
For more details, including those on the geometric composition (pasting) of 2-morphisms, see
\cite{kelly-street, maclane}.

More generally, the concept of a {\em bicategory} (or a {\em weak 2-category}) $\Cc$ is obtained by
relaxing the condition of strict associativity of $\otimes$ and of the unit property of the $\1_a$
by replacing them with canonical associativity 2-isomorphisms
\[
\alpha_{E,F,G}: (E\otimes F)\otimes G\Rightarrow E\otimes (F\otimes G), \quad
a\buildrel G\over\to b\buildrel F\over\to c\buildrel E\over\to d,
\]
and the unit 2-isomophisms
\[
\lambda_E: u\otimes\1_a\Rightarrow E, \rho_E: \1_b\otimes E\Rightarrow E, \quad
u: x\to y,
\]
satisfying the coherence conditions, which include the Mac Lane pentagon for the $\alpha_{E,F,G}$,
see \cite{benabou}. 

Even more generally, we will use the term {\em semi-bicategory} for a  structure similar to a
bicategory but where no unit 1-morphisms are assumed to exist. 
  
To any small bicategory $\Cc$ one can associate its {\em nerve} $N\C$, see \cite{street, bullejos}
for the strict case and \cite{duskin} for the general (weak) case.  This is a simplicial set with
$N_n\C$ being he set of the data consisting of:
\begin{itemize}
\item[(0)] Objects $a_0, ..., a_n$; 

\item[(1)] 1-morphisms $E_{ij}: a_i\to a_j$, $i\leq j$;

\item[(2)] 2-morphisms $u_{ijk}: E_{ik}\Rightarrow E_{jk}\otimes E_{ij}$, $i\leq j\leq k$, satisfying the
compatibility conditions:

\item[(3)] For each $0\leq i_0\leq i_1\leq i_2\leq i_3\leq n$ the tetrahedron formed by 
the $a_{i_\nu}$, $E_{i_\nu, i_{\nu'}}$ and $u_{i_{\nu}, i_{\nu'}, i_{\nu''}}$,
is {\em 2-commutative}. This means that after we paste the 
two halves of its boundary, we get two 2-morphisms
  \[
  E_{i_0i_3}\Rightarrow
 (E_{i_2 i_3}\otimes E_{i_1 i_2})\otimes E_{i_0i_1} , \quad 
 E_{i_0i_3}
\Rightarrow 
  E_{i_2 i_3}\otimes (E_{i_1 i_2}\otimes E_{i_0i_1})
   \]
   of which the second one is the composition of the first one with the associativity
   isomorphism $\alpha_{E_{i_2, i_3}, E_{i_1, i_2}, E_{i_0, i_1}}$.
\end{itemize}
  
If $\Cc$ is a semi-bicategory, then the above construction defines a semi-simplicial set $N\Cc$,
still called the nerve of $\Cc$.

\begin{defi}
A {\em weak} (resp. {\em strict}) {\em (2,1)-category} is a bicategory (resp, a strict 2-category)
$\Cc$ such that each category $\Homc_\CC(x,y)$ is a groupoid, i.e., all the 2-morphisms in $\Cc$
are invertible. 
\end{defi}
 
The following is then a straightforward application of the formalism of pasting in bicategories.

\begin{prop}\label{prop:(2,1)-segal}
If $\Cc$ is a weak (2,1)-category, then $\N\C$ is a quasi-category. \qed
\end{prop}

\begin{warn}
In using the term ``$\infty$-categories" for quasi-categories it is important to keep in mind that
a 2-category in the classical sense (even a strict one) does not, 
  in general, 
  give rise
to a $\infty$-category (unless its 2-morphisms are invertible). 
\end{warn}
 
Combined with the Joyal-Tierney equivalence \eqref{eq.totalization}, the proposition
implies that $X=\tau^!N\C$ is a $1$-Segal space whenever $\C$ is a (2,1)-category. This $1$-Segal
space can be more directly described as follows:  $X_n=B(\C_n)$, where $\C_n$ is the category
(groupoid) whose objects are chains of composable 1-morphisms $x_0\buildrel u_1\over\lra \cdots
\buildrel u_n\over\lra x_n$, and morphisms are 2-commutative ladders, i.e., systems of 1- and
2-morphisms as depicted:
\[
\xymatrix{
\ar @{} [rd] |{\Swarrow}
x_0\ar[r]^{u_1}\ar[d] & 
\ar @{} [rd] |{\Swarrow}
x_1\ar[r]^{u_2}\ar[d]&
\cdots\ar[r]^{u_{n-1}}&
\ar @{} [rd] |{\Swarrow}
x_{n-1}\ar[r]^{u_n}\ar[d]&
x_n\ar[d]
\cr
y_0\ar[r]^{v_1}&y_1\ar[r]^{v_2}&\cdots\ar[r]^{v_{n-1}}&y_{n-1}\ar[r]^{v_n}&y_n
}
\]
This reduces to Example \ref {ex.1segclass}(b) when $\C$ is a usual category considered as a
2-category with all 2-morphisms being identities.
\end{ex}

\begin{ex}[(Monoidal categories)] \label{ex:1-segal-monoidal}
A bicategory $\C$ with one object $\pt$ is the same as a monoidal category with unit object $(\Ac,
\otimes, \1)$: objects of $\Ac$ correspond to 1-morphisms in $\C$ (from $\pt$ to $\pt$), the
monoidal structure $\otimes$ in $\Ac$ gives the composition of 1-morphisms, and morphisms in $\Ac$
give 2-morphisms in $\C$.  A semi-bicategory with one object is the same as a monoidal category
$(\Ac, \otimes)$, but possibly without a unit object. 

Thus   a weak (2,1) category $\Cc$  with one object is the same as a monoidal category $(\Ac,
\otimes, \1)$  which is a groupoid. 

If $\Cc$ is a strict (2,1)-category (i.e., the monoidal structure in $\Ac$ is strictly associative),
the $1$-Segal space $X=\tau^! N\C$.  can be desrcribed in terms of the monoidal structure, similarly
to the construction of the classifying space of a group.

More precisely, for $n\geq 0$ let $\Bc ar_n(\Ac)$ be the category whose objects are sequences $(A_1,
..., A_n)$ of objects of $\Ac$ and morphisms are sequences of isomorphisms. For $n=0$ we put $\Bc
ar_0(\Ac)=pt$ to be the punctual category. For $i=0,..., n$ we define the face functors
\[
	\partial_i(A_1, ..., A_n) = 
	\begin{cases} 
		(A_2, ..., A_n), \quad\text{if}\quad i=0;\\
		(A_1, ..., A_i\otimes A_{i+1}, ..., A_n), \quad\text{if}\quad i=1, ..., n-1;\\
		(A_1, ..., A_{n-1}), \quad\text{if}\quad i=n.
	\end{cases}
\]
and define the degeneration functors in the standard way by inserting the unit object $\1$. This
makes $\Bc ar_\bullet(\Ac)$ into a simplicial groupoid. The simplicial space $X$ formed by the
realizations $X_n = B \Bc ar_n(\Ac)$ is the $1$-Segal space corresponding to $(\Ac, \otimes, \1)$ as
above. 
\end{ex}

\vfill\eject

\bibliographystyle{halpha} 
\bibliography{refs}

\end{document}